\documentclass{article}
\usepackage[utf8]{inputenc}
\usepackage{amssymb,amsmath, amsthm, mathabx}
\usepackage{amsmath,amsfonts,amssymb,amsthm,epsfig,epstopdf,titling,url,array}
\usepackage{tikz}
\usetikzlibrary{shapes,snakes}
\usepackage{color, enumerate}
\usepackage{comment, authblk}
\usepackage{hyperref}
\usepackage{pgfplots}
\usepackage{bm}
\usepackage{stmaryrd}
\usepackage{graphicx}
\usepackage{caption}
\usepackage{subcaption}
\usepackage{a4wide}
\usepackage{titlesec}
\usepackage[a4paper, total={6in, 8in}]{geometry}

\titleformat*{\subsection}{\large\bfseries}
\numberwithin{equation}{section}

\pgfplotsset{compat=newest}
\pgfplotsset{plot coordinates/math parser=false}
\newlength\figureheight
\newlength\figurewidth

\DeclareMathOperator{\bA}{\mathbf{A}}
\DeclareMathOperator{\bC}{\mathbf{C}}
\DeclareMathOperator{\bD}{\mathbf{D}}
\DeclareMathOperator{\bS}{\mathbf{S}}
\DeclareMathOperator{\bv}{\mathbf{v}}

\DeclareMathOperator{\bw}{\mathbf{w}}

\DeclareMathOperator{\bbC}{\mathbb{C}}
\DeclareMathOperator{\bbE}{\mathbb{E}}
\DeclareMathOperator{\bbN}{\mathbb{N}}
\DeclareMathOperator{\bbP}{\mathbb{P}}

\DeclareMathOperator{\sI}{\mathcal{I}}
\DeclareMathOperator{\sG}{\mathcal{G}}
\DeclareMathOperator{\sW}{\mathcal{W}}

\DeclareMathOperator{\one}{\mathbf{1}}

\newcommand*{\conj}[1]{\overline{#1}}
\newcommand*{\wt}[1]{\widetilde{#1}}
\newtheorem{thm}{Theorem}[section]
\newtheorem{exam}[thm]{Example}
\newtheorem{prop}[thm]{Proposition}
\newtheorem{lem}[thm]{Lemma}

\newtheorem{assum}[thm]{Assumption}
\newtheorem{defn}[thm]{Definition}
\newtheorem{claim}[thm]{Claim}

\renewcommand{\Im}{{\rm{Im}}}
\renewcommand{\Re}{{\rm{Re}}}

\allowdisplaybreaks

\title{Local circular law for the product of a deterministic matrix with a random matrix}
\author[1]{Haokai Xi \thanks{E-mail: haokai@math.wisc.edu.}}
\author[1]{Fan Yang  \thanks{E-mail: fyang75@math.wisc.edu.}}
\author[1]{Jun Yin \thanks{E-mail: jyin@math.wisc.edu. Partially supported by NSF Career Grant DMS-1552192 and Sloan fellowship.}}
\affil[1]{Department of Mathematics, University of Wisconsin-Madison}

\begin{document}
\maketitle

\begin{abstract}
It is well known that the spectral measure of eigenvalues of a rescaled square non-Hermitian random matrix with independent entries satisfies the circular law. We consider the product $TX$, where $T$ is a deterministic $N\times M$ matrix and $X$ is a random $M\times N$ matrix with independent entries having zero mean and variance $(N\wedge M)^{-1}$. We prove a general local circular law for the empirical spectral distribution (ESD) of $TX$ at any point $z$ away from the unit circle under the assumptions that $N\sim M$, and the matrix entries $X_{ij}$ have sufficiently high moments. More precisely, if $z$ satisfies $||z|-1|\ge \tau$ for arbitrarily small $\tau>0$, the ESD of $TX$ converges to $\tilde \chi_{\mathbb D}(z) dA(z)$, where $\tilde \chi_{\mathbb D}$ is a rotation-invariant function determined by the singular values of $T$ and $dA$ denotes the Lebesgue measure on $\mathbb C$. The local circular law is valid around $z$ up to scale $(N\wedge M)^{-1/4+\epsilon}$ for any $\epsilon>0$. Moreover, if $|z|>1$ or the matrix entries of $X$ have vanishing third moments, the local circular law is valid around $z$ up to scale $(N\wedge M)^{-1/2+\epsilon}$ for any $\epsilon>0$. 
\end{abstract}

\begin{section}{Introduction}

\noindent{\bf Circular law for non-Hermitian random matrices.} The study of the eigenvalue spectral of non-Hermitian random matrices goes back to the celebrated paper \cite{Ginibre} by Ginibre, where he calculated the joint probability density for the eigenvalues of non-Hermitian random matrix with independent complex Gaussian entries. The joint density distribution is integrable with an explicit kernel (see \cite{Ginibre,Mehta}), which allowed him to derive the circular law for the eigenvalues. For the Gaussian random matrix with real entries, the joint distribution of the eigenvalues is more complicated but still integrable, which leads to a proof of the circular law as well \cite{Borodin2009,Edelman1997,Forrester2007,Sinclair2007}.

For the random matrix with non-Gaussian entries, there is no explicit formula for the joint distribution of the eigenvalues. However, in many cases the eigenvalue spectrum of the non-Gaussian random matrices behaves similarly to the Gaussian case as $N\to \infty$, known as the universality phenomena. A key step in this direction is made by Girko in \cite{Girko}, where he partially proved the circular law for non-Hermitian matrices with independent entries. The crucial insight of the paper is the {\it Hermitization technique}, which allowed Girko to translate the convergence of complex empirical measures of a non-Hermitian matrix into the convergence of logarithmic transforms for a family of Hermitian matrices, or, to be more precise,
\begin{equation}\label{intro_Herm_tech}
\text{Tr}\log[(X-z)^\dag(X-z)] = \log \left[\det((X-z)^\dag(X-z)) \right],
\end{equation}
 with $X$ being the random matrix and $z\in \mathbb C$. Due to the singularity of the log function at $0$, the small eigenvalues of $(X-z)^\dag(X-z)$ play a special role. The estimate on the smallest singular value of $X-z$ was not obtained in \cite{Girko}, but the gap was remedied later in a series of paper. Bai \cite{Bai1997,Bai2006} analyzed the ESD of $(X-z)^\dag(X-z)$ through its Stieltjes transform and handled the logarithmic singularity by assuming bounded density and bounded high moments for the entries of $X$. Lower bounds on the smallest singular values were given by Rudelson and Vershynin \cite{Rud_Annal,RudVersh_square}, and subsequently by Tao and Vu \cite{TaoVu_circular}, Pan and Zhou \cite{PanZhou_circular} and G\H otze and Tikhomirov \cite{gotze2010} under weakened moments and smoothness assumptions. The final result was presented in \cite{tao2010}, where the circular law is proved under the optimal $L^2$ assumption. These papers studied the circular law in the global regime, i.e. the convergence of ESD on subsets containing $\eta N$ eigenvalues for some small constant $\eta >0$. Later in a series of papers \cite{local_circular,local_circularII,local_circularIII}, Bourgade, Yau and Yin proved the {\it local} version of the circular law up to the optimal scale $N^{-1/2+\epsilon}$ under the assumption that the distributions of the matrix entries satisfy a uniform sub-exponential decay condition. In \cite{TaoVu_local}, the local universality was proved by Tao and Vu under the assumption of first four moments matching the moments of a Gaussian random variable.

\begin{figure}[htb]
\centering
\includegraphics[width=\columnwidth]{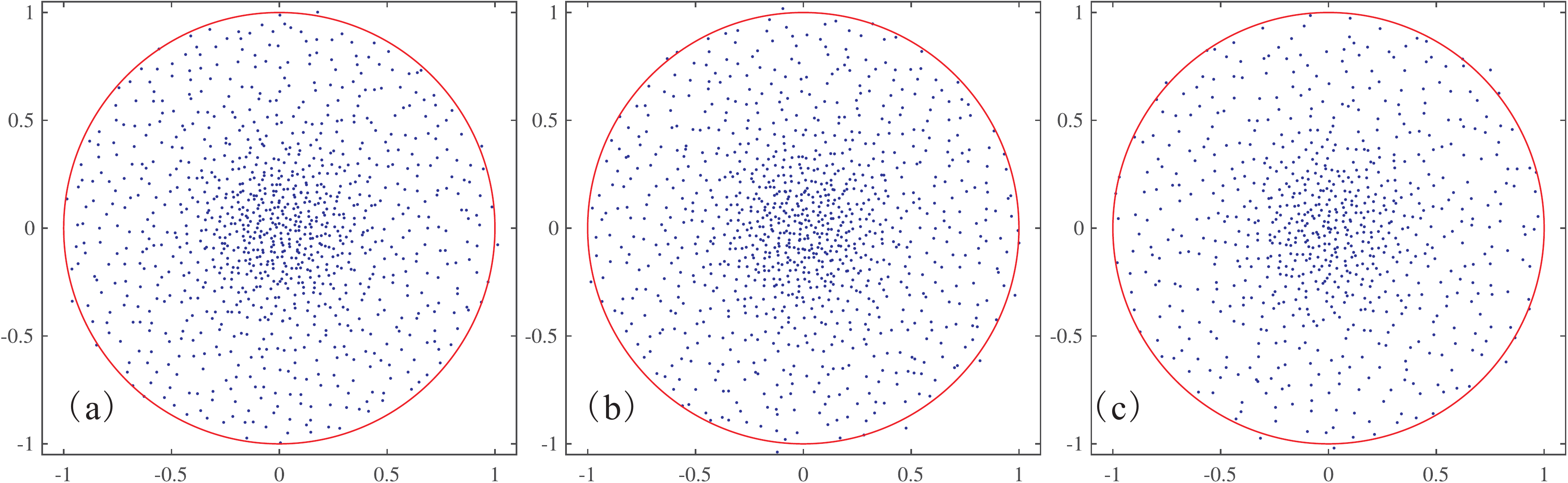}
\caption{The eigenvalue distribution of the product $TX$ of a deterministic $N\times M$ matrix $T$ with a Gaussian random $M\times N$ matrix $X$. The entries of $X$ have zero mean and variance $(N\wedge M)^{-1}$, and $TT^\dag$ has $0.5(N\wedge M)$ eigenvalues as $2/17$ and $0.5(N\wedge M)$ eigenvalues as $32/17$. (a) $N=M=1000$. (b) $N=1000$, $M=2000$. (c) $N=1500$, $M=750$.}
\label{fig_circular}
\end{figure}

In this paper, we study the ESD of the product of a deterministic $N\times M$ matrix $T$ with a random $M\times N$ matrix $X$, where we assume $N\sim M$. In Figure \ref{fig_circular}, we plot the eigenvalue distribution of $TX$ when $T$ have two distinct singular values (except the trivial zero singular values). The goal of this paper is to prove a local circular law for the ESD of $TX$ at any point $z$ away from the unit circle. Following the idea in \cite{local_circular}, the key ingredients for the proof are (a) the upper bound for the largest singular value of $TX-z$, (b) the lower bound for the least singular value of $TX-z$, and (c) rigidity of the singular values of $TX-z$. The upper bound for the largest singular value can be obtained by controlling the norm of $TX-z$ through a standard large deviation estimate (see e.g. \cite{Handbook_DS,Random_polytopes,RudVersh_rect} and (\ref{norm_upperbound})). The lower bound for the least singular value of $TX-z$ follows from the results in e.g. \cite{RudVersh_square} and \cite{TaoVu_circular} (see also Lemma \ref{lemm_leastsing}). Thus the bulk of this paper is devoted to establish (c).

\vspace{5pt}

\noindent{\bf Basic ideas.} To obtain the rigidity of the singular values of $Tx-z$, we study the ESD of $Q:=(TX-z)^\dag (TX-z)$ using Stieltjes transform as in \cite{local_circular}. We normalize $X$ so that its entries have variance $(N\wedge M)^{-1}$. Then $Q$ is an $N\times N$ Hermitian matrix with eigenvalues being typically of order 1. We denote its resolvent by $R(w):=(Q-w)^{-1}$, where $w=E+i\eta$ is a spectral parameter with positive imaginary part $\eta$. Then the Stieltjes transform of the ESD of $Q$ is equal to $N^{-1}\text{Tr}\, R(w)$, and we have the convergence estimate
\begin{equation}\label{intro_average_law}
N^{-1}\text{Tr}\, R(w) \approx m_c(w)
\end{equation}
with high probability for large $N$. Here $m_c$ is the Stieltjes transform of the asymptotic eigenvalue density, and the convergence in (\ref{intro_average_law}) is referred to as the {\it averaged law}. By taking the imaginary part of (\ref{intro_average_law}), it is easy to see that a control of the Stieltjes transform yields a control of the eigenvalue density on a small scale of order $\eta$ around $E$ (which contains an order $\eta N$ eigenvalues). A {\it local law} is an estimate of the form (\ref{intro_average_law}) for all $\eta \gg N^{-1}$. Such local laws have been a cornerstone of the modern random matrix theory. In \cite{local_law}, a local law was first derived for Wigner matrices. Subsequently in \cite{local_circular}, a local law for the resolvent of $(X-z)^\dag (X-z)$ was established to prove the local circular law.

In generalizing the proof in \cite{local_circular} to our setting, a main difficulty is that the entries of $TX$ are not independent. We will use a new comparison method proposed in \cite{Anisotropic}, which roughly states that if the local laws hold for $R(w)$ with Gaussian $X$, then they also hold in the case of a general $X$. For definiteness, we assume $N=M$ for now, and $T$ is a square matrix with singular decomposition $T=UDV$. For a Gaussian $X \equiv X^{Gauss}$, we have $V X^{Gauss}U \stackrel{d}{=} \tilde X^{Gauss} $, where $\tilde X$ is another Gaussian random matrix. Then for the determinant in (\ref{intro_Herm_tech}),
\begin{equation}
\det(TX^{Gauss}-z) {=} \det (DVX^{Gauss}U-z) \stackrel{d}{=} \det (D\tilde X^{Gauss}-z).\label{eqn_intro_idea}
\end{equation}
The problem is now reduced to the study of the singular values of $D\tilde X^{Gauss}-z$, which has independent entries. Notice the entries of $D\tilde X^{Gauss}$ are not identically distributed, which will make our proof much more complicated. However, this issue can be handled, e.g. as in \cite{Semicircle}, where a local law was obtained for generalized Wigner matrices with non-identically distributed entries.

To use the comparison method invented in \cite{Anisotropic}, it turns out the averaged local law from (\ref{intro_average_law}) is not sufficient. We have to control not only the trace of $R(w)$, but also the matrix $R(w)$ itself by showing that $R(w)$ is close to some deterministic matrix $\Pi(w)$, provided that $\eta \gg N^{-1}$. This closeness can be established in the sense of individual matrix entries $R_{ij}(w) \approx \Pi_{ij}(w)$ (see e.g. \cite{local_circular,Bulk_univ}). We call such an estimate an {\it entrywise local law}. More generally, in \cite{isotropic,isotropic_deform} the following closeness was established for {\it generalized matrix entries}:
\begin{equation}\label{intro_isotropic}
\langle \mathbf v, R(w) \mathbf u \rangle \approx \langle \mathbf v, \Pi(w) \mathbf u\rangle,\ \ \eta \gg N^{-1}, \ \ \forall \|\mathbf v\|_2,\|\mathbf u\|_2=1.
\end{equation}
We call the estimate in (\ref{intro_isotropic}) an {\it anisotropic local law}. (If $\Pi$ is a scalar matrix, (\ref{intro_isotropic}) is also referred to as an {\it isotropic local law}, in the sense that $R(w)$ is approximately isotropic for large $N$.) This kind of anisotropic local law is needed in applying the method in \cite{Anisotropic}.
Here we outline the three steps to establish the anisotropic local law for $Q=(TX-z)^\dag (TX-z)$: (A) the entrywise local law and averaged local law when $T$ is diagonal (Theorem \ref{law_squareD}); (B) the anisotropic local law when $T$ is diagonal (Theorem \ref{law_squareD}); (C) the anisotropic local law and averaged local law when $T$ is a general (rectangular) matrix (Theorem \ref{law_wideT}). 

In performing Step (A), our proof is basically based on the methods in \cite{local_circular}. However,  our multi-variable self-consistent equations and their solutions are much more complicated here. Thus a key part of the proof is to establish some basic properties of the asymptotic eigenvalue density and prove the stability of the self-consistent equations under small perturbations. These work need some new ideas and analytic techniques (see Appendix \ref{appendix1}). In performing Step (B), we applied and extended the polynomialization method developed in \cite[section 5]{isotropic}.  Finally, as remarked around (\ref{eqn_intro_idea}), (B) implies the anisotropic local law for a Gaussian $X$ and a general $T$. Based on this fact we perform Step (C) using a self-consistent comparison argument in \cite{Anisotropic}. With the averaged local law proved in Step (C), we can prove the local circular law for $TX$. In general, the averaged local law we get is up to the non-optimal scale $\eta \gg (N\wedge M)^{-1/2}$. As a result, we can only prove the local circular law for $TX$ up to the scale $(N\wedge M)^{-1/4 + \epsilon}$. A new observation is that the non-optimal averaged local law can lead to the optimal local circular law for $TX$ outside the unit circle (i.e. $|z|>1$) (see Section \ref{subsection_proofmain}). To prove the optimal local circular law inside the unit circle (i.e. $|z|<1$), we need the optimal averaged local law up to the scale $\eta \gg (N\wedge M)^{-1}$, which can be obtained under the extra assumption that the entries of $X$ have vanishing third moments.


\vspace{5pt}

\noindent{\bf Conventions.} The fundamental large parameter is $N$ and we assume that $M$ is comparable to $N$ (see (\ref{assm0})). All quantities that are not explicitly constant may depend on $N$, and we usually omit $N$ from our notation. We use $C$ to denote a generic large positive constant, which may depend on fixed parameters and whose value may change from one line to the next. Similarly, we use $c$ or $\epsilon$ to denote a generic small positive constant. If a constant depend on a quantity $a$, we use $C(a)$ or $C_a$ to indicate this dependence. We use $\tau>0$ in various assumptions to denote a small positive constant, and use $\zeta,\tau'$ to denote constants that depend on $\tau$ and may be chosen arbitrarily small. All constants $C$, $c$ and $\epsilon$ may depend on $\tau$; we neither indicate nor track this dependence.

For any (complex) matrix $A$, we use $A^\dag$ to denote its conjugate transpose, $A^T$ the transpose, $\|A\|$ the operator norm and $\|A\|_{HS}$ the Hilbert-Schmidt norm. We use the notation $\mathbf v=(v_i)_{i=1}^n$ for a vector in $\mathbb C^n$, and denote its Euclidean norm by $|\mathbf v|\equiv \|\mathbf v\|_2$. We usually write the $n\times n$ identity matrix $I_n$ as $1$ without causing any confusions.

For two quantities $A_N$ and $B_N>0$ depending on $N$, we use the notations $A_N = O(B_N)$ and $A_N \sim B_N$ to mean $|A_N| \le CB_N$ and $C^{-1}B_N \le |A_N| \le CB_N$, respectively, for some positive constant $C>0$. We use $A_N=o(B_N)$ to mean $|A_N| \le c_N B_N$ for some positive constant $c_N\to 0$ as $N\to \infty$. If $A_N$ is a matrix, we use the notations $A_N = O(B_N)$ and $A_N=o(B_N)$ to mean $\|A_N\| = O(B_N)$ and $\|A_N\|=o(B_N)$, respectively.

\vspace{10pt}

\noindent{\bf Acknowledgements.} The third author would like to thank Terence Tao, Mark Rudelson and Roman Vershynin for fruitful discussions and valuable suggestions.

\end{section}

\begin{section}{The main results}

In this section, we state and prove the main result of this paper.
In Section \ref{subsection_assumptions}, we define our model and list our main assumptions. In Section  \ref{subsection_maintheorem}, we first define the asymptotic eigenvalue density $\rho_{2c}$ of $Q=(TX-z)^\dag(TX-z)$, and then state the main theorem---Theorem \ref{main_theorem}---of this paper.
Its proof depends crucially on local estimates of the resolvent of $Q$, which are presented
in Section \ref{subsection_locallaws}. In Section \ref{subsection_proofmain}, we prove Theorems \ref{main_theorem} based on the local estimates stated in Section \ref{subsection_locallaws}.

\begin{subsection}{Definition of the model}\label{subsection_assumptions}

In this paper, we want to understand the local statistics of the eigenvalues of $TX-zI$, where $T$ is a deterministic $N\times M$ matrix, $X$ is a random $M\times N$ matrix, $z\in \mathbb C$ and $I$ is the identity operator.
We assume $M\sim N$, i.e.
\begin{equation}\label{assm0}
\tau\le \frac{M}{N} \le \tau^{-1}
\end{equation}
for some small $\tau>0$. We assume the entries $X_{i\mu}$ of $X$ are independent (not necessarily identically distributed) random variables satisfying
\begin{equation}\label{assm1}
 \bbE X_{i\mu}=0,\hskip 10pt \bbE|X_{i\mu}|^2=\frac{1}{N \wedge M}
\end{equation}
for all $1\le i \le M,1\le \mu \le N$. For definiteness, in this paper we only focus on the case where all matrix entries are real. However, our results and proofs also hold, after minor changes, in the complex case if we assume in addition $\mathbb EX_{i\mu}^2=0$ for $X_{i\mu} \in \mathbb C$.
We assume that for all $p\in\mathbb N$, there is an $N$-independent constant $C_p$ such that
\begin{equation}\label{assm2}
 \bbE|\sqrt {N\wedge M} X_{i\mu}|^p\le C_p
\end{equation}
for all $1\le i \le M,1\le \mu \le N$. We define $\Sigma := TT^\dag$, and assume the eigenvalues of $\Sigma$ satisfy that
\begin{equation}\label{assm3}
 \tau^{-1} \ge \sigma_1 \ge \sigma_{2} \ge \dots \ge \sigma_{N\wedge M} \ge \tau
\end{equation}
and all other eigenvalues are $0$. We can normalize $T$ by multiplying a scalar such that
\begin{equation}\label{assm4}
\frac{1}{N\wedge M}\sum_{i=1}^{N\wedge M} \sigma_i =1.
\end{equation}
We summarize our basic assumptions here for future reference.
\begin{assum}\label{main_assump}
We suppose that (\ref{assm0}), (\ref{assm1}), (\ref{assm2}), (\ref{assm3}) and (\ref{assm4}) hold.
\end{assum}

\end{subsection}

\begin{subsection}{The main theorem}\label{subsection_maintheorem}
Our main result is Theorem \ref{main_theorem}. To state it, we need to define the asymptotic eigenvalue density function for $Q$. We first introduce the self-consistent equations, and the asymptotic eigenvalue density will be closely related to their solutions.
Define
\begin{equation}\label{pi_measure}
\rho_{\Sigma} : = \frac{1}{N\wedge M}\sum_{i=1}^{N\wedge M} \delta_{\sigma_i}
\end{equation}
as the empirical spectral density of $\Sigma$. Let $n:=|\text{supp}\, \rho_{\Sigma}|$ be the number of distinct nonzero eigenvalues of $\Sigma$, which are denoted as
\begin{equation}\label{ordering_si}
\tau^{-1} \ge s_1 > s_{2} > \dots > s_n \ge \tau.
\end{equation}
Let $l_i$ be the multiplicity of $s_i$. By (\ref{assm4}), $l_i$ and $s_i$ satisfy the normalization conditions
\begin{equation}
\frac{1}{N\wedge M}\sum_{i=1} ^ n l_i =1, \ \ \frac{1}{N\wedge M}\sum_{i=1}^n l_i s_i =1.
\end{equation}
For each $w\in\mathbb C_+ :=\{ w\in \mathbb C: \text{Im}\, w >0 \}$, we define the self-consistent equations of $(m_1,m_2)$ as
\begin{align}
& \frac{1}{{m_2}} =  - w (1+m_1)  + \frac{{\left| z \right|^2 }}{{1 + m_1 }}, \label{eq_self1} \\
& m_1= \frac{1}{N}\sum_{i=1}^n l_i s_i \left[{-w\left( {1 + s_i m_2} \right)+ \frac{\left| z \right|^2}{1+m_1} }\right]^{-1}. \label{eq_self2}
\end{align}
If we plug (\ref{eq_self1}) into (\ref{eq_self2}), we get the self-consistent equation for $m_1$ only,
\begin{equation}\label{eq_self3}
m_1= \frac{1}{N}\sum_{i=1}^n l_i s_i \left[{-w\left( {1 + \frac{s_i }{- w (1 + m_1)  + \frac{\left| z \right|^2 }{1+m_1} }} \right) + \frac{\left| z \right|^2}{1+m_1} } \right]^{-1}.
\end{equation}
The next lemma states that the solution to (\ref{eq_self3}) in $\mathbb C_+$ is unique if $z$ is away from the unit circle. It is proved in Appendix \ref{subsection_append3}.

\begin{lem} \label{lemm_density}
Fix $z \in \mathbb C$ such that $|z|\ne 1$. For $w\in \mathbb C_+$, there exists at most one analytic function $m_{1c, z, \Sigma}(w): \mathbb C_+ \to \mathbb C_+$ such that (\ref{eq_self3}) holds and $wm_{1c, z, \Sigma}(w) \in \mathbb C_+$.
Moreover, $m_{1c, z, \Sigma, N}(w)$ is the Stieltjes transform of a positive integrable function $\rho_{1c}$ with compact support in $[0,\infty)$.
\end{lem}
We shall abbreviate $m_{1c}(w):=m_{1c, z, \Sigma}(w)$. We also define $m_{2c}(w):=m_{2c, z, \Sigma}(w)$ by taking $m_1=m_{1c}(w)$ in (\ref{eq_self1}). Obviously, $m_{2c}$ is also an analytic function of $w$. Furthermore, for any $w\in \mathbb C_+$ we have $m_{2c}(w), wm_{2c}(w) \in \mathbb C_+$ by using (\ref{eq_self1}) and $m_{1c},wm_{1c}\in \mathbb C_+$. We define two functions on $\mathbb R$ as
\begin{equation}\label{eqn_rho1c}
\rho_{1,2c}(x)=\frac{1}{\pi}\lim_{\eta\searrow 0} \text{Im}\, m_{1,2c}(x+i\eta), \ \ x\in \mathbb R.
\end{equation}
It is easy to see that $\rho_{1,2c} \ge 0$ and $\text{supp}(\rho_{1,2c})\subseteq [0,\infty)$. Moreover, $\text{supp}\, \rho_{2c}=\text{supp}\, \rho_{1c}$ by (\ref{eq_self1}). We shall call $\rho_{2c}$ the asymptotic eigenvalue density of $Q=(TX-z)^\dag(TX-z)$ (for a reason that will be made clear during the proof in Section \ref{section_weaklaw}).
Since $\text{Im}(wm_{2c}) \ge 0$, we have
$$\text{Im} \left[-w\left( {1 + s_i m_{2c}} \right)+ \frac{\left| z \right|^2}{1+m_{1c}} \right] \le -\text{Im}\,w,$$
and (\ref{eq_self2}) gives $|m_{1c}| \le {1}/{\text{Im}\,w} \to 0$ as $\text{Im}\,w \to \infty$. Similarly, $|m_{2c}| \le {1}/{\text{Im}\,w} \to 0$ as $\text{Im}\,w \to \infty$. Thus $m_{1,2c}(w)$ is indeed the Stieltjes transform of $\rho_{1,2c}$,
\begin{equation}
m_{1,2c}(w)= \int_{\mathbb R} \frac{\rho_{1,2c}(x)}{x-w}dx.\label{stj_rho1}
\end{equation}

We now state the basic properties of $\rho_{1c}$ and $\rho_{2c}$, which can be obtained by studying the solutions $m_{1,2c}(w)$ to the self-consistent equations (\ref{eq_self1}) and (\ref{eq_self3}) when $w\in (0,\infty)$. Here we extend the definition of $m_{1,2c}$ continuously down to the real axis by setting
$$m_{1,2c}(x) = \lim_{\eta\searrow 0} m_{1,2c}(x+i\eta), \ \ x\in \mathbb R.$$
As a convention, for $w\in \overline{\mathbb C_+}$, we take $\sqrt{w}$ to be the branch with positive imaginary part. Define
$m:=\sqrt{w}(1+m_1)$ and $m_c:=\sqrt{w}(1+m_{1c}).$ Equation (\ref{eq_self3}) then becomes
\begin{equation}
f(\sqrt{w},m)=0,\label{selfm_reduced}
\end{equation}
where
\begin{align}\label{eqn_def_fwm}
f(\sqrt{w},m)= -\sqrt{w} + m + \frac{1}{N}\sum_{i=1}^n l_i s_i \frac{m(m^2 - |z|^2)}{\sqrt{w}m^3 - (s_i + |z|^2)m^2 -\sqrt{w}|z|^2 m + |z|^4}.
\end{align}
The following lemma gives the basic structure of $\text{supp}\, \rho_{1,2c}$. Its proof is given in Appendix \ref{subsection_append1}.

\begin{lem}\label{lemm_rho}
Fix $\tau \le \left||z|^2-1\right| \le \tau^{-1}$. The support of $\rho_{1,2c}$ is a union of connected components:
\begin{equation}\label{support_rho1c}
{\rm{supp}} \, \rho_{1,2c} \cap (0,+\infty) = \left( \bigcup_{1\le k\le L} [e_{2k}, e_{2k-1}] \right) \cap (0,\infty),
\end{equation}
where $L \equiv L(n)\in \mathbb N$ and $C_1\tau^{-1}\ge e_1 > e_2 > \ldots > e_{2L} \ge 0$ for some constant $C_1>0$ that does not depend on $\tau$. If $|z|^2\le 1-\tau$, we have $e_{2L}=0$; if $1+\tau \le |z|^2 \le 1+\tau^{-1}$, $e_{2L} \ge \epsilon(\tau)$ for some constant $\epsilon(\tau)>0$. 
Moreover, for every $e_i>0$, there exists a unique $m_c(e_i)$ such that
\begin{equation}
{\partial_m f} (\sqrt{e_i},m_c(e_i)) =0. \label{equationEm2}
\end{equation}
\end{lem}
We shall call $e_i$'s the edges of $\rho_{1c}$.
%
For any $w\in (0,\infty)$ and $1\le i \le n$, the cubic polynomial $\sqrt{w}m^3 - (s_i + |z|^2)m^2 -\sqrt{w}|z|^2 m + |z|^4$ in (\ref{eqn_def_fwm}) has three distinct roots $a_i(w)>0$, $b_i(w)>0$ and $-c_i(w)<0$ (see Lemma \ref{lemm_tech_f}). Our next assumption on $\rho_{\Sigma}$ and $|z|$ takes the form of the following regularity conditions.

\begin{defn}\label{def_regular}
(Regularity) Fix $\tau \le \left||z|^2-1\right| \le \tau^{-1}$ and a small constant $\epsilon > 0$.

(i) We say that the edge $e_k\ne 0$, $k=1, \ldots, 2L$, is regular if
\begin{equation}
\min_{1\le i \le n} \{|m_c(e_k) - a_i(e_k)|,|m_c(e_k) - b_i(e_k)|, |m_c(e_k) + c_i(e_k)| \}\ge \epsilon , \label{regular1}
\end{equation}
and
\begin{equation}
\left|\partial_m^2 f(\sqrt{e_k},m_c(e_k))\right| \ge \epsilon . \label{regular2}
\end{equation}
In the case $|z|^2\le 1- \tau$, we always call $e_{2L}=0$ a regular edge.

(ii) We say that the bulk components $[e_{2k}, e_{2k-1}]$ is regular if for any fixed $\tau'>0$ there exists a constant $c(\tau,\tau')>0$ such that the density of $\rho_{1c}$ in $[e_{2k}+\tau', e_{2k-1}-\tau']$ is bounded from below by $c$.

\end{defn}

\noindent{\it{Remark 1:}} The edge regularity conditions (i) has previously appeared (may be in slightly different forms) in several works on sample covariance matrices and Wigner matrices \cite{Regularity2, Regularity1,Regularity5, Anisotropic,Regularity3,Regularity4}. The conditions (\ref{regular1}) and (\ref{regular2}) guarantees a regular square-root behavior of $\rho_{1c}$ near $e_k$ and ensures that the gap in the spectrum of $\rho_{1c}$ adjacent to $e_k$ does not close for large $N$ (Lemma \ref{lemm_edge_reg}),
\begin{equation}\label{edge_gap}
\min_{l\ne k} |e_{l}-e_k| \ge \epsilon
\end{equation}
for some constant $\epsilon>0$. The bulk regularity condition (ii) was introduced in \cite{Anisotropic}. It imposes a lower bound on the density of eigenvalues away from the edges. Without it, one can have points in the interior of $\text{supp}\, \rho_{1c}$ with an arbitrarily small density and our arguments would fail.

\vspace{5pt}

\noindent{\it{Remark 2:}} The regularity conditions in Definition \ref{def_regular} are stable under perturbations of $|z|$ and $\rho_{\Sigma}$. In particular, fix $\rho_{\Sigma}$, suppose the regularity conditions are satisfied at $z=z_0$ with $\tau  \le | {\left| z_0 \right|^2 - 1} | \le \tau ^{-1} $. Then for sufficiently small $c>0$, the regularity conditions hold uniformly in $z\in \left\{z: ||z|-|z_0||\le c\right\}$. For a detailed discussion, see the remark at the end of Section \ref{subsection_append3}.

\vspace{5pt}

We will use the following notion of stochastic domination, which was first introduced in \cite{Average_fluc} and subsequently used in many works on random matrix theory, such as \cite{isotropic,principal,local_circular,Delocal,Semicircle,Anisotropic}. It simplifies the presentation of the results and their proofs by systematizing statements of the form ``$\xi$ is bounded by $\zeta$ with high probability up to a small power of $N$".

\begin{defn}[Stochastic domination]
(i) Let
\[\xi=\left(\xi^{(N)}(u):N\in\bbN, u\in U^{(N)}\right),\hskip 10pt \zeta=\left(\zeta^{(N)}(u):N\in\bbN, u\in U^{(N)}\right)\]
be two families of nonnegative random variables, where $U^{(N)}$ is a possibly $N$-dependent parameter set. We say $\xi$ is stochastically dominated by $\zeta$, uniformly in $u$, if for any (small) $\epsilon>0$ and (large) $D>0$, 
\[\sup_{u\in U^{(N)}}\bbP\left[\xi^{(N)}(u)>N^\epsilon\zeta^{(N)}(u)\right]\le N^{-D}\]
for large enough $N\ge N_0(\epsilon, D)$, and we use the notation $\xi\prec\zeta$. Throughout this paper the stochastic domination will always be uniform in all parameters that are not explicitly fixed (such as matrix indices, and $w$ and $z$ that take values in some compact sets). Note that $N_0(\epsilon, D)$ may depend on quantities that are explicitly constant, such as $\tau$ and $C_p$ in (\ref{assm0}), (\ref{assm2}) and (\ref{assm3}).

(ii) If for some complex family $\xi$ we have $|\xi|\prec\zeta$, we also write $\xi \prec \zeta$ or $\xi=O_\prec(\zeta)$.
We also extend the definition of $O_\prec(\cdot)$ to matrices in the weak operator sense as follows. Let $A$ be a family of complex square random matrices and $\zeta$ a family of nonnegative random variables. Then we use $A=O_\prec(\zeta)$ to mean $\|A\|\prec\zeta$, where $\|A\|$ is the operator norm of $A$.

(iv) We say that an event $\Xi$ holds with high probability if $1-1(\Xi)\prec 0$.
\end{defn}

In the following, we denote the eigenvalues of $TX$ as $\mu_j$, $1\le j \le N$. We are now ready to state our main theorem, i.e. the general local circular law for $TX$.

\begin{thm}[Local circular law for $TX$] \label{main_theorem}
Suppose Assumption \ref{main_assump} holds, and $\tau  \le  | {\left| z_0 \right|^2 - 1} | \le \tau ^{ - 1} $ for any $N$ ($z_0$ can depend on $N$). Suppose $\rho_\Sigma$ (defined in (\ref{pi_measure})) and $|z_0|$ are such that all the edges and bulk components of $\rho_{1c}$ are regular in the sense of Definition \ref{def_regular}. We assume in addition that the entries of $X$ have a density bounded by $N^{C_2}$ for some $C_2>0$. Let $F$ be a smooth non-negative function which may depend on $N$, such that $\|F\|_\infty \le C_1$, $\|F'\|_\infty \le N^{C_1}$ and $F(z)=0$ for $|z|\ge C_1$, for some constant $C_1>0$ independent of $N$. Let $F_{z_0,a}(z) = K^{2a}F(K^a(z-z_0))$, where $K: = N\wedge M$. Then $TX$ has $(N-K)$ trivial zero eigenvalues, and for the other eigenvalues $\mu_j$, $1\le j \le K$, we have
\begin{equation}\label{eq_main0}
\frac{1}{K}\sum_{j=1}^{K} F_{z_0,a}(\mu_j) - \frac{1}{\pi} \int F_{z_0,a}(z)\tilde \chi_{\mathbb D}(z) dA(z) \prec K^{-1/2+2a} \|\Delta F\|_{L^1},
\end{equation}
for any $a\in (0,1/4]$. Here
\begin{equation}\label{def_chitilde}
\tilde \chi_{\mathbb D}(z):=\frac{1}{4}\int_0^\infty (\log x) \Delta_z \rho_{2c}(x,z)dx,
\end{equation}
where $\rho_{2c}\equiv \rho_{2c, z, \Sigma}$ is defined in (\ref{eqn_rho1c}). If $1+\tau  \le \left| z_0 \right|^2  \le 1+\tau ^{ - 1} $ or the entries of $X$ have vanishing third moments,
\begin{equation}
\mathbb EX_{i\mu}^3=0 , \label{assm_3rdmoment}
\end{equation}
for $1\le i \le M,1\le \mu \le N$, then we have the improved result
\begin{equation}\label{eq_main}
\frac{1}{K}\sum_{j=1}^{K} F_{z_0,a}(\mu_j) - \frac{1}{\pi} \int F_{z_0,a}(z)\tilde \chi_{\mathbb D}(z) dA(z) \prec K^{-1+2a} \|\Delta F\|_{L^1},
\end{equation}
for any $a\in (0,1/2]$. If $N=M$, the bounded density condition for the entries of $X$ is not necessary.
\end{thm}

\noindent{\it{Remark 1:}} Note that $F_{z_0,a}(z) = K^{2a}F(K^a(z-z_0))$ is an approximate delta function obtained from rescaling $F$ to the size of order $K^{-a}$ around $z_0$. Thus (\ref{eq_main0}) gives the general circular law up to scale $K^{-1/4+\epsilon}$, while (\ref{eq_main}) gives the general circular law up to scale $K^{-1/2+\epsilon}$. The $\tilde \chi_{\mathbb D}$ in (\ref{def_chitilde}) gives the distribution of the eigenvalues of $TX$. It is rotationally symmetric, because $\rho_{2c}(x,z)$ only depends on $|z|$ (see (\ref{eq_self1}) and (\ref{eq_self2})). When $T$ is the identity matrix, $\tilde \chi_{\mathbb D}$ becomes the indicator function $\chi_{\mathbb D}$ on the unit disk $\mathbb D$, and we get the well-known local circular law for $X$ \cite{local_circular}. For a general $T$, we do not have much understanding of $\tilde \chi_{\mathbb D}$ so far. 
This will be one of the topics of our future study. Also, we have assumed that $z$ is strictly away from the unit circle. Our proof may be extended to the $|z - 1|=o(1)$ case if we have a better understanding of the solutions $m_{1,2c}$ to equations (\ref{eq_self1}) and (\ref{eq_self2}).

\vspace{5pt}

\noindent{\it{Remark 2:}} As explained in the Introduction, the basic strategy of this paper is first to prove the anisotropic local law for the resolvent of $Q$ when $X$ is Gaussian, and then to get the anisotropic local law for a general $X$ through comparison with the Gaussian case. Without (\ref{assm_3rdmoment}), our comparison arguments do not give the anisotropic local law up to the optimal scale, so we can only prove the weaker bound (\ref{eq_main0}). We will try to remove this assumption in future works.

\vspace{5pt}

\noindent{\it{Remark 3:}} In the statement of the theorem, we have included an extra bounded density condition. This is only used in Lemma \ref{lemm_leastsing} to give a lower bound for the smallest singular value of $TX-z$. Thus it can be removed if we have a stronger result about the smallest singular value.

\vspace{5pt}

We conclude this section with two examples verifying the regularity conditions of Definition \ref{def_regular}.

\begin{exam}[Bounded number of distinct eigenvalues] \label{example1}
We suppose that $n$ is fixed, and that $s_1,\ldots,s_n$ and $\rho_{\Sigma}(\{s_1\}),\ldots, \rho_{\Sigma}(\{s_n\})$ all converge as $N\to \infty$. We suppose that $\lim_{N} e_{k} > \lim_N e_{k+1}$ for all $k$, and furthermore for all $e_k$ we have $\partial_m^2 f(\sqrt{e_k},m_c(e_k))\ne 0$. Then it is easy to check that all the edges and bulk components are regular in the sense of Definition \ref{def_regular} for small enough $\epsilon$. 
\end{exam}

\begin{exam}[Continuous limit] \label{example2}
We suppose $\rho_{\Sigma}$ is supported in some interval $[a,b]\subset(0,\infty)$, and that $\rho_{\Sigma}$ converges in distribution to some measure $\rho_\infty$ that is absolutely continuous and whose density satisfies $\tau \le d\rho_\infty(E)/dE \le \tau^{-1}$ for $E\in [a,b]$. Then there are only a small number (which is independent of $n$) of connected components for $\text{supp}\, \rho_{1c}$, and all the edges and bulk components are regular. See the remark at the end of Section \ref{subsection_append1}.
\end{exam}

\end{subsection}

\begin{subsection}{Hermitization and local laws for resolvents}\label{subsection_locallaws}

In the following, we use the notation
\begin{equation}\label{def_Yz}
Y\equiv Y_z:= TX-zI,
\end{equation}
where $I$ is the identity matrix. Following Girko's Hermitization technique \cite{Girko}, the first step in proving the local circular law is to understand the local statistics of singular values of $Y$. In this subsection, we present the main local estimates concerning the resolvents $\left(YY^\dag - w\right)^{-1}$ and $\left(Y^\dag Y - w\right)^{-1}$.
These results will be used later to prove Theorem \ref{main_theorem}.

Our local laws can be formulated in a simple, unified fashion using a $2N\times 2N$ block matrix, which is a linear function of $X$.

\begin{defn}[Index sets] \label{def_indexsets} 
We define the index sets
 \[\sI_1:=\{1,...,N\},\hskip 10pt \mathcal \sI_1^{M}:=\{1,\ldots, M\}, \hskip 10pt\sI_2:=\{N+1,...,2N\}, \hskip 10pt\sI:=\sI_1\cup\sI_2,\hskip 10pt\sI^M:=\sI_1^M\cup\sI_2.\]
We will consistently use the latin letters $i,j\in\sI_1$ or $\sI_1^M$, greek letters $\mu,\nu\in\sI_2$, and $s,t\in\sI$. We label the indices of the matrices according to
 $$X= (X_{i\mu}:i\in \mathcal I_1^M, \mu \in \mathcal I_2),  \hskip 10pt T= (T_{ij}:i\in \mathcal I_1, j \in \mathcal I_1^M).$$
When $M=N$, we always identify $\mathcal I_1^M$ with $\mathcal I_1$. For $i\in \mathcal I_1$ and $\mu \in \mathcal I_2$, we introduce the notations $\bar i:=i+N \in \mathcal I_2$ and $\bar\mu:=\mu-N \in \mathcal I_1$.
\end{defn}

\begin{defn}[Groups]
For an $\mathcal I \times \mathcal I$ matrix $A$, we define the $2\times 2$ matrix $A_{[ij]}$ as
\begin{equation}\label{Aij_group}
A_{[ij]}=\left( {\begin{array}{*{20}c}
   {A_{ij} } & {A_{i\bar j} }  \\
   {A_{\bar i j} } & {A_{\bar i\bar j} }  \\
\end{array}} \right).
\end{equation}
We shall call $A_{[ij]}$ a diagonal group if $i=j$, and an off-diagonal group otherwise .
\end{defn}

\begin{defn}[Linearizing block matrix]\label{def_linearHG}
For $w:=E+i\eta\in \bbC_+$, we define the $\sI\times\sI$ matrix
 \begin{equation}
   H(w)\equiv H(T,X,z,w): = \left( {\begin{array}{*{20}c}
   { - w I} & w^{1/2}Y  \\
   {w^{1/2} Y^\dag} & { - wI}  \\
   \end{array}} \right),
 \end{equation}
where we take the branch of $\sqrt{w}$ with positive imaginary part. Define the $\mathcal I\times\mathcal I$ matrix
 \begin{equation}\label{eqn_defG}
   G(w)\equiv G(T,X,z,w):=H(w)^{-1},
 \end{equation}
 as well as the $\mathcal I_1 \times \mathcal I_1$ and $\mathcal I_2 \times \mathcal I_2$ matrices
\begin{equation}
G_L(w)={\left( {YY^\dag   - w} \right)^{ - 1} },\ \ \ G_R(w)={\left( {Y^\dag Y  - w} \right)^{ - 1} }.
\end{equation}
Throughout the following, we frequently omit the argument $w$ from our notations.
\end{defn}

By Schur's complement formula, it is easy to see that
\begin{equation}\label{eqn_schurmatrix}
G\left( w \right) = \left( {\begin{array}{*{20}c}
   {G_L} & {w^{-1/2}G_LY}  \\
   {w^{-1/2}Y^\dag  G_L} & {w^{ - 1} Y^\dag  G_LY - w^{ - 1} I}  \\
\end{array}} \right) = \left( {\begin{array}{*{20}c}
   {w^{-1}YG_RY^\dag   - w^{-1}I} & {w^{-1/2}YG_R}  \\
   {w^{-1/2}G_RY^\dag  } & {G_R}  \\
\end{array}} \right).
\end{equation}
Therefore a control of $G$ immediately yields controls of the resolvents $G_L$ and $G_R$.

In the following, we only consider the $N\le M$ case. The $N>M$ case, as we will see, will be built easily upon $N\le M$ case. We introduce a deterministic matrix $\Pi$, which will be proved to be close to $G$ with high probability.

\begin{defn}[Deterministic limit of $G$]
Suppose $N\le M$ and $T$ has a singular decomposition
\begin{equation}\label{singular_decompT}
T=U\bar DV, \ \ \bar D=(D,0),
\end{equation}
where $D=\text{diag}(d_1,d_2,\ldots,d_N)$ is a diagonal matrix. Define $\pi_{{[i]}c}$ to be the $2\times 2$ matrix such that
\begin{equation}\label{def_pi}
\left(\pi_{{[i]}c}\right)^{ - 1} = \left(\begin{matrix}-w(1+ |d_i|^2m_{2c}) & -w^{1/2}z\\
 -w^{1/2}\bar z & -w(1+m_{1c})\\
\end{matrix}\right).
\end{equation}
Let $\Pi_d$ be the $2N\times 2N$ matrix with $(\Pi_d)_{[ii]} =\pi_{{[i]}c}$ and all other entries being zero. Define
\begin{equation}\label{def_Pi}
\Pi\equiv \Pi(\Sigma,z,w):=\left( {\begin{array}{*{20}c}
   { U} & {0}  \\
   {0} & {U}  \\
   \end{array}} \right) \Pi_d\left( {\begin{array}{*{20}c}
   { U^\dag} & {0}  \\
   {0} & {U^\dag}  \\
   \end{array}} \right) = \left(\begin{matrix} -(1+m_{1c})A(\Sigma) & w^{-1/2}zA(\Sigma)\\
 w^{-1/2}\bar z A(\Sigma) & -(1+m_{2c}\Sigma)A(\Sigma)\\
\end{matrix}\right),
\end{equation}
where $\Sigma=TT^\dag$ and $A(\Sigma) = \left[ w(1+ m_{2c}\Sigma)(1+m_{1c}) - |z|^2\right]^{-1}.$
\end{defn}


\begin{defn}[Averaged variables]
Suppose $N\le M$. Define the averaged random variables
\begin{equation}
 m_1:=\frac{1} {N} \sum_{i\in \mathcal I_1} \left(\bar \Sigma G\right)_{ii}, \ \ m_2:=\frac{1} {N} \sum_{\mu\in \mathcal I_2} \left(\bar \Sigma G\right)_{\mu\mu}, \label{def_M}
\end{equation}
where
\begin{equation}
\bar \Sigma:=\left( {\begin{array}{*{20}c}
   { \Sigma} & {0}  \\
   {0} & {I}  \\
   \end{array}} \right).
\end{equation}
Define $\pi_{[i]}$ to be the $2\times 2$ matrix such that
\begin{equation}\label{def_pi_i}
\left(\pi_{[i]}\right)^{-1}  =  \left( {\begin{array}{*{20}c}
   { - w(1+ \left|d_i\right|^2 m_2) } & {-w^{1/2}z}  \\
   {-w^{1/2}\bar z} & { - w(1 + m_1)} \end{array}} \right) .
 \end{equation}
\end{defn}

\noindent{\it{Remark:}} Note that under the above definition we have
$$m_2=\frac{1}{N}\text{Tr}\, G_R=\frac{1}{N}\text{Tr}\, G_L,$$
which is the Stieltjes transform of the empirical eigenvalue density of $YY^\dag$ and $Y^\dag Y$. Moreover, we will see from the proof that $m_{1,2c}$ are the almost sure limits of $m_{1,2}$ as $N\to \infty$ with
\begin{equation}\label{eq_m1c}
 m_{1c}=\frac{1} {N} \sum_{i\in \mathcal I_1} \left(\bar \Sigma \Pi\right)_{ii}, \ \ m_{2c}=\frac{1} {N} \sum_{\mu\in \mathcal I_2} \left(\bar \Sigma \Pi\right)_{\mu\mu}.
\end{equation}

The following two propositions summarize the properties of $\rho_{1,2c}$ and $m_{1,2c}$ that are needed to understand the main results in this section. They are proved in Appendix \ref{appendix1}. In Fig.\,\ref{fig_rho2c} we plot $\rho_{2c}$ for the example from Fig.\,\ref{fig_circular} in the cases $|z|>1$ and $|z|<1$, respectively.

\begin{figure}[htb]
\centering
\includegraphics[width=\columnwidth]{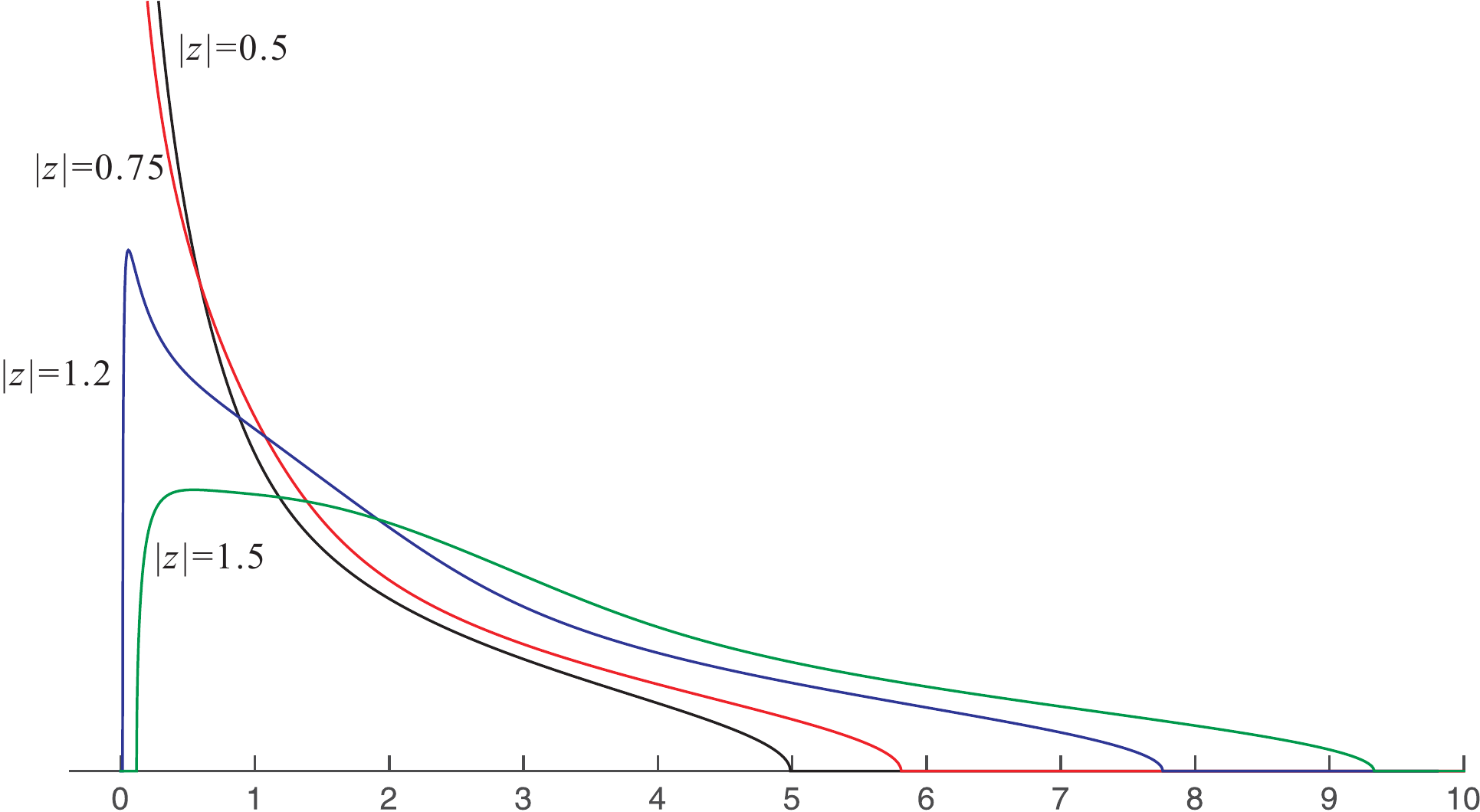}
\caption{The densities $\rho_{2c}(x,z)$ when $|z|=0.5$, $0.75,$ $1.2,$ $1.5$. Here $\rho_{\Sigma}= 0.5\delta_{\sqrt{2/17}} + 0.5\delta_{4\sqrt{2/17}}$. }
\label{fig_rho2c}
\end{figure}

\begin{prop}[Basic properties of $\rho_{1,2c}$]\label{prop_rho1c}
Fix $\epsilon>0$. The density $\rho_{1c}$ is compactly supported in $[0,\infty)$ and the following properties regarding $\rho_{1c}$ hold.

(i) The support of $\rho_{1c}$ is $\bigcup_{1\le k\le L(n)} [e_{2k}, e_{2k-1}] $ where $e_1 > e_2 > \ldots > e_{2L} \ge 0$. If $1+\tau \le |z|^2 \le 1+\tau^{-1}$, then $e_{2L} \ge \epsilon$; if $|z|^2\le 1-\tau$, then $e_1=0$.

(ii) Suppose $[e_{2k}, e_{2k-1}]$ is a regular bulk component. For any $\tau'>0$, if $x\in [e_{2k}+\tau', e_{2k-1}-\tau']$, then $\rho_{1c}(x) \sim 1$.

(iii) Suppose $e_j$ is a nonzero regular edge. If $j$ is even, then $\rho_{1c}(x)\sim \sqrt{x-e_{j}}$ as $x\to e_{j}$ from above. Otherwise if $j$ is odd, then $\rho_{1c}(x)\sim \sqrt{e_{j}-x}$ as $x\to e_{j}$ from below.

(iv) If $|z|^2 \le 1-\tau$, then $\rho_{1c}(x)\sim {x}^{-1/2}$ as $x \searrow e_{2L}=0$.

The same results also hold for $\rho_{2c}$. In addition, $\rho_{2c}$ is a probability density. 
\end{prop}

\begin{prop}\label{prop_roughbound}
The preceding proposition implies that, uniformly in $w$ in any compact set of $\mathbb C_+$,
\begin{equation}
|m_{1,2c}(w)|=O(|w|^{-1/2}).
\end{equation}
Moreover, if $1+\tau \le |z|^2 \le 1+\tau^{-1}$, then $|m_{1,2c}(w)| \sim 1$ for $w$ in any compact set of $\mathbb C_+$; if $|z|^2 \le 1-\tau$, then $|m_{1,2c}(w)| \sim |w|^{-1/2}$ for $w$ in any compact set of $\mathbb C_+$.
\end{prop}

We will consistently use the notation $E+i\eta$ for the spectral parameter $w$. In this paper, we regard the quantities $E(w)$ and $\eta(w)$ as functions of $w$ and usually omit the argument $w$. In the following we would like to define several spectral domains of $w$ that will be used in the proof. 

\begin{defn}[Spectral domains]
Fix a small constant $\zeta>0$ which may depend on $\tau$. The spectral parameter $w$ is always assumed to be in the fundamental domain
\begin{equation}\label{eq_domainD}
\mathbf D\equiv \mathbf D(\zeta,N) := \{w\in \mathbb C_+: 0\le E \le \zeta^{-1}, N^{-1+\zeta}|m_{2c}|^{-1}\le \eta\le \zeta^{-1}\}.
\end{equation}
unless otherwise indicated. Given a regular edge $e_k$, we define the subdomain
\begin{equation}
\mathbf D_{k}^e \equiv \mathbf D_{k}^e (\zeta, \tau',N):=\{w\in \mathbf D(\zeta,N): |E-e_k| \le \tau' , E\ge 0\}.
\end{equation}
Corresponding to a regular bulk component $[e_{2k},e_{2k-1}]$, we define the subdomain
\begin{equation}
\mathbf D_{k}^b\equiv \mathbf D_{k}^b(\zeta, \tau',N):=\{w\in \mathbf D(\zeta,N): E\in [e_{2k}+\tau',e_{2k-1}-\tau']\}.
\end{equation}
For the component outside ${\rm supp}\,\rho_{1c}$, we define the subdomain
\begin{equation}\label{def_domain_Do}
\mathbf D^o\equiv \mathbf D^o(\zeta, \tau',N):=\{w\in \mathbf D(\zeta,N): {\rm{dist}}(E,{\rm{supp}} \,\rho_{1c})\ge \tau'\}.
\end{equation}
We also need the following domain with large $\eta$,
\begin{equation}\label{eq_largedomainD}
 {\mathbf D}_L \equiv  {\mathbf D}_L(\zeta) := \{w\in \mathbb C_+: 0\le E \le \zeta^{-1}, \eta \ge \zeta^{-1} \},
\end{equation}
and the subdomain of $\mathbf D \cup \mathbf D_L$,
\begin{equation}\label{eq_subdomainD}
\widehat{\mathbf D} \equiv \widehat{\mathbf D}(\zeta, N) := \{w\in \mathbf D(\zeta,N) : \eta \ge N^{-1/2+\zeta}|m_{2c}|^{-1} \} \cup {\mathbf D}_L(\zeta).
\end{equation}
We call $\mathbf S$ a regular domain if it is a regular $\mathbf D_{k}^e$ or $\mathbf D_{k}^b$ domain, a $\mathbf D^o$ domain or a $\mathbf D_L$ domain.
\end{defn}

\noindent{\it{Remark:}} In the definition of $\mathbf D$, we have suppressed the explicit $w$-dependence. Notice that when $|z|^2<1-\tau$, since $|m_{2c}|\sim |w|^{-1/2}$ as $w\to 0$, we allow $\eta \sim |w| \sim N^{-2+2\zeta}$ in $\mathbf D$. In the definition of $\mathbf D^e_k$, the condition $E\ge 0$ is only for the edge at $0$ when $|z|^2\le 1-\tau$. 

\vspace{5pt}

Now we are prepared to state the various local laws satisfied by $G$ defined in (\ref{eqn_defG}). 
Let
\begin{equation}\label{eq_defpsi}
\Psi\equiv \Psi(w) := \sqrt {\frac{\Im \left(m_{1c} + m_{2c} \right) }{{N\eta }} } + \frac{1}{N\eta}
\end{equation}
be the deterministic control parameter.

\begin{defn}[Local laws]\label{def_local_laws}
Suppose $N\le M$. Recall $G\equiv G(T,X,z,w)$ defined in (\ref{eqn_defG}) and $\Pi \equiv \Pi(\Sigma,z,w)$ defined in (\ref{def_Pi}). Let $\mathbf S$ be a regular domain.

(i) We say that the entrywise local law holds with parameters $(T,X,z,\mathbf S)$ if
\begin{equation}
\left[G(T,X,z,w) - \Pi(\Sigma,z,w)\right]_{st} \prec \Psi(w)
\end{equation}
uniformly in $w\in \mathbf S$ and $s,t\in \mathcal I$. 

(ii) We say that the anisotropic local law holds with parameters $(T,X,z,\mathbf S)$ if
\begin{equation}
\left\|G(T,X,z,w) - \Pi(\Sigma,z,w)\right\| \prec \Psi(w)
\end{equation}
uniformly in $w\in \mathbf S$. 

(iii) We say that the averaged local law holds with parameters $(T,X,z,\mathbf S)$ if
\begin{equation}
\left|m_2(T,X,z,w) - m_{2c}(\Sigma,z,w)\right| \prec \frac{1}{N\eta}
\end{equation}
uniformly in $w\in \mathbf S$.

\end{defn}


The local laws for $G$ with a general $T$ will be built upon the following result with a diagonal $T$.

\begin{thm}[Local laws when $T$ is diagonal] \label{law_squareD}
Fix $\tau  \le  | {\left| z \right|^2 - 1} | \le \tau ^{ - 1} $. Suppose Assumption \ref{main_assump} holds, $N=M$, and $T\equiv D:=diag(d_1,...,d_N)$ is a diagonal matrix. Let $\mathbf S$ be a regular domain. Then the entrywise local law, anisotropic local law and averaged local law hold with parameters $(D,X,z,\mathbf S)$. 
\end{thm}


\vspace{5pt}

Now suppose that $N\le M$ and $T$ is an $N\times M$ matrix such that the eigenvalues of $\Sigma$ satisfy (\ref{assm3}) and (\ref{assm4}). Consider the singular decomposition $T=U\bar DV$, where $U$ is an $N\times N$ unitary matrix, $V$ is an $M\times M$ unitary matrix and $\bar D=(D,0)$ is an $N\times M$ matrix such that $D=\text{diag}(d_1,d_2,\ldots,d_N)$. Then we have
\begin{equation}
TX-z = UDV_1X-z, \label{TX1}
\end{equation}
where $V_1$ is an $N\times M$ matrix and $V_2$ is an $(M-N)\times M$ matrix defined through
$V = \left( {\begin{array}{*{20}c}
   { V_1 }  \\
   {V_2}  \\
\end{array}} \right).$ If $X=X^{Gauss}$ is Gaussian, then $V_1 X^{Gauss} \stackrel{d}{=} \tilde X^{Gauss} U^\dag$ with $\tilde X$ being an $N\times N$ Gaussian random matrix. Then by the definition of $G$ in (\ref{eqn_defG}),
\begin{equation}\label{eqn_comparison1}
G(T,X^{Gauss},z,w) \stackrel{d}{=}  \left( {\begin{array}{*{20}c}
   { U} & {0}  \\
   {0} & {U}  \\
   \end{array}} \right) G(D,\tilde X^{Gauss},z,w) \left( {\begin{array}{*{20}c}
   { U^\dag} & {0}  \\
   {0} & {U^\dag}  \\
   \end{array}} \right).
\end{equation}
Since the anisotropic local law holds for $G(D,\tilde X^{Gauss},z,w)$ by Theorem \ref{law_squareD}, we get immediately the anisotropic local law for $G(T,X^{Gauss},z,w)$. The next theorem states that the anisotropic local law holds for general $TX$ provided that the anisotropic local law holds for $TX^{Gauss}$.
----

\begin{thm}[Anisotropic local law when $N\le M$] \label{law_wideT}
Fix $\tau  \le  | {\left| z \right|^2 - 1} | \le \tau ^{ - 1} $. Suppose Assumption \ref{main_assump} holds and $N\le M$. Let $T=U\bar DV$ be a singular decomposition of $T$, where $\bar D=(D,0)$ with $D=\text{diag}(d_1,d_2,\ldots,d_N)$. Let $\mathbf S$ be a regular domain. Then the anisotropic local law and averaged local law hold with parameters $(T,X,z,\mathbf S \cap \widehat {\mathbf D})$. If in addition (\ref{assm_3rdmoment}) holds, then the anisotropic local law and averaged local law hold with parameters $(T,X,z,\mathbf S)$.
\end{thm}

Finally we turn to the $N>M$ case. Suppose $T=U\bar DV$ is a singular decomposition of $T$, where $U$ is an $N\times N$ unitary matrix, $V$ is an $M\times M$ unitary matrix and $\bar D=\left( {\begin{array}{*{20}c}
   {D}  \\
   {0}  \\
\end{array}} \right)$ is an $N\times M$ matrix such that $D=\text{diag}(d_1,d_2,\ldots,d_M)$.
Let $U=(U_1,U_2)$, where $U_1$ has size $N\times M$ and $U_2$ has size $N\times (N-M)$. Following Girko's idea of Hermitization \cite{Girko}, to prove the local circular law in Theorem \ref{main_theorem} when $N>M$, it suffices to study $\det(TX-z)$ (see (\ref{eqn_hermitizatio}) below), for which we have
 \begin{align}
\det(TX-z)  & = \det \left( {\begin{array}{*{20}c}
   { DVXU_1-z} &  DVXU_2\\
   {0} & {-z} \\
\end{array}} \right) = \det(V^T D^T U_1^T X^T - z)(-z)^{N-M}. \label{tall_reduce}
\end{align}
Comparing with (\ref{TX1}), we see that this case is reduced to the $N\le M$ case, with the only difference being that the extra $(-z)^{N-M}$ term corresponds to the $N-M$ zero eigenvalues of $TX$. Thus we make the following claim.



\begin{claim} \label{law_tallT}
The $N < M$ case of Theorem \ref{main_theorem} implies the $N>M$ case of Theorem \ref{main_theorem}.
\end{claim}


\end{subsection}

\begin{subsection}{Proof of Theorem \ref{main_theorem}} \label{subsection_proofmain}

By Claim \ref{law_tallT}, it suffices to assume $N\le M$. Our main tool will be Theorem \ref{law_wideT}. A major part of the proof follows from \cite[Section 5]{local_circular}.
The following lemma collects basic properties of stochastic domination $\prec$, which will be used tacitly during the proof and throughout this paper.

\begin{lem}[Lemma 3.2 in \cite{isotropic}]\label{lem_stodomin}
(i) Suppose that $\xi (u,v)\prec \zeta(u,v)$ uniformly in $u\in U$ and $v\in V$. If $|V|\le N^C$ for some constant $C$, then
$$\sum_{v\in V} \xi(u,v) \prec \sum_{v\in V} \zeta(u,w)$$
uniformly in $u$.

(ii) If $\xi_1 (u)\prec \zeta_1(u)$ uniformly in $u\in U$ and $\xi_2 (u)\prec \zeta_2(u)$ uniformly in $u\in U$, then
$$\xi_1(u)\xi_2(u) \prec \zeta_1(u)\zeta_2(u)$$
uniformly in $u\in U$.

(iii) Suppose that $\Psi(u)\ge N^{-C}$ is deterministic and $\xi(u)$ is a nonnegative random variable such that $E\xi(u)^2 \le N^C$ for all $u$. Then if $\xi(u)\prec \Psi(u)$ uniformly in $u$, we have
$$\mathbb E\xi(u) \prec \Psi(u)$$
uniformly in $u$.
\end{lem}

The Girko's Hermitization technique \cite{Girko} can be reformulated as the following (see e.g. \cite{Single_ring}): for any smooth function $g$,
\begin{align}
\frac{1}{N}\sum_{i=1}^N g(\mu_j) & = \frac{1}{4\pi N} \int \Delta g(z) \sum_{j=1}^N \log(\mu_j-z)(\bar \mu_j - \bar z)dA(z) \nonumber\\
& = \frac{1}{4\pi N} \int \Delta g(z) \log \left|\det(Y(z)Y^\dag(z)) \right|dA(z) = \frac{1}{4\pi N} \int \Delta g(z) \sum_{j=1}^N \log \lambda_j(z) dA(z) , \label{eqn_hermitizatio}
\end{align}
where $0\le \lambda_1 \le \lambda_2 \le \ldots \le \lambda_N$ are the ordered eigenvalues of $Y(z)Y^\dag(z)$. For $g=F_{z_0,a}$, we use the new variable $\xi=N^a(z-z_0)$ to write the  above equation as
\begin{align}\label{eigenvalue_delta}
\frac{1}{N}\sum_{i=1}^N F_{z_0,a}(\mu_j) = \frac{N^{-1+2a}}{4\pi } \int (\Delta F)(\xi) \sum_{j=1}^N \log \lambda_j(z) dA(\xi) .
\end{align}
Define the classical location $\gamma_j(z)$ of the $j$-th eigenvalue of $Y(z)Y^\dag(z)$ by
\begin{equation}\label{defn_gammaj}
\int_0^{\gamma_j(z)}\rho_{2c}(x)dx = \frac{j}{N}, \ \ 1\le j \le N.
\end{equation}
By Proposition \ref{prop_rho1c}, we have that for any $\delta>0$
\begin{equation}\label{eigenvalue_law}
\left| \sum_{j=1}^N \log \gamma_j(z) - N\int_0^\infty (\log x) \rho_{2c}(x,z) dx \right| \le \sum_{j=1}^N N \int_{\gamma_{j-1}(z)}^{\gamma_j(z)} \left| \log \gamma_j(z) -\log x \right| \rho_{2c}(x,z)dx \le N^\delta
\end{equation}
for large enough $N$. Suppose we have the bound
\begin{equation}\label{eigen_rigidity}
\left| \sum_j \log \lambda_j - \sum_ j \log \gamma_j \right|\prec N^{b}.
\end{equation}
Plugging (\ref{eigenvalue_law}) and (\ref{eigen_rigidity}) into (\ref{eigenvalue_delta}), we get
\begin{align*}
\frac{1}{N}\sum_{i=1}^N F_{z_0}(\mu_j) & = \frac{N^{ 2a}}{4\pi } \int (\Delta F)(\xi) \int_0^\infty (\log x) \rho_{2c}(x,z) dx  dA(\xi) + O_\prec( N^{-1+b+2a} \|\Delta F\|_{L_1})\\
& = \frac{1}{4\pi } \int F(\xi) \int_0^\infty (\log x) \Delta_z \rho_{2c}(x,z) dx  dA(\xi)+ O_\prec(N^{-1+b+2a} \|\Delta F\|_{L_1}).
\end{align*}
Thus we obtain (\ref{eq_main0}) if we can prove (\ref{eigen_rigidity}) for $b=1/2$, and we obtain (\ref{eq_main}) if we can can prove (\ref{eigen_rigidity}) for $b=0$
when $1+ \tau  \le  {\left| z_0 \right|^2 } \le 1+\tau ^{-1}$ or the assumption (\ref{assm_3rdmoment}) holds.

We need the following lemma which is a consequence of Theorem \ref{law_wideT}. Recall (\ref{support_rho1c}) and (\ref{edge_gap}), the number of components $L$ has order $1$ and each component $[e_{2k}, e_{2k-1}]$ contains order $N$ of $\gamma_j$'s. We define the classical number of eigenvalues to the left of the edge $e_k$, $1\le k \le 2L$, as
\begin{equation}
N_k: = \left\lceil N\int_0^{e_k}  \rho_{2c}(x)\right\rceil.
\end{equation}
Note that $N_{2L}=0$, $N_{1}=N$ and $N_{2k+1}=N_{2k}$, $1 \le k \le L-1$. 

\begin{lem}[Singular value rigidity]
Fix a small $\epsilon>0$. 

(i) If the averaged local law holds with parameters $(T,X,z,\mathbf D(\zeta,N) \cap \widehat {\mathbf D}(\zeta,N))$ for arbitrarily small $\zeta$, then the following estimates hold. For any $e_{2k}>0$ and $N_{2k}+N^{1/2+\epsilon} \le j \le N_{2k-1}-N^{1/2+\epsilon}$,
\begin{equation}\label{eq_rigid20}
\frac{|\lambda_j-\gamma_j|}{\gamma_j} \prec {\left(\min\left\{\frac{j-N_{2k}}{N},\frac{N_{2k-1}-j}{N}\right\}\right)^{-1/3} N^{-1/2}}.
\end{equation}
In the case $|z|^2\le 1-\tau$ with $e_{2L}=0$, we have for any $N_{2L}+N^{1/2+\epsilon} \le j \le N_{2L-1}-N^{1/2+\epsilon}$, 
\begin{equation}\label{eq_rigid10}
\frac{|\lambda_j-\gamma_j|}{\gamma_j} \prec j^{-1} \left(\frac{N_{2L-1}-j}{N}\right)^{-1/3} N^{1/2}.
\end{equation}
Moreover, if $ 1+\tau \le |z|^2 \le 1+\tau^{-1}$, then for any fixed $0< c < e_{2L}$,
\begin{equation}\label{eq_rigid200}
\# \{ j: 0 < \lambda_j < c \}\prec 1.
\end{equation}

(ii) If the averaged local law holds with parameters $(T,X,z,\mathbf D(\zeta,N))$ for arbitrarily small $\zeta$, then the following estimates hold. For any $e_{2k}>0$ and $N_{2k}+N^{\epsilon} \le j \le N_{2k-1}-N^{\epsilon}$,
\begin{equation}\label{eq_rigid2}
\frac{|\lambda_j-\gamma_j|}{\gamma_j} \prec {\left(\min\left\{\frac{j-N_{2k}}{N},\frac{N_{2k-1}-j}{N}\right\}\right)^{-1/3} N^{-1}}.
\end{equation}
In the case $|z|^2\le 1-\tau$ with $e_{2L}=0$, we have for any $N_{2L}+N^{\epsilon} \le j \le N_{2L-1}-N^{\epsilon}$,
\begin{equation}\label{eq_rigid1}
\frac{|\lambda_j-\gamma_j|}{\gamma_j} \prec j^{-1} \left(\frac{N_{2L-1}-j}{N}\right)^{-1/3}.
\end{equation}
\end{lem}
\begin{proof}
The proof is similar to the proof of \cite[Lemma 5.1]{local_circular}. See also \cite[Theorem 2.10]{isotropic} or \cite[Theorem 7.6]{Semicircle}
\end{proof}

Using (\ref{eq_rigid20}) and (\ref{eq_rigid10}), we get that 
\begin{equation}\label{eqn_lambda_b1}
 \sum_{N_{2k}+N^{1/2+\epsilon} \le j \le N_{2k-1}-N^{1/2+\epsilon}} \left|\log \lambda_j - \log \gamma_j  \right| \prec \sum_{N_{2k}+N^{1/2+\epsilon} \le j \le N_{2k-1}-N^{1/2+\epsilon}} \frac{ \left| \lambda_j - \gamma_j  \right|}{\gamma_j}  \prec N^{1/2}.
\end{equation}
Through a standard large deviation estimate, we have the following bound (see e.g. \cite{Handbook_DS,Random_polytopes,RudVersh_rect}), 
\begin{equation}\label{norm_upperbound}
\mathbb P(\|X\|>t) \le e^{-c_0t^2N} \ \text{ for } \ t> C_0,
\end{equation}
where $c_0,C_0>0$ are constants. Thus we have 
\begin{equation}\label{largest_singular}
\lambda_j \le \|Y\|^2 \le  (\|T\|\|X\|+|z|)^2 \prec 1, \ \ 1\le j \le N.
\end{equation}
Together with Lemma \ref{lemm_leastsing} concerning the smallest singular value of $TX-z$, we get
\begin{equation}\label{eqn_lambda_b31}
\sum_{k=1}^{2L} \sum_{|j -e_k|< N^{1/2+\epsilon}} \left|\log \lambda_j \right| \prec N^{1/2+\epsilon}.
\end{equation}
Since $\left|\log \gamma_j\right| \prec 1$ by Proposition \ref{prop_rho1c}, we conclude
\begin{equation}\label{eqn_lambda_b2}
\sum_{k=1}^{2L} \sum_{|j -e_k|< N^{1/2+\epsilon}} \left|\log \lambda_j - \log \gamma_j  \right| \prec N^{1/2+\epsilon}.
\end{equation}
Combining (\ref{eqn_lambda_b1})-(\ref{eqn_lambda_b2}), we get for any $\epsilon>0$,
\begin{equation}
\sum_{1\le j \le N } \left|\log \lambda_j - \log \gamma_j  \right| \prec N^{1/2+\epsilon}
\end{equation}
for large enough $N$. This implies (\ref{eigen_rigidity}) for $b=1/2$. 
If in addition the assumption (\ref{assm_3rdmoment}) holds, the averaged local law holds with parameters $(T,X,z,\mathbf D(\zeta,N))$ for arbitrarily small $\zeta$ by Theorem \ref{law_wideT}. Then we can prove (\ref{eigen_rigidity}) for $b=0$ using the better bounds (\ref{eq_rigid2}) and (\ref{eq_rigid1}). 

Finally we prove that when $|z_0|^2 \ge 1+\tau$, with the bounds (\ref{eq_rigid20}) we can still prove the estimate (\ref{eigen_rigidity}) for $b=0$. By the averaged local law and the definition of $\gamma_j$ in (\ref{defn_gammaj}), we have
\begin{equation}\label{improve_bound}
\left| \sum_{j=1}^N \frac{1}{\lambda_j - i\eta} - \sum_{j=1}^N \frac{1}{\gamma_j - i\eta} \right|\prec \frac{1}{\eta}, 
\end{equation}
uniformly in $N^{-1/2+\epsilon} \le \eta \le N^{1/2}$. Taking integral of (\ref{improve_bound}) over $\eta$ from $N^{-1/2+\epsilon}$ to $N^{1/2}$, we get
\begin{equation} 
\left| \sum_{j=1}^N \log \left( \frac{\lambda_j - iN^{-1/2+\epsilon}} {\gamma_j - iN^{-1/2+\epsilon}} \right)-\sum_{j=1}^N \log \left( \frac{\lambda_j - iN^{1/2}} {\gamma_j - iN^{1/2}} \right)  \right| \prec 1.
\end{equation}
Then we use (\ref{eq_rigid20}) and the bound (\ref{largest_singular}) to estimate that
$$ \left| \sum_{j=1}^N \log \left( \frac{\lambda_j - iN^{1/2}} {\gamma_j - iN^{1/2}} \right) \right| \prec \sum_{j=1}^N \left| \left(\lambda_j - \gamma_j\right) N^{-1/2}\right| \prec N^\epsilon.$$
Thus we conclude 
\begin{equation} \label{improve1}
\left| \sum_{j=1}^N \log \left( \frac{\lambda_j - iN^{-1/2+\epsilon}} {\gamma_j - iN^{-1/2+\epsilon}} \right)  \right| \prec N^\epsilon.
\end{equation}
Using $\gamma_j \sim 1$, (\ref{eq_rigid200}) and (\ref{least_bound}), we get  
\begin{align} 
\left| \sum_{j=1}^N \log \left( \frac{\lambda_j - iN^{-1/2+\epsilon}} {\gamma_j - iN^{-1/2+\epsilon}} \right) - \sum_{j=1}^N \log \frac{\lambda_j  } {\gamma_j }  \right| & \prec 1+ \left| \sum_{\lambda_j \ge c} \log \left( \frac{\lambda_j - iN^{-1/2+\epsilon}} {\gamma_j - iN^{-1/2+\epsilon}} \right) - \sum_{\lambda_j \ge c} \log \frac{\lambda_j  } {\gamma_j }  \right| \nonumber\\
& \prec 1+\sum_{\lambda_j \ge c} \left| \left(\lambda_j - \gamma_j\right) N^{-1/2+\epsilon}\right|  \prec N^{2\epsilon}.\label{improve2}
\end{align}
Combing (\ref{improve1}) and (\ref{improve2}), we conclude (\ref{eigen_rigidity}) for $b=0$.

\begin{lem}[Lower bound on the smallest singular value]\label{lemm_leastsing}
 If $N<M$ and the entries of $X$ have a density bounded by $N^{C_3}$ for some $C_3>0$, then
\begin{equation}
|\log \lambda_1(z)| \prec 1 \label{least_bound}
\end{equation}
holds uniformly for $z$ in any fixed compact set. If $N=M$, the bounded density condition is not necessary.
\end{lem}
\begin{proof}
To prove (\ref{least_bound}), we need to prove that
\begin{equation}
\mathbb P\left(\lambda_1(z)\le e^{-N^\epsilon}\right) \le N^{-C} \label{least_pf}
\end{equation}
for any $\epsilon,C>0$. In the case $N=M$ without the bounded density assumption, we have $\lambda_1(z) \ge \tau \lambda'_1(z),$
where $\lambda'_1(z)$ is the smallest singular values of $X-T^{-1}z$. Following \cite{RudVersh_square} or \cite[Theorem 2.1]{TaoVu_circular}, we have $|\log \lambda'_1(z)| \prec 1$, which further proves (\ref{least_bound}).

Now we turn to the case $N<M$ with the bounded density assumption. By (\ref{TX1}) we have that
$$TX-z= UD(V_1X-D^{-1}U^{-1}z)=:UD\tilde Y(z).$$
Hence it suffices to control the smallest singular value of $\tilde Y(z)$, call it $\tilde \lambda_1(z)$.
Notice the columns $\tilde Y_1, \ldots, \tilde Y_N$ of $\tilde Y(z)$ are independent vectors. From the variational characterization
$$\tilde \lambda_1(z)=\min_{|u|=1}\|\tilde Y(z)u\|^2,$$
we can easily get
\begin{align}
\tilde \lambda_1(z)^{1/2} \ge N^{-1/2} \min_{1\le k \le N} \text{dist}\left(\tilde Y_k,\text{span}\{\tilde Y_l,l\ne k\}\right) = N^{-1/2} \min_{1\le k \le N} \left|\langle \tilde Y_k,u_k\rangle\right|,
\end{align}
where $u_k$ is the unit normal vector of $\text{span}\{\tilde Y_l,l\ne k\}$ and hence is independent of $\tilde Y_k$. By conditioning on $u_k$, we get immediately
\begin{equation}
\mathbb P(\tilde \lambda_1(z)\le N^{-C_0}) \le C N^{-C_0/2 + C_3 + 3/2}, \label{least_density}
\end{equation}
which is a much stronger result than (\ref{least_pf}). Here we have used Theorem 1.2 of \cite{RudVersh_smallball} to conclude that $\langle \bar Y_k,u_k\rangle$ for fixed $u_k$ has density bounded by $CN^{C_3}$.
\end{proof}

\end{subsection}

\begin{subsection}{Outline of the paper}

The rest of this paper is devoted to the proof of Theorems \ref{law_squareD} and \ref{law_wideT}. In Section \ref{section_tools}, we collect the basics tools that we shall use throughout the proof.
In Section \ref{section_weaklaw}, we perform step (A) of the proof by proving the entrywise local law and averaged local law in Theorem \ref{law_squareD} under the assumption that $T$ is diagonal. We first prove a weak version of the entrywise local law in Sections \ref{subsection_selfeqn}-\ref{subsection_weak_proof}, and then improve the weak law to the strong entrywise local law and averaged local law in Sections \ref{subsection_proofstrong}-\ref{section_smallwz}. In Section \ref{section_isotropiclaw}, we perform step (B) of the proof by proving the anisotropic local law in Theorem \ref{law_squareD} using the entrywise local law proved in Section \ref{section_weaklaw}.
Finally in Section \ref{section_comparison} we finish the step (C) of the proof, where using Theorem \ref{law_squareD}, we prove Theorem \ref{law_wideT} with a self-consistent comparison method.

The first part of Appendix \ref{appendix1} establishes the basic properties of $\rho_{1,2c}$ stated in Lemma \ref{lemm_rho} and Proposition \ref{prop_rho1c}. In Sections \ref{subsection_append2} and \ref{subsection_append3}, we establish some key estimates on $m_{1,2c}$ and the stability of the self-consistent equation (\ref{eq_self3}) on regular domains.

\end{subsection}

\end{section}

\begin{section}{Basic tools}\label{section_tools}

In this preliminary section, we collect various identities and estimates that we shall use throughout the following. 

\begin{defn}[Minors]
For $J\subset \mathcal I$, we define the minor $H^{(J)}:=\{H_{st}:s,t \in \mathcal I\setminus J\}$, and correspondingly $G^{(J)}:=(H^{(J)})^{-1}$. Let $[J]:=\{s\in\mathcal I: s \in J \text{ or } \bar s \in J\}$. We also denote $H^{[J]}:=\{H_{st}:s,t \in \mathcal I\setminus [J]\}$ and $G^{[J]}:=(H^{[J]})^{-1}$. We abbreviate $(\{s\})\equiv (s)$, $(\{s, t\})\equiv (st)$, $[\{s\}]\equiv [s] $ and  $[\{s, t \}]=[st]$.
\end{defn}

Notice that by the definition, we have $H_{st}^{(J)}=0$ and $G_{st}^{(J)}=0$ if $s\in J$ or $t\in J$.

\begin{lem}{(Resolvent identities).}

 \begin{itemize}
\item[(i)]
For $i\in \mathcal I_1$ and $\mu\in \mathcal I_2$, we have
\begin{equation}
\frac{1}{{G_{ii} }} =  - w - w \left( {YG^{\left( i \right)} Y^\dag  } \right)_{ii} ,\ \frac{1}{{G_{\mu \mu } }} =  - w  - w \left( {Y^\dag  G^{\left( \mu  \right)} Y} \right)_{\mu \mu }.\label{resolvent2}
\end{equation}
For $i\ne j \in \mathcal I_1$ and $\mu \ne \nu \in \mathcal I_2$, we have
\begin{equation}
G_{ij}   = w G_{ii} G_{jj}^{\left( i \right)} \left( {YG^{\left( {ij} \right)} Y^\dag  } \right)_{ij},\ \ G_{\mu \nu }  = w G_{\mu \mu } G_{\nu \nu }^{\left( \mu  \right)} \left( {Y^\dag  G^{\left( {\mu \nu } \right)} Y} \right)_{\mu \nu }. \label{resolvent3}
\end{equation}

 \item[(ii)]
For $i\in \mathcal I_1$ and $\mu\in \mathcal I_2$, we have
\begin{align}
& G_{i\mu } = G_{ii} G_{\mu \mu }^{\left( i \right)} \left( { - w^{1/2}Y_{i\mu }  +  w{\left( {YG^{\left( {i\mu } \right)} Y} \right)_{i\mu } } } \right),\label{resolvent6} \\
& G_{\mu i}  = G_{\mu \mu } G_{ii}^{\left( \mu  \right)} \left( { - w^{1/2}Y_{\mu i}^\dag   + w\left( {Y^\dag  G^{\left( {\mu i} \right)} Y^\dag  } \right)_{\mu i} } \right).\label{resolvent7}
\end{align}

 \item[(iii)]
 For $r \in \mathcal I$ and $s,t \in \mathcal I \setminus \{r\}$,
\begin{equation}
G_{st}^{\left( r \right)}  = G_{st}  - \frac{{G_{sr} G_{rt} }}{{G_{rr} }}, \ \ \frac{1}{{G_{ss} }} = \frac{1}{{G_{ss}^{(r)} }} - \frac{{G_{sr} G_{rs} }}{{G_{ss} G_{ss}^{(r)} G_{rr} }}. \label{resolvent8}
\end{equation}

 \item[(iv)]
All of the above identities hold for $G^{(J)}$ instead of $G$ for $J\subset \mathcal I$.
\end{itemize}
\label{lemm_resolvent}
\end{lem}
\begin{proof}
All these identities can be proved using Schur's complement formula. They have been previously derived and summarized e.g. in \cite{Semicircle, EdgraphI, Bulk_univ}.
\end{proof}

\begin{lem}{(Resolvent identities for $G_{[ij]}$ groups).}\label{lemm_resolvent_group}
  \begin{itemize}
  \item[(i)] For $i\in \mathcal I_1$, we have
    \begin{equation}\label{eq_res1}
    G_{[ii]}^{ - 1}  = H_{[ii]}-\sum_{k,l \ne i} H_{[ik]}G_{[kl]}^{[i]}H_{[li]}.
    \end{equation}
    For $i\ne j \in \mathcal I_1$, we have
    \begin{align}
    G_{[ij]} & =  - G_{\left[ {ii} \right]} \sum_{k\ne i} H_{[ik]}G^{[i]}_{[kj]}  =  - \sum_{k\ne j} G^{[j]}_{[ik]} H_{[kj]}G_{\left[ {jj} \right]}  \label{eq_res2} \\
    & =  - G_{\left[ {ii} \right]} H_{[ij]}G_{[jj]}^{\left[ i \right]}  + G_{\left[ {ii} \right]}\sum_{k,l \notin \{i,j\}} H_{[ik]}G^{[ij]}_{[kl]}H_{[lj]}G_{[jj]}^{\left[ i \right]} . \label{eq_res2_2}
    \end{align}
  \item[(ii)] For $k \in \mathcal I_1$ and $i,j \in \mathcal I_1\setminus \{k\}$,
    \begin{equation}\label{eq_res3}
    G_{[ij]}^{\left[ k \right]}  = G_{[ij]}  - {G_{[ik]} G^{-1}_{[kk]} G_{[kj]} },
    \end{equation}
    and
    \begin{equation}\label{eq_res4}
    G_{\left[ {ii} \right]}^{ - 1}  = \left( {G_{\left[ {ii} \right]}^{\left[ k \right]} } \right)^{ - 1}  - G_{\left[ {ii} \right]}^{ - 1}G_{\left[ {ik} \right]} G_{\left[ {kk} \right]}^{ - 1} G_{\left[ {ki} \right]} \left( {G_{\left[ {ii} \right]}^{\left[ k \right]} }\right)^{ - 1} .
    \end{equation}
  \item[(iii)] All of the above identities hold for $G^{[J]}$ instead of $G$ for $J\subset \mathcal I$.
  \end{itemize}
\end{lem}
\begin{proof}
These identities can be proved using Schur's complement formula. The details are left to the reader.
\end{proof}

Next we introduce the spectral decomposition of $G$. Let
$$Y = \sum\limits_{k = 1}^N {\sqrt {\lambda_k } \xi_k } \zeta _{\bar k}^\dag$$
be the singular decomposition of $Y$, where $\lambda_1\ge \lambda_2 \ge \ldots \ge \lambda_N \ge 0$ and $\{\xi_{k}\}_{k=1}^{N}$ and $\{\zeta_{\bar k}\}_{k=1}^{N}$ are orthonormal bases of $\mathbb C^{\mathcal I_1}$ and $\mathbb C^{\mathcal I_2}$ respectively. Then by (\ref{eqn_schurmatrix}), we have
\begin{equation}
G\left( w \right) = \sum\limits_{k = 1}^N \frac{1}{\lambda_k-w}\left( {\begin{array}{*{20}c}
   {{\xi _k \xi _k^\dag  }} & {w^{-1/2}\sqrt {\lambda _k } \xi _k \zeta _{\bar k}^\dag}  \\
   {w^{-1/2} \sqrt {\lambda _k } \zeta _{\bar k} \xi _k^\dag  } & {\zeta _{\bar k} \zeta _{\bar k}^\dag }  \\
\end{array}} \right). \label{singular_rep}
\end{equation}

\begin{defn}[Generalized entries]
For $\mathbf v,\mathbf w \in \mathbb C^{\mathcal I}$, $s\in \mathcal I$ and an $\mathcal I\times \mathcal I$ matrix $A$, we shall denote
\begin{equation}
A_{\mathbf{vw}}:=\langle \mathbf v,A\mathbf w\rangle, \ \ A_{\mathbf{v}s}:=\langle \mathbf v,A\mathbf e_s\rangle, \ \ A_{s\mathbf{w}}:=\langle \mathbf e_s,A\mathbf w\rangle,
\end{equation}
where $\mathbf e_s$ is the standard unit vector.
\end{defn}

Given vectors $\mathbf v\in \mathbb C^{\mathcal I_1}$ and $\mathbf w\in \mathbb C^{\mathcal I_2}$, we always identify them with their natural embeddings $\left( {\begin{array}{*{20}c}
   {\mathbf v}  \\
   0 \\
\end{array}} \right)$ and $\left( {\begin{array}{*{20}c}
   0  \\
   \mathbf w \\
\end{array}} \right)$ in $\mathbb C^{\mathcal I}$.
The exact meanings will be clear from the context.

\begin{lem}
Fix $\tau >0$. The following estimates hold uniformly for any $w\in \bD(\tau,N)$. We have
\begin{equation}
\left\| G \right\| \le C\eta ^{ - 1} ,\ \left\| {\partial _w G} \right\| \le C\eta ^{ - 2}. \label{eq_gbound}
\end{equation}
Let $\mathbf v \in \mathbb C^{\mathcal I_1}$ and $\mathbf w \in \mathbb C^{\mathcal I_2}$, we have the bounds
\begin{align}
& \sum\limits_{\mu  \in \mathcal I_2 } {\left| {G_{\mathbf w\mu } } \right|^2 } = \sum\limits_{\mu  \in \mathcal I_2 } {\left| {G_{\mu \mathbf w} } \right|^2 }  = \frac{{\Im \, G_{\mathbf w\mathbf w} }}{\eta },\label{eq_gsq1} \\
& \sum\limits_{i \in \mathcal I_1 }  \left| {G_{\mathbf v i} } \right|^2 = \sum\limits_{i \in \mathcal I_1 }  \left| {G_{i\mathbf v} } \right|^2  = \frac{\Im \, G_{\mathbf v\mathbf v}}{\eta} , \label{eq_gsq2} \\
& \sum\limits_{i \in \mathcal I_1 } {\left| {G_{\mathbf wi} } \right|^2 } = \sum\limits_{i \in \mathcal I_1 } {\left| {G_{i\mathbf w} } \right|^2 } = \left| w \right|^{ - 1} {G}_{\mathbf w\mathbf w}  + \bar w\left| w \right|^{ - 1} \frac{{\Im \, G_{\mathbf w\mathbf w} }}{\eta } , \label{eq_gsq3} \\
& \sum\limits_{\mu \in \mathcal I_2 } {\left| {G_{\mathbf v \mu} } \right|^2 } = \sum\limits_{\mu \in \mathcal I_2 } {\left| {G_{\mu \mathbf v} } \right|^2 } = \left| w \right|^{ - 1} {G}_{\mathbf v\mathbf v}  + \bar w\left| w \right|^{ - 1} \frac{{\Im \, G_{\mathbf v\mathbf v} }}{\eta } . \label{eq_gsq4}
 \end{align}
All of the above estimates remain true for $G^{(J)}$ instead of $G$ for $J\subset \mathcal I$. 
\label{lemma_Im}
\end{lem}
\begin{proof}
The estimates in (\ref{eq_gbound}) follow from (\ref{singular_rep}). For any unit vectors $\mathbf {x,y} \in \mathbb C^{\mathcal I_1}$, we have
$$\left|\left\langle \mathbf x, G\mathbf y\right\rangle \right| \le \sum\limits_{k = 1}^N {\frac{{\left|\langle \mathbf x, \xi _k\rangle\right| \left| \langle \xi _k^\dag, \mathbf y\rangle \right| }}{\left|\lambda _k  - w \right|}} \le \frac{1}{\eta}\left[ \sum\limits_{k = 1}^N \left|\langle \mathbf x, \xi _k\rangle\right|^2 \right]^{1/2} \left[\sum\limits_{k = 1}^N\left| \langle \xi _k^\dag, \mathbf y\rangle \right|^2 \right]^{1/2} =\frac{1}{\eta}.$$
For any unit vectors $\mathbf x\in \mathbb C^{\mathcal I_1}$ and $\mathbf y \in \mathbb C^{I_2}$, we have
\begin{align*}
\left|\langle \mathbf x, G\mathbf y\rangle \right| \le \left|w\right|^{-1/2}\sum\limits_{k = 1}^N \frac{{\sqrt {\lambda _k } \left|\langle \mathbf x,\xi _{ k}\rangle \right| \left|\langle \zeta _{\bar k}^\dag, \mathbf y\rangle  \right|}}{\left|\lambda _k  - w \right|} \le \sum\limits_{k = 1}^N \frac{1}{2\eta}\left(\left|\langle \mathbf x,\xi _{ k}\rangle \right|^2 + \left|\langle \zeta _{\bar k}^\dag, \mathbf y\rangle  \right|^2\right)=\frac{1}{\eta},
\end{align*}
where we have used that for $w=E+i\eta$, $\left|w\right|^{-1/2}{\sqrt {\lambda _k }}/{\left|\lambda _k  - w \right|} \le  {\eta}^{-1}.$
For the other two blocks of $G$, we can prove similar estimates. This implies (\ref{eq_gbound}).
It is trivial to generalize the proof to $\partial_w G$, where $\eta^{-2}$ comes from the $(\lambda_k-w)^{-2}$ factor of $\partial_w G$.
For (\ref{eq_gsq1}), we observe that
$$ \frac{{{\mathop{\rm Im}\nolimits} \, G_{\mathbf w\mathbf w} }}{\eta } = \frac{1}{\eta }{\mathop{\rm Im}\nolimits} \sum\limits_{k = 1}^N {\frac{{\left\langle {\mathbf w,\zeta _{\bar k} } \right\rangle \langle {\zeta _{\bar k}^\dag  ,\mathbf w} \rangle }}{{\lambda _k  - w}}}  = \sum\limits_{k = 1}^N {\frac{{\left| {\left\langle {\mathbf w,\zeta _k } \right\rangle } \right|^2 }}{{\left( {\lambda _k  - E} \right)^2  + \eta ^2 }}}$$
and by (\ref{eqn_schurmatrix})
\begin{align}\label{eqn_interstep}
\sum\limits_{\mu  \in \mathcal I_2 } {\left| {G_{\mathbf w\mu } } \right|^2 } = \sum\limits_{\mu  \in \mathcal I_2 } \left\langle \mathbf w, G _R{e_\mu  } \right\rangle \langle e_\mu, G_R^\dag \mathbf w \rangle  = {\left\langle \mathbf w ,G_RG_R^\dag  \mathbf w \right\rangle} = \sum\limits_{k = 1}^N {\frac{{\left| {\left\langle {\mathbf w,\zeta _k } \right\rangle } \right|^2  }}{{\left( {\lambda _k  - E} \right)^2  + \eta ^2 }} }  .
\end{align}
Similarly, we can prove the identity for $\sum\limits_{\mu  \in \mathcal I_2 } {\left| {G_{\mu \mathbf w} } \right|^2 }$ and (\ref{eq_gsq2}). For identity (\ref{eq_gsq3}), first we can prove $\sum\limits_{i \in \mathcal I_1 } {\left| {G_{\mathbf wi} } \right|^2 } = \sum\limits_{i \in \mathcal I_1 } {\left| {G_{i\mathbf w} } \right|^2 }$ using (\ref{singular_rep}). 
Then we use (\ref{eqn_schurmatrix}) and (\ref{eqn_interstep}) to get
\begin{align}
 \sum\limits_{i \in \mathcal I_1 } {\left| {G_{\mathbf wi} } \right|^2 } & = \left| w \right|^{ - 1} \left( {G_R Y^\dag  YG_R^\dag  } \right)_{\mathbf w\mathbf w}  = \left| w \right|^{ - 1} \left[ {G_R \left( {Y^\dag  Y - \bar w} \right)G_R^\dag  } \right]_{\mathbf w\mathbf w}  + \bar w\left| w \right|^{ - 1} \left( {G_R G_R^\dag  } \right)_{\mathbf w\mathbf w}  \nonumber\\
& = \left| w \right|^{ - 1}  {G}_{\mathbf w\mathbf w}  + \bar w\left| w \right|^{ - 1} \left( {G_R G_R^\dag  } \right)_{\mathbf w\mathbf w}  = \left| w \right|^{ - 1} {G}_{\mathbf w\mathbf w}  + \bar w\left| w \right|^{ - 1} \frac{{{\mathop{\rm Im}\nolimits} \, G_{\mathbf w\mathbf w} }}{\eta } .
 \end{align}
Identity (\ref{eq_gsq4}) can be proved in a similar way. 
\end{proof}

The following Lemma give useful large deviation bounds. See Theorem B.1 and Lemmas B.2-B.4 in \cite{Delocal} for the proof. See also Theorem C.1 of \cite{Semicircle}.
\begin{lem}\label{large_deviation}
(Large deviation bounds) Let $(X_i^{(N)})$, $(Y_i^{(N)})$ be independent families of random variables and $(a_{ij}^{(N)})$, $(b_{i}^{(N)})$ be deterministic. Suppose all entries $X_i^{(N)}$ and $Y_i^{(N)}$ are independent and satisfies (\ref{assm1}) and (\ref{assm2}). Then we have the following bounds:
\begin{equation}
\sum\limits_i {b_i X_i }  \prec \frac{\left( {\sum\limits_i {\left| {b_i^{} } \right|^2 } } \right)^{1/2}}{\sqrt{N}} ,\ \ \sum\limits_{i,j} {a_{ij} X_i } Y_j  \prec \frac{\left( {\sum\limits_{i,j} {\left| {a_{ij} } \right|^2 } } \right)^{1/2}}{N} ,\ \ \sum\limits_{i \ne j} {a_{ij} X_i } Y_j  \prec \frac{\left( {\sum\limits_{i \ne j} {\left| {a_{ij} } \right|^2 } } \right)^{1/2}}{N}.
\end{equation}
If the coefficients $(a_{ij}^{(N)})$ and $(b_{i}^{(N)})$ depend on some parameter $u$, then all of the above estimates are uniform in $u$.
\end{lem}

We have stated some basic properties of $\rho_{1,2c}$ and $m_{1,2c}$ in Lemma \ref{lemm_rho} and Proposition \ref{prop_rho1c}. Now we collect more estimates for $m_{1,2c}$ that will be used in the proof. The next lemma is proved in Appendix \ref{subsection_append2}. For $w = E+i\eta \in \mathbf D$, we define the distance to the spectral edge through
\begin{equation}\label{eqn_def_kappa}
\kappa \equiv \kappa(E):= \min_{1\le k\le 2L, e_k >0} |E -e_k| .
\end{equation}
Notice in the $|z|<1$ case, we do not take into consideration the edge at $e_{2L}=0$.

\begin{lem}\label{lemm_m1_4case}

Fix $\tau>0$ and suppose $\tau \le ||z|^2-1|\le\tau^{-1}$. We denote $w = E+i\eta$.

\begin{itemize}

\item[Case 1]
Fix $\tau'>0$. Suppose the bulk component $[e_{2k},e_{2k-1}]$ is regular in the sense of Definition \ref{def_regular}. Then for $w\in \mathbf D_{k}^b(\zeta,\tau',N)$, we have
\begin{equation}
|1+m_{1c}| \sim {\rm{Im}}\, m_{1c} \sim 1, \ |m_{2c}|\sim {\rm{Im}}\, m_{2c} \sim 1.  \label{estimate1_bulk}
\end{equation}

\item[Case 2]
Fix $\tau'>0$. Then for $w\in \mathbf D^o(\zeta,\tau',N)$, we have
\begin{equation}
{\rm{Im}}\, m_{1,2c} \sim \eta, \ |1+m_{1c}| \sim 1,\ |m_{2c}|\sim 1.  \label{estimate1_out}
\end{equation}

\item[Case 3]
Suppose $e_k\ne 0$ is a regular edge. Then for $w\in \mathbf D_k^e(\zeta,\tau',N)$, if $\tau'>0$ is small enough,
\begin{equation}
{\rm{Im}}\, m_{1,2c} \sim
  \begin{cases}
      \sqrt{\kappa + \eta}    \hfill & \text{ if $E\in {\rm{supp}}\, \rho_{1,2c}$} \\
      \eta/\sqrt{\kappa+\eta} \hfill & \text{ if $E\notin {\rm{supp}}\, \rho_{1,2c}$} \\
  \end{cases}, \ |1+m_{1c}| \sim 1, \ |m_{2c}|\sim 1. \  \label{estimate1_edge}
\end{equation}

\item[Case 4]
Suppose $|z|^2\le 1-\tau$. We take $e_{2L}=0$ and $\tau'>0$ to be small enough. Then for $w\in \mathbf D_{2L}^e(\zeta,\tau',N)$, if $\Im\, w\ge \tau'$, we have
\begin{equation}
|1+m_{1c}| \sim {\rm{Im}}\, m_{1c} \sim 1, \ |m_{2c}|\sim {\rm{Im}}\, m_{2c} \sim 1; \label{estimate11_0}
\end{equation}
if $|w|\le 2\tau'$, we have
\begin{equation}
m_{1c} = i\frac{\sqrt{t}} {\sqrt{w}}+ O(1), \ m_{2c}= \frac{i\sqrt{t}}{\sqrt{w}(t+|z|^2)}+ O(1), \label{estimate12_0}
\end{equation}
for some constant $t >0$, and
\begin{equation}
{\rm{Im}}\,  m_{1,2c} \sim |w|^{-1/2}.  \label{estimate13_0}
\end{equation}

\item[Case 5]
For $w\in \mathbf D_L(\zeta)$, we have
\begin{equation}
|m_{1c}| \sim {\rm{Im}}\, m_{1c} \sim \frac{1}{\eta}, \ |m_{2c}|\sim {\rm{Im}}\, m_{2c} \sim \frac{1}{\eta}. \label{estimate1L}
\end{equation}
\end{itemize}

In Cases 1-4, we have
\begin{equation}
\left|{w\left( {1 + s_i m_{2c}} \right)(1+m_{1c}) - \left| z \right|^2 }\right| \ge c, \label{estimate2_bulk}
\end{equation}
where $c>0$ is some constant that may depend on $\tau$ and $\tau'$. In Case 5, we have
\begin{equation}
\left|{w\left( {1 + s_i m_{2c}} \right)(1+m_{1c}) - \left| z \right|^2 }\right| \ge \eta, \label{estimate2_large}
\end{equation}
\end{lem}

Note that the uniform bounds (\ref{estimate2_bulk}) and (\ref{estimate2_large}) guarantee that the matrix entries of $\Pi(w)$ remain bounded. We have the following Lemma, which is prove in Appendix \ref{subsection_append2}.

\begin{lem}\label{cor_basicestimate}
In Cases 1-4 of Lemma \ref{lemm_m1_4case}, we have
\begin{equation}
\|\pi_{[i]c}\|\le C |w|^{-1/2}, \ \ \left\|\left(\pi_{[i]c}\right)^{-1} \right\| \le C|w|^{1/2}, \label{estimate_Piw12}
\end{equation}
and in Case 5 of Lemma \ref{lemm_m1_4case}, we have
\begin{equation}
\|\pi_{[i]c}\|\le {C}{\eta}^{-1}, \ \ \left\|\left(\pi_{[i]c}\right)^{-1} \right\| \le C\eta. \label{estimate_PiwL}
\end{equation}
For all the cases in Lemma \ref{lemm_m1_4case}, 
\begin{equation}
\Im\, \Pi_{\mathbf v\mathbf v} \le C \Im(m_{1c}+m_{2c}),  \label{estimate_PiImw}
\end{equation}
uniformly in $w$ and any deterministic unit vector $\mathbf v \in \mathbb C^{\mathcal I}$.
\end{lem}

The self-consistent equation (\ref{eq_self3}) can be written as
\begin{equation}
\Upsilon(w,m_1)=0, \label{stab_Dm1}
\end{equation}
where
\begin{equation}\label{def_stabD}
\Upsilon(w,m_1) = m_1 + \frac{1}{N}\sum_{i=1}^n l_i s_i (1+m_1)\left[{w\left( {1 + s_i \frac{1+m_1}{- w (1 + m_1)^2  + \left| z \right|^2 }} \right)(1+m_1) - \left| z \right|^2 } \right]^{-1}.
\end{equation}
The stability of (\ref{stab_Dm1}) roughly says that if $\Upsilon(w,m_1)$ is small and $m_1(w')-m_{1c}(w')$ is small for $w':= w+ iN^{-10}$, then $m_1(w)-m_{1c}(w)$ is small. For an arbitrary $w\in \mathbf D$, we define the discrete set
\begin{align}\label{eqn_def_L}
L(w):=\{w\}\cup \{w'\in \mathbf D: \text{Re}\, w' = \text{Re}\, w, \text{Im}\, w'\in [\text{Im}\, w, 1]\cap (N^{-10}\mathbb N)\} ,
\end{align}
Thus, if $\text{Im}\, w \ge 1$ then $L(w)=\{w\}$, and if $\text{Im}\, w<1$ then $L(w)$ is a 1-dimensional lattice with spacing $N^{-10}$ plus the point $w$. Obviously, we have $|L(w)|\le N^{10}$.

\begin{defn}[Stability of (\ref{stab_Dm1})] \label{def_stability}
We say that (\ref{stab_Dm1}) is stable on $\mathbf D$ if the following holds. Suppose that $N^{-2}|m_{1c}| \le \delta(w) \le (\log N)^{-1}|m_{1c}|$ for $w\in \mathbf D$ and that $\delta$ is Lipschitz continuous with Lipschitz constant $\le N^4$. Suppose moreover that for each fixed $E$, the function $\eta \mapsto \delta(E+i\eta)$ is non-increasing for $\eta>0$. Suppose that $u_1:\mathbf D\to \mathbb C$ is the Stieltjes transform of a positive integrable function. Let $w\in \mathbf D$ and suppose that for all $w'\in L(w)$ we have 
\begin{equation}\label{Stability0}
\left|\Upsilon(w,u_1)\right| \le \delta(w).
\end{equation}
Then
\begin{equation}
\left|u_1(w)-m_{1c}(w)\right|\le \frac{C\delta}{\sqrt{\kappa+\eta+\delta}},\label{Stability1}
\end{equation}
for some constant $C>0$ independent of $w$ and $N$.

We say that (\ref{stab_Dm1}) is stable on $\mathbf D_L$ if for $0\le \delta(w) \le (\log N)^{-1}|m_{1c}|$, (\ref{Stability0}) implies 
\begin{equation}
\left|u_1(w)-m_{1c}(w)\right|\le {C\delta},\label{StabilityL}
\end{equation}
for some constant $C>0$ independent of $w$ and $N$.
\end{defn}

This stability condition has previously appeared in \cite{isotropic,local_circular,Anisotropic}. In \cite{Anisotropic}, for example, the stability condition was established under various regularity assumptions. In the following lemma, we establish the stability on each regular domain.
The proof is presented in Appendix \ref{subsection_append3}. This lemma leaves the case $|w|^{1/2}+|z|^2 = o(1)$ alone. We will handle this case in a different way in Section \ref{section_smallwz}.

\begin{lem}\label{lemm_stability} Fix $\tau>0$ and let $\tau'>0$ be sufficiently small depending on $\tau$. Let $\tau \le ||z|^2-1|\le\tau^{-1}$.
\begin{itemize}
\item[Case 1]
Suppose the bulk component $[e_{2k},e_{2k-1}]$ is regular in the sense of Definition \ref{def_regular}. Then (\ref{stab_Dm1}) is stable on $\mathbf D_{k}^b(\zeta,\tau',N)$ in the sense of Definition \ref{def_stability}.

\item[Case 2]
(\ref{stab_Dm1}) is stable on $ \mathbf D^o(\zeta,\tau',N)$ in the sense of Definition \ref{def_stability}.

\item[Case 3]
Suppose $e_k\ne 0$ is a regular edge in the sense of Definition \ref{def_regular}. Then (\ref{stab_Dm1}) is stable on $\mathbf D_k^e(\zeta,\tau',N)$ in the sense of Definition \ref{def_stability}.

\item[Case 4]
Suppose $|z|^2\le 1-\tau$ and $e_{2L}=0$. If $|w|^{1/2}+|z|^2 \ge \epsilon$ for some constant $\epsilon>0$, then (\ref{stab_Dm1}) is stable on $\mathbf D_{2L}^e(\zeta,\tau',N)$ in the sense of Definition \ref{def_stability}.

\item[Case 5]
(\ref{stab_Dm1}) is stable on $ \mathbf D_L(\zeta)$ in the sense of Definition \ref{def_stability}.

\end{itemize}
\label{lemma_bulk}
\end{lem}


\end{section}

\begin{section}{Entrywise local law when $T$ is diagonal} \label{section_weaklaw}

In this section we prove the entrywise local law and averaged local law in Theorem \ref{law_squareD} when $T$ is diagonal. The proof is similar to the previous proofs of entrywise locals laws in e.g. \cite{isotropic,principal, local_circular, Anisotropic}. We basically follow the ideas in \cite{local_circular}, and we will provide necessary details for the parts that are different from the previous proofs.

The main novel observation of this section is that the self-consistent equations (\ref{eq_self1}) and (\ref{eq_self2}) can be ``derived" from the random matrix model by an application of Schur's complement formula. It is helpful to give a heuristic argument here. We introduce the conditional expectation
\[\bbE_{[i]}[\cdot]:=\bbE[\cdot\mid H^{[i]}],\]
i.e. the partial expectation in the randomness of the $i$ and $\bar i$-th rows and columns of $H$. For the diagonal $G_{[ii]}$ group, we ignore formally the random fluctuations in (\ref{eq_res1}) to get that
\begin{align}
G_{[ii]}^{ - 1}  & \approx \mathbb E_{[i]} H_{[ii]}-\sum_{k,l \ne i} \mathbb E_{[i]} \left(H_{[ik]}G_{[kl]}^{[i]}H_{[li]}\right) = \left(\begin{matrix}-w& -w^{1/2}z\\
 -w^{1/2}\bar z& -w\\
\end{matrix}\right)-\frac w N\sum_{k} \left(\begin{matrix} |d_i|^2G^{[i]}_{\bar k \bar k} & 0 \\
0& |d_k|^2G^{[i]}_{k k}\\
\end{matrix}\right) \nonumber\\
& = \left(\begin{matrix}-w& -w^{1/2}z\\
 -w^{1/2}\bar z & -w\\
\end{matrix}\right)-w \left(\begin{matrix} |d_i|^2 m_{2}& 0\\
0& m_{1}\\
\end{matrix}\right), \label{selfcons}
\end{align}
where we use the definition of $m_1$ and $m_2$ in (\ref{def_M}).
The $11$ entry of (\ref{selfcons}) gives the equation
\begin{align}\label{self_dervGii}
G_{ii}  \approx \frac{{ - 1 - m_{1} }}{{w\left( {1 + |d_i |^2 m_{2} } \right)\left( {1 + m_{1} } \right) - \left| z \right|^2 }} ,
\end{align}
from which we get that
$$ G_{ii} \left[-w\left(1 + |d_i |^2 m_2\right) + \frac{|z|^2}{1+m_1}\right] \approx 1.$$
Summing over $i$ and using that $N^{-1}\sum_i G_{ii} = N^{-1} \sum_\mu G_{\mu\mu}=m_2 $, the above equation becomes
$$ -w\left(m_2 + m_1 m_2\right) + \frac{|z|^2 m_2}{1+m_1} \approx 1,$$
which gives (\ref{eq_self1}). Multiplying (\ref{self_dervGii}) with $|d_i|^2$ and summing over $i$, we get the self-consistent equation (\ref{eq_self2}). In this section we give a justification of these approximations.

Before we start the proof, we make the following remark. In this section we mainly focus on the domain $\mathbf D$. On the domain $\mathbf D_L$, the proofs are much simpler and we only describe them briefly. The parameter $z$ can be either inside or outside of the unit circle. Recall Lemmas \ref{lemm_m1_4case} and \ref{lemm_stability}, the domain $\mathbf D$ of $w$ can be divided roughly into four cases: $w$ near a {\it nonzero} regular edge, $w\to 0$, $w$ in the bulk, or $w$ outside the spectrum. In this section we will only consider the case $|z|^2 \le 1-\tau$ since it covers all four different behaviors. Notice in this case $|m_{1,2c}(w)| \sim |w|^{-1/2}$ for $w$ in any compact set of $\mathbb C_+$ by Proposition \ref{prop_roughbound}. Also due to the remark above Lemma \ref{lemm_stability}, in Sections \ref{subsection_selfeqn}-\ref{subsection_proofstrong}, we assume $|w|^{1/2}+|z|^2 \ge c$ for some $c>0$. We will handle the $|w|^{1/2}+|z|^2=o(1)$ case in Section \ref{section_smallwz}.

\begin{subsection}{The self-consistent equations}\label{subsection_selfeqn}
To begin with, we prove the following weak version of the entrywise local law.
\begin{prop}[Weak entrywise law]\label{thm_weaklaw}
Fix $\left|z \right|^2\le 1-\tau$ and a small constant $c>0$. Suppose Assumption \ref{main_assump} holds, $N=M$ and $T\equiv D:=diag(d_1,...,d_N)$. Then for any regular domain $\bS\subset \bD$,
\begin{equation}
\mathop {\max }\limits_{i,j \in \sI_1} \left\| {\left( {G(w) - \Pi(w) } \right)_{[ij]} } \right\| \prec \frac{1}{|w|^{1/2}}\left(\frac{|w|^{1/2}}{N\eta}\right)^{1/4}
\end{equation}
for all $w\in \bS$ such that $|w|^{1/2}+|z|^2 \ge c$. For $w\in \bD_L$, we have
\begin{equation}\label{weak_lawL}
\mathop {\max }\limits_{i,j \in \sI_1} \left\| {\left( {G(w) - \Pi(w) } \right)_{[ij]} } \right\| \prec \frac{1}{\eta}\sqrt{\frac{1}{N}}.
\end{equation}
\end{prop}

For the purpose of proof, we define the following random control parameters.

\begin{defn}[Control parameters]
Suppose $N=M$ and $T\equiv D:=diag(d_1,...,d_N)$. We define
\begin{equation}\label{eqn_randomerror}
\Lambda : = \mathop {\max }\limits_{i,j \in \sI_1} \left\| {\left( {G - \Pi } \right)_{[ij]} } \right\|,\ \ \Lambda _o : = \mathop {\max }\limits_{i \ne j \in \sI_1} \left\| {\left( {G - \Pi } \right)_{[ij]} } \right\|.
\end{equation}
For $J\subseteq \mathcal I$, define the averaged variables $m^{(J)}_{1,2}$ ($m_{1,2}^{[J]}$) by replacing $G$ in (\ref{def_M}) with $G^{(J)}$ ($G^{[J]}$), i.e.
\begin{equation}
m^{(J)}_1:=\frac{1}{N}\sum_{i\notin J} |d_i|^2 G^{(J)}_{ii}, \ \ m^{(J)}_2:=\frac{1}{N}\sum_{\mu \notin J} G^{(J)}_{\mu\mu}. \label{def_M2}
\end{equation}
The averaged error and the random control parameter are defined as 
\begin{equation}\label{eq_defpsitheta}
\theta:=|m_1-m_{1c}|+|m_2-m_{2c}|, \ \ \Psi _\theta  : = \sqrt {\frac{{\Im \left(m_{1c} + m_{2c} \right)  + \theta }}{{N\eta }}} + \frac{1}{N\eta}.
\end{equation}
\end{defn}

\noindent{\it{Remark:}} By (\ref{assm3}), we immediately get that
\begin{equation}\label{eq_Imm12}
\tau \Im\, m^{(J)}_1 \le \Im\, m^{(J)}_2 \le \tau^{-1}\Im\, m^{(J)}_2,
\end{equation}
and $\theta=O(\Lambda)$, since $|m_1-m_{1c}|\le \tau^{-1}\Lambda, \ \ |m_2-m_{2c}|\le \Lambda.$

\vspace{5pt}

We introduce the $Z$ variables
\begin{equation*}
  Z_{[i]}^{[J]}:=(1-\bbE_{[i]})\left(G_{[ii]}^{[J]}\right)^{-1}.
\end{equation*}
By the identity (\ref{eq_res1}) we have
\begin{equation}
G_{[ii]}^{-1}  =  \mathbb E_{[i]} G_{[ii]}^{-1} + Z_{[i]} = \left( {\begin{array}{*{20}c}
   { - w - w\left|d_i\right|^2m^{[i]}_2 } & {-w^{1/2}z}  \\
   {-w^{1/2}\bar z} & { - w - wm^{[i]}_1} \end{array}} \right) + Z_{[i]}, \label{resolvent_Gii}
\end{equation}
where
\begin{equation}\label{eqn_Zdef}
Z_{[i]} = w \left( {\begin{array}{*{20}c}
   { |d_i|^2 m^{[i]}_2 - |d_i|^2\left( {XG^{\left[ {i} \right]} X^\dag  } \right)_{ii} } & {w^{-1/2} d_i X_{i\bar i} -   \left( {DXG^{\left[ {i} \right]} DX} \right)_{i\bar i} }  \\
   {w^{-1/2} \bar d_i X_{\bar ii}^\dag - \left( {X^\dag  D^\dag G^{\left[ {i} \right]} X^\dag  D^\dag } \right)_{\bar ii} } & { m^{[i]}_1 - \left( {X^\dag  D^\dag G^{\left[ {i} \right]} D X} \right)_{\bar i\bar i} }  \\
\end{array}} \right).
\end{equation}
\begin{lem}
For $J\subseteq \mathcal I_1$, the following crude bound on the difference between $m_a$ and $m_a^{[J]}$ ($a=1,2$) holds:
\begin{equation}
\left| {m_a  - m_a^{\left[ J \right]} } \right| \le \frac{{C\left| J \right|}}{{N\eta }},\ \ a= 1,2,\label{m_T}
\end{equation}
where $C=C(\tau)$ is a constant depending only on $\tau$.
\end{lem}
\begin{proof}
For $i\in\sI_1$, we have
\begin{align}
|m_1-m_1^{(i)}|& =\frac{1}{N}\left|\sum_{k\in\sI_1}  |d_k|^2\frac{G_{ki}G_{ik}}{G_{ii}}\right| \le \frac{\tau^{-1}}{N|G_{ii}|} \sum_{k\in\sI_1} |G_{ik}|^2 = \frac{\tau^{-1}}{N\eta}\frac{\Im \, G_{ii}}{\left|G_{ii}\right|}\le \frac{\tau^{-1}}{N\eta}  \label{rough_boundmi}
\end{align}
where in the first step we use (\ref{resolvent8}), in the second and third steps the equality (\ref{eq_gsq2}). Similarly, using (\ref{resolvent8}) and (\ref{eq_gsq3}) we get
\begin{align*}
|m_1^{(i)}-m_1^{(i\bar i)}| & = \frac{1}{N}\left|\sum_{k\in\sI_1}|d_k|^2\frac{G^{(i)}_{k\bar i}G^{(i)}_{\bar ik}}{G^{(i)}_{\bar i\bar i}}\right| \le \frac{\tau^{-1}}{N|G^{(i)}_{\bar i\bar i}|} \left( \frac{{G}^{(i)}_{\bar i\bar i}}{|w|}  + \frac{\bar w}{\left| w \right|} \frac{{\Im \, G^{(i)}_{\bar i\bar i} }}{\eta }\right)   \le \frac{2\tau^{-1}}{N\eta}.
\end{align*}
By induction on the indices in $[J]$, we can prove (\ref{rough_boundmi}). The proof for $m_2$ is similar.
\end{proof}

\begin{lem}\label{Z_lemma}
Suppose $|z|^2 \le 1-\tau$. For $i\in \mathcal I_1$, we have
\begin{align}
& |\left(Z_{[i]}\right)_{11}| \prec  \left|w\right|\sqrt {\frac{{\Im \, m_2^{\left[i\right]} }}{{N\eta }}}, \ \  |\left(Z_{[i]}\right)_{22}| \prec  \left|w\right|\sqrt {\frac{{\Im\, m_1^{\left[i\right]} }}{{N\eta }}},\label{Z_lemma1}\\
& |\left(Z_{[i]}\right)_{st}| \prec \left|w\right|\left(\frac{\left|w\right|^{-1/2}}{\sqrt{N}}+ \sqrt {\frac{|m_1^{\left[ i \right]}|}{{N \left|w\right|}}}+\sqrt {\frac{{\Im\, m_1^{\left[ i \right]} }}{{N\eta }}}\right) \text{ for } s\ne t \in \{1,2\}, \label{Z_lemma2}
\end{align}
uniformly in $w\in \mathbf D \cup \mathbf D_L$. In particular, these imply that
\begin{equation}
Z_{[i]} \prec |w| \Psi_\theta, \label{Zestimate}
\end{equation}
uniformly in $w\in \mathbf D$, and
\begin{equation}
Z_{[i]} \prec |w| (N\eta)^{-1/2}, \label{ZestimateL}
\end{equation}
uniformly in $w\in \mathbf D_L$. 
\end{lem}
\begin{proof}
Apply the large deviation Lemma \ref{large_deviation} to $Z_{[i]}$ in (\ref{eqn_Zdef}), we get that
\begin{align*}
\left| \frac{\left(Z_{[i]}\right)_{11} }{w} \right| \prec & \frac{1 }{N}\left[ {\left( {\sum\limits_\mu {\left| {G_{\mu\mu}^{[i]} } \right|^2 } } \right)^{1/2}  + \left( {\sum\limits_{\mu \ne \nu} {\left| {G_{\mu\nu}^{[i]} } \right|^2 } } \right)^{1/2} } \right] \le \frac{C }{N}\left( {\sum\limits_{\mu,\nu} {\left| {G_{\mu\nu}^{[i]} } \right|^2 } } \right)^{1/2} \\
  = & \frac{C }{N}\left( {\sum\limits_\mu \frac{{\mathop{\rm Im}\nolimits}\, G_{\mu\mu}^{[i]} }{\eta} } \right)^{1/2} = C \sqrt { \frac{ \Im\, m_2^{[i]}  } {N\eta} } .
\end{align*}
where in the third step we use the equality (\ref{eq_gsq1}). Similarly we can prove the bound for $\left(Z_{[i]}\right)_{22}$ using Lemma \ref{large_deviation} and (\ref{eq_gsq2}). 
Now we consider $\left(Z_{[i]}\right)_{12}$. First, we have $X_{i\bar i} \prec N^{-1/2}$ by (\ref{assm2}). For the other part, we use Lemma \ref{large_deviation} and (\ref{eq_gsq4}) to get that
\begin{align}
\left| {\left( {DXG^{\left[ i \right]} DX} \right)_{i\bar i} } \right| & \prec \frac{1}{N}\left( {\sum\limits_{j,\mu } {\left|d_j\right|^2\left| {G_{\mu j}^{\left[ i \right]} } \right|^2 } } \right)^{1/2}   = \frac{1}{N}\left[ {\sum\limits_j {\left|d_j\right|^2\left( {\left| w \right|^{ - 1} G_{jj}^{\left[ i \right]}  + \frac{{\bar w}}{{\left| w \right|}}\frac{{\Im\, G_{jj}^{\left[ i \right]} }}{\eta }} \right)} } \right]^{1/2}  \nonumber\\
& \le \left[ {\frac{{ | {m_1^{\left[ i \right]} } |}}{{N\left| w \right|}} + \frac{{\Im\, m_1^{\left[ i \right]} }}{{N\eta }}} \right]^{1/2}  \le C\left( {\sqrt {\frac{{ | {m_1^{\left[ i \right]} } |}}{{N\left| w \right|}}}  + \sqrt {\frac{{\Im\, m_1^{\left[ i \right]} }}{{N\eta }}} } \right).
\end{align}
Similarly we can prove the estimate for $\left(Z_{[i]}\right)_{21}$.

Now we prove (\ref{Zestimate}). By the definitions (\ref{eq_defpsitheta}) and using (\ref{m_T}), we get that
\begin{align}
\left|\left(Z_{[i]}\right)_{11}\right| & \prec |w|\sqrt{\frac{{\Im\, m_2^{[i]} }}{N\eta} } = |w|\sqrt {\frac{{\Im\,m_{2c}  + \Im \left( {m_2^{[i]}  - m_2 } \right) + \Im \left( {m_2  - m_{2c} } \right)}}{{N\eta }}}  \le C|w| \Psi _\theta .
\end{align}
We can estimate $\left(Z_{[i]}\right)_{22}$ and the third term in (\ref{Z_lemma2}) in a similar way. For the Cases 1-4 in Lemma \ref{lemm_m1_4case}, we have $|m_{1c}| \sim 1$ for $|w|\sim 1$, $\Im\, m_{1c} \sim |w|^{-1/2} \sim |m_{1c}|$ for $|w| \to 0$, and $\eta \le C\Im\, m_{1c}$. Thus
$$\sqrt {\frac{{\left| {m_{1c} } \right|}}{{N\left| w \right|}}} \le \frac{C}{{\sqrt N }}\le C\Psi_\theta \text{ for } |w|\sim 1, \ \ \sqrt {\frac{{\left| {m_{1c} } \right|}}{{N\left| w \right|}}}  \le C\sqrt {\frac{{\Im\, m_{1c} }}{{N\eta }}}\le C\Psi_\theta \text{ for } |w|\to 0.$$ 
Then for the second term in (\ref{Z_lemma2}), we have that
\begin{align*}
 \sqrt {\frac{{| {m_1^{\left[ i \right]} } |}}{{N\left| w \right|}}}  & \le C\left( \frac{1}{N\eta}+\sqrt { \frac{\theta }{{N\eta }}}  + \sqrt {\frac{{\left| {m_{1c} } \right|}}{{N\left| w \right|}}}  \right)  \le C \Psi _\theta .
\end{align*}
This concludes (\ref{Zestimate}).
Finally, the estimate (\ref{ZestimateL}) follows directly from (\ref{Z_lemma1}), (\ref{Z_lemma2}) and (\ref{eq_gbound}).
\end{proof}

\begin{lem}\label{lemm_selfcons_weak}
Suppose $|z|^2 \le 1-\tau$. Define the $w$-dependent event $\Xi(w):=\{\theta\le |w|^{-1/2}(\log N)^{-1}\}$. Then we have that for $w\in \mathbf D$,
\begin{equation}
\one(\Xi)m_2= \one(\Xi) \left[\frac{1+m_1}{-w\left(1 + m_1 \right)^2 + |z|^2}+ O_\prec (\Psi_\theta)\right] , \ \ \one(\Xi)\Upsilon(w,m_1)\prec \one(\Xi)\Psi_\theta,\ \label{selfcons_lemm}
\end{equation}
where $\Upsilon$ is defined in (\ref{def_stabD}).
For $w\in \bD_L$, we have
\begin{equation}\label{selfcons_eq1L}
 m_2= \frac{1+m_1}{-w\left(1 + m_1 \right)^2 + |z|^2}+ O_\prec \left(\eta^{-1}(N\eta)^{-1/2}\right), \ \ \Upsilon(w,m_1)\prec \eta^{-1}\left(N\eta\right)^{-1/2}.
\end{equation}
\label{selfcons_lemm2}
\end{lem}
\begin{proof}
Using (\ref{resolvent_Gii}), we get
\begin{equation}
G_{[ii]}^{-1} = \pi_{[i]}^{-1} + \epsilon_{[i]}, \label{resolvent_G}
\end{equation}
where $\pi_{[i]}$ is defined in (\ref{def_pi_i}) and
$$\epsilon_{[i]}  =  w\left( {\begin{array}{*{20}c}
   { \left|d_i\right|^2 \left(m_2 - m^{[i]}_2\right) } & {0}  \\
   {0} & { m_1- m^{[i]}_1} \end{array}} \right)+Z_{[i]}.$$
By (\ref{m_T})
and (\ref{Zestimate}), we get that $\epsilon_{[i]} \prec |w|\Psi_\theta$.
Let $B_{i}=\pi_{[i]}^{-1}-\pi_{{[i]}c}^{-1}$, where $\pi_{{[i]}c}$ is defined in (\ref{def_pi}). By (\ref{estimate_Piw12}) and the definition of $\Xi$, we have
$\one(\Xi)\|B_{i}\pi_{{[i]}c}\| \le C(\log N)^{-1}.$ Thus we have the expansion
\begin{equation}
\one(\Xi) \pi_{[i]}= \one(\Xi) (\pi_{{[i]}c}^{-1} +B_{i})^{-1} = \one(\Xi) \pi_{[i]c} \left(1 - B_{i}\pi_{{[i]}c} + (B_{i}\pi_{{[i]}c})^2  + \ldots \right) = \one(\Xi) (\pi_{{[i]}c} + \epsilon_a),
\end{equation}
where $\epsilon_a$ can be estimated as $\one(\Xi)\|\epsilon_a\| \le \one(\Xi)C|w|^{-1/2}(\log N)^{-1}. $
This shows that $\one(\Xi)\|\pi_{[i]}\|=\one(\Xi)O(|w|^{-1/2})$, and so $\one(\Xi)\left\|\epsilon_{[i]} \pi_{[i]}\right\| \prec \one(\Xi)|w|^{1/2}\Psi_\theta \le \one\left( \Xi  \right)CN^{ - \zeta /2} $ by the definition of $\mathbf D$ in (\ref{eq_domainD}).
Again we do the expansion for (\ref{resolvent_G}),
\begin{equation}
\one(\Xi)G_{[ii]} = \one(\Xi)\left(\pi_{[i]}^{-1}+\epsilon_{[i]}\right)^{-1}  = \one(\Xi) \pi_{[i]}\left(1 + \sum_{l=1}^{\infty}\left(-\epsilon_{[i]}\pi_{[i]}\right)^l\right)=\one(\Xi)\left(\pi_{[i]}+\epsilon_b\right), \label{self_matrix}
\end{equation}
where $\one(\Xi)\|\epsilon_b \| \prec \one(\Xi) \Psi_\theta$. Now the $11$ entry of (\ref{self_matrix}) gives that
\begin{equation}
\one(\Xi)G_{ii} = \one(\Xi) \frac{{ - 1 - m_1 }}{{w\left( {1 + |d_i |^2 m_2 } \right)\left( {1 + m_1 } \right) - \left| z \right|^2 }} + \one(\Xi)O_\prec\left(\Psi_\theta\right) , \label{G11error}
\end{equation}
from which we get that
\begin{equation}\label{G11_temp}
\one(\Xi)G_{ii} \left[-w\left(1 + |d_i |^2 m_2\right) + \frac{|z|^2}{1+m_1}\right]= \one(\Xi) \left[1 + O_\prec\left(|w|^{1/2}\Psi_\theta\right)  \right].
\end{equation}
Here we use that
\[\one(\Xi)\left[-w\left(1 + |d_i |^2 m_2\right) + \frac{|z|^2}{1+m_1}\right]=O(|w|^{1/2}),\]
which follows from Proposition \ref{prop_roughbound} and the definition of $\Xi$. Summing (\ref{G11_temp}) over $i$,
$$\one(\Xi)\left[-w\left(m_2 + m_1 m_2\right) + \frac{|z|^2 m_2}{1+m_1}  \right]= \one(\Xi)\left[1+ O_\prec\left(|w|^{1/2}\Psi_\theta\right) \right] ,$$
which gives
\begin{equation}
\one(\Xi)m_2= \one(\Xi)\frac{1+m_1}{-w\left(1 + m_1 \right)^2 + |z|^2}+\one(\Xi) O_\prec\left(\Psi_\theta\right). \label{m12relation}
\end{equation}
Now plug (\ref{m12relation}) into (\ref{G11error}), multiply with $|d_i|^2$ and sum over $i$, we get
\begin{align}
\one(\Xi) m_1 = \one(\Xi) \left[\frac{1}{N}  \sum_{i=1}^n l_i s_i \frac{{ - 1 - m_1 }}{{w\left( 1 + s_i \frac{1+m_1}{-w\left(1 + m_1 \right)^2 + |z|^2} \right)\left( {1 + m_1 } \right) - \left| z \right|^2 }} + O_\prec\left(\Psi_\theta\right)\right],
\end{align}
where we use (\ref{estimate2_bulk}) and $\one(\Xi)(1+m_1)=\one(\Xi)O(|w|^{-1/2})$.
This concludes the proof.

Similarly, when $w\in \bD_L$, it is easy to prove (\ref{selfcons_eq1L}) using the estimates (\ref{ZestimateL}) and (\ref{eq_gbound}). Note that $|m_{1,2}| = O(\eta^{-1})$ by (\ref{eq_gbound}), which implies immediately the bounds $\|\pi_{[i]}\|=O(\eta^{-1})$ and $\| \left(\pi_{[i]}\right)^{-1} \| =O(\eta)$. Hence without introducing the event $\Xi$, we can obtain directly 
\begin{equation}
 G_{[ii]} = \pi_{[i]}+ O_\prec(\eta^{-1}(N\eta)^{-1/2}). \label{self_matrixL}
\end{equation}
The rest of the proof is essentially the same.
\end{proof}

Notice that applying Lemma \ref{lemm_stability} to (\ref{selfcons_eq1L}), we obtain $|m_{1,2}-m_{1,2c}|\prec \eta^{-1}(N\eta)^{-1/2}$. Plugging it into (\ref{self_matrixL}), we get immediately (\ref{weak_lawL}) for $w\in \mathbf D_L$. This proves the entrywise law on $\mathbf D_L$, since $\eta^{-1}N^{-1/2} \le C\Psi$ by the definition (\ref{eq_defpsi}) and the estimate (\ref{estimate1L}).

\end{subsection}

\begin{subsection}{The large $\eta$ case}

It remains prove Proposition \ref{thm_weaklaw} on domain $\mathbf D$. We would like to fix $E$ and then apply a continuity argument in $\eta$ by first showing that the rough bound $\Lambda\le |w|^{-1/2}(\log N)^{-1}$ in Lemma \ref{lemm_selfcons_weak} holds for large $\eta$. To start the argument, we first need to establish the estimates on $G$ when $\eta\sim 1$. The next lemma is a trivial consequence of (\ref{eq_gbound}). 

\begin{lem}
For any $w\in \bD$ and $\eta \ge c$ for fixed $c>0$, we have the bound
\begin{equation}
\mathop {\max }\limits_{s,t} \left| {G_{st} \left( w \right)} \right| \le C
\end{equation}
for some $C>0$. This estimate also holds if we replace $G$ with $G^{(J)}$ for $J\subset \mathcal I$.
\end{lem}

\begin{lem}
Fix $c>0$ and $|z|^2 \le 1-\tau$. We have the following estimate
\begin{equation}\label{lemm_largeeta}
\mathop {\max }\limits_{w \in \bD,\eta \ge c} \Lambda \left( w \right) \prec N^{ - 1/2 }.
\end{equation}
\end{lem}
\begin{proof}
By the previous lemma, we have $ |m^{[i]}_{1,2}|= O(1)$. So by Lemma \ref{Z_lemma}, $\|Z_{[i]} \| \prec N^{-1/2}$ uniformly in $\eta \ge c$. Then as in (\ref{resolvent_G}),
\begin{equation}
G_{[ii]} = \left(\pi_{[i]}^{-1} + \epsilon_{[i]}\right)^{-1}, \label{large_eta}
\end{equation}
where $\|\pi_{[i]}^{-1} \|=O(1)$ and $\|\epsilon_{[i]}\|\prec N^{-1/2}$.
Notice since $G_{[ii]} = O(1)$, we have the estimate
$$\pi_i = \left(G_{[ii]}^{-1}-\epsilon_{[i]}\right)^{-1} = G_{[ii]}\left(1-\epsilon_{[i]} G_{[ii]}\right)^{-1} = O_\prec(1).$$
Then we can expand (\ref{large_eta}) to get that
\begin{equation}\label{eq_gii}
G_{[ii]} = \pi_i+O_{\prec} \left(N^{-1/2}\right).
\end{equation}
The $11$ and $22$ entries of (\ref{eq_gii}) leads to the equations
\begin{align}
& m_1 = \frac{1}{N}\sum_{i=1}^N |d_i|^2\left[{{-w\left( {1 + |d_i |^2 m_2 } \right) + \frac{\left| z \right|^2}{1+m_1} }}\right]^{-1} + O_\prec\left(N^{-1/2}\right) , \label{sum_Gii1} \\
& m_2 = \frac{1}{N}\sum_{i=1}^N \left[{{-w\left( {1 + m_1 } \right) + \frac{\left| z \right|^2}{1+|d_i|^2m_2} }}\right]^{-1} + O_\prec\left(N^{-1/2}\right) . \label{sum_Gmu}
\end{align}
Our goal is to prove that $\Im\, m_{1,2} \ge C(\log N)^{-1}$ with high probability for some $C>0$.

Using the spectral decomposition (\ref{singular_rep}), we note that for $l> 1$,
\begin{align*}
& \frac{1}{N}\sum\limits_{|\lambda_k - E| \ge l\eta} \frac{|E-\lambda_k|}{(\lambda _k  - E)^2 +\eta^2} \le \frac{1}{l\eta}, \\
& \frac{1}{N}\sum\limits_{|\lambda_k - E| \le l\eta} \frac{|E-\lambda_k|}{(\lambda _k  - E)^2 +\eta^2} \le \frac{1}{N}\sum\limits_{|\lambda_k - E| \le l\eta} \frac{l\eta}{(\lambda _k  - E)^2 +\eta^2} \le l \Im\, m_2.
\end{align*}
Summing up these two inequalities and optimizing $l$, we get
\begin{equation}
|\Re\, m_2| \le 2\sqrt{\frac{\Im\, m_2}{\eta}}. \label{ReImm2}
\end{equation}
Assume that $\Im\, m_2 \le C(\log N)^{-1}$, then by (\ref{eq_Imm12}) we also have $\Im\, m_1 \le C\tau^{-1}(\log N)^{-1} $. From (\ref{ReImm2}), we get $|m_{2}|\le C(\log N)^{-1/2}$. Together with $\Im\, w=\eta \ge c$ and ${\rm Im} [|z|^2/(1+m_1)] < 0$, (\ref{sum_Gii1}) gives
\begin{equation}
|m_1| \le \frac{1}{N}\sum_i |d_i|^2\left| {\rm Im} \left[ {{-w\left( {1 + |d_i |^2 m_2 } \right) + \frac{\left| z \right|^2}{1+m_1} }}\right]\right|^{-1} + o(1) \le C
\end{equation}
with high probability. Using the above estimate and $|m_{2}|\le C(\log N)^{-1/2}$ we get $$ \left| -w\left( {1 + m_1 } \right) + \frac{\left| z \right|^2}{1+|d_i|^2m_2} \right|\le C \text{ with high probability.}$$ On the other hand
\begin{equation}
{\rm Im} \left[-w\left( {1 + m_1 } \right) + \frac{\left| z \right|^2}{1+|d_i|^2m_2}\right] \le -\Im \, w = -\eta,\label{Imeta_large}
\end{equation}
where we use $\Im [|z|^2/(1+|d_i|^2m_2)] < 0$ and
$$\Im (wm_1) = \Im \left[ \frac{1}{N}\sum\limits_{k = 1}^N |d_i|^2|\xi _k(i)|^2  \left( -1 + \frac{\lambda_k}{{\lambda _k  - w}} \right) \right] \ge 0.$$
Hence (\ref{sum_Gmu}) implies $\Im\, m_2 \ge c$ with high probability for some $c>0$. This contradicts $\Im\, m_2 \le C(\log N)^{-1}$. Thus $\Im\, m_2 \ge C(\log N)^{-1}$ with high probability for some $C>0$ , which also implies $\Im\, m_1 \ge C (\log N)^{-1} $ by (\ref{eq_Imm12}).

Now we can proceed as in Lemma \ref{selfcons_lemm2} and get that
\begin{equation}
 m_2= \frac{1+m_1}{-w\left(1 + m_1 \right)^2 + |z|^2}+O_\prec\left(N^{-1/2}\right), \ \ \Upsilon(w,m_1)\prec N^{1/2}. \label{m12relation_large}
\end{equation}
We omit the details. Applying Lemma \ref{lemm_stability} to (\ref{m12relation_large}), we conclude $|m_{1,2}-m_{1,2c}| \prec N^{ - 1/2}$ uniformly in $\eta\ge c$. 
By (\ref{eq_gii}), we get $\|(G-\Pi)_{[ii]}\|\prec{N^{-1/2}}$ uniformly in $\eta\ge c$ and $i\in \mathcal I_1$. Finally using (\ref{eq_res2_2}) and Lemmas \ref{lemma_Im}-\ref{large_deviation}, we can prove the off-diagonal estimate (see (\ref{est_offd})).
\end{proof}
\end{subsection}

\begin{subsection}{Proof of the weak entrywise local law}\label{subsection_weak_proof}

In this subsection, we finish the proof of Proposition \ref{thm_weaklaw} on domain $\mathbf D$. We shall fix the real part $E$ of $w=E+i\eta$ and decrease the imaginary part $\eta$. Recall Lemma \ref{lemm_selfcons_weak} is based on the condition $\Lambda\le |w|^{-1/2}(\log N)^{-1}$ (i.e. event $\Xi$). So far this is only established for large $\eta$ in (\ref{lemm_largeeta}). We want to show this condition for small $\eta$ also by using a continuity argument.

It is convenient to introduce the random function
$$v(w)=\max_{w'\in L(w)}\theta(w')|w'|^{1/2}\left(\frac{N\Im\, w'}{|w'|^{1/2}}\right)^{1/4},$$
where $L(w)$ is defined in (\ref{eqn_def_L}). Fix a regular domain $\mathbf S$, an $\epsilon<{\zeta}/{4}$ and a large $D>0$. Our goal is to prove that with high probability there is a gap in the range of $v$, i.e.
\begin{equation}\label{eq_cont30}
\bbP\left(v(w)\le N^\epsilon, v(w)>N^{3\epsilon/4}\right)\le N^{-D+21}
\end{equation}
for all $w\in \mathbf S$ and large enough $N\ge N(\epsilon, D)$. 

Suppose $v(w)\le N^\epsilon$, then it is easy to verify
\begin{equation}\label{newevent}
\theta(w')\le C|w'|^{-1/2}(\log N)^{-1}
\end{equation}
for all $w'\in L(w)$. Hence $\{v(w)\le N^{\epsilon}\}\subset \Xi(w')$ for all $w'\in \mathbf S \cap L(w)$. Then by (\ref{selfcons_lemm}), we have that for all $w'\in \bS \cap L(w)$, there exists an $N_0\equiv N_0(\epsilon, D)$ such that
\begin{equation}\label{eq_cont1}
P\left( {v(w) \le {N^\epsilon},\Upsilon (w') > \frac{N^{\epsilon}}{|w'|^{1/2}}\sqrt {\frac{{{{\left| {w'} \right|}^{1/2}}}}{{N\Im\, w'}}} } \right) \le {N^{ - D}},
\end{equation}
for all $N>N_0$. Taking the union bound we get
\begin{equation}\label{maxupsilon}
P\left( {v(w) \le {N^\epsilon},\max\limits_{w' \in L\left( w \right)} \Upsilon (w') \sqrt {\frac{{N\Im \, w'}}{{{{\left| {w'} \right|}^{-1/2}}}}}  > {N^{\epsilon}}} \right) \le {N^{ - D + 10}}.\end{equation}
Now consider the event
\begin{equation}\label{assump_temp}
\Xi_1:=\left\{v(w)\le N^\epsilon, \max\limits_{w'\in L(w)}\Upsilon (w') \sqrt { \frac{N\Im\,w'}{{\left| {w'} \right|}^{-1/2}}}\le N^\epsilon \right\}.
\end{equation}
Then $1(\Xi_1)\Upsilon(w')\le \delta \left( {w'} \right)$ for all $w'\in L(w)$ with $ \delta \left( {w'} \right) = \frac{N^{\epsilon}}{|w'|^{1/2}} \sqrt {\frac{{{{\left| {w'} \right|}^{ 1/2}}}}{{N\Im\, w'}}}.$ We now apply Lemma \ref{lemm_stability}. If $\kappa \ll 1$  (recall (\ref{eqn_def_kappa})), then $|w|\sim 1$ and we have
$$1(\Xi_1)|m_1(w')-m_{1c}(w')| \le C \sqrt{\delta(w')} \le C {N^{\epsilon/2}} {{\left( {\frac{1}{N\Im\, w'}} \right)}^{1/4}}$$
for all $w'\in L(w)$; if $\kappa \ge c > 0$ for some constant $c>0$, then
$$1(\Xi_1)|m_1(w')-m_{1c}(w')| \le C {\delta(w')}\le C \frac{N^{\epsilon}}{|w'|^{1/2}} {{\left( {\frac{{\left| {w'} \right|}^{ 1/2}}{N\Im\, w'}} \right)}^{1/2}}$$
for all $w'\in L(w)$.
Combining these two cases we get
\begin{equation}
1(\Xi_1)|m_1(w')-m_{1c}(w')| \le C \frac{N^{\epsilon/2}}{\left|w'\right|^{1/2}}{{\left( {\frac{{\left| {w'} \right|}^{1/2}}{N\Im\, w'}} \right)}^{1/4}}\label{eq_6.4_1}
\end{equation}
for all $w'\in L(w)$.
By (\ref{selfcons_lemm}), we have
$$1(\Xi_1)|m_2(w')-m_{2c}(w')|\prec 1(\Xi_1)|m_1(w')-m_{1c}(w')|+1(\Xi_1)\Psi_{\theta}\prec \frac{N^{\epsilon/2}}{\left|w'\right|^{1/2}}{{\left( {\frac{{\left| {w'} \right|}^{1/2}}{N\Im\, w'}} \right)}^{1/4}},$$
for all $w'\in \bS \cap L(w)$.
Combining this bound with (\ref{eq_6.4_1}), we see there is $N_1\equiv N_1(\epsilon,D)$ such that
\begin{equation}
\bbP\left(v(w)\le N^\epsilon, \max\limits_{w'\in L(w)}\Upsilon (w') \sqrt { \frac{N\Im\, w'}{{\left| {w'} \right|}^{-1/2}}}\le N^\epsilon, \max\limits_{w'\in L(w)}\theta(w') |w'|^{1/2} {{\left( {\frac{N\Im\, w'}{{\left| {w'} \right|}^{1/2}} } \right)}^{1/4}} >N^{3\epsilon/4} \right)\le N^{-D}\label{eq_cont2}
\end{equation}
for $N\ge \max\{N_0,N_1\}$. Adding (\ref{maxupsilon}) and (\ref{eq_cont2}), we get
\begin{equation*}
\bbP\left(v(w)\le N^\epsilon,  \max\limits_{w'\in L(w)}\theta(w') |w'|^{1/2} {{\left( {\frac{N\Im\, w'}{{\left| {w'} \right|}^{1/2}} } \right)}^{1/4}} >N^{3\epsilon/4} \right)\le N^{-D+11}.
\end{equation*}
Taking the union bound over $L(w)$ we get (\ref{eq_cont30}) for all $N\ge\max\{N_0, N_1\}$.

Now we conclude the proof of Proposition \ref{thm_weaklaw} by combining (\ref{eq_cont30}) with the large $\eta$ estimate (\ref{lemm_largeeta}).
We choose a lattice $\Delta\subset \bS$ such that $|\Delta|\le N^{20}$ and for any $w\in\bS$ there is a $w'\in\Delta$ with $|w'-w|\le N^{-9}$.
Taking the union bound we get
\begin{equation}\label{eq_cont4}
\bbP\left(\exists w\in\Delta: v(w)\in (N^{3\epsilon/4}, N^\epsilon]\right)\le N^{-D+41}.
\end{equation}
Since $v$ has Lipshcitz constant bounded by, say, $N^6$, then we have
\begin{equation}\label{eq_cont4}
\bbP\left(\exists w\in\bS: v(w)\in (2N^{3\epsilon/4}, N^\epsilon/2]\right)\le N^{-D+41}.
\end{equation}
Combining with (\ref{lemm_largeeta}), we see that there exists $N_2\equiv N_2(\epsilon,D)$ such that for $N> N_2 $,
\begin{align*}
\bbP\left( {\forall w \in \bS:v(w) \le 2N^{3\epsilon/4}} \right) \ge 1 - 2{N^{ - D + 41}}.
\end{align*}
Since $\epsilon$ and $D$ are arbitrary, the above inequality shows that $v(w)\prec 1$ uniformly in $w\in \bS$, or
\begin{equation}\label{eq_weakaverage}
\theta(w) \prec \frac{1}{|w|^{1/2}}\left(\frac{|w|^{1/2}}{N\eta}\right)^{1/4}.
\end{equation}
In particular we see that for all $w\in \bS$, the event $\Xi$ holds with high-probability.

Now using (\ref{self_matrix}) and (\ref{eq_weakaverage}), we get
\begin{equation}
\left\| {{G_{[ii]}} - {\pi_{[i]c}}} \right\| \le \left\| {{G_{[ii]}} - \pi_{[i]}} \right\| + \left\| {\pi_{[i]} - \pi_{[i]c}} \right\| \prec {\Psi _\theta } + \theta \prec \frac{1}{|w|^{1/2}}\left(\frac{|w|^{1/2}}{N\eta}\right)^{1/4}.\label{diagonal_weak2}
\end{equation}
To conclude Proposition \ref{thm_weaklaw}, it remains to prove the estimate for the off-diagonal entries. By (\ref{m_T}), it is not hard to see that
\begin{equation}
\left\| G^{\left[ J \right]}_{[ii]}  - \pi_{[i]c} \right\| \prec \frac{1}{|w|^{1/2}}\left(\frac{|w|^{1/2}}{N\eta}\right)^{1/4}
\end{equation}
for any $|J|\le l$ with $l\in \mathbb N$ fixed. Thus we have 
$G^{\left[ J \right]}_{[ii]} = O\left(|w|^{-1/2}\right)$ and $\left(G^{\left[ J \right]}_{[ii]}\right)^{-1} = O\left(|w|^{1/2}\right)$ with high probability.
Let $i\ne j\in \mathcal I_1$, using (\ref{eq_res2_2}) and the above diagonal estimates, we get that
\begin{align}
\left\|G_{[ij]}\right\|  \prec |w|^{-1}\frac{|w|^{1/2}}{\sqrt{N}} + |w|^{-1} \left\| \sum_{k,l\notin \{i,j\}} H_{[ik]}G^{[ij]}_{[kl]} H_{[lj]}\right\| \prec \Psi_\theta \prec \frac{1}{|w|^{1/2}}\left(\frac{|w|^{1/2}}{N\eta}\right)^{1/4},\label{est_offd}
\end{align}
where, as in the proof of Lemma \ref{Z_lemma}, we use Lemmas \ref{lemma_Im} and \ref{large_deviation} to obtain that
\begin{equation}
|w|^{-1} \left\| \sum_{k,l\notin \{i,j\}} H_{[ik]}G^{[ij]}_{[kl]} H_{[lj]}\right\|=\left\| \left(\begin{matrix} \sum_{k,l\notin \{i,j\}} X_{i \bar k} G^{[ij]}_{\bar k \bar l}X^\dag_{\bar l j} & \sum_{k,l \notin \{i,j\}} X_{i \bar k} G^{[ij]}_{\bar k l} X_{l \bar j}\\
\sum_{k,l\notin \{i,j\}} X^\dag_{\bar i k} G^{[ij]}_{k \bar l}X^\dag_{\bar l j}& \sum_{k,l\notin \{i,j\}} X^\dag_{\bar i k} G^{[ij]}_{k l} X_{l \bar j}\\
\end{matrix}\right) \right\| \prec \Psi_\theta .
\end{equation}
\end{subsection}

%

\begin{subsection}{Proof of the strong enterywise local law}\label{subsection_proofstrong}

In this section, we finish the proof of the (strong) entrywise local law in Theorem \ref{law_squareD} on domain $\mathbf D$ and under the condition $|w|^{1/2}+|z|^2 \ge c$.
In Lemma \ref{selfcons_lemm2}, we have proved an error estimate of the self-consistent equations of $m_{1,2}$ linearly in $\Psi_\theta$. The core part of the proof is to improve this estimate to quadratic in $\Psi_\theta$. For the sequence of random variables $Z_{[i]}$, we define the averaged quantities
$$\left[ Z \right] = \frac{1}{N}\sum\limits_{i=1}^N {\pi_{[i]}Z_{[i]}\pi_{[i]} },\ \ \left\langle Z \right\rangle = \frac{1}{N}\sum\limits_{i=1}^N |d_i|^2{\pi_{[i]}Z_{[i]}\pi_{[i]} }.$$
The following Lemma is an improvement of Lemma \ref{selfcons_lemm2}.

\begin{lem}
Fix $\left| z\right|^2 \le 1 - \tau$. Then for $w\in \mathbf D$,
\begin{equation}\label{selfcons_avg1}
m_2= \frac{1+m_1}{-w\left(1 + m_1 \right)^2 + |z|^2}+O_\prec (|w|^{1/2}\Psi_\theta^2+\|[Z]\|+\|\langle Z \rangle\|),
\end{equation}
and
\begin{equation}
\Upsilon(w,m_1)\prec |w|^{1/2}\Psi_\theta^2+\|[Z]\|+\|\langle Z \rangle\|. \label{selfcons_avg2}
\end{equation}
For $w\in \bD_L$, 
\begin{equation}\label{selfcons_avg1L}
m_2= \frac{1+m_1}{-w\left(1 + m_1 \right)^2 + |z|^2}+O_\prec \left((N\eta)^{-1}+\|[Z]\|+\|\langle Z \rangle\| \right),
\end{equation}
and
\begin{equation}
\Upsilon(w,m_1)\prec (N\eta)^{-1}+\|[Z]\|+\|\langle Z \rangle\| . \label{selfcons_avg2L}
\end{equation}
\end{lem}
\begin{proof}
The proof is almost the same as the one in Lemma \ref{selfcons_lemm2}, we only lay out the difference. We first consider the case $w\in \mathbf D$. By Proposition \ref{thm_weaklaw}, the event $\Xi$ holds with high probability. Hence without loss of generality, we may assume $\Xi$ holds throughout the proof. Using (\ref{eq_res3}), we get
\begin{equation}\label{better_M}
  \frac{1}{N}\sum_{k\in \mathcal I_1} \left( {\begin{array}{*{20}c}
   { |d_k|^2} & {0}  \\
   {0} & {1}  \\
   \end{array}} \right) \left( G_{[kk]} -G_{[kk]}^{\left[i\right]} \right)= \left( {\begin{array}{*{20}c}
   { |d_i|^2} & {0}  \\
   {0} & {1}  \\
   \end{array}} \right) \frac{G_{[ii]}}{N}+\frac{1}{N}\sum_{k\ne i} \left( {\begin{array}{*{20}c}
   { |d_k|^2} & {0}  \\
   {0} & {1}  \\
   \end{array}} \right) {G_{[ki]} G^{-1}_{[ii]} G_{[ik]} }.
\end{equation}
By Proposition \ref{thm_weaklaw}, (\ref{estimate_Piw12}) and (\ref{est_offd}), we have
$$\left\|{G_{[ki]} G^{-1}_{[ii]} G_{[ik]} }\right\| \prec |w|^{1/2}\Psi_\theta^2.$$
By Lemma \ref{lemm_m1_4case}, it is easy to verify that $\left\|{G_{[ii]}}/{N}\right\| \le C |w|^{1/2}\Psi_\theta^2.$
Plug it into (\ref{better_M}), we get
\begin{equation}\label{better_estimate_m12}
\left| {m_{1,2}^{[i]}  - m_{1,2} } \right| \prec |w|^{1/2}\Psi_\theta^2.
\end{equation}
Using (\ref{Zestimate}) and (\ref{better_estimate_m12}), the error $\epsilon_b$ in $(\ref{self_matrix})$ is 
\begin{align*}
\epsilon_b= O_\prec(|w|^{1/2}\Psi^2_\theta)-{\pi_{[i]}Z_{[i]}\pi_{[i]} } \left[1+O_\prec (|w|^{1/2} \Psi_\theta)\right] =O_\prec (|w|^{1/2} \Psi_\theta^2)-{\pi_{[i]}Z_{[i]}\pi_{[i]} }.
\end{align*}
Then following the arguments in Lemma \ref{selfcons_lemm2}, we can obtain the desired result on $\Xi$. 
For $w\in \mathbf D_L$, the proof is similar by using (\ref{weak_lawL}).
\end{proof}

In the following lemma we prove stronger bounds on $[Z]$ and $\langle Z\rangle$ by keeping track of the cancellation effects due to the average over the index $i$. The proof is given in Appendix \ref{appendix3}.  
\begin{lem}\label{fluc_aver}
(Fluctuation averaging) Fix $|z|^2 \le 1-\tau$. Suppose $\Phi$ and $\Phi_o$ are positive, $N$-dependent deterministic functions satisfying $N^{-1/2} \le \Phi, \Phi_o \le N^{-c}$ for some constant $c>0$. Suppose moreover that $\Lambda \prec |w|^{-1/2}\Phi$ and $\Lambda_o \prec |w|^{-1/2}\Phi_o$. Then for $w\in \mathbf D$,
\begin{equation}\label{flucaver_ZZ}
\|[Z] \| + \|\langle Z\rangle \| \prec  \left|w\right|^{-1/2} {\Phi _o^2}.
\end{equation}
\end{lem}

Now we finish the proof of the entrywise local law and averaged local law on the domain $\mathbf D$. By Proposition \ref{thm_weaklaw}, we can take in Lemma \ref{fluc_aver} 
$$\Phi_o = |w|^{1/2}\sqrt{\frac{\Im(m_{1c}+m_{2c})+|w|^{-3/8}(N\eta)^{-1/4}}{N\eta}},\ \ \ \Phi=\left(\frac{|w|^{1/2}}{N\eta}\right)^{1/4},$$
with $\Lambda_o\prec\Psi_\theta\prec |w|^{-1/2}\Phi_o$ and $\Lambda \prec \theta \prec |w|^{-1/2}\Phi$. 
Then (\ref{selfcons_avg2}) gives
$$\Upsilon(w,m_1)\prec\frac{|w|^{1/2}\Im(m_{1c}+m_{2c})+|w|^{1/4}(N\eta)^{-1/4}}{N\eta}.$$
Then using the stability Lemma \ref{lemm_stability},
\[|m_1-m_{1c}|\prec\frac{|w|^{1/2}\Im(m_{1c}+m_{2c})}{N\eta\sqrt{\kappa+\eta}}+\frac{|w|^{1/8}}{(N\eta)^{5/8}} \prec \frac{1}{N\eta}+\frac{|w|^{1/8}}{(N\eta)^{5/8}} \prec |w|^{-1/2}\left(\frac{|w|^{1/2}}{N\eta}\right)^{1/2+1/8}.\]
Here if $\sqrt{\kappa+\eta}\ge (\log N)^{-1}$, we use
\[\frac{|w|^{1/2}\Im(m_{1c}+m_{2c})}{N\eta\sqrt{\kappa+\eta}}\le \frac{C\log N}{N\eta}\prec \frac{1}{N\eta},\]
while if $\sqrt{\kappa+\eta}\le (\log N)^{-1}$, we have $\Im(m_{1c}+m_{2c})=O(\sqrt{\kappa+\eta})$, which also gives that
\[\frac{|w|^{1/2}\Im(m_{1c}+m_{2c})}{N\eta\sqrt{\kappa+\eta}}\prec \frac{1}{N\eta}.\]
We then use (\ref{selfcons_avg1}) to get that
\begin{equation}\label{estimatetheta_ind}
\theta\prec |m_1-m_{1c}|+\frac{|w|^{1/2}\Im(m_{1c}+m_{2c})+|w|^{1/4}(N\eta)^{-1/4}}{N\eta}\prec |w|^{-1/2}\left(\frac{|w|^{1/2}}{N\eta}\right)^{1/2+1/8}.
\end{equation}
Repeating the previous steps with the new estimate (\ref{estimatetheta_ind}), we get the bound
$$\theta\prec {|w|^{-1/2}\left(\frac{|w|^{1/2}}{N\eta}\right)^{\sum_{k=1}^l 1/2^k + 1/2^{l+2} }}$$
after $l$ iterations. This implies the averaged local law $\theta\prec(N\eta)^{-1}$ since $l$ can be arbitrarily large.
Finally as in (\ref{diagonal_weak2}) and (\ref{est_offd}), we have for $i\ne j$
\begin{align*}
 & \left\| {G_{[ii]}  - \pi_{[i]c} } \right\| + \left\| {G_{[ij]} } \right\| \prec \Psi _\theta   + \theta  \prec \sqrt {\frac{{{\mathop{\rm Im}\nolimits} (m_{1c}+m_{2c}) }}{{N\eta }}}  + \frac{1}{{N\eta }}.\end{align*}
This concludes the entrywise local law and averaged local law in Theorem \ref{law_squareD} when $|w|^{1/2}+|z|^2\sim 1$.

When $w\in \mathbf D_L$, we have proved the entrywise law (see the remark after (\ref{self_matrixL})). Also we can prove a similar result as Lemma \ref{fluc_aver}, which implies
\begin{equation} 
m_2= \frac{1+m_1}{-w\left(1 + m_1 \right)^2 + |z|^2}+O_\prec \left((N\eta)^{-1}\right), \ \ \Upsilon(w,m_1)\prec (N\eta)^{-1}.
\end{equation}
The averaged local law then follows from Lemma \ref{lemm_stability}.
We leave the details to the reader.
\end{subsection}

\begin{subsection}{Proof of Theorem \ref{law_squareD} when $|z|$ and $|w|$ are small}\label{section_smallwz}

In the previous proof, we did not include the case where $|w|^{1/2}+|z|^2\le \epsilon$ for some sufficiently small constant $\epsilon>0$. The only reason is that Lemma \ref{lemm_stability} does not apply in this case. In this section, we deal with this problem.

The main idea of this subsection is to use a different set of self-consistent equations, which has the desired stability when $|w|$ and $|z|$ are small. Multiplying (\ref{G11error}) with $|d_i|^2$ and summing over $i$, 
\begin{equation}
1(\Xi)m_1 = 1(\Xi) \left[ \frac{1}{N}\sum_{i=1}^n l_i s_i \frac{{ - 1 - m_1 }}{{w\left( {1 + s_i m_2 } \right)\left( {1 + m_1 } \right) - \left| z \right|^2 }} + O_\prec\left(\Psi_\theta\right) \right]. \label{selfeqn_m1}
\end{equation}
Recall that $\Sigma:=DD^\dag=D^\dag D$. We introduce a new matrix
\begin{equation}
   \tilde H(w) : = \left( {\begin{array}{*{20}c}
   { - w \Sigma^{-1}} & w^{1/2}(X-D^{-1}z)  \\
   {w^{1/2} (X-D^{-1}z)^\dag} & { - wI}  \\
   \end{array}} \right),
 \end{equation}
and define $\tilde G :=\tilde H^{-1}.$ By Schur's complement formula, the upper left block of $\tilde G$ is
$$\tilde G_L=\left[ (X-D^{-1}z)(X-D^{-1}z)^\dag - w\Sigma^{-1}\right]^{-1},$$
and the lower right block is equal to
$$\tilde G_R=\left[ (X-D^{-1}z)^\dag \Sigma (X-D^{-1}z) - w\right]^{-1}=\left[ (DX-z)^\dag (DX-z) - w\right]^{-1} =G_R.$$
Now we write $m_{1,2}$ in another way as
\begin{align}
m_1 & =\frac{1}{N}\text{Tr}\left[ D^\dag \left( {YY^\dag   - w} \right)^{ - 1} D\right] = \frac{1}{N}\text{Tr} \, \tilde G_L , \label{extram1}\\
m_2 & =\frac{1}{N} \text{Tr} \, \tilde G_R = \frac{1}{N}\text{Tr}\left[ (X-D^{-1}z)^\dag \Sigma (X-D^{-1}z) - w\right]^{-1} \nonumber\\
& = \frac{1}{N}\text{Tr} \left[(X-D^{-1}z)(X-D^{-1}z)^\dag \Sigma - w\right]^{-1} =\frac{1}{N}\text{Tr}\left(\Sigma^{-1}\tilde G_L \right) . \label{extram2}
\end{align}

We apply the arguments in the proof of Lemma \ref{lemm_selfcons_weak} to $\tilde H$, and get that
\begin{align}
\tilde G_{[ii]}^{ - 1}  = \left(\begin{matrix} -w|d_i|^{-2} - w m_2  &  -w^{1/2}zd_i^{-1}\\
  -w^{1/2}\bar z\bar d^{-1}_i &  -w-w m_1 \\
\end{matrix}\right) + O_{\prec}(|w| \Psi_\theta) ,
\end{align}
from which we get that
$$1(\Xi)\tilde G_{ii} =1(\Xi)\left[ \frac{-1-m_1}{w(|d_i|^{-2} + m_2)(1+ m_1)-|z|^2|d_i|^{-2}} + O_\prec(\Psi_\theta)\right].$$
Plugging this into (\ref{extram2}), we get
\begin{align}
1(\Xi)m_2 & = 1(\Xi)\left[\frac{1}{N}\sum_{i=1}^n \frac{l_i}{s_i} \frac{-1 - m_1}{w(s^{-1}_i+m_2)(1+ m_1)-|z|^2s_i^{-1}} + O_\prec(\Psi_\theta)\right]. \label{extram22}
\end{align}
We take the equations in (\ref{selfeqn_m1}) and (\ref{extram22}) as our new self-consistent equations, namely,
\begin{align}
1(\Xi)f_1(m_1, m_2) = 1(\Xi)O(\Psi_\theta) ,  \ \ 1(\Xi)f_2(m_1,m_2) = 1(\Xi)O(\Psi_\theta),\label{selfeqn_new}
\end{align}
where
\begin{align}
& f_1( m_1, m_2) := m_1 + \frac{1}{N}\sum_{i} l_i s_i \frac{{1 + m_{1} }}{{w\left( {1 + s_i m_{2} } \right)\left( {1 + m_{1} } \right) - \left| z \right|^2 }}  ,  \\
& f_2( m_1, m_2) := m_2 + \frac{1}{N}\sum_i l_i \frac{1 + m_1}{w(1+s_i m_2)(1 + m_1)-|z|^2 } .
\end{align}
According to the following lemma, this system of self-consistent equations are stable when $|w|$ and $|z|^2$ are small enough .

\begin{lem}
Suppose that $N^{-2}|w|^{-1/2} \le \delta(w) \le (\log N)^{-1}|w|^{-1/2}$ for $w\in \mathbf D$. Suppose $u_{1,2}:\mathbf D\to \mathbb C$ are Stieltjes transforms of positive integrable functions such that
$$\max\left\{\left|f_1(u_1,u_2)(w)\right|, \left|f_2(u_1,u_2)(w)\right|\right\}\le \delta(w).$$
Then there exists an $\epsilon>0$ such that if ${|w|}^{1/2}+|z|^2 \le \epsilon$, we have
\begin{equation}
 \left|u_1(w)-m_{1c}(w)\right|+\left|u_2(w)-m_{2c}(w)\right| \le {C\delta},
\end{equation}
for some constant $C>0$ independent of $w$, $z$ and $N$.
\end{lem}
\begin{proof}
The proof depends on the estimate of the Jacobian at $(m_{1c},m_{2c})$. By (\ref{estimate12_0}) and (\ref{defn_t}),
$$m_{1c}=\frac{i\sqrt{t_0}+O(|w|^{1/2}+|z|^2)}{\sqrt{w}},\ \ m_{2c}=\frac{it_0^{-1/2}+O(|w|^{1/2}+|z|^2)}{\sqrt{w}},$$
where $t_0=(N^{-1}\sum_{i=1}^n l_i/s_i )^{-1}$. Then we can calculate that
$$ \det\left(\begin{matrix} \partial_1 f_1  &  \partial_2 f_1 \\
  \partial_1 f_2 &  \partial_2 f_2 \\
\end{matrix}\right)_{u_{1,2}=m_{1,2c}} = \det \left(\begin{matrix} 1 + O(|z|^2)  &  t_0 + O(|w|^{1/2}+|z|^2)\\
  O(|z|^2) &  2+ O(|w|^{1/2}+|z|^2) \\
\end{matrix}\right) = 2+O(|w|^{1/2}+|z|^2).$$
We can conclude the stability by expanding $f_{1,2}(u_1,u_2)$ around $(m_{1c},m_{2c})$ and using a fixed point argument as in the proof of Lemma \ref{lemm_stability} in Section \ref{subsection_append3}.
\end{proof}

With this stability lemma, we can repeat all the arguments in the previous subsections to prove the entrywise local law and averaged local law when ${|w|}^{1/2}+|z|^2 \le \epsilon$.

\end{subsection}

\end{section}

\begin{section}{Anisotropic local law when $T$ is diagonal}\label{section_isotropiclaw}

In this section we prove the anisotropic local law in Theorem \ref{law_squareD} when $T$ is diagonal. The basic ideas of the proof follow from \cite[section 5]{isotropic}, and the core part of our proof is a novel way to perform the combinatorics.
By the Definition \ref{def_local_laws} (ii) and the definition of matrix norm, it suffices to prove the following proposition for generalized entries of $G$.

\begin{prop}\label{iso_prop}
Fix ${\left| z \right|^2 } \le 1 - \tau$ and suppose that the assumptions of Theorem \ref{law_squareD} hold. Then for any regular domain $\mathbf S$,
\begin{equation}
\left| {\left\langle {{\mathbf{u}},\left( {G(w) - \Pi(w) } \right){\mathbf v}} \right\rangle } \right| \prec \Psi
\end{equation}
uniformly in $w\in \bS$ and any deterministic unit vectors $\mathbf u,\mathbf v\in{\mathbb C}^{\mathcal I}$.
\end{prop}
It is equivalent to show that
\begin{equation}\label{eq_iso_1}
 \sum\limits_{i,j\in \mathcal I_1} u^\dag_{\left[ i \right]} \left( {G_{\left[ ij \right ]} - \Pi_{\left[ ij\right]}}\right ) v_{\left[ j \right]} \prec \Psi, \ \ u_{[i]}:=  \left( {\begin{array}{*{20}c}
   {u_i}   \\
   {u_{\bar i}} \\
   \end{array}} \right), \ \ v_{[j]}:=  \left( {\begin{array}{*{20}c}
   {v_j}   \\
   {v_{\bar j}} \\
   \end{array}} \right).
\end{equation}
By the entrywise local law, 
\begin{align*}
 \left| {\sum\limits_{i,j} { u_{\left[ i \right]}^\dag \left( {{G_{\left[ ij \right ]}} - \Pi_{\left[ ij\right]}} \right) v_{\left[ j \right]} } } \right| & \le  \sum_i \left\lVert{{G_{\left[ ii \right ]}} - \Pi_{\left[ ii\right]}} \right\rVert\left|u_{\left[ i \right]}\right| \left|v_{\left[ i \right]}\right| + \left| {\sum\limits_{i \ne j} {{{u}^\dag_{\left[ i \right]}} {G_{\left[ ij \right ]}}{{v}_{\left[ j \right]}}} } \right| \prec  \Psi  + \left| {\sum\limits_{i \ne j} {{{u}^\dag_{\left[ i \right]}} {G_{\left[ ij \right ]}}{{v}_{\left[ j \right]}}} } \right|.
\end{align*}
Thus to show (\ref{eq_iso_1}), it suffices to prove
\begin{equation} \label{eq_iso_11}
\left| {\sum\limits_{i \ne j} {{{u}^\dag_{\left[ i \right]}} {G_{\left[ ij \right ]}}{{v}_{\left[ j \right]}}} } \right|\prec \Psi .
\end{equation}
Notice from the entrywise law, we can only get
\begin{equation*}
\left| {\sum\limits_{i \ne j} {{{u}^\dag_{\left[ i \right]}} {G_{\left[ ij \right ]}}{{v}_{\left[ j \right]}}} } \right|\prec \Psi \|\mathbf u\|_1 \|\mathbf v\|_1 \le N\Psi,
\end{equation*}
using $\|\mathbf u\|_1 \le N^{1/2} \|\mathbf u\|_2$ and $\|\mathbf v\|_1 \le N^{1/2} \|\mathbf v\|_2$. In particular, this estimate of the $\ell^1$ norm is sharp when $\mathbf u,\mathbf v$ are delocalized, i.e. their entries have size of order $N^{-1/2}$.

The estimate (\ref{eq_iso_11}) follows from the Chebyshev's inequality if we can prove the following lemma.

\begin{lem}\label{iso_lemm_1}
Suppose the assumptions in Proposition \ref{iso_prop} hold. For any even $p\in 2\mathbb N$, there exists a constant $C_p$ which is independent of $N$ such that
\[\mathbb E{\left| {\sum\limits_{i \ne j} {{{u}^\dag_{\left[ i \right]}} {G_{\left[ ij \right ]}}{{v}_{\left[ j \right]}}} } \right|^p} \le {C_p}{\Psi ^p}.\]
\end{lem}

The proof of Lemma $\ref{iso_lemm_1}$ is based on the polynomialization method developed in \cite[section 5]{isotropic}. Again we only give the proof for $w\in \mathbf D$. When $w\in \mathbf D_L$, the proof is almost the same.

\subsection{Rescaling and partition of indices}
For our purpose, it is convenient to define the rescaled matrix
\begin{equation}\label{iso_intro_R}
{R^{(J)}}: = w^{1/2} G^{(J)},
\end{equation}
for any $J\subset\sI$ and $|J|\le l$ for some fixed $l$. Consequently we define the control parameter $\Phi$
\begin{equation}\Phi=\left|w\right|^{1/2}\Psi.\end{equation}
By the entrywise law, for $w\in \mathbf D$,
\begin{equation}\label{rescaled_Rorder}
R_{\left[ii\right]}^{(J)}=O_{\prec}(1), \ \ \left(R_{\left[ii\right]}^{(J)}\right)^{-1}=O_{\prec}(1), \ \ R_{\left[ij\right]}^{(J)}=O_\prec(\Phi) \text{ for } i\ne j
\end{equation}
under the above scaling. Now to prove Lemma \ref{iso_lemm_1}, it is equivalent to prove
\begin{equation}\label{eq_iso_2}
\mathbb E{\left| {\sum\limits_{i \ne j} {{{u}^\dag_{\left[ i \right]}} {R_{\left[ ij \right ]}}{{v}_{\left[ j \right]}} } } \right|^p} \le {C_p}{\Phi ^p}.
\end{equation}
We expand the product in (\ref{eq_iso_2}) as
$${ \left| {  \sum\limits_{i \ne j} {{{u}^\dag_{\left[ i \right]}} {R_{\left[ ij \right]}} {{v}_{\left[ j \right]}}} } \right|^p} = \sum\limits_{ {i_k \ne j_k} \in {{\cal I}_1}}\prod\limits_{k = 1}^{p/2} {{u}^\dag_{\left[ i_k \right]}R_{\left[{i_k}{j_k}\right]}{v}_{\left[ j_k \right]}}  \cdot \prod\limits_{k = p/2 + 1}^p \overline {{u}^\dag_{\left[ i_k \right]}R_{\left[{i_k}{j_k}\right]}{v}_{\left[ j_k \right]}} . $$
Formally, we regard $\{i_1,...,i_p, j_1,...,j_p\}$ as the set of $2p$ (index) variables that take values in $\mathcal I_1$. Let $\mathcal B_p$ be the collection of all partitions of $\{i_1,...,i_p, j_1,...,j_p\}$ such that $i_k,j_k$ are not in the same block for all $k=1,...,p$. For $\Gamma \in \mathcal B_p$, let $n(\Gamma)$ be the number of its blocks and define a set of $\mathcal I_1$-valued variables as
\begin{equation}\label{iso_defn_L}
L(\Gamma):=\{b_1,...,b_{n(\Gamma)}\}.
\end{equation}
Now it is convenient to regard $\Gamma$ as a symbol-to-symbol function
\begin{equation}\label{def_gammapart}
\Gamma: \{i_1,...,i_p,j_1,...,j_p\}\rightarrow L(\Gamma),
\end{equation}
such that each $\Gamma^{-1}\left(b_k\right)$ is a block of the partition. Then we can rewrite the sum as
\begin{equation}
\left| { \sum\limits_{i \ne j} {{{u}^\dag_{\left[ i \right]}} {R_{\left[ ij \right ]}}{{v}_{\left[ j \right]}} } } \right|^p =\sum\limits_{\Gamma  \in {{\cal B}_p}} {\sum\limits_{\scriptstyle b_l \in {{\cal I}_1},\atop
\scriptstyle l = 1,...,n(\Gamma)}^* {\prod\limits_{k = 1}^{p/2} {{u}^\dag_{\left[ \Gamma(i_k) \right]}R_{\left[{\Gamma(i_k)}{\Gamma(j_k)}\right]}{v}_{\left[ \Gamma(j_k) \right]}}  \cdot \prod\limits_{k = p/2 + 1}^p \overline {{u}^\dag_{\left[ \Gamma(i_k) \right]}R_{\left[{\Gamma(i_k)}{\Gamma(j_k)}\right]}{v}_{\left[ \Gamma(j_k) \right]}}} },\label{sum_1}
\end{equation}
where $\Sigma^*$ denote the summation subject to the condition that the values of $b_1,...b_n$ are ordered as $b_1 < b_2 <\ldots <b_n$.
We pick one term from the above summation and denote
\begin{equation}\label{iso_defn_Delta}
\Delta(\Gamma) := {\prod\limits_{k = 1}^{p/2} {{u}^\dag_{\left[ \Gamma(i_k) \right]}R_{\left[{\Gamma(i_k)}{\Gamma(j_k)}\right]}{v}_{\left[ \Gamma(j_k) \right]}}  \cdot \prod\limits_{k = p/2 + 1}^p \overline {{u}^\dag_{\left[ \Gamma(i_k) \right]}R_{\left[{\Gamma(i_k)}{\Gamma(j_k)}\right]}{v}_{\left[ \Gamma(j_k) \right]}}}.
\end{equation}

\noindent{\bf Notations:} For any $b_k\in L$, we can define a corresponding $\mathcal I_2$-valued variable $\bar b_k$ in the obvious way, and we denote
\begin{equation}\label{iso_defn_barL}
[L]:=\{b_1,...,b_{n},\conj {b_1},...,\conj {b_{n}}\}.
\end{equation}
For notational convenience, we will also use letters $i,j,k,l$ to denote the symbols in $L$.

\subsection{String and string operators}

During the proof we will frequently use the following resolvent identities for rescaled matrix $R$. They follows immediately from Lemma \ref{lemm_resolvent_group}.

\begin{lem}[Resolvent identities for $R_{[ij]}$ groups]\label{iso_lemm_2}

For $k\notin J$ and $i, j \in \mathcal I_1 \setminus J\cup\{k\}$, we have
\begin{align}
& R_{\left[ij\right]}^{ \left[J\right] } = R_{\left[ij\right]}^{ {\left[Jk\right]} } + {R_{\left[ik\right]}^{\left[ J \right]}\left({R_{\left[kk\right]}^{\left[ J \right]}}\right)^{-1} R_{\left[kj\right]}^{\left[ J \right]}}, \label{eq_diagR} \\
& \left({R_{\left[ii\right]}^{\left[ J \right]}}\right)^{-1} = \left({R_{\left[ii\right]}^{\left[ {Jk} \right]}}\right)^{-1} - {\left({R_{\left[ii\right]}^{\left[ J \right]}}\right)^{-1}R_{\left[ik\right]}^{\left[ J \right]}\left({R_{\left[kk\right]}^{\left[ {J} \right]}}\right)^{-1}R_{\left[ki\right]}^{\left[ J \right]}\left({R_{\left[ii\right]}^{\left[ Jk \right]}}\right)^{-1}},\label{eq_diagR2} \\
& \left({R_{\left[ii\right]}^{\left[ J \right]} }\right)^{-1} =  w^{-1/2}H_{\left[ii\right]}^{\left[J\right]}  - w^{-1}\sum\limits_{l,l' \notin J\cup \{i\}}H_{\left[il\right]}^{\left[J\right]}R_{\left[ll'\right]}^{\left[Ji\right]}H_{\left[l'i\right]}^{\left[J\right]} .\label{eq_diagR3}
\end{align}
Furthermore, for $i\ne j$ and $L$ defined in (\ref{iso_defn_L}), we have
\begin{equation}
\label{iso_S12}
R_{\left[ij\right]}^{ \left[L\backslash\{ij\}\right] } =  R_{\left[ii\right]}^{ \left[L\backslash\{ij\}\right] }S_{\left[ij\right]}R_{\left[jj\right]}^{ \left[L\backslash\{j\}\right] },\ \text{ with } \
 S_{\left[ij\right]} = - w^{-1/2}H_{\left[ij\right]}+w^{-1}\sum\limits_{k,l \notin L} H_{\left[ik\right]}R_{\left[kl\right]}^{\left[L\right]}H_{\left[lj\right]}.
\end{equation}
\end{lem}
In this section, we expand the $R$ variables in $\Delta(\Gamma)$ using the identities in Lemma \ref{iso_lemm_2}. During the expansion, we need to distinguish carefully between an algebraic expression and its values as a random variable.

\begin{defn}[Strings]
Let $\mathfrak A$ be an alphabet containing all symbols that may appear during the expansion, such as $R_{[ij]}^{[J]}$, $\left({R_{[ij]}^{[J]}}\right)^{-1}$, $S_{[ij]}$, ${u}^\dag_{[i]}$ and ${v}_{[j]}$ for $i, j, J\subset L(\Gamma)$. We define a string $\mathbf s$ to be a formal expression consisting of the symbols from $\mathfrak A$, and denote by $\left\llbracket\bf s\right\rrbracket$ the random variable represented by it. Let $\mathfrak M$ be the collection of all possible strings. We denote an empty string by $\emptyset$. 
\end{defn}

Given a string $\mathbf s$, after an expansion of $R$'s in it, we will get a different string $\mathbf s'$. However they represent the same random variable $\left\llbracket\bf s\right\rrbracket=\left\llbracket\bf s'\right\rrbracket$. During the proof, we will identify more elements of $\mathfrak A$ (see the symbols in (\ref{iso_decomp_1})).

To perform the expansions in a systematical way, we define the following operators acting on strings. We call the symbols $R_{\left[ij\right]}^{ [J] }$,  $\left({R_{\left[ij\right]}^{ [J] }}\right)^{-1}$ to be {\it maximally expanded} if $J\cup \{i,j\}=L$. We call a string $\mathbf s$ to be {\it maximally expanded} if all the $R$ symbols in $\mathbf s$ is maximally expanded.

\begin{defn}[String operators]\label{defn_stringoperator}

(i) Define an operator $\tau_0^{(k)}$ for $\Omega\in\mathfrak M$, in the following sense. Find the first $R_{[ij]}^{[J]}$ in $\Omega$ such that $k\notin J\cup \{i,j\}$, or the first $\left({R_{[ii]}^{[J]}}\right)^{-1}$ such that $k\notin J\cup \{i\}$. If $R_{[ij]}^{[J]}$
is found, replace it with $R_{[ij]}^{[Jk]}$; if  $\left({R_{[ii]}^{[J]}}\right)^{-1}$ is found, replace it with $\left({R_{[ii]}^{[Jk]}}\right)^{-1}$; if neither is found, $\tau_0^{(k)}(\Omega) = \Omega$ and we say that $\tau_0^{(k)}$ is trivial for $\Omega$.

(ii) Define an operator $\tau_1^{(k)}$ for $\Omega\in\mathfrak M$, in the following sense. Find the first $R_{[ij]}^{[J]}$ in $\Omega$ such that $k\notin J\cup \{i,j\}$, or the first $\left({R_{[ii]}^{[J]}}\right)^{-1}$ such that $k\notin T\cup \{i\}$. If $R_{[ij]}^{[J]}$
is found, replace it with ${R_{\left[ik\right]}^{\left[ J \right]}\left({R_{\left[kk\right]}^{\left[ J \right]}}\right)^{-1} R_{\left[kj\right]}^{\left[ J \right]}}$; if  $\left({R_{[ii]}^{[J]}}\right)^{-1}$ is found, replace it with ${-\left({R_{\left[ii\right]}^{\left[ J \right]}}\right)^{-1}R_{\left[ik\right]}^{\left[ J \right]}\left({R_{\left[kk\right]}^{\left[ {J} \right]}}\right)^{-1}R_{\left[ki\right]}^{\left[ J \right]}\left({R_{\left[ii\right]}^{\left[ Jk \right]}}\right)^{-1}}$;
if neither is found, $\tau_1^{(k)}(\Omega)=\emptyset$ and we say that $\tau_1^{(k)}$ is null for $\Omega$.

(iii) Define an operator $\rho$ for $\Omega\in\mathfrak M$, in the following sense. Find each maximally expanded $R_{\left[ij\right]}^{ \left[L\backslash\{ij\}\right] }$ in $\Omega$ and replace it with $R_{\left[ii\right]}^{ \left[L\backslash\{ij\}\right] }S_{[ij]}R_{\left[jj\right]}^{ \left[L\backslash\{j\}\right] }$. If nothing is found,  $\rho(\Omega)=\Omega$.
\end{defn}

According to Lemma \ref{iso_lemm_2}, for any $\Omega\in\mathfrak M$ we have
\begin{equation}\label{operator_decompose}
\left\llbracket \left(\tau_0^{(k)}+\tau_1^{(k)}\right)(\Omega)\right\rrbracket = \left\llbracket\Omega\right\rrbracket, \ \
\left\llbracket\rho(\Omega)\right\rrbracket=\left\llbracket\Omega\right\rrbracket
\end{equation}

\begin{defn}
Define the function ${\cal F}_{\rm{d-max}}:\mathfrak M\rightarrow \mathbb N$ (where the subscript ``d-max" stands for ``distance to being maximally expanded") through
$${\cal F}_{\rm{d-max}}\left({R_{\left[ij\right]}^{ \left[J\right] }}^*\right)=\left|L\backslash\left(J\cup\{i,j\}\right)\right|,$$
where $*$ could be $1$ or $-1$, and
$${\cal F}_{\rm{d-max}}(\Omega) = \sum\limits_{R\text{ variables in }\Omega}{\cal F}_{\rm{d-max}}(R).$$
Define another function ${\cal F}_{\rm{off}}:\mathfrak M\rightarrow \mathbb N$ with ${\cal F}_{\rm{off}}(\Omega)$ being the number of off-diagonal symbols in $\Omega$.
\end{defn}

By off-diagonal symbols, we mean the terms of the form $A_{st}$ with $s\notin \{t,\bar t\}$ or $A_{[ij]}$ with $i\ne j$, e.g. $R_{[ij]}^{[J]}$ and $S_{[ij]}$ with $i\ne j$. Later we will define other types of off-diagonal symbols (see (\ref{iso_decomp_1})). Note that a $R$ symbol is maximally expanded if and only if ${\cal F}_{\rm{d-max}}(R)=0$ and a string $\Omega$  is maximally expanded if and only if ${\cal F}_{\rm{d-max}}(\Omega)=0$. The next two lemmas are almost trivial by Definition \ref{defn_stringoperator}.

\begin{lem}
If $\tau_0^{(k)}(\Omega) = \Omega$ and $\tau_1^{(k)}(\Omega) = \emptyset$,
\begin{equation}\label{RNum_1}
 {\cal F}_{\rm{d-max}}\left(\tau_0^{(k)}(\Omega)\right) = {\cal F}_{\rm{d-max}}(\Omega), \ \
 {\cal F}_{\rm{d-max}}\left(\tau_1^{(k)}(\Omega)\right) = 0;
\end{equation}
otherwise,
\begin{equation}\label{RNum_11}
 {\cal F}_{\rm{d-max}}\left(\tau_0^{(k)}(\Omega)\right) = {\cal F}_{\rm{d-max}}(\Omega) - 1,\ \
 {\cal F}_{\rm{d-max}}\left(\tau_1^{(k)}(\Omega)\right) \le {\cal F}_{\rm{d-max}}(\Omega) + 4n(\Gamma).
\end{equation}
For $\rho$, we have
\begin{equation}\label{RNum_3}
 {\cal F}_{\rm{d-max}}\left(\rho(\Omega)\right) = {\cal F}_{\rm{d-max}}(\Omega) + a,
\end{equation}
where $a$ is the number of maximally expanded off-diagonal $R$'s in $\Omega$.
\end{lem}

\begin{lem}
For any $\Omega\in\mathfrak M$, we have
\begin{equation}\label{offNum_1}
 {\cal F}_{\rm{off}}\left(\tau_0^{(k)}(\Omega)\right) = {\cal F}_{\rm{off}}(\Omega), \ \  {\cal F}_{\rm{off}}\left(\rho(\Omega)\right) = {\cal F}_{\rm{off}}(\Omega),
\end{equation}
and
\begin{equation}\label{offNum_2}
 {\cal F}_{\rm{off}}(\Omega) + 1 \le {\cal F}_{\rm{off}}\left(\tau_1^{(k)}(\Omega)\right) \le {\cal F}_{\rm{off}}(\Omega) + 2 \ \text{ if }\ \tau_1^{(k)}(\Omega)\neq \emptyset.
\end{equation}
\end{lem}

\subsection{Expansion of the strings}

For simplicity of notations, throughout the rest of this section we omit the complex conjugates on the right hand side of (\ref{iso_defn_Delta}) (if we keep the complex conjugates, the proof is the same but with slightly heavier notations). Suppose the right hand side of (\ref{iso_defn_Delta}) is represented by a string $\Omega_\Delta$. Given a binary word ${\bf w} = a_1a_2...a_m$ with $a_i\in\{0,1\}$, we define the operation
\begin{equation}\label{eq_oper}
(\Omega_\Delta)_{\mathbf w} = \rho\tau_{a_{m}}^{(b_{m})}\cdots \rho\tau_{a_2}^{(b_2)}\rho\tau_{a_1}^{(b_1)}\left(\Omega_\Delta\right)
\end{equation}
where $b_{qn+r} := b_{r}$ (recall (\ref{iso_defn_L})) for any $1\le r \le n$ and $q\in\mathbb N$. So a binary words $\bw$ uniquely determines an operator composition.
By (\ref{operator_decompose}),
$\left\llbracket (\Omega_\Delta)_{\mathbf w0}\right\rrbracket +\left\llbracket (\Omega_\Delta)_{\mathbf w1}\right\rrbracket = \left\llbracket(\Omega_\Delta)_{\bw}\right\rrbracket$
and so we get
$$\sum\limits_{|\mathbf w|=m}\left\llbracket(\Omega_\Delta)_{\bw}\right\rrbracket = \left\llbracket\Omega_\Delta\right\rrbracket$$
for any $m\ge 1$, where $|\mathbf w|$ is the length of $\bw$.

\begin{lem}\label{iso_lem_3} 
Given any $\bw$ such that $|\mathbf w|=(n^2+1)(p+6l_0)$ and $(\Omega_\Delta)_{\bw}\neq \emptyset$, either ${\cal F}_{\rm{off}}((\Omega_\Delta)_{\bw})\ge l_0:= \left({8}/{\zeta} + 2\right)p,$ or $(\Omega_\Delta)_{\bw}$ is maximally expanded.
\end{lem}
\begin{proof}
We use $m_0$ to denote the number of $0$'s in $\bw$, and $m_1$ to denote the number of $1$'s.
Furthermore, we use $m_0^{(0)}$ to denote the number of $0$'s corresponding to the trivial $\tau_0$'s, and $m_0^{(1)}$ to denote the number of $0$'s corresponding to the non-trivial $\tau_0$'s. Assume ${\cal F}_{\rm{off}}((\Omega_\Delta)_{\bw})< l_0$ and $(\Omega_\Delta)_{\bw}$ is not maximally expanded. By (\ref{offNum_1})-(\ref{offNum_2}), $m_1\le l_0 - p\le l_0$. By (\ref{RNum_1})-(\ref{RNum_3}),
 $${\cal F}_{\rm{d-max}}((\Omega_\Delta)_{\bw})\le {\cal F}_{\rm{d-max}}(\Omega_\Delta)+l_0+4nm_1-m_{0}^{(1)}.$$
Using ${\cal F}_{\rm{d-max}}(\Omega_\Delta)=np$, we get a rough estimate $m_{0}^{(1)}+m_1 < n(p+6l_0)$. By pigeonhole principle, there are at least $n$ $0$'s in a row in $\bw$ that correspond to trivial $\tau_0$'s. This indicates that $(\Omega_\Delta)_{\bw}$ is maximally expanded, which gives a contradiction.
\end{proof}

\begin{lem}\label{iso_lem_trivial}
There exists constants $C_{p,l_0}, C_{p,\zeta}>0$ such that
\begin{equation}
 \sum\limits_{\Gamma  \in {{\cal B}_p}} \sum\limits_{\scriptstyle b_l \in {{\cal I}_1},\atop
\scriptstyle l = 1,...,n(\Gamma)}^*\left|\mathbb E\sum\limits_{\scriptstyle |\mathbf w|=(n^2+1)(p+6l_0), \atop\scriptstyle {\cal F}_{\rm{off}}((\Omega_{\Delta(\Gamma)})_{\bw})\ge l_0}\left\llbracket(\Omega_{\Delta(\Gamma)})_{\bw}\right\rrbracket\right|\le C_{p,l_0} N^{2p} \Phi^{l_0}\le C_{p,\zeta}\Phi^p.
\end{equation}
\end{lem}
\begin{proof}
The first bound is due to the fact that each summand is bounded by $C\Phi^{l_0}$ and there are at most $N^{2p}$ of them. For the second bound, we used $\Phi\le CN^{-\zeta/2}.$
\end{proof}

This lemma shows that all the strings with sufficiently many off-diagonal symbols contributes at most $\Phi^p$. It only remains to handle the maximally expanded strings. Define a diagonal symbol as
\begin{equation}
S_{\left[ii\right]} := -\begin{pmatrix}0 & d_iX_{i\bar{i}}\\
\bar d_iX_{\bar{ii}}^{\dagger} & 0
\end{pmatrix}
 + w^{-1}\sum\limits_{k,l \notin L} H_{\left[ik\right]}R_{\left[kl\right]}^{\left[L\right]}H_{\left[li\right]},
\end{equation}
such that
\begin{equation}\label{diag_exp_1}
\left(R_{\left[ii \right]}^{\left[L\backslash\{i\}\right]}\right)^{-1} = \begin{pmatrix}-w^{1/2} & -z\\
-\bar{z} & -w^{1/2}
\end{pmatrix}
 - S_{\left[ii\right]}.
\end{equation}
Notice all the $R$ symbols in a maximally expanded string is diagonal. We taylor expand $R_{\left[ii \right]}^{\left[L\backslash\{i\}\right]}$ as
\begin{align}
R_{\left[ii \right]}^{\left[L\backslash\{i\}\right]}=&\left[w^{-1/2}\pi_{[i]c}^{-1}+\left(S_{\left[ii\right]}-B_i\right)\right]^{-1} =\sum\limits_{k=0}^{l_0-1} \tilde\pi_{ic}\left[\left(S_{[ii]}-B_i\right)\tilde\pi_{ic}\right]^{k}+O_{\prec}\left(\Phi^{l_0}\right),\label{diag_exp_2}
\end{align}
where $\tilde\pi_{[i]c}=w^{1/2}\pi_{[i]c}$, $B_i=\begin{pmatrix}w^{1/2}|d_i|^2m_{2c} & 0\\
0 & w^{1/2}m_{1c}
\end{pmatrix}$, and for the error term,
\[ S_{[ii]}-B_i=w^{-1/2}Z_{[i]}^{[L\setminus \{i\}]}+w^{1/2}\begin{pmatrix}|d_i|^2(m_{2c}-m_{2}^{[L]}) & 0\\
0 & m_{1c}-m_{1}^{[L]}
\end{pmatrix} \prec \Phi \]
by (\ref{Zestimate}) and the averaged local law. Now for all maximally expanded $(\Omega_\Delta)_{\bw}$ with  $|\mathbf w|=(n^2+1)(p+6l_0)$, denote by $\sigma\left\llbracket(\Omega_\Delta)_{\bw}\right\rrbracket$ the expression after plugging in (\ref{diag_exp_1}) and (\ref{diag_exp_2}) without the tail terms. Similar to Lemma \ref{iso_lem_trivial}, we have
\[\sum\limits_{\Gamma  \in {{\cal B}_p}} {\sum\limits_{\scriptstyle b_l \in {{\cal I}_1},\atop
\scriptstyle l = 1,...,n(\Gamma)}^* \left|\mathbb E{\sum\limits_{\scriptstyle |\mathbf w|=(n^2+1)(p+6l_0), \atop
\scriptstyle (\Omega_\Delta)_{\mathbf w}\text{ maximally expanded}}\left( \left\llbracket(\Omega_{\Delta(\Gamma)})_{\bw}\right\rrbracket-\sigma\left\llbracket(\Omega_{\Delta(\Gamma)})_{\bw}\right\rrbracket \right)}\right| } \le {C_{p,\zeta }}{\Phi ^p}.\]
From the above bound and Lemmas \ref{iso_lem_3}, \ref{iso_lem_trivial}, we see that to prove (\ref{eq_iso_2}), it suffices to show
\begin{equation}\label{eq_iso_goal2}
\sum\limits_{\Gamma  \in {{\cal B}_p}} {\sum\limits_{\scriptstyle b_l \in {{\cal I}_1},\atop
\scriptstyle l = 1,...,n(\Gamma)}^* {\left|\mathbb E\sum\limits_{\scriptstyle |\mathbf w|=(n^2+1)(p+6l_0), \atop
\scriptstyle (\Omega_\Delta)_{\bw}\text{ maximally expanded}}\sigma \left\llbracket(\Omega_{\Delta(\Gamma)})_{\bw}\right\rrbracket \right|} } \le {C_{p,\zeta}}{\Phi ^p}.\end{equation}

We write $\sigma\left\llbracket(\Omega_\Delta)_{\bw}\right\rrbracket$ as a sum of monomials in terms of $S_{[ij]}$,
\begin{equation}\label{eq_iso_mono1}
\sigma\left\llbracket(\Omega_\Delta)_{\bw}\right\rrbracket=\sum\limits_{i}M(\bw, \Delta(\Gamma),i),
\end{equation}
where $i$ is an index to label these monomials. Notice that after plugging ($\ref{eq_iso_mono1}$) into (\ref{eq_iso_goal2}), the number of summands $M(\bw, \Delta(\Gamma),i)$ inside the expectation only depends on $p$ and $\zeta$. Thus to show (\ref{eq_iso_goal2}), it suffices to prove the following lemma.
\begin{lem}
Fix any $\Gamma  \in {{\cal B}_p}$ and binary word $\mathbf w$ with $|\mathbf w|=(n^2+1)(p+6l_0)$. Suppose $(\Omega_\Delta)_{\bw}$ is maximally expanded. Let $M(\bw, \Delta(\Gamma))$ be an monomial in $\sigma\left\llbracket(\Omega_{\Delta(\Gamma)})_{\bw}\right\rrbracket$. We have
\begin{equation}\label{eq_iso_goal3} {\sum\limits_{\scriptstyle b_l \in {{\cal I}_1},  l = 1,...,n(\Gamma)}^* \left|\mathbb E{M(\bw, \Delta(\Gamma))} \right|} \le {C_{p,\zeta }}{\Phi ^p}\end{equation}
for some constant $C_{p,\zeta}$ that only depends on $p$ and $\zeta$.
\end{lem}

For the rest of this section, we fix a $\Gamma\in{{\cal B}_p}$ and a maximally expanded $(\Omega_{\Delta(\Gamma)})_{\bw}$ with $|\mathbf w|=(n^2+1)(p+6l_0)$. Then we fix a monomial $M(\bw, \Delta(\Gamma))$ in $\sigma\left\llbracket(\Omega_{\Delta(\Gamma)})_{\bw}\right\rrbracket$. Let $\Omega_M$ be the string form of $M(\bw,\Delta(\Gamma))$ in terms of $S_{[ij]}$.
It is not hard to see that
\begin{equation}\label{Foff1}
{\cal F}_{\rm{off}}\left(\Omega_M\right) = {\cal F}_{\rm{off}}\left((\Omega_\Delta)_{\bw}\right).
\end{equation}
Now we decompose $S_{[ij]}$ as
\begin{equation}\label{iso_decomp_1}
S_{\left[ij\right]}=S^X_{i\bar j}+S^X_{\bar i j}+S^{R}_{i\bar j}+S^{R}_{i j}+S^{R}_{\bar i\bar j}+S^{R}_{\bar i j},
\end{equation}
where we define the following symbols in $\mathfrak A$:
\begin{equation}\label{SX}
S^X_{i\bar j}:=d_iX_{i\bar{j}}\begin{pmatrix}0 & 1\\
0 & 0
\end{pmatrix},\ \ S^X_{\bar i j}:=\bar{d_i}X_{\bar{i}j}^{\dagger}\begin{pmatrix}0 & 0\\
1 & 0
\end{pmatrix},
\end{equation}
\begin{equation}\label{SR1}
S^{R}_{i\bar j}:=\sum\limits _{k,l \notin L} d_id_lX_{i\bar{k}}X_{l\bar{j}}\begin{pmatrix}0 & R_{\bar k l}^{\left[L\right]}\\
0 & 0
\end{pmatrix},\ \ S^{R}_{i j}:=\sum\limits _{k,l \notin L} d_i\bar{d_l}X_{i\bar{k}}X_{\bar{l}j}^{\dagger}\begin{pmatrix}R_{\bar k \bar l}^{\left[L\right]} & 0\\
0 & 0
\end{pmatrix},
\end{equation}
\begin{equation}\label{SR2}
S^{R}_{\bar i\bar j}:=\sum\limits _{k,l\notin L} \bar{d_i}d_lX_{\bar{i}k}^{\dagger}X_{l\bar{j}}\begin{pmatrix}0 & 0\\0 & R_{kl}^{\left[L\right]}
\end{pmatrix},\ \ S^{R}_{\bar i j}:=\sum\limits _{k,l\notin L} \bar{d_i}\bar{d_l}X_{\bar{i}k}^{\dagger}X_{\bar{l}j}^{\dagger}\begin{pmatrix}0 & 0\\
R_{k\bar l }^{\left[L\right]} & 0
\end{pmatrix}.
\end{equation}
We expand $S_{[ij]}$'s of $M(\bw,\Delta(\Gamma))$ as in (\ref{iso_decomp_1}), and write $M(\bw,\Delta(\Gamma))$ as a sum of monomials in terms of $S_{st}^X$ and $S_{st}^R$,
\begin{equation}\label{eqn_monomial2}
M(\bw,\Delta(\Gamma))=\sum\limits_{i}Q(\bw,\Delta(\Gamma),i),
\end{equation}
where $i$ is an index to label these monomials. Again it is not hard to see that
\begin{equation}\label{Foff2}
{\cal F}_{\rm{off}}\left(\Omega_Q\right)={\cal F}_{\rm{off}}\left(\Omega_M\right) = {\cal F}_{\rm{off}}\left((\Omega_\Delta)_{\bw}\right).
\end{equation}
Since the number of summands in (\ref{eqn_monomial2}) is independent of $N$, to prove (\ref{eq_iso_goal3}) it suffices to show
\begin{equation}\label{eq_iso_goal4}
{\sum\limits_{\scriptstyle b_l \in {{\cal I}_1}, l = 1,...,n(\Gamma)}^* \left|\mathbb E{Q(\bw, \Delta(\Gamma))} \right|} \le {C_{p,\zeta}}{\Phi ^p}
\end{equation}
for any monomial $Q(\bw,\Delta(\Gamma))$ in (\ref{eqn_monomial2}). Throughout the following, we fix a $Q(\bw,\Delta(\Gamma))$ with nonzero expectation, and denote by $\Omega_Q$ the string form of $Q(\bw,\Delta(\Gamma))$ in terms of $S_{st}^X$ and $S_{st}^R$. Notice the $R$ variables in $S^R_{st}$ are maximally expanded. As a result, the $S^X_{st}$ variables are independent of $S^R_{st}$ variables in $Q(\bw, \Delta(\Gamma))$.
Therefore we make the following observation: if $S_{st}^X$ appears as a symbol in $\Omega_Q$, then $\Omega_Q$ contains at least two of them.

\begin{defn}\label{def_iso_5}
Recall $\Gamma$ defined in (\ref{def_gammapart}). Let $h$ be the number of blocks of $\Gamma$ whose size is $1$, i.e.
\begin{equation}
h:=\sum\limits_{l=1}^{n(\Gamma)}\one\left(\left|\Gamma^{-1}(b_l)\right|=1\right).
\end{equation}
For $l=1,...,n$, define
$$I_l:=\left|\{i_1, \ldots, i_p\} \cap \Gamma^{-1}(b_l)\right|, \ \ J_l:=\left|\{j_1, \ldots, j_p\} \cap \Gamma^{-1}(b_l)\right|.$$
\end{defn}

\begin{lem}
Suppose for any $b_1,...,b_n$ taking distinct values in $\sI_1$,
\begin{equation}\label{eq_iso_3}
|\bbE Q(\bw,\Delta(\Gamma))|\le CN^{-h/2}\Phi^p\prod_{l=1}^n\left|u_{[b_l]}\right|^{I_l}\left|v_{[b_l]}\right|^{J_l}
\end{equation}
holds for some constant $C$ independent of $N$. Then the estimate (\ref{eq_iso_goal4}) holds.
\end{lem}
\begin{proof}
By Cauchy-Schwarz inequality,
\[\sum\limits_{k=1}^N {\left|{u}_{\left[k\right]}\right|^{a}\left|{v}_{\left[k\right]}\right|^{b} \le }
\begin{cases}
N^{1/2} &\text{if } a+b=1\\
1 &\text{if } a+b\geq 2
\end{cases}.\]
Then using $h = \sum\limits_{l=1}^{n}\one\left(I_l+J_l=1\right),$ we get
\[\sum\limits_{\scriptstyle b_l \in {{\cal I}_1}, l = 1,...,n(\Gamma)}^* \left|\mathbb E{Q(\bw, \Delta(\Gamma))}\right| \le C \Phi^pN^{-h/2}\prod_{l=1}^n\sum\limits_{b_l\in\mathcal I_1}\left|{u}_{\left[b_l\right]}\right|^{I_l}\left|{v}_{\left[b_l\right]}\right|^{J_l}\le C \Phi^p.\]
\end{proof}

Hence it suffices to prove (\ref{eq_iso_3}). The key is to extract the $N^{-h/2}$ factor from $\bbE Q(\bw,\Delta(\Gamma))$. For this purpose, we need to keep track of the indices in $L$ during the expansion.

\begin{defn}
 Define a function ${\cal F}_{\rm{in}}:L\times\mathfrak M\rightarrow \mathbb N$ with
 ${\cal F}_{\rm{in}}(l,\Omega)$ giving the number of times $l$ or $\bar l$ appears as an index of off-diagonal $R$ or $S$ in $\Omega$.
\end{defn}

The following lemma follows immediately from Definition \ref{defn_stringoperator} and the expansions we have done to obtain $\Omega_Q$ from $(\Omega_{\Delta})_{\bw}$.
\begin{lem}\label{lem_Fin_change}
(1) For any string $\Omega$, if $\tau_0^{(k)}$ is not trivial for $\Omega$, then
\begin{equation}\label{iNum_1}
 {\cal F}_{\rm{in}}\left(l, \tau_0^{(k)}(\Omega)\right) = {\cal F}_{\rm{in}}(l, \Omega), \ \
 {\cal F}_{\rm{in}}\left(l, \tau_1^{(k)}(\Omega)\right) = {\cal F}_{\rm{in}}(l, \Omega) + 2\delta_{kl}.
\end{equation}
(2) For any string $\Omega$,
\begin{equation}\label{iNum_3}
 {\cal F}_{\rm{in}}\left(l, \rho(\Omega)\right) = {\cal F}_{\rm{in}}(l, \Omega).
\end{equation}
(3) For any maximally expanded $(\Omega_\Delta)_{\mathbf w}$,
\begin{equation}\label{iNum_4}
{\cal F}_{\rm{in}}(l, \Omega_Q)={\cal F}_{\rm{in}}(l,(\Omega_{\Delta})_{\bw}).
\end{equation}
\end{lem}

Let $\Omega_Q^X$ be the substring of $\Omega_Q$ containing only $S^X$ symbols, and $\Omega_Q^R$ be the substring of $\Omega_Q$ containing only $S^R$ symbols. Define
\begin{equation}\label{Vnode}
\mathcal V := \{l\in L|\ {\cal F}_{\rm{in}}(l,\Omega_{\Delta}) = 1\},
\end{equation}
and
\begin{equation}\label{Vnode1}
\mathcal V_0 := \{l\in L|\ {\cal F}_{\rm{in}}(l,\Omega_{\Delta}) = 1\text{ and }{\cal F}_{\rm{in}}(l, \Omega_Q^X)=0\},
\end{equation}
\begin{equation}\label{Vnode2}
\mathcal V_1 := \{l\in L|\ {\cal F}_{\rm{in}}(l,\Omega_{\Delta}) = 1\text{ and }{\cal F}_{\rm{in}}(l, \Omega_Q^X)\ge 2\}.
\end{equation}
Recall the observation above Definition \ref{def_iso_5}, $\mathcal V= \mathcal V_0 \cup \mathcal V_1$ and
\[h=|\mathcal V|=|\mathcal V_0|+|\mathcal V_1|.\]
Let $n_X$ be the number of off-diagonal $S^X$ symbols in $\Omega_Q^X$ and $n_R$ be the number of off-diagonal $S^R$ symbols in $\Omega_Q^R$. Notice that $n_{o}:=n_X+n_R$ is the total number of off-diagonal symbols in $\Omega_Q$.

\subsection{Introduction of graphs and conclusion of the proof}

We introduce the graphs to conclude the proof of (\ref{eq_iso_3}). We use a connected graph to represent the string $\Omega_Q$, call it by $\mathfrak G_{Q0}$. The indices in $[L]$ are represented by black nodes in $\mathfrak G_{Q0}$. The $S_{st}^X$ or $S_{st}^R$ symbols in $\Omega_Q$ are represented by edges connecting the nodes $s$ and $t$. We also define colors for the nodes and edges, where the color set for nodes is $\{black, white\}$ and the color set for edges is $\{S^X, S^R, X, R\}$. In $\mathfrak G_{Q0}$, all the nodes are black, all $S^X$ edges are assigned $S^X$ color and all $S^R$ edges are assigned $S^R$ color. We show a possible graph in Fig. \ref{M_0}. In this subsection, we identify an index with its node representation, and a symbol with its edge representation.

\begin{defn}
Define function $\deg$ on the nodes set $[L]$, where $\deg(l)$ is the number of $S^R$ edges connecting to the node $l$.
\end{defn}

By Lemma \ref{lem_Fin_change}, we see that for any $l\in\mathcal V_0$,
\begin{equation}\label{eq_eo}
{\cal F}_{\rm{in}}(l, \Omega_Q) \equiv \deg(l) + \deg(\bar l) \equiv 1 \pmod 2.
\end{equation}
Hence
\begin{align}
|\mathcal V_0| = \sum\limits_{l \in\mathcal V_0} \left[{\cal F}_{\rm{in}}\left(l, \Omega_Q\right) \mod 2\right] \le \sum\limits_{l \in \mathcal V_0}\left[\left(\deg(l\right) \mod 2)+\left(\deg(\bar l)\mod 2\right)\right].\label{eq_eo2}
\end{align}

Now we expand the $S^R$ edges. Take the $S^{R}_{i\bar j}$ edge as an example (recall (\ref{SR1})). We replace the $S^{R}_{i\bar j}$ edge with an $R$-group, defined as following. We add two white colored nodes to represent the summation indices $\bar k, l \notin [L]$, two $X$-colored edges to represent $X_{i\bar{k}}$ and $X_{l\bar{j}}$, and a $R$-colored edge connecting $\bar k$ and $l$ to represent
$\begin{pmatrix}0 & R_{\bar k l}^{\left[L\right]}\\
0 & 0
\end{pmatrix} .$ We call the subgraph consisting of the three new edges and their nodes an $R$-group. If $i=j$, we call it a diagonal $R$-group; otherwise, call it an off-diagonal $R$-group. We expand all $S^{R}$ edges in $\mathfrak G_{Q0}$ into $R$-groups and call the resulting graph $\mathfrak G_{Q1}$. For example, after expanding the $S^R$ edges in Fig.\,\ref{M_0}, we get the graph in Fig.\,\ref{M_1}. In the graph $\mathfrak G_{Q1}$, the $R$ edges, $X$ edges and $S^X$ edges are mutually independent, since the $R$ symbols are maximally expanded, and the white nodes are different from the black nodes.

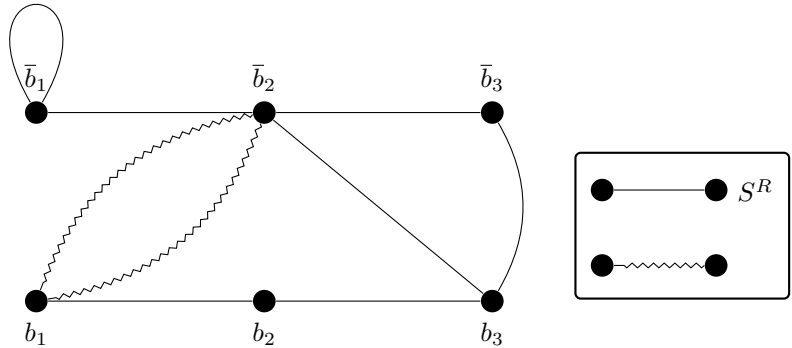
\begin{figure}[h]
\centering
\begin{tikzpicture}
 \tikzset{every loop/.style={min distance=20mm,in=60,out=120,looseness=-1}}
 \tikzset{dot/.style={circle,fill=#1,inner sep=3,minimum size=0.5pt}}
 \tikzset{zig/.style={decoration={
    zigzag,
    segment length=4,
    amplitude=.9,post=lineto,
    post length=2pt}}}
 \node (a2) [label=above:{$\overline b_1$}, dot] {} edge [loop] (a2);
 \node (a3) [label=below:{$b_1$}, below of=a2, xshift=0mm, yshift = -15mm, dot] {}; 
 \node (a4) [label=above:{$\overline b_2$}, right of=a2, xshift = 20mm, dot] {} edge  (a2) edge [bend left, decorate, zig]  (a3) edge [bend right, decorate, zig]  (a3);
 \node (a5) [label=below:{$b_2$}, right of=a3, xshift = 20mm, dot] {} edge  (a3) ;
 \node (a6) [label=above:{$\overline b_3$}, right of=a4, xshift = 20mm, dot] {} edge (a4);
 \node (a7) [label=below:{$b_3$}, right of=a5, xshift = 20mm, dot] {} edge [bend right] (a6) {} edge (a4) edge  (a5);
 \node (l) [draw=black,thick,rounded corners=2pt,below left=10mm, minimum width = 80pt, minimum height = 55pt, right of=a6,xshift = 25mm, yshift = -5mm]  {};
 \node (l1) [dot, above of=l, xshift = -30pt, yshift = -15pt] {};\node (l2) [dot, right of= l1, xshift = 5mm, label=right:{$S^R$}] {} edge (l1) ;
 \node (l3) [dot, below of = l1] {};\node (l4) [dot, right of= l3, xshift = 5mm] {} edge [decorate, zig, label=right:{$S^X$}] (l3) ;
\end{tikzpicture}
\caption{An example of the graph $\mathfrak G_{Q0}$.}
\label{M_0}
\end{figure}


\begin{figure}[htb]
\centering
\begin{tikzpicture}
 \tikzset{dot/.style={circle,fill=#1,inner sep=3,minimum size=0.5pt}}
 \tikzset{wdot/.style={circle,draw,inner sep=3,minimum size=0.5pt}}
 \tikzset{zig/.style={decoration={
    zigzag,
    segment length=4,
    amplitude=.9,post=lineto,
    post length=2pt}}}
 \node (a2) [label=right:{$\overline b_1$}, dot] {};
 \node (c1) [left of=a2, xshift=-6mm, yshift = 6mm, wdot] {} edge [decorate,zig]  (a2);
 \node (c2) [left of=a2, xshift=-7mm, yshift = -4mm, wdot] {} edge[decorate,zig] (a2) edge (c1);
 \node (b1) [above of=a2, xshift=3mm, yshift = 8mm, wdot] {} edge[decorate,zig] (a2);
 \node (b2) [above of=a4, xshift=-3mm, yshift = 8mm, wdot] {} edge[decorate,zig] (a4) edge (b1);
 \node (a3) [label=left:{$b_1$}, below of=a2, xshift=0mm, yshift = -15mm, dot] {}; 
 \node (a4) [label=below:{$\overline b_2$}, right of=a2, xshift = 20mm, dot] {} edge [bend left, decorate,zig] (a3) edge [bend right,decorate, zig] (a3);
 \node (a5) [label=left:{$b_2$}, right of=a3, xshift = 20mm, dot] {};
 \node (a6) [label=below:{$\overline b_3$}, right of=a4, xshift = 20mm, dot] {};
 \node (a7) [label=left:{$b_3$}, right of=a5, xshift = 20mm, dot] {};
 \node (b9) [above of=a4, xshift=3mm, yshift = 8mm, wdot] {} edge [decorate,zig] (a4);
 \node (b0) [above of=a6, xshift=-3mm, yshift = 8mm, wdot] {} edge [decorate,zig] (a6) edge  (b9);
 \node (b3) [right of=a6, xshift=8mm, yshift = -3mm, wdot] {} edge[decorate,zig] (a6);
 \node (b4) [right of=a7, xshift=8mm, yshift = 3mm, wdot] {} edge[decorate,zig] (a7) edge (b3);
 \node (b5) [below of=a3, xshift=3mm, yshift = -8mm, wdot] {} edge[decorate,zig] (a3);
 \node (b6) [below of=a5, xshift=-3mm, yshift = -8mm, wdot] {} edge[decorate,zig] (a5) edge (b5);
 \node (b7) [below of=a5, xshift=3mm, yshift = -8mm, wdot] {} edge[decorate,zig]  (a5);
 \node (b8) [below of=a7, xshift=-3mm, yshift = -8mm, wdot] {} edge[decorate,zig] (a7) edge (b7);
 \node (c3) [right of=a4, xshift=3mm, yshift=-4mm, wdot] {} edge[decorate,zig] (a4);
 \node (c4) [above of=a7, xshift=-5mm, yshift=3mm, wdot] {} edge[decorate,zig] (a7) edge (c3);
 \node (l) [draw=black,thick,rounded corners=2pt,below left=10mm, minimum width = 75pt, minimum height = 90pt, right of=a6,xshift = 40mm, yshift = -5mm]  {};
 \node (l1) [wdot, above of=l, xshift = -30pt, yshift = 0pt] {};\node (l2) [wdot, right of= l1, xshift = 5mm, label=right:{$R$}] {} edge (l1) ;
 \node (l3) [dot, below of = l1] {};\node (l4) [wdot, right of= l3, xshift = 5mm] {} edge [decorate, zig, label=right:{$X$}] (l3) ;
 \node (l5) [dot, below of = l3] {};\node (l6) [dot, right of= l5, xshift = 5mm] {} edge [decorate, zig, label=right:{$S^X$}] (l5) ;
\end{tikzpicture}
\caption{The resulting graph $\mathfrak G_{Q1}$ after expanding each $S^{R}$ in Fig.\,\ref{M_0} into $R$-groups.}
\label{M_1}
\end{figure}
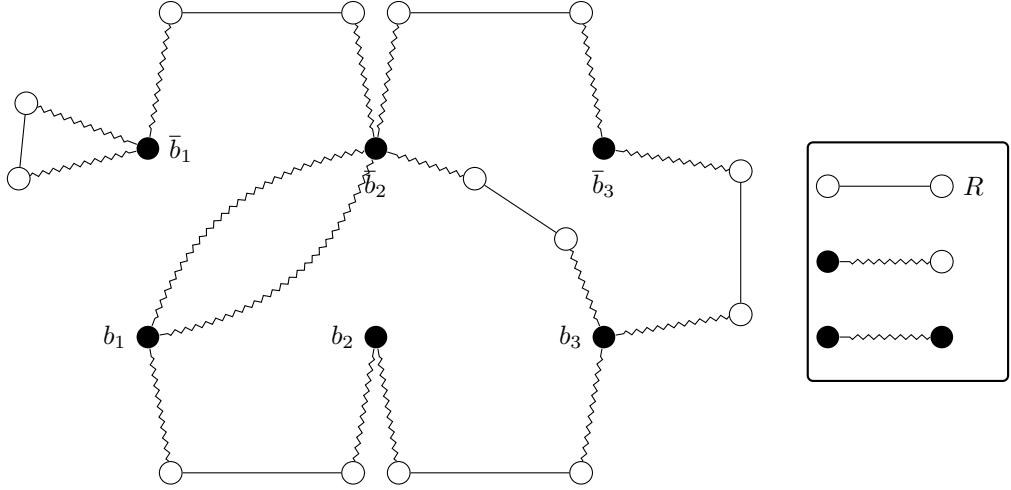

Notice that each white node represents a summation index. As we have done for the black nodes, we first partition the white nodes into blocks and then assign values to the blocks when doing the summation. Let $W$ be the set of all white nodes in $\mathfrak G_{Q1}$, and let $\mathcal W$ be the collection of all partitions of $W$. Fix a partition $\gamma \in\mathcal W$ and denote its blocks by $W_1,...,W_{m(\gamma)}$. If two white nodes of some off-diagonal $R$-group happen to lie in the same block, then we merge the two nodes into one diamond white node (Fig.\,\ref{diamond_node}). All the other white nodes are called normal (Fig.\,\ref{normal_node}). Let $n_R^{(d)}$ be the number of diamond nodes ($\le$ the number of diagonal $R$-edges in $\mathfrak G_{Q1}$). Then we trivially have
\begin{equation}\label{number_white_node}
\text{\# of white nodes} = -n_R^{(d)}+\sum_{k = 1}^{n} \left[{\deg \left( {{b_k}} \right) + \deg ({{\bar b}_k})} \right].
\end{equation}

\begin{figure}[htb]
\centering
 \begin{subfigure}[h]{4cm}\centering
  \begin{tikzpicture}
 \tikzset{every loop/.style={min distance=17mm,in=60,out=120,looseness=-1}}
 \tikzset{dot/.style={circle,fill=#1,inner sep=3,minimum size=0.5pt}}
 \tikzset{wdot/.style={diamond,draw,inner sep=2,minimum size=0.5pt}}
 \tikzset{zig/.style={decoration={
    zigzag,
    segment length=4,
    amplitude=.9,post=lineto,
    post length=2pt}}}
 \node (a2) [label=below:{$ $}, dot] {};
 \node (a4) [label=below:{$ $}, right of=a2, xshift = 17mm, dot] {};
 \node (b1) [above of=a2, xshift=12mm, yshift = 6mm, wdot] {} edge [decorate, zig] node[right] {} (a2) edge [decorate, zig] node[left] {} (a4) edge [loop] node[above] {} (b1);
  \end{tikzpicture}
  \caption{Diamond white node.}
  \label{diamond_node}
 \end{subfigure}
 \begin{subfigure}[h]{10cm}\centering
  \begin{tikzpicture}
 \tikzset{every loop/.style={min distance=20mm,in=60,out=120,looseness=-1}}
 \tikzset{dot/.style={circle,fill=#1,inner sep=3,minimum size=0.5pt}}
 \tikzset{wdot/.style={circle,draw,inner sep=3,minimum size=0.5pt}}
 \tikzset{zig/.style={decoration={
    zigzag,
    segment length=4,
    amplitude=.9,post=lineto,
    post length=2pt}}}
 \node (a2) [label=right:{$ $}, dot] {};
 \node (c1) [above of=a2, xshift=-4mm, yshift = 8mm, wdot] {} edge [decorate, zig] node[left] {} (a2);
 \node (c2) [above of=a2, xshift=4mm, yshift = 8mm, wdot] {} edge [decorate, zig] node[right] {} (a2) edge node[above] {} (c1);
 \node (a4) [label=below:{$ $}, right of=a2, xshift = 20mm, dot] {};
 \node (a6) [label=below:{$ $}, right of=a4, xshift = 20mm, dot] {};
 \node (b9) [above of=a4, xshift=3mm, yshift = 8mm, wdot] {} edge [decorate, zig] node[right] {} (a4);
 \node (b0) [above of=a6, xshift=-3mm, yshift = 8mm, wdot] {} edge [decorate, zig] node[left] {} (a6) edge node[above] {} (b9);
  \end{tikzpicture}
  \caption{Normal white nodes.}
  \label{normal_node}
 \end{subfigure}
 \caption{Two types of white nodes}
 \label{whitenodestype}
\end{figure}
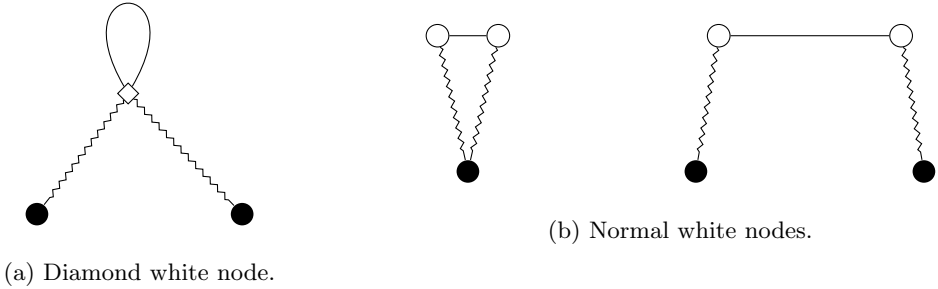

By (\ref{eq_eo2}), there are $|\mathcal V_0|$ black nodes with odd $\deg$ in $[\mathcal V_0]$ (where $[\mathcal V_0]$ is defined in the obvious way). WLOG, we assume these nodes are $b_1,...,b_{|\mathcal V_0|}$. To have nonzero expectation, each white block must contain at least two white nodes. Therefore for each $k=1,...,|\mathcal V_0|$, there exists a block connecting to $b_k$ which contains at least $3$ white nodes. Call such a block $W(b_k)$, and denote by $A(b_k)$ the set of the adjacent white nodes to $b_k$ in $W(b_k)$. (Note that the $W(b_k)$'s or $A(b_k)$'s are not necessarily distinct.) WLOG, let $W_{1},...,W_{d}$ be the distinct
blocks among all $W(b_k)$'s. Define
\[\mathcal V_{00} := \{b_k|\ A(b_k)\text{ has no normal white nodes,} \ 1\le k\le |\mathcal V_0|\},\]
and
\[\mathcal V_{01} := \{b_k|\ A(b_k)\text{ has at least one normal white node,}\ 1\le k\le |\mathcal V_0|\}.\]
The following lemma gives the key estimates we need.

\begin{lem}
For any partition $\gamma \in\mathcal W$,
\begin{equation}\label{claim3_1}
m(\gamma) \le \frac{-|\mathcal V_{00}|-|\mathcal V_{01}|/2- n_R^{(d)}+\sum_{k = 1}^{n} \left[{\deg \left( {{b_k}} \right) + \deg ({{\bar b}_k})} \right]}{2},
\end{equation}
and
\begin{equation}\label{claim3_2}
 n_X+n_R \ge p + |\mathcal V_1| + |\mathcal V_{00}|, \ \ n_X \ge |\mathcal V_{1}|, \ \  n_R^{(d)}\ge |\mathcal V_{00}|.
\end{equation}
\end{lem}
\begin{proof}
The second inequality of (\ref{claim3_2}) can be proved easily through
\[|\mathcal V_1| \le \left|\{k\in L|{\cal F}_{\rm{in}}(k, \Omega_Q^X)\ge 2\}\right| \le n_X.\]
Notice for $b_k \in \mathcal V_0$, $A(b_k)$ contains at least three diamond white nodes, while each of the white node is share by another $b_l$. Thus we trivially have $|\mathcal V_{00}| \le n_R^{(d)}.$

Now we prove (\ref{claim3_1}). A diamond white node is connected to two black nodes and a normal white node is connected to one black node. Hence a diamond white node belongs to two sets $A(b_{k_1}), A(b_{k_2})$, and a normal white node belongs to exactly one set $A(b_k)$. Therefore for each $i=1,...,d$, if $ W_{i}$ contains exactly one $A(b_k)$ then
 \[\left|W_{i}\right|\ge 3\ge 2+\mathbf 1_{\mathcal V_{01}}(b_k)+\frac{\mathbf 1_{\mathcal V_{00}}(b_k)}{2}.\]
Otherwise if $ W_{i}$ contains more than one $A(b_k)$, then
 \[\left|W_{i}\right|\ge \sum\limits_{b_k: A(b_k)\subseteq W_{i}}\left(2 \cdot \mathbf 1_{\mathcal V_{01}}(b_k)+\frac{3}{2}\cdot\mathbf 1_{\mathcal V_{00}}(b_k)\right)\ge 2+\sum\limits_{b_k: A(b_k)\subseteq W_{i}}\left(\mathbf 1_{\mathcal V_{01}}(b_k)+\frac{\mathbf 1_{\mathcal V_{00}}(b_k)}{2}\right).\]
Here the first inequality can be understood as following. For each black node $b_k$ with $A(b_k)\subseteq W_i$, we count the number of white nodes in $A(b_k)$ and add them together. During the counting, we assign weight-1 to a normal white node and weight-$1/2$ to a diamond white node (since it is shared by two different black nodes). If $b_k\in \mathcal V_{00}$, there are at least three diamond white nodes in $A(b_k)$ with total weight $\ge 3/2$; if $b_k\in \mathcal V_{01}$, there are at least one normal white node and two other white nodes in $A(b_k)$ with total weight $\ge 2$. Thus $\sum_{b_k: A(b_k)\subseteq W_{i}}\left(2 \cdot \mathbf 1_{\mathcal V_{01}}(b_k)+\frac{3}{2}\cdot\mathbf 1_{\mathcal V_{00}}(b_k)\right)$ is smaller than the number of white nodes in $W_i$. Then summing $|W_i|$ over $i$, we get
\[\sum\limits_{i=1}^{d}\left|W_{i}\right|\ge 2d+|\mathcal V_{01}|+\frac{|\mathcal V_{00}|}{2}.\]
For the other $m-d$ blocks, each of them contains at least two white nodes. Therefore
\[2m+|\mathcal V_{01}|+ \frac{|\mathcal V_{00}|}{2} \le\sum\limits_{i=1}^{d}\left|W_{i}\right|+2(m-d)\le -n_R^{(d)}+\sum_{k = 1}^{n} \left[{\deg \left( {{b_k}} \right) + \deg ({{\bar b}_k})} \right],\]
where we use (\ref{number_white_node}) in the last step. This proves (\ref{claim3_1}).

For $b_k \in \mathcal V_{00}$, $A(b_k)$ contains at least three white nodes from off-diagonal $R$-groups,
 \begin{align*}
 \mathcal V_{00}  \subseteq& \{b_k\in L|\ {\cal F}_{\rm{in}}(b_k,\Omega_\Delta) = 1\text{ and }{\cal F}_{\rm{in}}(b_k, \Omega_Q^R)\geq 3\}=:\mathcal V_2.
 \end{align*}
Recall (\ref{iNum_1})-(\ref{iNum_3}), only $\tau_1^{(k)}$ may increase ${\cal F}_{\rm{in}}$. Thus $\bw$ contains $\tau_1^{(b_k)}$ for each $b_k\in \mathcal V_1 \cup \mathcal V_2$ (recall the definition of $\mathcal V_1$ in (\ref{Vnode2})).
Therefore by (\ref{offNum_2}), (\ref{Foff2}) and the fact that $\mathcal V_{00}$ and $\mathcal V_1$ are disjoint,
\begin{align*}
n_X+n_R = {\cal F}_{\rm{off}}((\Omega_\Delta)_{\bw})\ge & {\cal F}_{\rm{off}}(\Omega_\Delta)+\left|\mathcal V_1\cup\mathcal V_{2}\right| \ge p+|\mathcal V_1|+|\mathcal V_{00}|.
\end{align*}
This proves the first inequality of (\ref{claim3_2}).
\end{proof}

By (\ref{assm2}) and (\ref{rescaled_Rorder}), a diagonal $R$ edge contributes $1$, an off-diagonal $R$ edge contributes $\Phi$, and $S^X$ or $X$ edge contributes $N^{-1/2}$. Denote
$$\mathcal  U = \prod_{l=1}^n\left|u_{[b_l]}\right|^{I_l}\left|v_{[b_l]}\right|^{J_l}.$$
Then using Lemma \ref{lem_stodomin}, we get
\begin{align*}
\left|\mathbb E Q(\bw,\Delta(\Gamma))\right| & \le C \mathcal U{\left( {{N^{ - 1/2}}} \right)}^{n_X}\sum\limits_{\gamma  \in \mathcal W} {\sum\limits_{\gamma(W_1), \ldots ,\gamma(W_{m})\in\sI\setminus L}^* {{\Phi ^{{n_R - n_R^{(d)}}}}\prod\limits_{k = 1}^n {{{\left( {{N^{ - 1/2}}} \right)}^{\deg \left( {{b_k}} \right) + \deg \left( {{{\bar b}_k}} \right)}}} } }\nonumber\\
& \le C \mathcal U N^{-n_X/2}\sum\limits_{\gamma  \in \mathcal W}N^{m-{\frac{{\sum\limits_{k = 1}^n {\deg \left( {{b_k}} \right) + \deg ({{\bar b}_k})} }}{2}}}\Phi^{n_R - n_R^{(d)}}\\
& \le C\mathcal U{N^{ - n_X/2}} \sum\limits_{\gamma  \in \mathcal W} {N^{\frac{{-|\mathcal V_{01}|-|\mathcal V_{00}|/2 - n_R^{(d)}}}{2}}} {\Phi ^{{n_R-n_R^{(d)}}}}\nonumber\\
& \le C \mathcal UN^{-h/2}\sum\limits_{\gamma  \in \mathcal W} {N^{-(n_X-|\mathcal V_1|)/2}} {N^{ - (n_R^{(d)}-|\mathcal V_{00}|)/2}}{\Phi ^{{n_R-n_R^{(d)}}}} \nonumber\\
& \le C \mathcal U {N^{ - h/2}}\sum_{\gamma \in \mathcal W} {\Phi ^{ n_X+n_R - |\mathcal V_1| - |\mathcal V_{00}| }}   \le C \mathcal U{N^{ - h/2}}{\Phi ^p},
\end{align*}
where in the third step we used (\ref{claim3_1}), in the fourth step $h=|\mathcal V|=|\mathcal V_1|+|\mathcal V_{00}|+|\mathcal V_{01}|$, in the fifth step ${N^{ - 1/2}} \leq \Phi$ and (\ref{claim3_2}), and in the last step (\ref{claim3_2}).
Thus we have proved (\ref{eq_iso_3}), which concludes the proof of Proposition \ref{iso_prop}.

\end{section}

\begin{section}{Anisotropic local law: self-consistent comparison}\label{section_comparison}
In this section we prove Theorem \ref{law_wideT}. We first prove the anisotropic and averaged local laws under the vanishing third moment assumption (\ref{assm_3rdmoment}). When $\eta \ge N^{-1/2+\zeta}|m_{2c}|^{-1}$, the anisotropic and averaged local laws can be established without assuming (\ref{assm_3rdmoment}). For convenience, we only consider the case $w\in \mathbf D$ and $|z|^2 \le 1-\tau$ in this section.
The proof for other cases is almost the same.

Following the notations in the arguments between Theorems \ref{law_squareD} and \ref{law_wideT},
\begin{align}
H(TX-z,w) = \overline T \left( {\begin{array}{*{20}c}
   { - w (D^\dag D)^{-1}} & w^{1/2}(V_1X-(UD)^{-1}z) \\
   {w^{1/2} (V_1X-(UD)^{-1}z)^\dag } & { - wI}  \\
   \end{array}} \right)\overline T^\dag, \ \  \overline T:=\left( {\begin{array}{*{20}c}
   { UD} & 0 \\
   0& I  \\
   \end{array}} \right).
 \end{align}
Now we define
 \begin{equation}\label{def_mathcalg}
 \mathcal G(w):=|w|^{1/2}\left( {\begin{array}{*{20}c}
   { - w (D^\dag D)^{-1}} & w^{1/2}\left(V_1X-(UD)^{-1}z\right) \\
   { w^{1/2}\left(V_1X-(UD)^{-1}z\right)^\dag } & { - w I}  \\
   \end{array}} \right)^{-1} = |w|^{1/2}\overline T^\dag G\overline T.
 \end{equation}
Since $T$ is invertible and $\|T\|+ \|T^{-1}\| \le \tau^{-1}$ by (\ref{assm3}), to prove the anisotropic law in Theorem \ref{law_wideT}, it suffices to show
 \begin{equation}\label{goal_ani}
 \| \sG(w)-\wt\Pi(w) \|\prec \Phi(w)
 \end{equation}
 where
 \begin{equation}\label{def_PiPhi}
 \wt\Pi(w):=|w|^{1/2}\overline T^\dag \Pi(w)\overline T ,\ \ \Phi(w):=|w|^{1/2}\Psi(w).
 \end{equation}
Notice we have $\|\wt\Pi\|=O(1)$ by (\ref{estimate_Piw12}). By the remark around (\ref{eqn_comparison1}), if $X=X^{Gauss}$ is Gaussian, then (\ref{goal_ani}) holds. Hence for a general $X$, it suffices to prove that
 \begin{equation}\label{Gaussian_starting}
\|\sG(X,w) - \sG(X^{Gauss},w)\| \prec \Phi(w).
 \end{equation}
Similar to Lemma \ref{lemma_Im}, it is easy to prove the following estimates for $\mathcal G$.

\begin{lem}\label{lem_comp_gbound}
For $i\in \mathcal I^M_1$, we define $\mathbf v_i=V_1 \mathbf e_i \in \mathbb R^{\mathcal I_1}$, i.e. $\mathbf v_i$ is the $i$-th column vector of $V_1$. Let $\mathbf u \in \mathbb R^{\mathcal I_1}$ and $\mathbf w \in \mathbb R^{\mathcal I_2}$, then we have for some constant $C>0$,
  \begin{align}
& \sum\limits_{\mu  \in \mathcal I_2 } {\left| {\sG_{\mathbf w\mu } } \right|^2 }  = |w|^{1/2}\frac{{\Im \sG_{\mathbf w\mathbf w} }}{\eta }, \label{eq_sgsq1}\\ 
& \sum\limits_{i \in \mathcal I_1^M }  \left| {\sG_{\mathbf u \mathbf v_i} } \right|^2  \le C|w|^{1/2} \frac{\Im \sG_{\mathbf u\mathbf u}}{\eta} , \label{eq_sgsq2} \\
& \sum\limits_{i \in \mathcal I_1^M } {\left| {\sG_{\mathbf w \mathbf v_i} } \right|^2 } \le C\left(\left| w \right|^{ - 1/2} {\sG}_{\mathbf w\mathbf w}  + \bar w\left| w \right|^{ - 1/2} \frac{{\Im \sG_{\mathbf w\mathbf w} }}{\eta }\right) , \label{eq_sgsq3} \\
& \sum\limits_{\mu \in \mathcal I_2 } {\left| {\sG_{\mathbf u \mu} } \right|^2 } \le C \left(\left| w \right|^{ - 1/2} {\sG}_{\mathbf u\mathbf u}  + \bar w\left| w \right|^{ - 1/2} \frac{{\Im \sG_{\mathbf u\mathbf u} }}{\eta } \right),\label{eq_sgsq4}
 \end{align}
\end{lem}
\begin{subsection}{Self-consistent comparison}\label{subsection_selfcomp}
Our proof basically follows the arguments in \cite[Section 7]{Anisotropic} with some minor modifications. Thus we will not write down all the details for the proof. By polarization, it suffices to show the following proposition. 

\begin{prop}\label{comparison_prop}
Fix ${\left| z \right|^2 } \le 1 - \tau$ and suppose that the assumptions of Theorem \ref{law_wideT} hold. If (\ref{assm_3rdmoment}) holds or $\eta \ge N^{-1/2+\zeta}|m_{2c}|^{-1}$, then for any regular domain $\mathbf S \subseteq \mathbf D$,
 \begin{equation}\label{goal_ani2}
\left\langle \mathbf v, \left(\sG(w)-\wt\Pi(w)\right) \mathbf v \right\rangle \prec \Phi(w)
\end{equation}
uniformly in $w\in \bS$ and any deterministic unit vectors $ \mathbf v\in{\mathbb C}^{\mathcal I}$.
\end{prop}

We first assume that (\ref{assm_3rdmoment}) holds. Then we will show how to modify the arguments to prove the $\eta \ge N^{-1/2+\zeta}|m_{2c}|^{-1}$ case.
The proof consists of a bootstrap argument from larger scales to smaller scales in multiplicative increments of $N^{-\delta}$, where
\begin{equation}
 \delta \in\left(0,\frac{\zeta}{2C_0}\right), \label{assm_comp_delta}
\end{equation}
with $C_0> 0$ being a universal constant that will be chosen large enough in the proof. For any $\eta\ge \left|m_{1c}\right|^{-1}N^{-1+\zeta}$, we define
\begin{equation}\label{eq_comp_eta}
\eta_l:=\eta N^{\delta l} \text{ for } \ l=0,...,L-1,\ \ \ \eta_L:=1.
\end{equation}
where
$L\equiv L(\eta):=\max\left\{l\in\mathbb N|\ \eta N^{\delta(l-1)}<1\right\}.$
Note that $L\le2\delta^{-1}$.

By (\ref{eq_gbound}), the function $w\mapsto \sG(w)-\widetilde\Pi(w)$ is Lipschitz continuous in $\mathbf S$ with Lipschitz constant bounded by $CN^3$.
Thus to prove (\ref{goal_ani2}) for all $w\in \mathbf S$, it suffices to show (\ref{goal_ani2}) holds for all $w$ in some discrete but sufficiently dense subset $\widehat {\mathbf S} \subset \mathbf S$. We will use the following discretized domain $\widehat\bS$.
\begin{defn}
Let $\widehat{\mathbf S}$ be an $N^{-10}$-net of $\mathbf S$ such that $ |\widehat{\mathbf S} |\le N^{20}$ and
\[E+i\eta\in\widehat{\mathbf S}\Rightarrow E+i\eta_l\in\widehat{\mathbf S}\text{ for }l=1,...,L(\eta).\]
\end{defn}

The bootstrapping is formulated in terms of two scale-dependent properties ($\bA_m$) and ($\bC_m$) defined on the subsets
\[\widehat{\mathbf S}_m:=\left\{w\in\widehat{\mathbf S}\mid\text{Im} \, w\ge N^{-\delta m}\right\}.\]
${(\bA_m)}$ For all $w\in\widehat \bS_m$, all deterministic unit vector $\mathbf v$, and all $X$ satisfying (\ref{assm1})-(\ref{assm2}), we have
\begin{equation}\label{eq_comp_Am}
 \text{Im}\sG_{\mathbf v\mathbf v}(w)\prec|w|^{1/2}\text{Im} \left[m_{1c}(w)+m_{2c}(w)\right]+N^{C_0\delta}\Phi(w).
\end{equation}
${(\bC_m)}$ For all $w\in\widehat \bS_m$, all deterministic unit vector $\mathbf v$, and all $X$ satisfying (\ref{assm1})-(\ref{assm2}), we have
\begin{equation}\label{eq_comp_Cm}
 \left|\sG_{\mathbf v\mathbf v}(w)-\widetilde\Pi_{\mathbf v\mathbf v}(w)\right|\prec N^{C_0\delta}\Phi(w).
\end{equation}
It is trivial to see that property ${(\mathbf A_0)}$ holds. Moreover, it is easy to observe the following result.

\begin{lem}\label{lemm_boot2}
For any $m$, property ${(\mathbf C_m)}$ implies property $(\mathbf A_m)$.
\end{lem}
\begin{proof}
This result follows from (\ref{estimate_PiImw}).
\end{proof}

The key step is the following induction result.
\begin{lem}\label{lemm_boot}
For any $1\le m\le2\delta^{-1}$, property $(\mathbf A_{m-1})$ implies property $(\mathbf C_m)$.
\end{lem}

Combining Lemmas \ref{lemm_boot2} and \ref{lemm_boot}, we conclude that (\ref{eq_comp_Cm}) holds for all $w\in\widehat{\mathbf S}$. Since $\delta$ can be chosen arbitrarily small under the condition (\ref{assm_comp_delta}), we conclude that (\ref{goal_ani2}) holds for all $w\in\widehat{\mathbf S}$, and Proposition \ref{comparison_prop} follows. What remains now is the proof of Lemma \ref{lemm_boot}. Denote
\begin{equation}\label{eq_comp_F(X)}
 F_{\mathbf v}(X,w)=\left|\sG_{\mathbf{vv}}(X,w)-\widetilde\Pi_{\mathbf {vv}}(w)\right|.
\end{equation}
By Markov's inequality, it suffices to prove the following lemma.
\begin{lem}\label{lemm_comp_0}
 Fix $p\in 2\mathbb N$ and $m\le 2\delta^{-1}$. Suppose that the assumptions of Proposition \ref{comparison_prop},  (\ref{assm_3rdmoment}) and property $(\mathbf A_{m-1})$ hold. Then we have
 \begin{equation}
  \mathbb EF_{\mathbf v}^p(X,w)\le\left( N^{C_0\delta}\Phi(w)\right)^p
 \end{equation}
 for all $w\in{\widehat{\mathbf S}}_m$ and all deterministic unit vector $\mathbf v$.
\end{lem}
In the following, we prove Lemma \ref{lemm_comp_0}. 
First, in order to make use of the assumption $(\mathbf A_{m-1})$, which has spectral parameters in $\widehat{\mathbf S}_{m-1}$, to get some estimates for spectral parameters in $\widehat{\mathbf S}_{m}$, we shall use the following rough bounds for $\mathcal G_{\mathbf{xy}}$.

\begin{lem}\label{lemm_comp_1}
For any $w=E+i\eta\in\mathbf S$ and $\mathbf x,\mathbf y\in \mathbb C^{\mathcal I}$,  we have
\begin{align*}
\left|\sG_{\mathbf x\mathbf y}(w)-\wt\Pi_{\mathbf x\mathbf y}(w)\right|\prec & N^{2\delta}\sum\limits_{l=1}^{L(\eta)} \left[\textnormal{Im}\sG_{\mathbf x_1\mathbf x_1}(E+i\eta_l)+\textnormal{Im}\sG_{\mathbf x_2\mathbf x_2}(E+i\eta_l) \right.\\
& \left. +\textnormal{Im}\sG_{\mathbf y_1\mathbf y_1}(E+i\eta_l)+\textnormal{Im}\sG_{\mathbf y_2\mathbf y_2}(E+i\eta_l)\right]+|\mathbf x||\mathbf y|,
\end{align*}
where $\mathbf x=\left( {\begin{array}{*{20}c}
   {\mathbf x}_1   \\
   {\mathbf x}_2 \\
   \end{array}} \right)$ and $\mathbf y=\left( {\begin{array}{*{20}c}
   {\mathbf y}_1   \\
   {\mathbf y}_2 \\
   \end{array}} \right)$ for ${\mathbf x}_1,{\mathbf y}_1\in\mathbb C^{\mathcal I_1}$ and ${\mathbf x}_2,{\mathbf y}_2\in\mathbb C^{\mathcal I_2}$.
\end{lem}
\begin{proof} The proof is similar to the one for \cite[Lemma 7.12]{Anisotropic}.\end{proof}
\begin{lem}\label{lemm_comp_2}
Suppose $(\mathbf A_{m-1})$ holds, then
 \begin{equation}\label{eq_comp_apbound}
  \sG(w)-\wt\Pi(w)=O_{\prec}(N^{2\delta})
 \end{equation}
 and
\begin{equation}\label{eq_comp_apbound2}
\textnormal{Im}\sG_{\mathbf v\mathbf v}\le N^{2\delta}\left[|w|^{1/2}\textnormal{Im}\left(m_{1c}(w)+m_{2c}(w)\right)+N^{C_0\delta}\Phi(w)\right]
\end{equation}
 for all $w\in \widehat{\mathbf S}_m$ and all deterministic unit vector $\mathbf v$
\end{lem}
\begin{proof}
Let $w=E+i\eta \in \widehat{\mathbf S}_m$. Then $E+i\eta_l \in \widehat{\mathbf S}_{m-1}$ for $l=1,\ldots, L(\eta)$, and (\ref{eq_comp_Am}) gives $\textnormal{Im}\sG_{\mathbf v\mathbf v}(w)\prec 1.$ The estimate (\ref{eq_comp_apbound}) now follows immediately from Lemma \ref{lemm_comp_1}. To prove (\ref{eq_comp_apbound2}), we remark that if $s(w)$ is the Stieltjes transform of any positive integrable function on $\mathbb R$, the map $\eta \mapsto \eta\Im\, s(E+i\eta)$ is nondecreasing and the map $\eta \mapsto \eta^{-1} \Im\, s(E+i\eta)$ is nonincreasing. We apply them to $|w|^{-1/2}\Im\,\sG_{\mathbf v\mathbf v}(E+i\eta)$ and $\Im\, m_{1,2c}(E+i\eta)$ to get for $w_1=E+i\eta_1\in \widehat{\mathbf S}_{m-1}$,
\begin{align*}
\textnormal{Im}\sG_{\mathbf v\mathbf v}(w) & \le N^{\delta}\frac{|w|^{1/2}}{|w_1|^{1/2}}\textnormal{Im}\sG_{\mathbf v\mathbf v}(w_1)\prec N^{\delta}\left[|w|^{1/2}\textnormal{Im}\left(m_{1c}(w_1)+m_{2c}(w_1)\right)+N^{C_0\delta}\frac{|w|^{1/2}}{|w_1|^{1/2}}\Phi(w_1)\right] \\
& \le N^{2\delta}\left[|w|^{1/2}\textnormal{Im}\left(m_{1c}(w)+m_{2c}(w)\right)+N^{C_0\delta}\Phi(w)\right],
\end{align*}
where we use $\Phi(w):=|w|^{1/2}\Psi(w)$ and the fact that $\eta \mapsto \Psi(E+i\eta)$ is nonincreasing, which is clear from the definition (\ref{eq_defpsi}).
\end{proof}

Now we apply the self-consistent comparison method presented in \cite[Section 7]{Anisotropic} to prove Lemma \ref{lemm_comp_0}. To organize the proof, we divide it into two small subsections.

\begin{subsubsection}{Interpolation and expansion}
\begin{defn}[Interpolating matrices]
Introduce the notation $X^0:=X^{Gauss}$ and $X^1:=X$. Let $\rho_{i\mu}^0$ and $\rho_{i\mu}^1$ be the laws of $X_{i\mu}^0$ and $X_{i\mu}^1$, respectively, for $i\in \mathcal I_1^M$ and $\mu\in \mathcal I_2$. For $\theta\in [0,1]$, we define the interpolated law
$$\rho_{i\mu}^\theta := (1-\theta)\rho_{i\mu}^0+\theta\rho_{i\mu}^1.$$
We shall work on the probability space consisting of triples $(X^0,X^\theta, X^1)$ of independent $\mathcal I_1^M\times \mathcal I_2$ random matrices, where the matrix $X^\theta=(X_{i\mu}^\theta)$ has law
\begin{equation}\label{law_interpol}
\prod_{i\in \mathcal I_1^M}\prod_{\mu\in \mathcal I_2} \rho_{i\mu}^\theta(dX_{i\mu}^\theta).
\end{equation}
For $\lambda \in \mathbb R$, $i\in \mathcal I_1^M$ and $\mu\in \mathcal I_2$, we define the matrix $X_{(i\mu)}^{\theta,\lambda}$ through
\[\left(X_{(i\mu)}^{\theta,\lambda}\right)_{j\nu}:=\begin{cases}X_{i\mu}^{\theta}&\text{ if }(j,\nu)\ne (i,\mu)\\\lambda&\text{ if }(j,\nu)=(i,\mu)\end{cases}.\]
We also introduce the matrices \[\sG^{\theta}(w):=\sG\left(X^{\theta},w\right),\ \ \ \sG^{\theta, \lambda}_{(i\mu)}(w):=\sG\left(X_{(i\mu)}^{\theta,\lambda},w\right),\]
according to (\ref{def_mathcalg}) and the Definition \ref{def_linearHG}.
\end{defn}

We shall prove Lemma \ref{lemm_comp_0} through interpolation matrices $X^\theta$ between $X^0$ and $X^1$. It holds for $X^0$ by the the anisotropic law (\ref{goal_ani}) (see the remark above (\ref{Gaussian_starting})).
\begin{lem}\label{Gaussian_case}
Lemma \ref{lemm_comp_0} holds if $X=X^0$.
\end{lem}

Using (\ref{law_interpol}) and fundamental calculus, we get the following basic interpolation formula.
\begin{lem}\label{lemm_comp_3}
 For $F:\mathbb R^{\mathcal I_1^M \times\mathcal I_2}\rightarrow \mathbb C$ we have
\begin{equation}\label{basic_interp}
\frac{d}{d\theta}\mathbb E F(X^\theta)=\sum\limits_{i\in\mathcal I_1^M}\sum\limits_{\mu\in\mathcal I_2}\left[\mathbb E F\left(X^{\theta,X_{i\mu}^1}_{(i\mu)}\right)-\mathbb E F\left(X^{\theta,X_{i\mu}^0}_{(i\mu)}\right)\right]
\end{equation}
 provided all the expectations exists.
\end{lem}

We shall apply Lemma \ref{lemm_comp_3} with $F(X)=F_{\mathbf v}^p(X,w)$ for $F_{\mathbf v}(X,w)$ defined in (\ref{eq_comp_F(X)}). The main work is devoted to prove the following self-consistent estimate for the right-hand side of (\ref{basic_interp}).

\begin{lem}\label{lemm_comp_4}
 Fix $p\in 2\mathbb N$ and $m\le 2\delta^{-1}$. Suppose (\ref{assm_3rdmoment}) and $\mathbf{(A_{m-1})}$ holds, then we have
 \begin{equation}
  \sum\limits_{i\in\mathcal I_1^M}\sum\limits_{\mu\in\mathcal I_2}\left[\mathbb EF_{\mathbf v}^p\left(X^{\theta,X_{i\mu}^1}_{(i\mu)}\right)-\mathbb EF_{\mathbf v}^p\left(X^{\theta,X_{i\mu}^0}_{(i\mu)}\right)\right]=
  O\left((N^{C_0\delta}\Phi)^p+\mathbb EF_{\mathbf v}^p(X^\theta,w)\right)
 \end{equation}
 for all $\theta\in[0,1]$, all $w\in\widehat{\mathbf{S}}_m$, and all deterministic unit vector $\mathbf v$.
\end{lem}
Combining Lemmas \ref{Gaussian_case}, \ref{lemm_comp_3} and \ref{lemm_comp_4} with a Gr\"onwall argument, we can conclude the proof of Lemma \ref{lemm_comp_0} and hence Proposition \ref{comparison_prop}.

In order to prove Lemma \ref{lemm_comp_4}, we compare $X^{\theta,X_{i\mu}^0}_{(i\mu)}$ and $X^{\theta,X_{i\mu}^1}_{(i\mu)}$ via a common $X^{\theta,0}_{(i\mu)}$, i.e. under the assumptions of Lemma \ref{lemm_comp_4}, we will prove
\begin{equation}\label{lemm_comp_5}
\sum\limits_{i\in\mathcal I^M_1}\sum\limits_{\mu\in\mathcal I_2}\left[\mathbb EF_{\mathbf v}^p\left(X^{\theta,X_{i\mu}^u}_{(i\mu)}\right)-\mathbb EF_{\mathbf v}^p\left(X^{\theta,0}_{(i\mu)}\right)\right]=
  O\left((N^{C_0\delta}\Phi)^p+\mathbb EF_{\mathbf v}^p(X^\theta,w)\right)
 \end{equation}
for all $u\in \{0,1\}$, all $\theta\in[0,1]$, all $w\in\widehat{\mathbf{S}}_m$, and all deterministic unit vector $\mathbf v$.

Underlying the proof of (\ref{lemm_comp_5}) is an expansion approach which we will describe below. Throughout the rest of the proof, we suppose that $(\mathbf A_{m-1})$ holds. Also the rest of the proof is performed at a single $w\in \widehat{\mathbf{S}}_m$. Define the $\mathcal I\times \mathcal I$ matrix $\Delta_{(i\mu)}^\lambda$ through
\begin{equation}
 \left(\Delta_{(i\mu)}^{\lambda}\right)_{st}:=\lambda\delta_{is}\delta_{\mu t}+\lambda\delta_{it}\delta_{\mu s}.
\end{equation}
Then we have for any $\lambda,\lambda'\in \mathbb R$ and $K\in \mathbb N$,
\begin{equation}\label{eq_comp_expansion}
\sG_{(i \mu)}^{\theta,\lambda'} = \sG_{(i\mu)}^{\theta,\lambda}+\sum\limits_{k=1}^{K}\alpha^k \sG_{(i\mu)}^{\theta,\lambda}\left(\overline V\Delta_{(i\mu)}^{\lambda-\lambda'}\overline V^\dag \sG_{(i\mu)}^{\theta,\lambda}\right)^k+\alpha^{K+1}\sG_{(i\mu)}^{\theta,\lambda'}\left(\overline V\Delta_{(i\mu)}^{\lambda-\lambda'}\overline V^\dag \sG_{(i\mu)}^{\theta,\lambda}\right)^{K+1},
\end{equation}
where
$\overline V:=\begin{pmatrix}V_1 & 0\\ 0 & I\end{pmatrix}$ and $\alpha:=\frac{w^{1/2}}{|w|^{1/2}}.$
The following result provides a priori bounds for the entries of $\sG_{(i\mu)}^{\theta,\lambda}$.
\begin{lem}\label{lemm_comp_6}
 Suppose that $y$ is a random variable satisfying $|y|\prec N^{-1/2}$. Then
 \begin{equation}\label{comp_eq_apriori}
   \sG_{(i\mu)}^{\theta,y}-\wt\Pi=O_{\prec}(N^{2\delta})
 \end{equation}
 for all $i\in\sI^M_1$ and $\mu\in\sI_2$.
\end{lem}
\begin{proof} See \cite[Lemma 7.14]{Anisotropic}. \end{proof}

In the following, for simplicity of notations we introduce $f_{(i\mu)}(\lambda):=F_{\mathbf v}^p(X_{(i\mu)}^{\theta, \lambda})$. We use $f_{(i\mu)}^{(n)}$ to denote the $n$-th derivative of $f_{(i\mu)}$. By Lemma \ref{lemm_comp_6} and expansion (\ref{eq_comp_expansion}) we get the following result.
\begin{lem}
Suppose that $y$ is a random variable satisfying $|y|\prec N^{-1/2}$. Then for fixed $n\in\bbN$,
  \begin{equation}
  \left|f_{(i\mu)}^{(n)}(y)\right|\prec N^{2\delta(n+p)}.
 \end{equation}
\end{lem}
By this lemma, the Taylor expansion of $f_{(i\mu)}$ gives
\begin{equation}\label{eq_comp_taylor}
f_{(i\mu)}(y)=\sum\limits_{n=0}^{4p}\frac{y^n}{n!}f^{(n)}_{(i\mu)}(0)+O_\prec(\Phi^p),
\end{equation}
provided $C_0$ is chosen large enough in (\ref{assm_comp_delta}). Therefore we have for $u\in\{0,1\}$,
\begin{align*}
\mathbb EF_{\mathbf v}^p\left(X^{\theta,X_{i\mu}^u}_{(i\mu)}\right)-\mathbb EF_{\mathbf v}^p\left(X^{\theta,0}_{(i\mu)}\right)=&\bbE\left[f_{(i\mu)}\left(X_{i\mu}^u\right)-f_{(i\mu)}(0)\right]\\=&\bbE f_{(i\mu)}(0)+\frac{1}{2N}\bbE f_{(i\mu)}^{(2)}(0)+\sum\limits_{n=4}^{4p}\frac{1}{n!}\bbE f^{(n)}_{(i\mu)}(0)\bbE\left(X_{i\mu}^u\right)^n+O_\prec(\Phi^p),
\end{align*}
where we used that $X_{i\mu}^u$ has vanishing first and third moments and its variance is $1/N$. Thus to show (\ref{lemm_comp_5}), we only need to prove for $n=4,5,...,4p$,
\begin{equation}\label{eq_comp_est}
N^{-n/2}\sum\limits_{i\in\mathcal I^M_1}\sum\limits_{\mu\in\mathcal I_2}\left|\bbE f^{(n)}_{(i\mu)}(0)\right|=O\left((N^{C_0\delta}\Phi)^p+\mathbb EF_{\mathbf v}^p(X^\theta,w)\right),\end{equation}
where we have used (\ref{assm2}). In order to get a self-consistent estimate in terms of the matrix $X^\theta$ on the right-hand side of (\ref{eq_comp_est}), we want to replace $X^{\theta,0}_{(i\mu)}$ in $f_{(i\mu)}(0):=F_{\mathbf v}^p(X_{(i\mu)}^{\theta, 0})$ with $X^\theta = X_{(i\mu)}^{\theta, X_{(i\mu)}^\theta}$. 
\begin{lem}
Suppose that
\begin{equation}\label{eq_comp_selfest}
N^{-n/2}\sum\limits_{i\in\mathcal I_1^M}\sum\limits_{\mu\in\mathcal I_2}\left|\bbE f^{(n)}_{(i\mu)}(X_{i\mu}^\theta)\right|=O\left((N^{C_0\delta}\Phi)^p+\mathbb EF_{\mathbf v}^p(X^\theta,w)\right)
\end{equation}
holds for $n=4,...,4p$, Then (\ref{eq_comp_est}) holds for $n=4,...,4p$.
\end{lem}
\begin{proof}
From (\ref{eq_comp_taylor}) we can get
\begin{equation}\label{eq_comp_taylor2}
f_{(i\mu)}^{(l)}(0)=f_{(i\mu)}^{(l)}(y)-\sum\limits_{n=1}^{4p-l}\frac{y^n}{n!}f^{(l+n)}_{(i\mu)}(0)+O_\prec(N^{l/2}\Phi^p).
\end{equation}
The result follows by repeatedly applying (\ref{eq_comp_taylor2}). The details can be found in \cite[Lemma 7.16]{Anisotropic}.
\end{proof}
\end{subsubsection}

\begin{subsubsection}{Conclusion of the proof with words}\label{section_words}
What remains now is to prove (\ref{eq_comp_selfest}). In order to exploit the detailed structure of the derivatives on the left-hand side of (\ref{eq_comp_selfest}), we introduce the following algebraic objects.

\begin{defn}[Words]\label{def_comp_words}
Given $i\in \mathcal I^M_1$ and $\mu\in \mathcal I_2$. Let $\sW$ be the set of words of even length in two letters $\{\mathbf i, \bm{\mu}\}$. We denote the length of a word $w\in\sW$ by $2n(w)$ with $n(w)\in \mathbb N$. We use bold symbols to denote the letters of words. For instance, $w=\mathbf t_1\mathbf s_2\mathbf t_2\mathbf s_3\cdots\mathbf t_n\mathbf s_{n+1}$ denotes a word of length $2n$.
Define $\sW_n:=\{w\in \mathcal W: n(w)=n\}$ to be the set of words of length $2n$.
We require that each word $w\in \sW_n$ satisfies that $\mathbf t_l\mathbf s_{l+1}\in\{\mathbf i\bm{\mu},\bm{\mu}\mathbf i\}$ for all $1\le l\le n$.

Next we assign each letter $*$ its value $[*]$ through $[\mathbf i]:=\bv_i$, $[\mathbf {\mu}]:=\mu,$ where $\mathbf v_i\in \mathbb C^{\mathcal I_1}$ is defined in Lemma \ref{lem_comp_gbound} and is regarded as a summation index. Note that it is important to distinguish the abstract letter from its value, which is a summation index. Finally, to each word $w$ we assign a random variable $A_{\mathbf v, i, \mu}(w)$ as follows. If $n(w)=0$ we define
 $$A_{\mathbf v, i, \mu}(W):=\sG_{\mathbf v\mathbf v}-\wt\Pi_{\mathbf v\mathbf v}.$$
 If $n(w)\ge 1$, say $w=\mathbf t_1\mathbf s_2\mathbf t_2\mathbf s_3\cdots\mathbf t_n\mathbf s_{n+1}$, we define
 \begin{equation}\label{eq_comp_A(W)}
 A_{\mathbf v, i, \mu}(W):=\sG_{\bv[\mathbf t_1]}\sG_{[\mathbf s_2][\mathbf t_2]}\cdots \sG_{[\mathbf s_n][\mathbf t_n]}\sG_{[\mathbf s_{n+1}]\bv}.
 \end{equation}
\end{defn}

Notice the words are constructed such that, by (\ref{eq_comp_expansion}),
\[\left(\frac{\partial}{\partial X_{i\mu}}\right)^n \left(\mathcal G_{\mathbf v\mathbf v}-\widetilde\Pi_{\mathbf v\mathbf v}\right)=(-\alpha)^n n!\sum_{w\in \mathcal W_n} A_{\mathbf v, i, \mu}(w)\]
for $n=0,1,2,\ldots$, which gives that
\begin{align*}
\left(\frac{\partial}{\partial X_{i\mu}}\right)^n F_{\bv}^p(X)=(-\alpha)^n n! & \sum_{n_1+\cdots+n_p=n}\prod_{r=1}^{p/2}\frac{1}{n_r! n_{r+p/2}!} \\
& \times \left(\sum_{w_r\in\sW_{n_r}}\sum_{w_{r+p/2}\in\sW_{n_{r+p/2}}}A_{\mathbf v, i, \mu}(w_r)\overline{A_{\mathbf v, i, \mu}(w_{r+p/2})}\right).
\end{align*}
Then to prove (\ref{eq_comp_selfest}), it suffices to show that
\begin{equation}
N^{-n/2}\sum\limits_{i\in\mathcal I^M_1}\sum\limits_{\mu\in\mathcal I_2}\left|\bbE\prod_{r=1}^{p/2}A_{\mathbf v, i, \mu}(w_r)\overline{A_{\mathbf v, i, \mu}(w_{r+p/2})}\right|=O\left((N^{C_0\delta}\Phi)^p+\mathbb EF_{\mathbf v}^p(X^\theta,w)\right)\label{eq_comp_goal1}
\end{equation}
for  $4\le n\le 4p$ and all words $w_1,...,w_p\in \sW$ satisfying $n(w_1)+\cdots+n(w_p)=n$.
To avoid the unimportant notational complications coming from the complex conjugates, we in fact prove that
\begin{equation}\label{eq_comp_goal2}
N^{-n/2}\sum\limits_{i\in\mathcal I_1^M}\sum\limits_{\mu\in\mathcal I_2}\left|\bbE\prod_{r=1}^{p}A_{\mathbf v, i, \mu}(w_r)\right|=O\left((N^{C_0\delta}\Phi)^p+\mathbb EF_{\mathbf v}^p(X^\theta,w)\right),
\end{equation}
and the proof of $(\ref{eq_comp_goal1})$ is essentially the same but with slightly heavier notations. Treating empty words separately, we find it suffices to prove
\begin{equation}
\label{eq_comp_goal3}N^{-n/2}\sum\limits_{i\in\mathcal I_1^M}\sum\limits_{\mu\in\mathcal I_2}\bbE\left|A^{p-q}_{\mathbf v, i, \mu}(w_0)\prod_{r=1}^{q}A_{\mathbf v, i, \mu}(w_r)\right|=O\left((N^{C_0\delta}\Phi)^p+\mathbb EF_{\mathbf v}^p(X^\theta,w)\right)
\end{equation}
for  $4\le n\le 4p$, $1\le q \le p$, and $w_r$ such that $n(w_0)=0$, $\sum_r n(w_r)=n$ and $n(w_r)\ge 1$ for $r\ge 1$.

To estimate (\ref{eq_comp_goal3}) we introduce the quantity
\begin{equation}\label{eq_comp_Rs}
\mathcal R_s:=|\mathcal G_{\mathbf v \mathbf v_s}|+|\mathcal G_{\mathbf v_s \mathbf v}|.
\end{equation}
for $s\in \sI$, where as a convention we let $\mathbf v_\mu=e_\mu$ for $\mu\in\sI_2$.

\begin{lem}\label{lem_comp_A}
  For $w\in\sW$ we have the rough bound
  \begin{equation}
  |A_{\mathbf v, i, \mu}(w)|\prec N^{2\delta(n(w)+1)}.\label{eq_comp_A1}
  \end{equation}
  Furthermore, for $n(w)\ge 1$ we have
  \begin{equation}
  |A_{\mathbf v, i, \mu}(w)|\prec(\mathcal R_i^2+\mathcal R_\mu^2)N^{2\delta(n(w)-1)}.\label{eq_comp_A2}
  \end{equation}
  For $n(w)=1$ we have better bound
  \begin{equation}
  |A_{\mathbf v, i, \mu}(w)|\prec \mathcal R_i\mathcal R_\mu.\label{eq_comp_A3}
  \end{equation}
\end{lem}
\begin{proof}
 (\ref{eq_comp_A1}) follows immediately from the rough bound (\ref{eq_comp_apbound}) and definition (\ref{eq_comp_A(W)}).  For (\ref{eq_comp_A2}) we break $A_{\mathbf v, i, \mu}(w)$ into $\sG_{\bv[\mathbf t_1]}(\sG_{[\mathbf s_2][\mathbf t_2]}\cdots \sG_{[\mathbf s_n][\mathbf t_n]})^{1/2}$ times $(\sG_{[\mathbf s_2][\mathbf t_2]}\cdots \sG_{[\mathbf s_n][\mathbf t_n]})^{1/2}\sG_{[\mathbf s_{n+1}]\bv}$ and use Cauchy-Schwarz inequality. (\ref{eq_comp_A3}) follows from the constraint $\mathbf t_1\ne\mathbf s_2$ in the definition (\ref{eq_comp_A(W)}).
\end{proof}
By pigeonhole principle, if $n\le 2q-2$ there exists at least two words $w_r$ with $n(w_r)=1$. Therefore by Lemma \ref{lem_comp_A} we have
\begin{equation}\label{eq_comp_r1}
 \left|A^{p-q}_{\mathbf v, i, \mu}(w_0)\prod_{r=1}^{q}A_{\mathbf v, i, \mu}(w_r)\right|\prec N^{2\delta(n+q)}F_{\bv}^{p-q}(X)\left(\one(n\ge 2q-1)(\mathcal R_i^2+\mathcal R_\mu^2)+\one(n\le 2q-2)\mathcal R_i^2\mathcal R_\mu^2\right).
\end{equation}
Then by Lemma \ref{lem_comp_gbound},
\begin{align}
 \frac{1}{N}\sum_{i\in\sI_1^M}\mathcal R_i^2+ \frac{1}{N}\sum_{\mu\in\sI_2}\mathcal R_{\mu}^2 & \prec\frac{|w|^{1/2}\Im \sG_{\mathbf v\mathbf v}+\eta|w|^{-1/2}\sG_{\mathbf v\mathbf v}}{N\eta} \nonumber\\
& \prec N^{2\delta}\frac{|w|\Im (m_{1c}+m_{2c})+|w|^{1/2}N^{C_0\delta}\Phi}{N\eta}\prec N^{(C_0+2)\delta}\Phi^2, \label{eq_comp_r2}
\end{align}
where in the second step we used the two bounds in Lemma \ref{lemm_comp_2}, $|w|^{-1/2} \eta =O(|w| \Im\, m_{1c})$ by Lemma \ref{lemm_m1_4case}, and in the last step the definition of $\Phi$. Using the same method we can get
\begin{equation}\label{eq_comp_r3}
\frac{1}{N^2}\sum_{i\in\sI_1^M}\sum_{\mu\in\sI_2}\mathcal R_i^2\mathcal R_\mu^2\prec \left(N^{(C_0+2)\delta}\Phi^2\right)^2.
\end{equation}
Plugging (\ref{eq_comp_r2}) and (\ref{eq_comp_r3}) into (\ref{eq_comp_r1}), we get that the left-hand side of
(\ref{eq_comp_goal3}) is bounded by
\[N^{-n/2+2}N^{2\delta(n+q+2)}\bbE F_{\bv}^{p-q}(X)\left(\one(n\ge 2q-1)\left(N^{C_0\delta/2}\Phi\right)^2+\one(n\le 2q-2)\left(N^{C_0\delta/2}\Phi\right)^4\right).\]
Using $\Phi \ge cN^{-1/2}$, we find that the left hand side of (\ref{eq_comp_goal3}) is bounded by
\begin{align*}
 & N^{2\delta(n+q+2)} \bbE F_{\bv}^{p-q}(X)\left(\one(n\ge 2q-1)\left(N^{C_0\delta/2}\Phi\right)^{n-2}+\one(n\le 2q-2)\left(N^{C_0\delta/2}\Phi\right)^n\right)\\
 &\le \bbE F_{\bv}^{p-q}(X)\left(\one(n\ge 2q-1)\left(N^{C_0\delta/2+12\delta}\Phi\right)^{n-2}+\one(n\le 2q-2)\left(N^{C_0\delta/2+12\delta}\Phi\right)^n\right)
\end{align*}
where we used that $q\le n$ and $n\ge 4$. Choose $C_0\ge 25$, then by (\ref{assm_comp_delta}) we have $N^{C_0\delta/2+12\delta} \le N^{\zeta/2}$ and
hence $N^{C_0\delta/2+12\delta}\Phi\le 1$. Moreover, if $n\ge 4$ and $n\ge 2q-1$, then $n\ge q+2$. Therefore we conclude that the left-hand side of $(\ref{eq_comp_goal3})$ is bounded by
\begin{equation}
\bbE F_{\bv}^{p-q}(X)\left(N^{C_0\delta}\Phi\right)^q.
\end{equation}
Now (\ref{eq_comp_goal3}) follows from Holder's inequality. This concludes the proof of (\ref{eq_comp_selfest}), and hence of (\ref{lemm_comp_5}), and then of Lemma \ref{lemm_boot}. This finishes the proof of Proposition \ref{comparison_prop} under the assumption (\ref{assm_3rdmoment}).

\vspace{5pt}

In the rest of this section, we prove Proposition \ref{comparison_prop} when $\eta \ge N^{-1/2+\zeta}|m_{2c}|^{-1}$. In this case, we can verify that
\begin{equation}\label{eq_comp_boundPhi}
\Phi \le N^{-1/4-\zeta/2}.
\end{equation}
Following the previous arguments, we see that it suffices to prove the estimate ($\ref{eq_comp_selfest}$) for $n=3$. In other words, we need to prove the following lemma. 
\begin{lem}\label{lemm_comparison_big}
Fix $1\le m \le 2\delta^{-1}$ and $p\in 2\mathbb N$. Let $w\in \widehat{\mathbf S}_m \cap \widehat {\mathbf D}$ (recall (\ref{eq_subdomainD})) and suppose $(\mathbf A_{m-1})$ holds. Then we have
\begin{equation}\label{eq_comp_selfest_generalX}
N^{-3/2}\sum\limits_{i\in\mathcal I_1^M}\sum\limits_{\mu\in\mathcal I_2}\left|\bbE f^{(3)}_{(i\mu)}(X_{i\mu}^\theta)\right|=O\left((N^{C_0\delta}\Phi)^p+\mathbb EF_{\mathbf v}^p(X^\theta,w)\right).
\end{equation}
\end{lem}
\begin{proof}
The main new ingredient of the proof is a further iteration step at a fixed $w$. Suppose
\begin{equation}\label{comp_geX_iteration}
\mathcal G-\tilde\Pi=O_\prec(N^{2\delta}\phi)
\end{equation}
for some $\phi\le 1$. By the a priori bound (\ref{eq_comp_apbound}), (\ref{comp_geX_iteration}) holds for $\phi=1$. Assuming (\ref{comp_geX_iteration}), we shall prove a self-improving bound of the form
\begin{equation}\label{comp_geX_self-improving-bound}
N^{-3/2}\sum\limits_{i\in\mathcal I_1^M}\sum\limits_{\mu\in\mathcal I_2}\left|\bbE f^{(3)}_{(i\mu)}(X_{i\mu}^\theta)\right|=O\left((N^{C_0\delta}\Phi)^p+(N^{-\zeta/4}\phi)^p+\mathbb EF_{\mathbf v}^p(X^\theta,w)\right).
\end{equation}
Once (\ref{comp_geX_self-improving-bound}) is proved, we can use it iteratively to get an increasingly accurate bound for the left hand side of (\ref{eq_comp_Cm}). After each step, we obtain a better a priori bound (\ref{comp_geX_iteration}) where $\phi$ is reduced by $N^{-\zeta/4}$. Hence after $O(\zeta^{-1})$ iterations we can get (\ref{eq_comp_selfest_generalX}).

As in Section \ref{section_words}, to prove (\ref{comp_geX_self-improving-bound}) it suffice to show 
\begin{equation}\label{comp_geX_words}
N^{-3/2}\left|\sum\limits_{i\in\mathcal I_1^M}\sum\limits_{\mu\in\mathcal I_2}A^{p-q}_{\mathbf v, i, \mu}(w_0)\prod_{r=1}^{q}A_{\mathbf v, i, \mu}(w_r)\right|\prec F_{\bv}^{p-q}(X)(N^{(C_0-1)\delta}\Phi + N^{-\zeta/2}\phi)^q,
\end{equation}
which follows from
\begin{equation}\label{comp_geX_words2}
N^{-3/2}\left|\sum\limits_{i\in\mathcal I_1^M}\sum\limits_{\mu\in\mathcal I_2}\prod_{r=1}^{q}A_{\mathbf v, i, \mu}(w_r)\right|\prec (N^{(C_0-1)\delta}\Phi + N^{-\zeta/2}\phi)^q.
\end{equation}
Each of the three cases $q=1,\, 2,\, 3$ can be proved as in \cite[Lemma 12.7]{Anisotropic}, and we leave the details to the reader. This concludes Lemma \ref{lemm_comparison_big}.
\end{proof}
 
\end{subsubsection}
\end{subsection}

\begin{subsection}{Averaged local law for $TX$}\label{section_averageTX}
In this section we prove the averaged local law in Theorem \ref{law_wideT}. Again for convenience, we only consider the case $w\in \mathbf D$ and $|z|^2 \le 1-\tau$. First we assume (\ref{assm_3rdmoment}) holds.
The anisotropic local law proved in the previous section gives a good a priori bound. In analogy to (\ref{eq_comp_F(X)}), we define
\begin{align*}
\wt F(X,w) : &=|w|^{1/2} |m_2(w)-m_{2c}(w)| =\left|\frac{1}{N}\sum\limits_{\nu\in\sI_2}\sG_{\nu\nu}(w)-|w|^{1/2}m_{2c}(w)\right|.
\end{align*}
Since $\Phi^2 =O(|w|^{1/2}/{(N\eta)})$, it suffices to show that $\wt F\prec \Phi^2$.
Following the argument in Section \ref{subsection_selfcomp}, analogous to (\ref{eq_comp_selfest}), we only need to prove that
\begin{equation}\label{eq_comp_selfestAvg}
N^{-n/2}\sum\limits_{i\in\mathcal I_1^M}\sum\limits_{\mu\in\mathcal I_2}\left|\bbE \left(\frac{\partial}{\partial X_{i\mu}}\right)^n\wt F^p(X)\right|=O\left((N^{\delta}\Phi^2)^p+\mathbb E\wt F^p(X)\right)
\end{equation}
for all $n=4,...,4p$. Here $\delta>0$ is an arbitrary positive constant. Analogously to (\ref{eq_comp_goal2}), it suffices to prove that for $n=4,...,4p$,
\begin{equation}\label{eq_comp_goalAvg}
N^{-n/2}\sum\limits_{i\in\mathcal I_1^M}\sum\limits_{\mu\in\mathcal I_2}\left|\bbE\prod_{r=1}^{p}\left(\frac{1}{N}\sum_{\nu\in\sI_2}A_{ \mathbf e_\nu, i, \mu}(w_r)\right)\right|=O\left((N^{\delta}\Phi^2)^p+\mathbb E\wt F^p(X)\right)
\end{equation}
for $\sum_r n(w_r)=n$. The only difference in the definition of $A_{\mathbf v, i, \mu}(w)$ is that when $n(w)=0$, we define
\[A_{\mathbf v, i, \mu}(w):=\sG_{\mathbf v\mathbf v}-|w|^{1/2}m_{2c}.\]

Similar to (\ref{eq_comp_Rs}) we define
\begin{equation}\label{eq_comp_RsAvg}
\mathcal R_{\nu, s}:=|\mathcal G_{\nu \mathbf v_s}|+|\mathcal G_{\mathbf v_s \nu}|.
\end{equation}
By the anisotropic local law, $\sG-\wt\Pi=O_{\prec}(\Phi)$. Hence combining with Lemma \ref{lem_comp_gbound} and (\ref{estimate_PiImw}), we get
\begin{equation}\label{eq_comp_r22}
 \frac{1}{N}\sum_{\nu\in\sI_2}\mathcal R_{\nu,s}^2\prec\frac{|w|^{1/2}\Im \sG_{\mathbf v_s\mathbf v_s}}{N\eta}\prec \frac{|w|\Im (m_{1c}+m_{2c})+|w|^{1/2}\Phi}{N\eta}=O(\Phi^2).
\end{equation}
Using the anisotropic local law again, we get $\sG=O_\prec(1)$. Then we have 
\begin{equation}\label{average_bound}
\left|\frac{1}{N}\sum_{\nu\in\sI_2}A_{ \mathbf e_\nu, i, \mu}(w)\right|\prec \frac{1}{N}\sum_{\nu\in\sI_2}\left(\mathcal R_{\nu,i}^2+\mathcal R_{\nu,\mu}^2\right)\prec \Phi^2\text{ for }n(w)\ge 1.
\end{equation}
Following (\ref{average_bound}), for $n\ge 4$, the left-hand side of (\ref{eq_comp_goalAvg}) is bounded by
\[\bbE\wt F^{p-q}(X)(\Phi^2)^q.\]
Applying Holder's inequality, we conclude the proof.



Then we prove the averaged local law when $\eta \ge N^{-1/2+\zeta}|m_{2c}|^{-1}$. It suffices to prove 
\begin{equation}\label{comp_avg_geX_self-improving-bound}
N^{-3/2}\left|\sum\limits_{i\in\mathcal I_1^M}\sum\limits_{\mu\in\mathcal I_2}\bbE \left(\frac{\partial}{\partial X_{i\mu}}\right)^3\wt F^p(X)\right|=O\left(\left(\frac{|w|^{1/2}}{N\eta}\right)^p+\mathbb E\wt F^p(X)\right).
\end{equation}
Analogous to (\ref{eq_comp_goalAvg}), it is reduced to show that
\begin{equation}\label{eq_comp_goalAvg_genX}
N^{-3/2}\left|\sum\limits_{i\in\mathcal I_1^M}\sum\limits_{\mu\in\mathcal I_2}\bbE\prod_{r=1}^{q}\left(\frac{1}{N}\sum_{\nu\in\sI_2}A_{ \mathbf e_\nu, i, \mu}(w_r)\right)\right|=O\left(\left(\frac{|w|^{1/2}}{N\eta}\right)^q+\mathbb E\wt F^q(X)\right)
\end{equation}
where $q$ is the number of words with nonzero length. Again we can prove the three cases $q=1,\, 2,\, 3$ as in \cite[Lemma 12.8]{Anisotropic}, and we leave the details to the reader. This concludes the averaged law.

\end{subsection}

\end{section}

\begin{appendix}

\section{Properties of $\rho_{1,2c}$ and Stability of (\ref{eq_self3}) }
\label{appendix1}


\subsection{Proof of Lemma \ref{lemm_rho} and Proposition \ref{prop_rho1c}}\label{subsection_append1}

We now prove Lemma \ref{lemm_rho}. First is a technical lemma for $f$ defined in (\ref{eqn_def_fwm}).
\begin{lem}\label{lemm_tech_f}
For $w>0$ and $|z|>0$, $f$ can be written as
\begin{equation}
f(\sqrt{w},m) = -\sqrt{w} + m + w^{-1/2} + \frac{1}{N}\sum_{i=1}^{n}l_i s_i \left( \frac{A_i}{m - a_i} + \frac{B_i}{m - b_i} + \frac{C_i}{m + c_i} \right), \label{PFD_f}
\end{equation}
where we have the following estimates for the poles and the coefficients,
\begin{align}
& \max\left(|z|,\frac{s_i+|z|^2}{\sqrt{w}}\right) < a_i < \frac{s_i + |z|^2}{\sqrt{w}} + |z|, \ \ a_n < a_{n-1} < \ldots < a_1 ,\label{bounda}\\
&  0< b_1 < b_2 < \ldots < b_n < \min\left(|z|,\frac{|z|^2}{\sqrt{w}}\right), \label{boundb}\\
& \frac{-(s_i + |z|^2)+ \sqrt{(s_i + |z|^2)^2+4w|z|^2} }{2\sqrt{w}} < c_i < |z|, \ \ c_1 < c_2 < \ldots < c_n , \label{boundc}
\end{align}
and
\begin{equation}
0 < A_i \le 2\frac{s_i+|z|^2 + \sqrt{w}|z|}{w}, \ 0 < B_i \le 2\frac{s_i+|z|^2 + \sqrt{w}|z|}{w}, \ 0 < C_i \le \frac{s_i+|z|^2 + \sqrt{w}|z|}{w}. \label{boundABC}
\end{equation}
\end{lem}

\begin{proof}
The proof is based on basic algebraic arguments. Let $$p_i = \sqrt{w}m^3 - (s_i + |z|^2)m^2 -\sqrt{w}|z|^2 m + |z|^4.$$
It is easy to verify that
$$\Delta=18(s_i + |z|^2)w|z|^6 + 4(s_i + |z|^2)^3 |z|^4 + (s_i + |z|^2)^2 w|z|^4 + 4w^2 |z|^6 - 27w|z|^8 > 0.$$
Thus $p_i$ has three distinct real roots. By the form of $p_i$, we see that there are two positive roots and one negative root, call them $a_i> b_i > 0 > -c_i$.
Now we perform the partial fraction expansion for the rational functions in (\ref{eqn_def_fwm}),
\begin{equation}
\frac{m^2 - |z|^2}{\sqrt{w}m^3 - (s_i + |z|^2)m^2 -\sqrt{w}|z|^2 m + |z|^4} = \frac{A'_i}{m-a_i}  + \frac{B'_i}{m-b_i} - \frac{C'_i}{m+c_i}, \label{PFD}
\end{equation}
where
\begin{equation}
A'_i=\frac{ a_i^2 - |z|^2}{\sqrt{w}(a_i-b_i)(a_i+c_i)}, \ B'_i=\frac{ b_i^2 - |z|^2}{\sqrt{w}(b_i-a_i)(b_i+c_i)}, \ C'_i=\frac{ - c_i^2 + |z|^2}{\sqrt{w}(c_i+a_i)(c_i+b_i)}. \label{ABC}
\end{equation}
We take $s_i=0$ in $p_i$ and call the resulting polynomial as
$$p_0 = \sqrt{w}m^3 - |z|^2m^2 -\sqrt{w}|z|^2 m + |z|^4 = \sqrt{w}\left(m-\frac{|z|^2}{\sqrt{w}}\right)\left(m^2-|z|^2\right),$$
which has roots $m=\pm |z|, |z|^2/\sqrt{w}$. By (\ref{ordering_si}), we have $p_1 < p_2 < \ldots < p_n < p_0$ for all $m\ne 0$. Comparing the graphs of $p_i$'s (as cubic functions of $m$) for $0\le i \le n$, we get that
\begin{equation}\label{bound1_ab}
\max\left(|z|,\frac{|z|^2}{\sqrt{w}}\right) < a_n < a_{n-1} < \ldots < a_1, \ 0<b_1 < b_2 < \ldots < b_n < \min\left(|z|,\frac{|z|^2}{\sqrt{w}}\right),
\end{equation}
and
\begin{equation}\label{bound1_c}
0 < c_1 < c_2 < \ldots < c_n < |z|.
\end{equation}
Thus we get (\ref{boundb}). By these bounds, we see that $a_i^2 - |z|^2>0$, $b_i^2 - |z|^2<0$ and $-c_i^2+|z|^2>0$, which, by (\ref{ABC}), give that $A'_i>0$, $B'_i>0$ and $C'_i>0$. Plugging (\ref{PFD}) into $f$, we get immediately (\ref{PFD_f})
for $A_i=A'_i a_i$, $B_i=B'_i b_i$ and $C_i=C'_i c_i$. 

Now we compare $p_i$ with $p_i':= \sqrt{w}m^3 - (s_i + |z|^2)m^2 -\sqrt{w}|z|^2 m,$ which has roots $$m=0,\ \frac{(s_i + |z|^2) \pm \sqrt{(s_i + |z|^2)^2+4w|z|^2} }{2\sqrt{w}}.$$ Since $p_i'<p_i$ for all $m$, we get
\begin{equation}\label{bound2_a}
a_i < \frac{(s_i + |z|^2) + \sqrt{(s_i + |z|^2)^2+4w|z|^2} }{2\sqrt{w}} < \frac{s_i + |z|^2}{\sqrt{w}} + |z|,
\end{equation}
and
\begin{equation}\label{bound2_c}
c_i>\frac{-(s_i + |z|^2)+ \sqrt{(s_i + |z|^2)^2+4w|z|^2} }{2\sqrt{w}}.
\end{equation}
From (\ref{bound1_c}) and (\ref{bound2_c}), we get (\ref{boundc}). Then we compare $p_i$ with $p_i'':= \sqrt{w}m^3 - (s_i + |z|^2)m^2,$ which has roots $w=0$, $(s_i + |z|^2)/\sqrt{w}$. Notice $p_i''>p_i$ for $m>|z|^2/\sqrt{w}$ and $a_i > |z|^2/\sqrt{w}$, so we get $a_i> (s_i + |z|^2)/{\sqrt{w}}.$ Combining this bound with (\ref{bound1_ab}) and (\ref{bound2_a}), we get (\ref{bounda}).

Finally we estimate the coefficients $A_i$, $B_i$ and $C_i$. Using (\ref{ABC}) and (\ref{bounda})-(\ref{boundc}), we first can estimate that
\begin{align*}
& A'_i = \frac{ (a_i - |z|)(a_i+|z|)}{\sqrt{w}(a_i-b_i)(a_i+c_i)} \le \frac{ a_i+|z|}{\sqrt{w}(a_i+c_i)}\le \frac{2}{\sqrt{w}}, \\
& B'_i = \frac{ (|z|+b_i)(|z|-b_i)}{\sqrt{w}(a_i-b_i)(b_i+c_i)} \le \frac{ |z|+b_i}{\sqrt{w}(b_i+c_i)} \le 2\frac{s_i + |z|^2+\sqrt{w}|z|}{w|z|}, \\
& C_i' = \frac{ (|z|-c_i)(c_i+|z|)}{\sqrt{w}(c_i+a_i)(c_i+b_i)} \le \frac{ |z| -c_i}{\sqrt{w}(c_i+b_i)} \le \frac{s_i + |z|^2+\sqrt{w}|z|}{w|z|},
\end{align*}
from which we get that
\begin{align}
& A_i=A'_i a_i \le \frac{2}{\sqrt{w}} \left(\frac{s_i + |z|^2}{\sqrt{w}} + |z|\right) = 2\frac{s_i + |z|^2 +\sqrt{w}|z|}{w},\label{eq_app_A} \\
& B_i = B'_i b_i \le 2 \frac{s_i + |z|^2+\sqrt{w}|z|}{w|z|} |z| = 2\frac{s_i + |z|^2+\sqrt{w}|z|}{w}, \label{eq_app_B} \\
& C_i = C'_i c_i \le  \frac{s_i + |z|^2+\sqrt{w}|z|}{w|z|} |z| = \frac{s_i + |z|^2+\sqrt{w}|z|}{w}.\label{eq_app_C}
\end{align}
\end{proof}

In (\ref{PFD_f}), it is sometimes convenient to reorder the terms and rename the constants to write $f$ as
\begin{equation}\label{PFD_fanother}
f(m) = -\sqrt{w} + m + w^{-1/2} + \frac{1}{N}\sum_{k=1}^{2n} \frac{C^{+}_k}{m - x_k} + \frac{1}{N}\sum_{l=1}^{n} \frac{C^{-}_l}{m + y_l}.
\end{equation}
where all the constants $C^{+}_k$ and $C^{-}_l$ are positive, and we choose the order such that
\begin{equation}
0<x_1 < x_2 < \ldots < x_{2n}, \ 0<y_1 < y_2 < \ldots <y_{n} .
\end{equation}
Clearly, $f$ is smooth on the $3n+1$ open intervals of ${\mathbb R}$ defined by
$$I_{-n}:=(-\infty,-y_{n}), \ I_{-k}:=(-y_{k+1},-y_{k})\ (k=1, \ldots, n-1),\ I_0:=(-y_1,x_1),$$
$$I_k:=(x_{k},x_{k+1}) \ (k=1, \ldots, 2n-1),\ I_{2n} :=(x_{2n},+\infty).$$
Next, we introduce the multiset $\mathcal C$ of critical points of $f$ (as a function of $m$), using the conventions that a nondegenerate critical point is counted once and a degenerated critical point twice. First we will prove the following elementary lemma about the structure of $\mathcal C$ (see Fig.\,\ref{figf_z15} and \ref{figf_z05}).

\begin{lem}\label{lemm_app_cpt}
(Critical points) We have $|\mathcal C \cap I_{-n}|=|\mathcal C \cap I_{2n}|=1$ and $|\mathcal C \cap I_k|\in \{0,2\}$ for $k=-n+1,\ldots, 2n-1$.
\end{lem}

\begin{proof}
We omit the dependence of $f$ on $w$ for now. By (\ref{PFD_fanother}) we have
\begin{align*}
&f'(m) = 1 - \frac{1}{N}\sum_{k=1}^{2n} \frac{C^{+}_k}{\left(m - x_k\right)^2} - \frac{1}{N}\sum_{l=1}^{n} \frac{C^{-}_l}{\left(m + y_l\right)^2}, \ f''(m) = \frac{1}{N}\sum_{k=1}^{2n} \frac{2C^{+}_k}{\left(m - x_k\right)^3} + \frac{1}{N}\sum_{l=1}^{n} \frac{2C^{-}_l}{\left(m + y_l\right)^3}.
\end{align*}
We see that $f''$ is decreasing on all the intervals $I_k$ for $k=-n+1,\ldots,2n-1$. Thus there is at most one point $m\in I_k$ such that $f''(m)=0$. We conclude that $f$ has at most two critical points on $I_k$. By the boundary conditions of $f'$ on $\partial I_k$, we get $|\mathcal C \cap I_k|\in \{0,2\}$ for $k=-n+1,\ldots, 2n-1$. For $m<-y_{n}$, we have $f''(m)<0$, while for $m>x_{2n}$, we have $f''(m)>0$. By the boundary conditions of $f'$ on $\partial I_{-n}$ and $\partial I_{2n}$, we see that $f'$ decreases from $1$ to $-\infty$ when $m$ increases from $-\infty$ to $-y_n$, while $f'$ increases from $-\infty$ to $1$ when $m$ increases from $x_{2n}$ to $+\infty$. Hence we conclude that each of the intervals $(-\infty,-y_n)$ and $(x_{2n},+\infty)$ contains a unique critical point in it, i.e. $|\mathcal C \cap I_{-n}|=|\mathcal C \cap I_{2n}|=1$.
\end{proof}

From this lemma, we deduce that $|\mathcal C|=2p$ is even. We denote by $z_{2p}$ the critical point in $I_{-n}$, $z_{1}$ the critical point in $I_{2n}$, and $z_2 \ge \ldots \ge z_{2p-1}$ the $2p-2$ critical points in $I_{-n+1} \cup \ldots \cup I_{2n-1}$. For $k=1, \ldots, 2p$, we define the critical values $h_k:=f(z_k)$. The next lemma is crucial in establishing the basic properties of $\rho_{1c}$ (see e.g. Fig.\,\ref{figf_z15}).

\begin{lem} \label{lemm_ordering}
(Orderings of the critical values) The critical values are ordered as $h_1 \ge h_2 \ge \ldots \ge h_{2p}.$ Furthermore, there is an absolute constant $C_0>0$ independent of $\tau$ such that $h_k \in [-C_0(\tau^{-1}|w|^{-1/2} + |z|)-\sqrt{w},C_0(\tau^{-1}|w|^{-1/2} + |z|)-\sqrt{w}]$ for $k=1, \ldots, 2p$.
\end{lem}

\begin{proof}
Notice for the equation (\ref{selfm_reduced}), if we multiply both sides with the product of all denominators in $f$, we get a polynomial equation $P_{w}(m)=0$ with $P_{w}$ being a polynomial of degree $3n+1$. An immediate consequence is that for any fixed $w>0$ and $E \in \mathbb R$, $f(\sqrt{w},m)=E$ can have at most $3n+1$ roots in $m$. This fact is useful in the proof of this lemma and Lemma \ref{lemm_rho}.

For $i=-n, \ldots, 2n$, define the subset $J_i(w):=\{m\in I_i: \partial_m f(\sqrt{w}, m)>0\}$.
From Lemma \ref{lemm_app_cpt}, we deduce that if $i=-n+1,\ldots, 2n-1$, then $J_i \ne \emptyset$ if and only if $I_i$ contains two distinct critical points of $f$, in which case $J_i$ is an interval. Moreover, we have $J_{-n}=(-\infty,z_{2p})$ and $J_{2n}=(z_1,+\infty)$. Next, we observe that for any $-n\le i < j \le 2n$, we have $f(J_i) \cap f(J_j) =\emptyset$. Otherwise if there were $E\in f(J_i)\cap f(J_j)$, we would have $|\{x:f(x)=E\}|> 3n+1$. We hence conclude that the sets $f(J_i)$, $-n \le i \le 2n$ can be strictly ordered. The claim $h_1 \ge h_2 \ge \ldots \ge h_{2p}$ is now reformulated as
\begin{equation}
f(J_i) < f(J_j) \text{ whenever } i < j \text{ and } J_i,J_j \ne \emptyset. \label{order1}
\end{equation}

To prove (\ref{order1}), we use a continuity argument. Let $t\in (0,1]$ and introduce
\begin{equation*}
f^t(m) = -\sqrt{w} + m + w^{-1/2} + \frac{t}{N}\sum_{k=1}^{2n} \frac{C^{+}_k}{m - x_k} + \frac{t}{N}\sum_{l=1}^{n} \frac{C^{-}_l}{m + y_l}.
\end{equation*}
It is easy to check (\ref{order1}) holds for small enough $t>0$. We claim that
\begin{equation}
J_i\ne \emptyset \Rightarrow J_i^t \ne \emptyset \text{ for all } t\in (0,1]. \label{order2}
\end{equation}
This is trivial for $i=-n,2n$. Recall that for $-n+1 \le i \le 2n-1$, $J_i^t \ne \emptyset$ is equivalent to $I_i$ containing two distinct critical points. Moreover, $\partial_t \partial_m f^t(m) <0$ in $I_{-n+1}\cup \ldots \cup I_{2n-1}$, from which we deduce that the number of distinct critical points in each $I_i$, $i=-n+1, \ldots, 2n-1$, does not decreases as $t$ decreases. This proves (\ref{order2}).

Next, suppose that there exist $i<j$ such that $J_i , J_j \ne \emptyset$ and $f(J_i) > f(J_j)$. From (\ref{order2}), we deduce that $J_i^t , J_j^t \ne \emptyset$ for all $t\in (0,1]$. By a simple continuity argument, we get that $f^t(J_i^t) > f^t(J_j^t)$ for all $t\in (0,1]$. However, this is impossible for small enough $t$ as explained before (\ref{order2}). This concludes the proof of (\ref{order1}).

To prove the second statement of Lemma \ref{lemm_ordering}, we only need to show that $h_1 \le C_0(\tau^{-1}|w|^{-1/2} + |z|) - \sqrt{w}$ and $h_{2p}\ge -C_0(\tau^{-1}|w|^{-1/2} + |z|)-\sqrt{w}$ for some absolute constant $C_0$. We only give the proof for $h_1$; the proof for $h_{2p}$ is similar. At $z_1$, we have
\begin{equation*}
f(z_1) + \sqrt{w} \le (z_1 + y_n) \left[1 + \frac{1}{N}\sum_{k=1}^{2n} \frac{C^{+}_k}{\left(z_1- x_k\right)^2} + \frac{1}{N}\sum_{l=1}^{n} \frac{C^{-}_l}{\left(z_1 + y_l\right)^2}\right]+w^{-1/2} =2(z_1 + y_n)+w^{-1/2},
\end{equation*}
where we use
\begin{equation}\label{eqn_derf}
0=f'(z_1)= 1 - \frac{1}{N}\sum_{k=1}^{2n} \frac{C^{+}_k}{\left(z_1 - x_k\right)^2} - \frac{1}{N}\sum_{l=1}^{n} \frac{C^{-}_l}{\left(z_1 + y_l\right)^2}.
\end{equation}
Now we would like to estimate $z_1 + y_n$. Again using (\ref{eqn_derf}), we have that
$$\frac{1}{N}\sum_{k=1}^{2n} \frac{C^{+}_k}{\left(z_1 - x_{2n}\right)^2} + \frac{1}{N}\sum_{l=1}^{n} \frac{C^{-}_l}{\left(z_1 - x_{2n}\right)^2} \ge 1.$$
Then by (\ref{boundABC}) we get
$$z_1 - x_{2n} \le \sqrt{\frac{1}{N}\sum_{k=1}^{2n} C^{+}_k + \frac{1}{N}\sum_{l=1}^{n} C^{-}_l} \le \sqrt{5\frac{\tau^{-1}+|z|^2 + \sqrt{w}|z|}{w}}.$$
Using the above estimates and (\ref{bounda})-(\ref{boundc}), we obtain that
\begin{align*}
f(z_1) \le 2\left(\sqrt{5\frac{\tau^{-1}+|z|^2 + \sqrt{w}|z|}{w}}+\frac{s_1+|z|^2}{\sqrt{w}}+2|z|\right)+w^{-1/2} - \sqrt{w} \le C_0(\tau^{-1}|w|^{-1/2} + |z|)- \sqrt{w}.
\end{align*}
for some constant $C_0>0$ that does not depend on $\tau$.
\end{proof}

\begin{figure}[htb]
\centering
\begin{subfigure}{\textwidth}
\includegraphics[height=6cm]{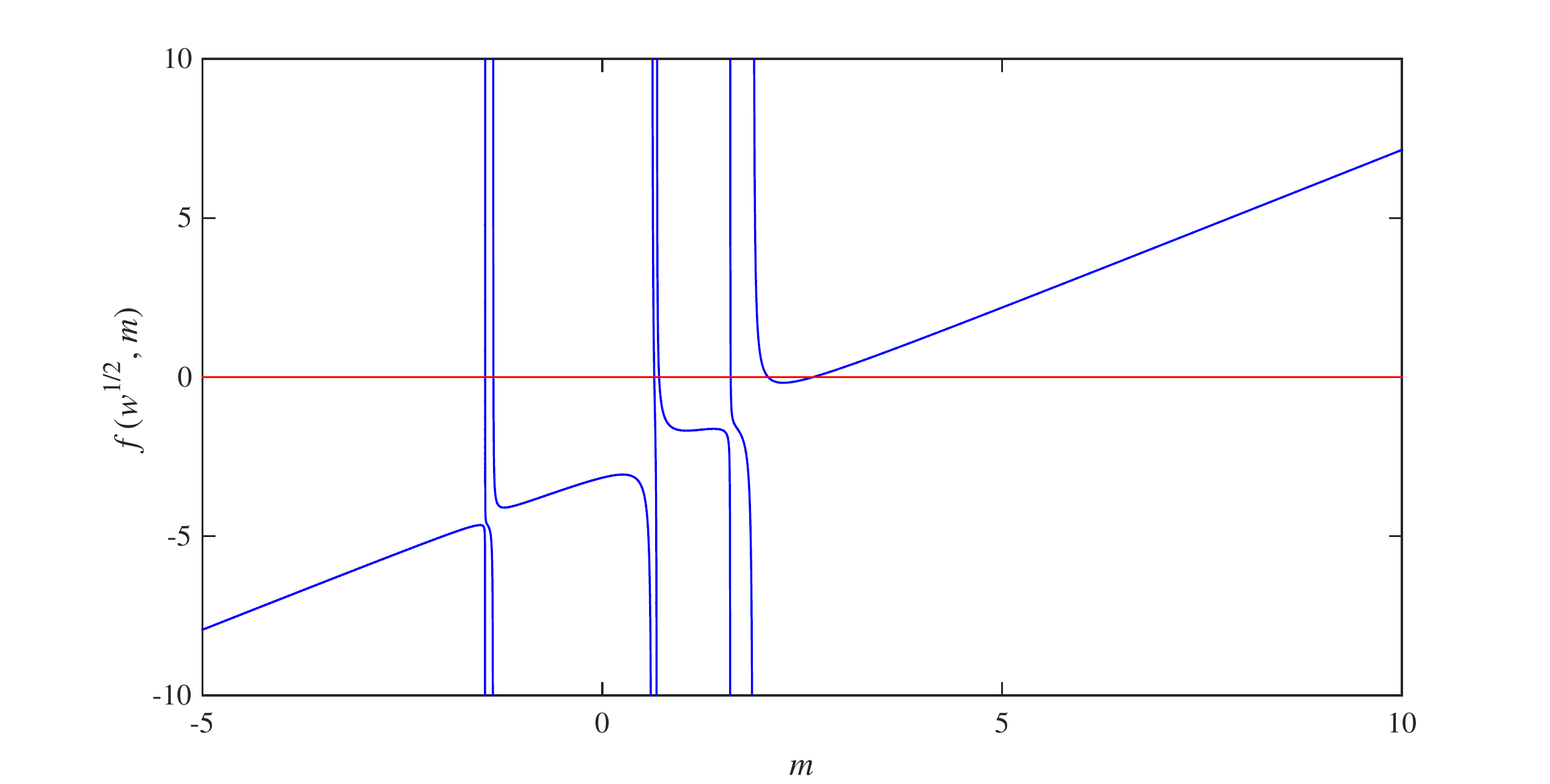}
\end{subfigure}
\begin{subfigure}{\textwidth}
\includegraphics[height=6cm]{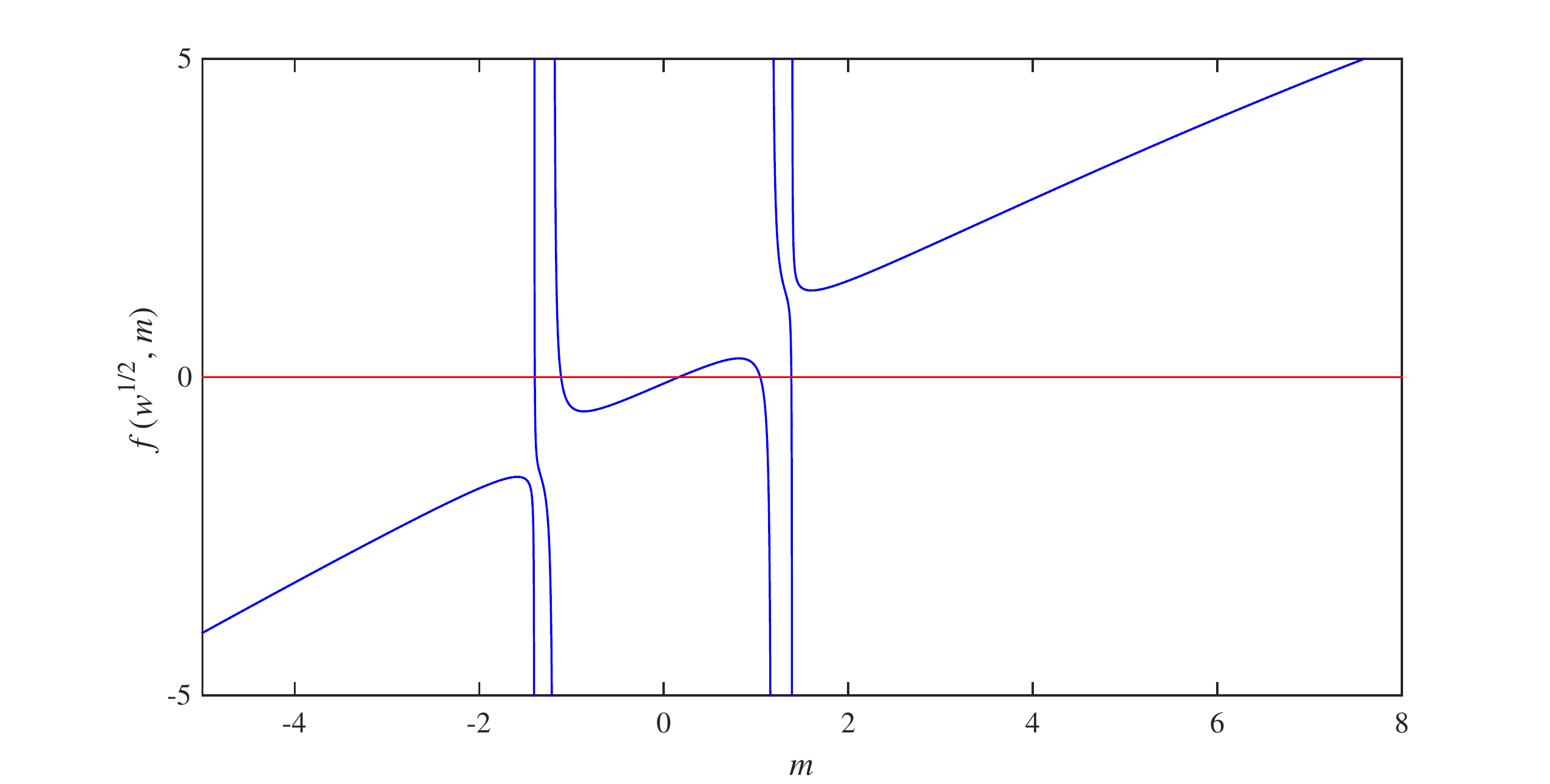}
\end{subfigure}
\caption{The graphs of $f(\sqrt{w},m)$ for the example from Figure \ref{fig_circular}, i.e. $\rho_{\Sigma}= 0.5\delta_{\sqrt{2/17}} + 0.5\delta_{4\sqrt{2/17}}$. We take $|z|=1.5$, and $w=10$ and $0.01$ in the upper and lower graphs, respectively. In the lower graph, we only plot the five branches near $m=0$. The remaining two branches are far away.}
\label{figf_z15}
\end{figure}

\begin{figure}[htb]
\centering
\begin{subfigure}{\textwidth}
\includegraphics[height=6cm]{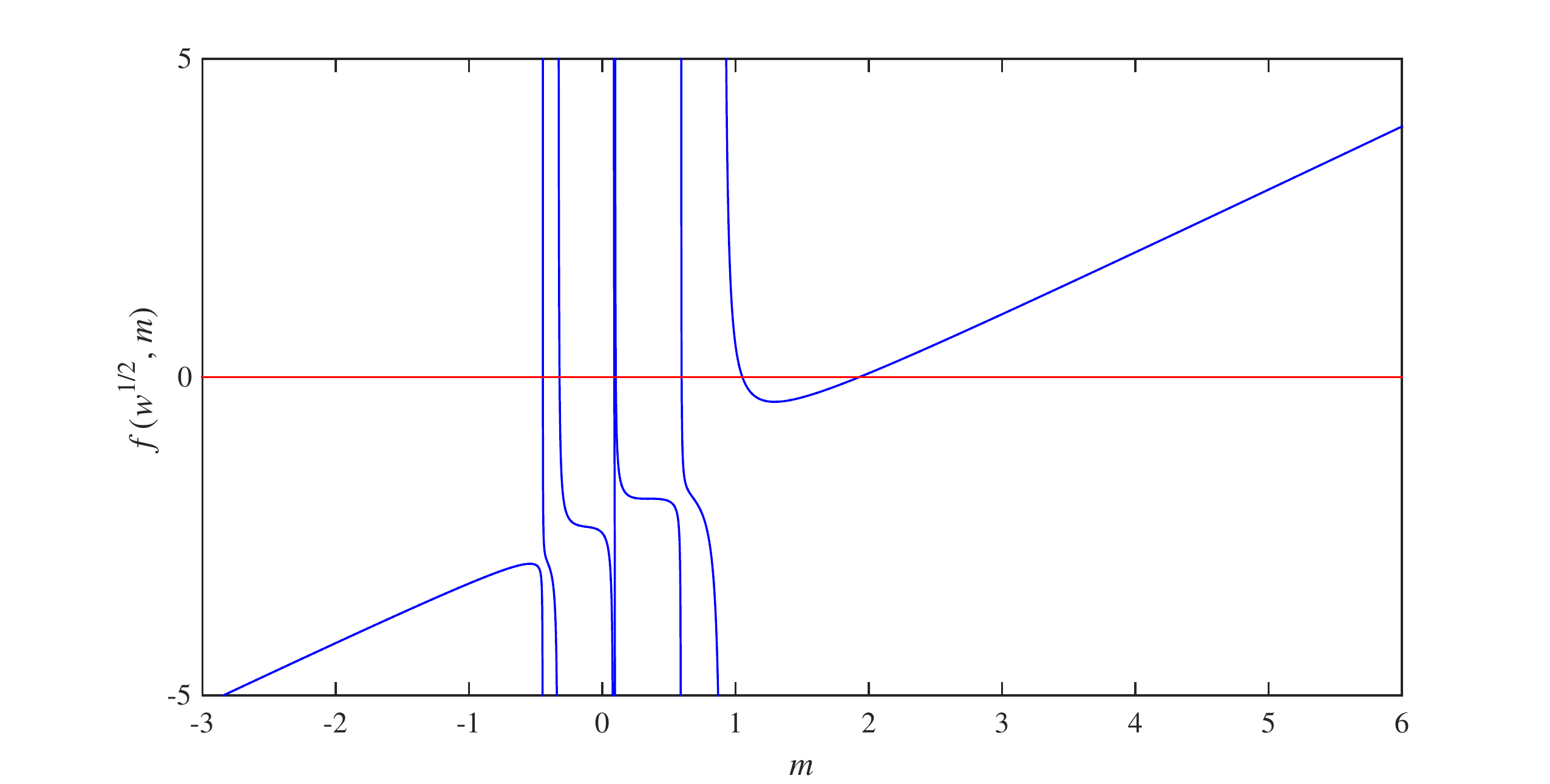}
\end{subfigure}
\begin{subfigure}{\textwidth}
\includegraphics[height=6cm]{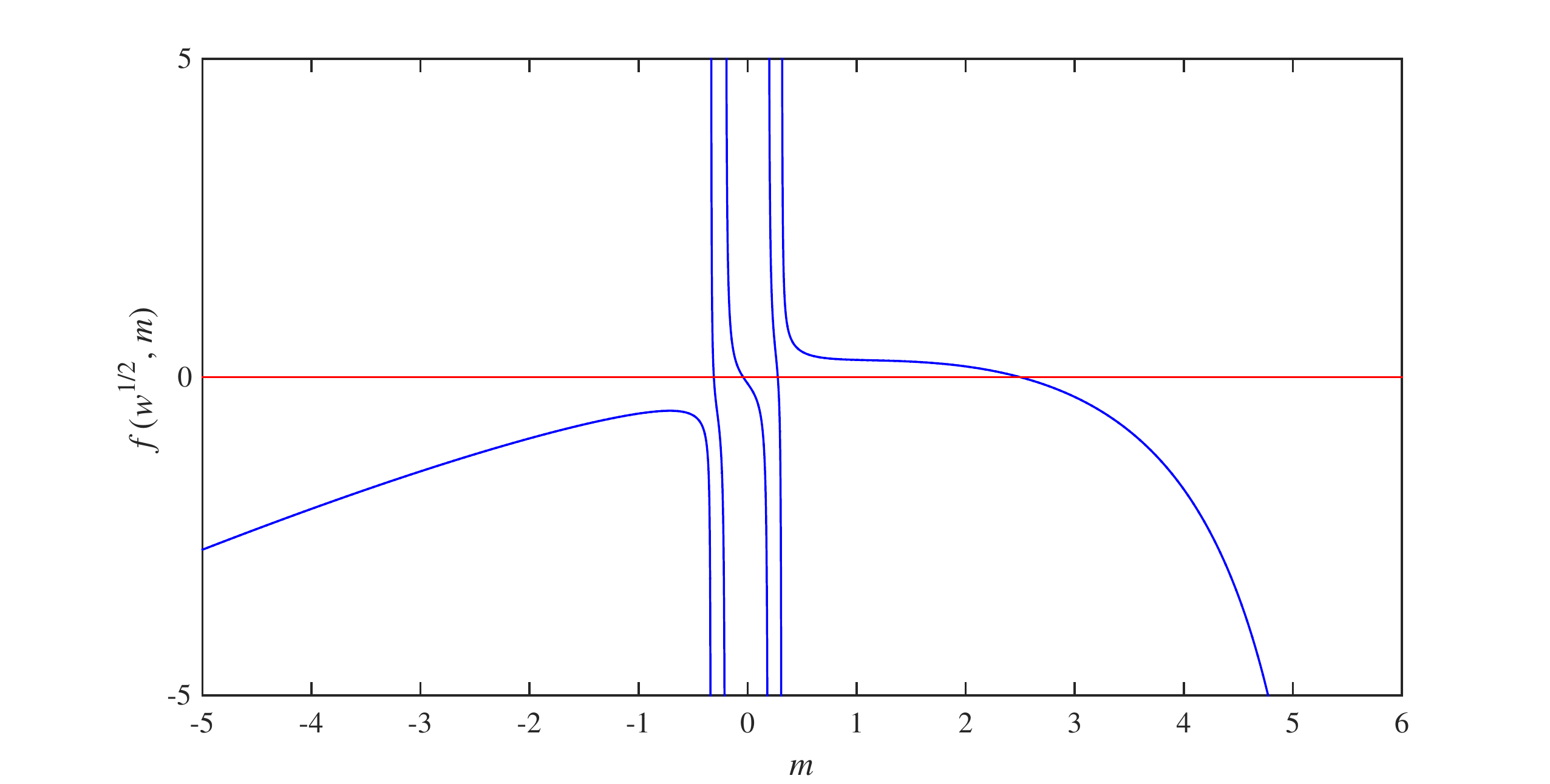}
\end{subfigure}
\caption{The graphs of $f(\sqrt{w},m)$ for the example from Figure \ref{fig_circular}, i.e. $\rho_{\Sigma}= 0.5\delta_{\sqrt{2/17}} + 0.5\delta_{4\sqrt{2/17}}$. We take $|z|=0.5$, and $w=6$ and $0.01$ in the upper and lower graphs, respectively. In the lower graph, we only plot the five branches near $m=0$. The remaining two branches are far away.}
\label{figf_z05}
\end{figure}

\begin{proof}[Proof of Lemma \ref{lemm_rho}]

Let $J(w):= \bigcup_{i=-n}^{2n} J_i(w)$. Given $w > 0$ such that $0 \in f(J(w))$, 
then the set $\{m\in \mathbb R: f(\sqrt{w}, m) = 0\}$ has $3n+1$ points. Since $f(\sqrt{w}, m) = 0$ has at most $3n+1$ solutions in $m$, we deduce that $m_c(w)$ is real and hence $m_{1c}(w)$ is also real. Since $m_{1c}$ is the Stieltjes transform of $\rho_{1c}$, we conclude that $w \notin \text{supp}\, \rho_{1c}$. On the other hand, suppose $w> 0$ and $0\notin f\left(J(w)\right)$. Then the set of preimages $\{ m \in \mathbb R: f(\sqrt{w}, m)=0 \} = \{ m \in \mathbb R: P_w(m)=0 \}$ has $3n-1$ points. Since $P_w(m)$ is a degree $3n+1$ polynomial with real coefficients, we conclude that $P_w$ has a unique root with positive imaginary part. By the uniqueness of the solution of $P_{w+ i \eta} $ in $\mathbb C_+$ (Lemma \ref{lemm_density}) and the continuity of the roots of $P_{w+ i \eta} $ in $\eta$, we conclude that $\text{Im}\, m_c(w)>0$ and $\text{Im}\, m_{1c}(w)>0$ by taking $\eta \searrow 0$, i.e. $w \in \text{supp}\, \rho_{1c}$. In sum, we get
\begin{equation}\label{eqn_supp}
\text{supp}\, \rho_{1c} = \overline{\left\{w> 0: 0 \notin f\left(J(w)\right) \right\}}.
\end{equation}

From Lemma \ref{lemm_ordering}, we see that there exists an absolute constant $C_1>0$ such that if $w \ge C_1 \tau^{-1}$, then $h_1(\omega) \le C_0(\tau^{-1}|w|^{-1/2} + |z|)-\sqrt{w} < 0$. Hence fix $w \ge C_1\tau^{-1}$,  we have $0 \in f(J_{2n}(w))$ and $w\notin \text{supp}\, \rho_{1c}$ (see the upper graphs in Fig.\,\ref{figf_z15} and \ref{figf_z05}). This shows that $\rho_{1c}$ is compactly supported in $[0, C_1 \tau^{-1}]$. Now we decrease $w$ so that $w< s_1 + |z|^2 +1$, then using (\ref{bounda}),
$$h_1 (w) > z_1 + w^{-1/2} - \sqrt{w} >\frac{s_1 + |z|^2 + 1- w}{\sqrt{w}}>0.$$
By continuity, there must be some $ 0< w < C\tau^{-1}$ such that $0 \notin f\left(J(w)\right)$. Thus $\text{supp}\, \rho_{1c} \ne \emptyset$. By (\ref{eqn_supp}), it is not hard to see that $\text{supp}\,\rho_{1c}$ is a disjoint union of (countably many) closed intervals,
\begin{equation}
{\rm{supp}} \, \rho_{1c} = \bigcup_{k} [e_{2k}, e_{2k-1}] ,
\end{equation}
where $C_1\tau^{-1}\ge e_1 \ge e_2 \ge \ldots $. Furthermore, for $e_i$ to be a boundary point, we must have that $0$ is a critical value of $f(\sqrt{e_i},m)$, i.e. there is a unique critical point $m=m_c(e_i)$ such that
\begin{equation}\label{equationEm}
f(\sqrt{e_i},m_c(e_i))=0, \ \ {\partial_m f} (\sqrt{e_i},m_c(e_i)) =0.
\end{equation}
Notice the two equations in (\ref{equationEm}) are equivalent to two polynomial equations in $(\sqrt{w},m)$ with order $3n+1$ and $6n$, respectively. By B\'ezout's theorem, there are at most finitely many solutions to (\ref{equationEm}). Hence there are finitely many $e_i$'s, call them $e_1 \ge e_2 \ge \ldots \ge e_{2L}$, where $L\equiv L(n)\in\mathbb N$. To prove the statement about $e_{2L}$, we use Lemma \ref{lemm_edge0} below. This concludes Lemma \ref{lemm_rho}.
\end{proof}

\begin{lem} \label{lemm_edge0}
If $1+\tau \le |z|^2 \le 1+\tau^{-1}$, there is a constant $\epsilon(\tau)>0$ so that $e_{2L} \ge \epsilon(\tau)$. If $|z|^2\le 1-\tau$, $e_{2L}=0$ and $\rho_{1c}(x) \sim x^{-1/2} $ when $x\searrow 0$.
\end{lem}

\begin{proof}
By this lemma, the behavior of the leftmost edge $e_{2L}$ changes essentially when $z$ crosses the unit circle. From the following proof, we see that the singularity happens at $|z|^2=N^{-1}\sum_{i=1}^n l_i s_i.$
Thus the fact that the singular circle has radius $1$ comes from our normalization (\ref{assm4}) for $T$.

We first study equation (\ref{selfm_reduced}) when $w \searrow 0$ in the case $1+\tau\le |z|^2 \le 1+\tau^{-1}$. We calculate the derivative of $f$ as
\begin{align}
\partial_m f(\sqrt{w}, m) = 1 & + \frac{1}{N}\sum_{i=1}^n l_i s_i \frac{ m^2 - |z|^2}{\sqrt{w}m^3 - (s_i + |z|^2)m^2 -\sqrt{w}|z|^2 m + |z|^4} \nonumber\\
& - \frac{m}{N}\sum_{i=1}^n l_i s_i \frac{ \sqrt{w}\left(m^2 - |z|^2\right)^2 +2s_i |z|^2 m}{\left[\sqrt{w}m^3 - (s_i + |z|^2)m^2 -\sqrt{w}|z|^2 m + |z|^4\right]^2} . \label{eqn_partialm}
\end{align}
It is easy to see that $J_0 \ne \emptyset$ for all $w>0$, since $\partial_m f(\sqrt{w}, 0)=1-|z|^{-2}>0$ (see the lower graph in Fig.\,\ref{figf_z15}). Call the end points of $J_0$ as $z_k(w)>0$ and $z_{k+1}(w)<0$. By the definition of $I_0$, we have $z_{k} < b_1 < |z|$.
Suppose $z_{k}=o(|z|)$ as $w\to 0$, then (\ref{eqn_partialm}) gives that $0 = 1 - |z|^{-2} + o(1) ,$ which gives a contradiction. Thus $z_{k}\sim |z|$ as $w\to 0$. Now using $\partial_m f(\sqrt{w},z_{k})=0$, we can estimate that
\begin{align}
f(\sqrt{w},z_{k}) & = -\sqrt{w}+ \frac{z_{k}^2}{N}\sum_{i=1}^n l_i s_i \frac{ \sqrt{w}\left(z_{k}^2 - |z|^2\right)^2 +2s_i |z|^2 z_{k}}{\left[\sqrt{w}z_{k}^3 - (s_i + |z|^2)z_{k}^2 - \sqrt{w}|z|^2 z_{k} + |z|^4\right]^2} \nonumber\\
& \ge -\sqrt{w}+\frac{1}{N}\sum_{i=1}^n l_i s_i \frac{ 2s_i |z|^2 z_{k}^3}{|z|^8} \ge c -\sqrt{w} \label{estmate_fI0}
\end{align}
for some $C>0$ independent of $w$, where in the second step we use that
$$\sqrt{w}z_{k}^3 - (s_i + |z|^2)z_{k}^2 - \sqrt{w}|z|^2 z_{k} + |z|^4 >0, \text{ and } \sqrt{w}z_{k}^3 - (s_i + |z|^2)z_{k}^2 - \sqrt{w}|z|^2 z_{k}<0$$
which come from that $0< z_{k} < b_i $ for all $1\le i \le n$. By (\ref{estmate_fI0}), we can find $\epsilon$ small enough such that $f(\sqrt{w},z_{k})>0$ for all $0<w\le \epsilon$. In this case $0\in f(J_0({w}))$ and hence $w\notin \text{supp}\,\rho_{1c}$. In fact, it is not hard to see that there is a solution $m_0 = \sqrt{w}|z|^2/(|z|^2-1) + o(\sqrt{w}) \in I_0$ such that $f(\sqrt{w},m_0)=0$ and $\partial_m f(\sqrt{w},m_0)>0$. This proves the first statement of Lemma \ref{lemm_edge0}.

Now we study equation (\ref{selfm_reduced}) when $|z|^2\le 1-\tau$ and $w \to 0$. For later purpose, we allow $w$ to be complex and prove a more general result than what we need for this lemma. Let $w=0$ in the equation (\ref{selfm_reduced}), we get
$m=0$ or
\begin{align}
0= 1 + \frac{1}{N}\sum_{i=1}^n l_i s_i \frac{m^2 - |z|^2}{- (s_i + |z|^2)m^2+|z|^4}.
\end{align}
We define
\begin{align}\label{g_anotherform}
g(x) = 1 + \frac{1}{N}\sum_{i=1}^n l_i s_i \frac{ x - |z|^2}{- (s_i + |z|^2)x+|z|^4} = \frac{|z|^2}{N}\sum_{i=1}^n l_i \frac{ -x + |z|^2 - s_i}{- (s_i + |z|^2)x+|z|^4}.
\end{align}
It is easy to see that
$g$ is smooth and decreasing on the intervals defined through
$$K_{1}:=\left(-\infty, \frac{|z|^4}{s_1 + |z|^2}\right),\ K_i:=\left(\frac{|z|^4}{s_{i-1} + |z|^2}, \frac{|z|^4}{s_{i} + |z|^2}\right) \ (i=2,\ldots, n), \ K_{n+1}:=\left(\frac{|z|^4}{s_{n} + |z|^2}, \infty\right).$$
By the boundary values of $g$ on these intervals, we see that $g(x)$ has exactly one zero on intervals $K_i$ for $i=1,\ldots,n$, and has no zero on $K_{n+1}$. Since $g(x)=0$ is equivalent to a polynomial equation of order $n$, it has at most $n$ solutions. We conclude that all of its solutions are real. Obviously the zeros on the intervals $K_i$ are positive for $i=2,\ldots,n$. Now we study the zero on $K_{1}$. Observe that $g(0) = 1 - |z|^{-2} < 0$ (as $|z|^2\le 1-\tau$), the zero on $K_1$ is negative, call it $-t$. Moreover, we can verify that $g(-\tau^{-1})>0$ by (\ref{g_anotherform}), so $t<\tau^{-1}$. If $|z|^2\ge \tau/2$, then by the concavity of $g$ on the $K_1$, we get
\begin{equation}\label{eqn_estimate_t1}
 t \ge \frac{g(0)}{g'(0)} \ge \frac{|z|^4(1-|z|^2)}{s_1} \ge \frac{\tau^4}{4} .
 \end{equation}
In the case $|z|^2 \le \tau/2$, we have $|z|^2 - s_n \le -\tau/2$ and $g(|z|^2 - s_n) \le 0$ by (\ref{g_anotherform}).
Hence we have 
\begin{equation}\label{eqn_estimate_t2}
-t\le |z|^2 - s_n \le -\tau/2.
\end{equation}
Combining (\ref{eqn_estimate_t1}) and (\ref{eqn_estimate_t2}), we get that $c \tau^4 \le t \le \tau^{-1}$ for some constant $c>0$.

Now we return to the self-consistent equation (\ref{selfm_reduced}). The previous discussions show that
$$f(0 , i\sqrt{t})=0,\ t \ge c\tau^4.$$
It is easy to see that there exists constants $c_1,\tau'>0$ such that
\begin{equation}\label{eqn_app_denom}
\left|- (s_i + |z|^2) m^2+|z|^4+ \sqrt{w}\left( m^3  - |z|^2 m\right)\right| \ge c_1 \text{  for  } |m-i\sqrt{t}|\le \tau'.
\end{equation}
First we consider the case $|z|\ge \epsilon >0$. Expanding $f(\sqrt{w},m)$ around $(0, i\sqrt{t})$ and using (\ref{eqn_app_denom}), 
\begin{equation}
0 = \partial_{\sqrt{w}} f(0,i\sqrt{t}) \sqrt{w} + \partial_m f(0,i\sqrt{t}) (m - i\sqrt{t}) + o(\sqrt{w}) + o(m - i\sqrt{t}). \label{self_expand}
\end{equation}
By (\ref{eqn_partialm}),
\begin{equation}\label{eqn_partialw}
\partial _{\sqrt{w}} f(\sqrt{w}, m) = -1 -\frac{ m^2}{N}\sum\limits_{i = 1}^n {l_i } s _i \frac{\left( m^2  - |z|^2 \right)^2}{{\left[- (s_i + |z|^2) m^2+|z|^4+ \sqrt{w}\left( m^3  - |z|^2 m\right)\right]^2 }},
\end{equation}
and (\ref{eqn_app_denom}), we get $\left|\partial _{\sqrt{w}} f(0,i\sqrt{t}) \right| \le C$ and
\begin{align}
\partial_m f(0,i\sqrt{t}) = \frac{t}{N}\sum_{i=1}^n l_i s_i \frac{ 2s_i |z|^2 }{\left[(s_i + |z|^2)t+|z|^4\right]^2} \ge c_2 \label{partialm_t}
\end{align}
for some $c_2>0$. Using (\ref{partialm_t}), we get from (\ref{self_expand}) that
\begin{equation}\label{eqn_app_largez}
m-i\sqrt{t} = O(\sqrt{w}), \ \text{ if } |z|\ge \epsilon.
\end{equation}
In particular, this shows that $|m| \approx \Im \, m \sim 1$ as $w \to 0$.

Then assume that $|z|^2< \epsilon$, for sufficiently small $\epsilon$. From $g(-t)=0$ and (\ref{g_anotherform}), we get that
\begin{align}
 \frac{1}{N}\sum_{i=1}^n l_i  \frac{ t + |z|^2 - s_i }{(s_i + |z|^2)t+|z|^4}=0.\label{equationt}
\end{align}
From the leading order term, we get $t^{-1}= t_0^{-1} + O(|z|^2)$, where $t_0:=\left(N^{-1} \sum_i l_i /s_i\right)^{-1}$. Expanding (\ref{equationt}) to the first order term of $|z|^2$, we get
\begin{equation}
t=t_0 +  \left(\frac{t_0^2}{N}\sum_i\frac{l_i}{s_i^2} -2 \right)|z|^2 + O(|z|^4).\label{defn_t}
\end{equation}
Now we write equation (\ref{selfm_reduced}) as
\begin{equation}
F(\sqrt{w},m)=0,
\end{equation}
where $F(\sqrt{w} , m): = f(\sqrt{w},m)/m.$ Expanding $F$ around $(0, i\sqrt{t})$ and using (\ref{eqn_app_denom}), we get
\begin{align}
0 = &\partial_{\sqrt{w}} F(0, i\sqrt{t}) \sqrt{w} + \partial_m F(0, i\sqrt{t}) (m - i\sqrt{t}) + \partial_m \partial_{\sqrt{w}}F(0, i\sqrt{t}) (m - i\sqrt{t})\sqrt{w} \nonumber\\
&+\frac{1}{2}\partial^2_{\sqrt{w}} F(0, i\sqrt{t}) w + \frac{1}{2}\partial^2_m F(0, i\sqrt{t}) (m - i\sqrt{t})^2 + o(w,|m - i\sqrt{t}|^2,|m - i\sqrt{t}|\sqrt{w}). \label{eqn_app_order2}
\end{align}
We can calculate that (the partial derivatives of $F$ can be obtained using (\ref{eqn_partialm}) and (\ref{eqn_partialw}))
\begin{align}
& \partial_{m} F(\sqrt{w}, i\sqrt{t}) = - \frac{ 2i |z|^2 + 2\sqrt{wt_0}}{t_0^{3/2}} +o(|z|^2,\sqrt{w}), \label{partialm_smallmw}\\
& \partial _{\sqrt{w}} F(\sqrt{w}, i\sqrt{t}) = \left(i|z|^2  + 2\sqrt{wt_0} \right)\frac{\sqrt{t_0}}{N}\sum\limits_{j = 1}^n \frac{l_j}{s^2_j} + o(|z|^2,\sqrt{w}). \label{partialw_smallmw}
\end{align}
From (\ref{partialm_smallmw}) and (\ref{partialw_smallmw}), we get that
\begin{align*}
& \partial_m F(0, i\sqrt{t}) = - \frac{ 2i |z|^2 }{t_0^{3/2}} +o(|z|^2),\ \partial_{\sqrt{w}} F(0, i\sqrt{t}) = \frac{i|z|^2\sqrt{t_0}}{N}\sum\limits_{j = 1}^n \frac{l_j}{s^2_j} + o(|z|^2), \\
& \partial_m \partial_{\sqrt{w}}F(0, i\sqrt{t}) = - \frac{ 2 }{t_0} +O(|z|^2),\ \partial^2_{\sqrt{w}} F(0, i\sqrt{t})= \frac{2t_0}{N}\sum\limits_{j = 1}^n \frac{l_j}{s^2_j} + O(|z|^2), \ \partial^2_m F(0, i\sqrt{t}) = O(|z|^2).
\end{align*}
Plugging the above results into (\ref{eqn_app_order2}), we get that
\begin{align}
0 = &\left[\frac{i|z|^2\sqrt{t_0}+\sqrt{w}t_0}{N}\sum\limits_{j = 1}^n \frac{l_j}{s^2_j}+ o(|z|^2)\right] \sqrt{w} + \left[ -2 \frac{ i |z|^2 + \sqrt{wt_0}}{t_0^{3/2}}+o(|z|^2) \right] (m - i\sqrt{t}) \nonumber\\
&  + o(w,|m - i\sqrt{t}|^2,|m - i\sqrt{t}|\sqrt{w}).
\end{align}
Observing that $ \left|i|z|^2\sqrt{t_0}+\sqrt{w}t_0\right| \sim |z|^2 + \sqrt{|w|}$, we get
\begin{equation}\label{eqn_app_smallz}
m - i\sqrt{t} = \left[\frac{t_0^2 }{2N}\sum\limits_{j = 1}^n \frac{l_j}{s^2_j} +O(|w|^{1/2}+|z|^2)\right]\sqrt{w}, \ \text{ if } |z|< \epsilon.
\end{equation}

Combing (\ref{eqn_app_largez}) and (\ref{eqn_app_smallz}), we get that if $|z|^2<1-\tau$, $m = i\sqrt{t} +O(\sqrt{w}) $ when $w \to 0$. In particular, this shows that $|m| \approx \Im\, m \sim 1$ when $w\to 0$. Finally we conclude the proof of Lemma \ref{lemm_edge0} by using that $m_{1c}(w)=m_c (w) w^{-1/2} - 1$.
\end{proof}

To prove Proposition \ref{prop_rho1c}, we need the following lemma, which is a consequence of the edge regularity conditions (\ref{regular1}) and (\ref{regular2}).

\begin{lem} \label{lemm_edge_reg}
Suppose $e_k \ne 0$ is a regular edge. Then $|m_{1c}(w)-m_{1c}(e_k)| \sim |w-e_k|^{1/2}$ as $w \to e_k$ and $\min_{l\ne k} |e_{l}-e_k| \ge \delta$ for some constant $\delta>0$.
\end{lem}

\begin{proof}
Denote $m_k:=m_c(e_k)$ and let $w\to e_k$. Notice by Lemma \ref{lemm_rho}, if $e_k\ne 0$, we have
\begin{equation}\label{eqn_ek}
\epsilon\le e_k \le C \tau^{-1}.
\end{equation}
Then we expand $f$ around $(\sqrt{e_k},m_k)$ to get that
\begin{align}
0 = & \partial_{\sqrt{w}}f(\sqrt{e_k},m_k)(\sqrt{w}-\sqrt{e_k})+\frac{1}{2}\partial_{m}^2f(\sqrt{e_k},m_k)(m_c(w)-m_k)^2 \nonumber\\
& +O\left[|\sqrt{w}-\sqrt{e_k}|^2+|m_c(w)-m_k|^3+|\sqrt{w}-\sqrt{e_k}||m_c(w)-m_k|\right], \label{expand_edge}
\end{align}
where by (\ref{eqn_partialw}),
\begin{equation}\label{partialfw_edge}
\partial _{\sqrt{w}} f(\sqrt{e_k}, m_k) = -1 -\frac{m_k^2}{N}\sum\limits_{i = 1}^n l_i s_i \frac{\left(m_k^2  - |z|^2 \right)^2}{{e_k(m_k - a_i)^2(m_k - b_i)^2(m_k + c_i)^2 }} ,
\end{equation}
and by (\ref{PFD_f}),
\begin{equation}\label{partialfm_edge}
\partial_{m}^2f(\sqrt{e_k},m_k) =  \frac{2}{N}\sum_{i=1}^{n}l_i s_i \left[ \frac{A_i}{(m_k - a_i)^3} + \frac{B_i}{(m_k - b_i)^3} + \frac{C_i}{(m_k + c_i)^3} \right],
\end{equation}
Applying (\ref{bounda})-(\ref{boundABC}), (\ref{eqn_ek}) and the conditions (\ref{regular1})-(\ref{regular2}) to (\ref{partialfw_edge}) and (\ref{partialfm_edge}), we get that
\begin{equation}\label{eqn_coeff1}
1 \le \left|\partial _{\sqrt{w}} f(\sqrt{e_k}, m_k)\right| \le C_1, \ \epsilon\le \left|\partial_{m}^2f(\sqrt{e_k},m_k)\right| \le C_2
\end{equation}
for some $C_1,C_2>0$. Similarly, if $|w-e_k|\le \tau'$  and $|m_c(w) - m_k| \le \tau'$ for some sufficiently small $\tau'$, using the condition (\ref{regular1}) we can get that
\begin{equation}\label{eqn_coeff1}
\max\left\{\left|\partial_{m}^3f(\sqrt{w},m_c(w))\right|,\left|\partial_{\sqrt{w}}^2f(\sqrt{w},m_c(w))\right|,\left|\partial_{m}\partial_{\sqrt{w}}f(\sqrt{w},m_c(w))\right| \right\} \le C_3.
\end{equation}
Plug them into equation (\ref{expand_edge}), for $|w-e_k|\le \tau'$  and $|m_c(w) - m_k| \le \tau'$, we get $|m_c(w)-m_k| \sim |\sqrt{w}-\sqrt{e_k}|^{1/2}$ and
\begin{equation}
 - \partial_{\sqrt{w}}f(\sqrt{e_k},m_k)(\sqrt{w}-\sqrt{e_k})+O( |\sqrt{w}-\sqrt{e_k}|^{3/2})=\frac{1}{2}\partial_{m}^2f(\sqrt{e_k},m_k)(m_c(w)-m_k)^2 . \label{exp_edge}
 \end{equation}
By (\ref{eqn_ek}), we immediately get that $|\sqrt{w}-\sqrt{e_k}|\sim |w-e_k|$ and $|m_c(w) - m_k| \sim |m_{1c}(w)-m_{1c}(e_k)|$, which proves the first part of the lemma.
By (\ref{exp_edge}), if $w$ is real and $|w-e_k|\le \tau'$, we have that
\begin{equation}
m_c(w)-m_k= \left[\frac{- 2\partial_{\sqrt{w}}f(\sqrt{e_k},m_k)}{\partial_{m}^2f(\sqrt{e_k},m_k)} +O( |\sqrt{w}-\sqrt{e_k}|^{1/2})\right]^{1/2}\left(\sqrt{w}-\sqrt{e_k}\right)^{1/2}.\label{exp_edge2}
 \end{equation}
Thus on a sufficiently small interval $U=[e_k-\delta,e_k+\delta]$, $m_c(w)$ has positive imaginary part for $w$ on one side of $e_k$ and $m_c(w)$ is real for $w$ on the other side. Hence $U$ does not contain another edge. This shows that $\min_{l\ne k} |e_{l}-e_k| \ge \delta.$ \end{proof}

\begin{proof}[Proof of Proposition \ref{prop_rho1c}]

The properties of $\rho_{1c}$ have been proved in Lemmas \ref{lemm_rho}, \ref{lemm_edge0} and \ref{lemm_edge_reg}, and included in the Definition \ref{def_regular}.
Since $\text{supp}\, \rho_{2c} =\text{supp}\, \rho_{1c}$ by the discussions after Lemma \ref{lemm_density}, we immediately get property (i) for $\rho_{2c}$. The conclusion $\rho_{2c}$ being a probability measure is due to the definition of $m_2$ in (\ref{def_M}) and the fact that $m_{2c}$ is the almost sure limit of $m_2$.

The properties (ii) and (iv) for $\rho_{2c}$ can be easily obtained by plugging $m_{1c}$ into (\ref{eq_self1}). To prove the property (iii) for $\rho_{2c}$, we need to know the behavior of $\text{Im}\, m_{2c}(w)$ when $w \to e_j$ along the real line. By (\ref{eq_self1}), it suffices to prove that if $|x - e_j |\le \tau'$ for some small enough $\tau'>0$, then
$$\left| - w(1+m_{1c})^2 + |z|^2\right|=\left|m_c^2 - |z|^2 \right|\ge \epsilon$$
for some constant $\epsilon>0$. Suppose that $\left|m_{c}^2(w)-|z|^2\right|=o(1)$. Plugging $m_c$ into $\partial_m f(\sqrt{w}, m_c)$ in (\ref{eqn_partialm}), and using condition (\ref{regular1}) and Lemma \ref{lemm_edge_reg}, we get that
\begin{equation}
\partial_mf(\sqrt{w},m_c(w))= -1 + O(|m_c^2-|z|^2|).
\end{equation}
Again using condition (\ref{regular1}) and Lemma \ref{lemm_edge_reg}, we can bound $\partial_{\sqrt{w}} \partial_m f(\sqrt{w},m_c(w)) $ and $\partial^2_m f(\sqrt{w},m_c(w)) $ for $w$ near $e_j$.
Thus we shall have that
\begin{equation}
0 = \partial_m f(\sqrt{e_j},m_c(e_j)) = \partial_mf(\sqrt{w},m_c(w)) + O(|w-e_j|^{1/2}) = -1 + O(|m_c^2-|z|^2|+|w-e_j|^{1/2}).\label{temp3}
\end{equation}
This gives a contradiction. Thus we must have a lower bound for $\left|m_{c}^2-|z|^2\right|$. \end{proof}

\noindent{\it{Remark:}} Here we add a small remark on Example \ref{example2}. Given the assumptions in Example \ref{example2}, it is easy to see that $f$ can only take critical values on intervals $I_{-n}$, $I_0$, $I_n$ and $I_{2n}$, since $\max\{|a_i -a_{i-1}|,|b_i -b_{i-1}|,|c_i -c_{i-1}|\}\to 0$ in this case. Thus the number of connected components of $\text{supp}\, \rho_{1c}$ is independent of $n$, and all the edges and the bulk components are regular as in Example \ref{example1}.

\subsection{Proof of Lemmas \ref{lemm_m1_4case} and \ref{cor_basicestimate}}\label{subsection_append2}


We first prove Lemma \ref{lemm_m1_4case}. We consider the five cases separately.

\vspace{10pt}

{\noindent{\it{Case 1}}:}
For $w=E+i\eta \in \mathbf D_{k}^b(\zeta,\tau',N)$, we have
\begin{equation}\label{eqn_app_Imm1}
m_{1c}(w)=\int_{\mathbb R} \frac{\rho_{1c}(x)}{x-(E+i\eta)}dx, \ \ \text{Im}\, m_{1c}(w)=\int_{\mathbb R} \frac{\rho_{1c}(x,z)\eta}{(x-E)^2+\eta^2}dx.
\end{equation}
By the regularity condition of Definition \ref{def_regular} (ii), 
we get immediately $\Im \, m_{1c} \sim 1$. Since $\Im \, m_{1c} \le |1+m_{1c}| \le C$ by Proposition \ref{prop_roughbound}, we get $|1+m_{1c}| \sim 1$. Notice $wm_{1c}$ can be expressed as
$$wm_{1c}(w)=\int_{\mathbb R} \frac{w\rho_{1c}(x,z)}{x-w}dx=-\int_{\mathbb R} \rho_{1c}(x,z) dx + \int_{\mathbb R} \frac{x\rho_{c}(x,z)}{x-w}dx.$$
By the same argument as above and using the fact that $x\ge \tau'$ for $x \in [e_{2k}+\tau',e_{2k-1}-\tau']$, we get $$\text{Im}(wm_{1c}) = \text{Im} \int_{\mathbb R} \frac{x\rho_{1c}(x,z)}{x-w}dx \sim 1.$$
Since the imaginary parts of $-w$ and ${|z|^2}/{(1+m_{1c})}$ are both negative, we get
\begin{equation}\label{eqn_app_bulk1}
\text{Im} \left[-w(1+m_{1c}) + \frac{|z|^2}{1+m_{1c}}\right] \le -\text{Im}(wm_{1c}).
\end{equation}
Using the bounds for $m_{1c}$ and $\Im\, m_{1c}$ proved above, it is easy to see
\begin{equation}\label{eqn_app_bulk2}
\left|-w(1+m_{1c}) + \frac{|z|^2}{1+m_{1c}}\right| = O(1).
\end{equation}
Equations (\ref{eqn_app_bulk1}) and (\ref{eqn_app_bulk2}) together give that $\text{Im}\, m_{2c} \sim 1$ and $ | m_{2c} |\sim 1$. Similarly, we can also prove that
$$ wm_{2c}=\left[{- (1 + m_{1c})   + \frac{\left| z \right|^2 }{w(1+m_{1c})} }\right]^{-1} \in \mathbb C_+$$
and $\text{Im}(wm_{2c}) \sim 1$. Now (\ref{estimate2_bulk}) follows from
$$\text{Im}\left( {w + s_i wm_{2c} } - \frac{\left| z \right|^2}{1+m_{1c}}\right) \ge s_i \text{Im}(wm_{2c}) .$$


{\noindent{\it{Case 2}}:} For $w=E+i\eta \in \mathbf D^o(\zeta,\tau',N)$, using (\ref{eqn_app_Imm1}) and $\text{dist}(E,\text{supp} \, \rho_{1,2c})\ge \tau'$, we immediately get $\Im \, m_{1,2c} \sim \eta$. Now we prove the other estimates.

We first prove (\ref{estimate2_bulk}). If $\eta\sim 1$, the proof is exactly the same as in Case 1. Hence we assume $\eta \le c'$, where $c'\equiv c'(\tau,\tau')>0$ is sufficiently small. We separate it into two cases.

(i) Suppose $E\sim 1$. We shall prove that
\begin{equation}
\min_i \{|m_c(w) - a_i(w)|,|m_c(w) - b_i(w)|, |m_c(w)+c_i(w)|\}\ge \epsilon', \label{outer_regularity}
\end{equation}
for some constant $\epsilon'$. This leads immediately to (\ref{estimate2_bulk}) since
\begin{equation}
\left|{w\left( {1 + s_i \frac{1 + m_{1c}}{- w (1 + m_{1c})^2   + {\left| z \right|^2 } }} \right)(1+m_{1c}) - \left| z \right|^2 }\right| = \left|\frac{\sqrt{w}(m_c - a_i)(m_c - b_i)(m_c+c_i)}{- m_{c}^2   + {\left| z \right|^2 } }\right| .
\end{equation}
For $p_i = \sqrt{E}m^3 - (s_i + |z|^2)m^2 -\sqrt{E}|z|^2 m + |z|^4$, it is not hard to prove that its roots $a_i(E)$, $b_i(E)$ and $-c_i(E)$ decrease as $E$ increase. Since $E\notin \text{supp}\, \rho_{1c}$, we have $m_{1c}(E)\in \mathbb R$ and
$$\frac{dm_{1c}(E)}{dE}=\int_{\mathbb R} \frac{\rho_{1c}(x,z)}{(x-E)^2}dx\ge 0.$$
So $m_{1c}(E)$ (and hence $m_c(E)$) increases as $E$ increases. If $e_k$ is the smallest edge that is bigger than $E$, then for $a_i(E)$ bigger than $m_c(E)$, we have that
\begin{equation}\label{eqn_app_right}
a_i(E)-m_c(E)\ge a_i(e_k)-m_c(e_k) + \epsilon(\tau') \ge \epsilon(\tau'),
\end{equation}
by using $|E-e_k| \ge \tau'$ (see (\ref{def_domain_Do})). On the other hand, If $e_{k-1}$ is the largest edge value that is smaller than $E$, then for $a_i(E)$ smaller than $m_c(E)$, we have that
\begin{equation}\label{eqn_app_left}
m_c(E)-a_i(E) \ge m_c(e_{k-1})-a_i(e_{k-1}) + \epsilon(\tau') \ge \epsilon(\tau').
\end{equation}
Applying the same arguments to $b_i(E)$ and $-c_i(E)$, we get
 \begin{equation}\label{eqn_app_Ereg}
\min_i \{|m_c(E) - a_i(E)|,|m_c(E) - b_i(E)|, |m_c(E)+c_i(E)|\}\ge \epsilon
\end{equation}
for $E\in (e_{2k-1},e_{2k})$ for some $k$. Now we are only left with the case $E<e_{2L}$, the rightmost edge, when $|z|^2\ge 1+\tau$. In this case, we have seen that $0<m_c(E)<b_i(E)$ for all $i$ in the proof of Lemma \ref{lemm_edge0}. Thus we can use (\ref{eqn_app_right}) to get lower bounds for $|m_c(E)-a_i(E)|$ and $|m_c(E)-b_i(E)|$. Since $c_i(E)\sim 1$ in this case (e.g. by (\ref{boundc}) and using $E, |z|\sim 1 $), $|m_c(E)+c_i(E)|\ge \epsilon$ is trivial. Again we get the estimate (\ref{eqn_app_Ereg}).

Then we consider $w=E+i\eta$ with $\eta \le c'$. First it is easy to check that $a_i(E+i\eta)$, $b_i(E+i\eta)$ and $c_i(E+i\eta)$ are continuous in $\eta$. On the other hand for $m_{c}(E+i\eta)$, we have
\begin{equation}\label{derv_m1_bound}
\partial_w m_{1c}(w)=\int_{\mathbb R} \frac{\rho_{1c}(x,z)}{(x-w)^2}dx \le C
\end{equation}
by the condition $\text{dist}(E,\text{supp} \, \rho_{1c})\ge \tau'$. Thus we immediately get $|m_{c}(E+i\eta)-m_{c}(E)|=O(\eta)$. Hence as long as $c'$ is small enough, (\ref{outer_regularity}) is true, which further gives (\ref{estimate2_bulk}).

(ii) Suppose $w=E+i\eta \to 0$, in which case we must have $|z|^2\ge 1+\tau$ and $E<e_{2L}$. Using $|m_{1,2c}(w)|\sim 1$ by Proposition \ref{prop_roughbound}, we can calculate directly that
$$\left|{w\left( 1 + s_i m_{2c} \right)(1+m_{1c}) - \left| z \right|^2 }\right| =  \left||z|^2 + O(w)\right| \ge c.$$This concludes the proof of (\ref{estimate2_bulk}).

Then we show that $|1+m_{1c}|\sim 1$ for $w \in \mathbf D^o$ and $\eta \le c'$. We again divide it into two cases.
First suppose $|w|\sim 1$. If $|m_c|$ can be arbitrarily small, then by (\ref{estimate2_bulk}) we get that
\begin{align*}
f(\sqrt{w},m_c) = -\sqrt{w} + O(m_c) \ne 0,
\end{align*}
which gives a contradiction.
Then suppose $w=E+i\eta \to 0$ when $|z|^2\ge 1+\tau$ and $E<e_{2L}$. We have seen in the proof of Lemma \ref{lemm_edge0} that $$m_c(E) = \sqrt{E}\frac{|z|^2}{|z|^2-1} + o\left(\sqrt{E}\right) \Rightarrow 1+m_{1c}(E) = \frac{|z|^2}{|z|^2-1}+o(1).$$ Then using (\ref{derv_m1_bound}), we get $$\left|1+m_{1c}(E+i\eta)\right|=\left| \frac{|z|^2}{|z|^2-1}+o(1)+O(\eta)\right| \sim 1.$$
Finally we have $|m_{2c}|\sim 1$ for $w \in \mathbf D^o$ and $\eta \le c'$ by Proposition \ref{prop_roughbound}. 

\vspace{10pt}

{\noindent{\it{Case 3}}:} For regular edge $e_k\ne 0$, we always have $e_k\ge \epsilon$ for some $\epsilon>0$ by Lemma \ref{lemm_edge0}. Thus we always have $|w|\sim 1$ for $w=E+i\eta \in \mathbf D_k^e(\zeta,\tau',N)$ as long as $\tau'$ is sufficiently small. If $\eta\sim 1$, then $\sqrt{\kappa + \eta} \sim \eta/\sqrt{\kappa+\eta} \sim 1$ and the proof is exactly the same as in Case 1. Now we pick $\tau'$ small and consider the case $\eta \le \tau'$. By the regularity assumption (\ref{regular1}) and Lemma \ref{lemm_edge_reg}, we have
\begin{equation}
\min_{1\le i \le n} \{|m_c(w) - a_i(w)|,|m_c(w) - b_i(w)|,|m_c(w)+c_i(w)| \}\ge \epsilon/2
\end{equation}
uniformly in $w \in \{ w\in \mathbf D_k^e(\zeta,\tau',N): \kappa(w)+\eta(w) \le 2\tau' \}$, provided $\tau'$ is sufficiently small. The above bound implies (\ref{estimate2_bulk}). If $m_c(w)\to 0$, then using (\ref{estimate2_bulk}) we get from $f(\sqrt{w},m_c)=0$ that $-\sqrt{w}+O(m_c)=0,$ which gives a contradiction. Thus we must have $|1+m_{1c}|\sim |m_c| \sim 1$. To show $|m_{2c}|\sim 1$, we can use Proposition \ref{prop_roughbound}. 

We still need to prove the estimates for $\text{Im}\,m_{1,2c}$ when $\eta \le \tau'$. Recall the expansion (\ref{exp_edge}) around $e_k$ and equation (\ref{exp_edge2}). Notice both $\partial_{\sqrt{w}}f(\sqrt{e_k},m_k)$ and $\partial_{m}^2f(\sqrt{e_k},m_k)$ are real (as $e_k$ and $m_k$ are real). Suppose $k$ is odd, then $\text{Im}\, m_c(E)=0$ for $E\searrow e_k$ (i.e. $E\notin \text{supp} \rho_c$) and $\text{Im}\, m_c(E)>0$ for $E\nearrow e_k$ (i.e. $E\in \text{supp} \rho_c$). Thus (\ref{exp_edge2}) gives
$$m_c(w)-m_k= C_k(w) (w-e_k)^{1/2} + D_k(w),$$
with $C_k>0$, $C_k \sim 1$, $|D_k|=O(|w-e_k|)$ and $\Im\, D_k \sim \eta$. Then for $E\ge e_k$, we have
$$\text{Im}\, m_c(E+i\eta) \sim \Im(\kappa + i\eta)^{1/2} + O(\eta) \sim \frac{\eta}{\sqrt{\kappa+\eta}},$$
and for $E\le e_k$, we have
$$\text{Im}\, m_c(E+i\eta) \sim \Im(-\kappa + i\eta)^{1/2}+ O(\eta) \sim \sqrt{\kappa+\eta}.$$
If $k$ is even, the proof is the same except that in this case
$$m_c(w)-m_k= C_k(w) (e_k-w)^{1/2} + D_k(w).$$
For $m_{1c}(w)$ and $m_{2c}(w)$, we get the conclusion by noticing ${w}\approx {e_k}$ and
$$\text{Im}\, m_{1c}= \text{Im}\left(w^{-1/2} m_{c}\right) \sim \text{Im}\, m_c(w) , \ \ \text{Im}\, m_{2c}=\text{Im}\left[\frac{m_c}{\sqrt{w}(-m_c^2 + |z|^2)}\right] \sim \text{Im}\, m_c(w) .$$


{\noindent{\it{Case 4}}:} Again if $\eta \sim 1$, the proof is the same as in Case 1. If $|w|\le 2\tau'$ for small enough $\tau'$, in the proof of Lemma \ref{lemm_edge0}, we have seen that $m_{c} = i \sqrt{t} + O(\sqrt{w})$, which gives the first equation in (\ref{estimate12_0}). Plugging it into (\ref{eq_self1}), we get the second equation in (\ref{estimate12_0}). Taking the imaginary part, we obtain (\ref{estimate13_0}). Finally using (\ref{estimate12_0}), we get (\ref{estimate2_bulk}) easily.

\vspace{10pt}

{\noindent{\it{Case 5}}:}
For $w=E+i\eta \in \mathbf D_L(\zeta,N)$, the bounds for $m_{1,2}$ and $\Im\, m_{1,2}$ in (\ref{estimate1L}) follows from (\ref{eqn_app_Imm1}) directly. 

Finally we prove Lemma \ref{cor_basicestimate}. The estimates (\ref{estimate_Piw12}) and (\ref{estimate_PiwL}) follow immediately from (\ref{def_pi}), (\ref{estimate2_bulk}) and (\ref{estimate2_large}). For (\ref{estimate_PiImw}), we can write
\begin{align*}
\Pi_{\mathbf v\mathbf v} = \left\langle \mathbf v, \left( {\begin{array}{*{20}c}
   { U} & {0}  \\
   {0} & {U}  \\
   \end{array}} \right) \Pi_d\left( {\begin{array}{*{20}c}
   { U^\dag} & {0}  \\
   {0} & {U^\dag}  \\
   \end{array}} \right) \mathbf v \right\rangle = \left(\Pi_d \right)_{ {\mathbf u} {\mathbf u}} = \sum_{i=1}^N \left\langle u_{[i]}, \pi_{[i]c} u_{[i]}\right\rangle,
\end{align*}
where
$$ \mathbf u:= \left( {\begin{array}{*{20}c}
   {U^\dag} & {0}  \\
   {0} & {U^\dag}  \\
   \end{array}} \right)  \mathbf v, \ \ u_{[i]}:=  \left( {\begin{array}{*{20}c}
   {u_i}   \\
   {u_{\bar i}} \\
   \end{array}} \right).$$
To control $\Im\, \Pi_{\mathbf v\mathbf v}$, it is enough to bound $\left\langle u_{[i]}, \pi_{[i]c} u_{[i]}\right\rangle$ for each $i$. 

We first consider Cases 1-4 of Lemma \ref{lemm_m1_4case}. By the definition of $\pi_{[i]c}$ in (\ref{def_pi}), we get
\begin{align*}
\Im\, \pi_{ii,c} & = |u_i|^2\Im\left[ - w(1+ |d_i|^2m_{2c}) +\frac{|z|^2}{1+m_{1c}}\right]^{-1} \le \frac{C}{|w|} \Im\left[ w(1+ |d_i|^2m_{2c}) - \frac{|z|^2}{1+m_{1c}} \right]\\
& = \frac{C}{|w|} \left[(1+ |d_i|^2\Re\, m_{2c})\Im\, w  +|d_i|^2(\Re\, w) \Im\, m_{2c}+ \frac{|z|^2}{|1+m_{1c}|^2}\Im\, m_{1c}\right],
\end{align*}
where in the second step we use (\ref{estimate2_bulk}) and $|1+m_{1c}|\sim |w|^{-1/2}$. In the first three cases of Lemma \ref{lemm_m1_4case}, we have $|w|\sim 1$ and $\Im\, w = O(\Im\, m_{1c})$, which give that $\Im\, \pi_{ii,c} \le C \Im{(m_{1c}+m_{2c})}$. In case 4 of Lemma \ref{lemm_m1_4case}, we use $|\Im\, w| + |\Re\, w| + |1+m_{1c}|^{-2} = O(|w|)$ and $\Im\, m_{1,2c}\sim |w|^{-1/2}$ to get that $\Im\, \pi_{ii,c}  \le C \Im{(m_{1c}+m_{2c})}$. Similarly we have the bound $\Im\, \pi_{\bar i\bar i,c} \le C\Im{(m_{1c}+m_{2c})}$.
Finally we can estimate the following term using similar methods,
\begin{align*}
\Im \left(\bar u_{\bar i} u_i \pi_{\bar i i,c}+\bar u_{i} u_{\bar i} \pi_{ i \bar i,c} \right)& =2\Re \left(\bar u_i u_{\bar i} z\right)\Im\left\{w^{-1/2}\left[ w(1+ |d_i|^2m_{2c})(1+m_{1c}) - |z|^2\right]^{-1} \right\}\\
&\le C\Re \left(\bar u_i u_{\bar i} z\right)\Im(m_{1c}+m_{2c}) \le C\left(|u_i|^2 + |u_{\bar i}|^2\right) \Im(m_{1c}+m_{2c}).
\end{align*}
Combining the above estimates we get $ \Im\left\langle u_{[i]}, \pi_{[i]c} u_{[i]}\right\rangle \le C|u_{[i]}|^2\Im(m_{1c}+m_{2c})$, which implies (\ref{estimate_PiImw}).
For the Case 5 of Lemma \ref{lemm_m1_4case}, we use (\ref{estimate1L}) and (\ref{estimate_PiwL}) to get
$$\Im\left\langle u_{[i]}, \pi_{[i]c} u_{[i]}\right\rangle \le  |u_{[i]}|^2\|\pi_{[i]c}\|\le C|u_{[i]}|^2\Im(m_{1c}+m_{2c}).$$

\subsection{Proof of Lemma \ref{lemm_stability} and Lemma \ref{lemm_density}}\label{subsection_append3}


We first prove Lemma \ref{lemm_stability}. During the proof, we also use the following equivalent definition of the stability expressed in terms of $m=\sqrt{w}(1+m_1)$, $u=\sqrt{w}(1+u_1)$ and $f(\sqrt{w},m)$. Suppose the assumptions in Definition \ref{def_stability} holds. Let $w\in \mathbf D$ and suppose that for all $w'\in L(w)$ we have $\left|f(\sqrt{w}, u)\right| \le |w|^{1/2} \delta(w).$ Then
\begin{equation}
\left|u(w)-m_{c}(w)\right|\le \frac{C |w|^{1/2}\delta}{\sqrt{\kappa+\eta+\delta}}.\label{Stability2}
\end{equation}

{\noindent{\it{Case 1}}:}
We take over the notations in Definition \ref{def_stability} and abbreviate $R:=f(\sqrt{w}, u)$, so that $|R|\le |w|^{1/2}\delta$. Then we write the equation $f(\sqrt{w},u)-f(\sqrt{w},m_c)=R$ as
\begin{equation}
\alpha(u-m_c)^2+\beta(u-m_c)=R, \label{stab_bulk_fixed}
\end{equation}
where using (\ref{PFD_f}), $\alpha$ and $\beta$ can be expressed as
\begin{equation}
\alpha :=\frac{1}{N}\sum_{i=1}^{n}l_i s_i \left[ \frac{A_i}{(u-a_i)(m_c - a_i)^2} + \frac{B_i}{(u-b_i)(m_c - b_i)^2} + \frac{C_i}{(u+c_i)(m_c + c_i)^2} \right], \label{def_alpha}
\end{equation}
and
\begin{equation}
\beta :=1 - \frac{1}{N}\sum_{i=1}^{n}l_i s_i \left[ \frac{A_i}{(m_c - a_i)^2} + \frac{B_i}{(m_c - b_i)^2} + \frac{C_i}{(m_c + c_i)^2} \right] = \partial_m f(\sqrt{w},m_c).\label{def_beta}
\end{equation}
We shall prove that
\begin{equation}
 |\alpha|+|\partial_u \alpha| \le C, \ |\beta| \sim 1,
\label{estimate_alphabeta1}
\end{equation}
for $w\in \mathbf D_k^b$ and $u$ satisfying $|u-m_c|\le (\log N)^{-1/3}$. If $|u-m_c|\le (\log N)^{-1/3}$, we also have $\text{Im}\, u \sim 1$. By (\ref{estimate2_bulk}), 
\begin{equation}
\min_i \{|m_c-a_i|,|m_c-b_i|,|m_c+c_i|\} \ge \epsilon \label{nonsingular_m}
\end{equation}
for some $\epsilon>0$. Replacing the $m_c$ in (\ref{estimate2_bulk}) with $u$, we also get that
\begin{equation}
\min_i \{|u-a_i|,|u-b_i|,|u+c_i|\} \ge \epsilon'
\label{nonsingular_u}
\end{equation}
for some $\epsilon'>0$. Using (\ref{nonsingular_m}) and (\ref{nonsingular_u}), we get immediately that $|\alpha|+|\partial_u \alpha|+|\beta|\le C$. What remains is the proof of the lower bound $|\beta| \ge c$. If $\text{Im}\, w \ge \epsilon$ for some constant $\epsilon>0$, the lower bound follows from Lemma \ref{lowerbound_beta1} below. If $\text{Im}\, w \le \epsilon$ for a sufficiently small $\epsilon$, the lower bound follows from Lemma \ref{lowerbound_beta2} below. Now given the bound (\ref{estimate_alphabeta1}), it is easy to prove (\ref{Stability2}) with a fixed point argument. This proves the stability of (\ref{stab_Dm1})


\begin{lem}
Suppose that ${\rm{Im}}\, w \sim 1$ and $|m_c|\sim {\rm{Im}}\, m_c \sim 1$. Then $\left|\partial_m f(\sqrt{w},m_c) \right| \ge c $ for some constant $c >0$.
\label{lowerbound_beta1}
\end{lem}
\begin{proof}
Using (\ref{stj_rho1}), $m_c=\sqrt{w}(1+m_{1c})$ and the conditions ${\rm{Im}}\, w \sim 1$, ${\rm{Im}}\, m_c \sim 1$, we can get that
\begin{equation}
\left|\frac{\partial_{\sqrt{w}} f (\sqrt{w},m_c)}{\partial_m f (\sqrt{w},m_c)}\right|=\left|\frac{\partial m_c}{\partial \sqrt{w}}\right| \le C \Rightarrow \left|\partial_{\sqrt{w}} f (\sqrt{w},m_c)\right| \le C \left|\partial_{m} f (\sqrt{w},m_c) \right|, \label{compareab}
\end{equation}
for some constant $C>0$. 
Now we assume that $\left|\partial_m f(\sqrt{w},m_c) \right|$ can be arbitrarily small. Then $\left|\partial_{\sqrt{w}} f(\sqrt{w},m_c) \right|$ can also be arbitrarily small. Denote $a:=\partial_m f(\sqrt{w},m_c)$ and $b:=\partial_{\sqrt{w}} f(\sqrt{w},m_c)$. Using (\ref{eqn_partialm}) and (\ref{eqn_partialw}), we get that
\begin{equation}
a = \frac{\sqrt{w}}{m_c}-\frac{ m_c}{N}\sum_{i=1}^n l_i s_i \frac{ \sqrt{w}\left( m_c^2 - |z|^2\right)^2 +2s_i |z|^2 m_c}{\left[- (s_i + |z|^2) m_c^2+|z|^4+ \sqrt{w}\left( m_c^3  - |z|^2 m_c\right)\right]^2} \label{am}
\end{equation}
and
\begin{equation}
b = -1 -\frac{ m_c^2}{N}\sum\limits_{i = 1}^n l_i s _i \frac{\left( m_c^2  - |z|^2 \right)^2}{{\left[- (s_i + |z|^2) m_c^2+|z|^4+ \sqrt{w}\left( m_c^3  - |z|^2 m\right)\right]^2 }}.\label{bm}
\end{equation}
Using (\ref{am}) and (\ref{bm}), we can get that
\begin{equation}
\frac{(\sqrt{w}m_c-|z|^2)|z|^2}{m_c}b - \frac{1}{2}(m_c^2-|z|^2)(m_ca-\sqrt{w}b) = \frac{(|z|^2-\sqrt{w}m_c)(m_c^2+|z|^2)}{m_c},\label{contradiction}
\end{equation}
where we use the equation $f(\sqrt{w},m_c)=0$ in the derivation. By our assumption, the left-hand side of (\ref{contradiction}) can be arbitrarily small. For the right-hand side of (\ref{contradiction}), we have $\left|m_c\right| \sim 1$ and $|\sqrt{w}m_c-|z|^2|\sim 1$ (because $\text{Im}\, (\sqrt{w}m_c) = \text{Im}\, (w+wm_{1c}) \sim 1$). Thus if $|m_c-i|z||\ge c'$ for some constant $c' >0$, we have $|m^2+|z|^2|\sim 1$, and
$$\left|\frac{(\sqrt{w}m_c-|z|^2)|z|^2}{m_c}b - \frac{1}{2}(m_c^2-|z|^2)(m_ca-\sqrt{w}b)\right| \sim 1,$$
which gives a contradiction. Thus we must have a lower bound $\left|\partial_m f(\sqrt{w},m_c) \right| \ge c$ if $|m-i|z|| \ge c'$.

We still need to deal with the case where $|m_c-i|z|| \le c'$ for some sufficiently small $c'$. Notice $|z|\sim 1$ in this case. Then we have
\begin{align}
\frac{\partial f}{\partial \sqrt{w}} (\sqrt{w},i|z|) = -1 + \frac{ |z|^2 }{N}\sum\limits_{i = 1}^n l_i s _i \frac{4 |z|^4 }{{\left[(s_i + |z|^2) |z|^2+|z|^4 - 2 i\sqrt{w}|z|^3\right]^2 }}. \label{frac_Li}
\end{align}
Denote $L_i:=(s_i + |z|^2) |z|^2+|z|^4 - 2 i\sqrt{w}|z|^3$. Since $i\sqrt{w}=i(x+iy)=ix-y$ with $x,y>0$ and $x,y\sim 1$, we have $\text{Re}\, L_i > 0$, $\text{Im}\, L_i < 0 $ and $|\text{Re}\, L_i|, |\text{Im}\, L_i|\sim 1$. Furthermore, $\text{Im}\, L_i^2 < 0$ and $|\text{Im}\, L_i^2|\sim 1$. Thus each fraction $4|z|^4/L_i^2$ in (\ref{frac_Li}) has positive imaginary part and all the imaginary have order 1. Therefore
$$\left|\frac{\partial f}{\partial \sqrt{w}} (\sqrt{w},i|z|)\right| \ge \text{Im}\left[\frac{\partial f}{\partial \sqrt{w}} (\sqrt{w},i|z|) \right]\sim 1.$$
Then by (\ref{compareab}), we get that
$\left|\partial_m f(\sqrt{w},i|z|) \right| \ge c$
for some $c>0$. Using (\ref{estimate2_bulk}), it is easy to see that
$$\partial_m f(\sqrt{w},m_c)=\partial_m f(\sqrt{w},i|z|)+O(\left|m_c-i|z|\right|). $$
Thus in the case $|m_c -i|z|| \to 0$, we still can find $c>0$ such that
$\left|\partial_m f(\sqrt{w},m_c) \right| \ge c$.
\end{proof}

\begin{lem}\label{lowerbound_beta2}
Suppose that $w \in \mathbf D_k^b$ and ${\rm{Im}}\, w \le \epsilon$. Then for sufficiently small $\epsilon>0$, we have $\left|\partial_m f(\sqrt{w},m_c) \right| \sim 1$.
\end{lem}
\begin{proof}
By (\ref{estimate1_bulk}) and (\ref{estimate2_bulk}), if $|w|\sim 1$ and $\text{Im}\, m\sim 1$, we have $\partial_{\sqrt{w}} \partial_m f(w,m_c) =O(1)$ and $\partial^2_m f(w,m_c) =O(1)$. Denote $w=E+i\eta$.
Taking the imaginary part of the following equation
\begin{equation}
0=f(\sqrt{E},m_c(E)) = -\sqrt{E} + m_c + E^{-1/2} + \frac{1}{N}\sum_{i=1}^{n}l_i s_i \left( \frac{A_i}{m_c - a_i} + \frac{B_i}{m_c - b_i} + \frac{C_i}{m_c + c_i} \right),
\end{equation}
and noticing that $A_i,B_i,C_i$ and $a_i,b_i,c_i$ are all positive real numbers for real $E$, we get
\begin{equation}
\frac{1}{N}\sum_{i=1}^{n}l_i s_i \left( \frac{A_i}{|m_c - a_i|^2} + \frac{B_i}{|m_c - b_i|^2} + \frac{C_i}{|m_c + c_i|^2} \right)= 1.
\end{equation}
Using the above equation, we get
\begin{align}
& \partial_m f(\sqrt{E},m_c(E)) = 1 - \frac{1}{N}\sum_{i=1}^{n}l_i s_i \left[ \frac{A_i}{(m_c - a_i)^2} + \frac{B_i}{(m_c - b_i)^2} + \frac{C_i}{(m_c + c_i)^2} \right] \nonumber\\
& = \frac{1}{N}\sum_{i=1}^{n}l_i s_i \left[ \frac{A_i}{|m_c - a_i|^2} -  \frac{A_i}{(m_c - a_i)^2} + \frac{B_i}{|m_c - b_i|^2} - \frac{B_i}{(m_c - b_i)^2} + \frac{C_i}{|m_c + c_i|^2}- \frac{C_i}{(m_c + c_i)^2} \right]. \label{eqn_app_Ref}
\end{align}
We look at, for example, the term
$$\frac{A_i}{|m_c - a_i|^2} -  \frac{A_i}{(m_c - a_i)^2} = \frac{A_i}{|m_c - a_i|^2}(1-e^{-2i\theta_i}),$$
where $m_c - a_i :=|m_c - a_i|e^{i\theta_i}$. Using $\Im\, m_c \sim 1$, it is easy to see that $\text{Re} (1-e^{-2i\theta_i})\ge c'$ for some constant $c'>0$. Applying the same estimates to the $B,C$ terms in (\ref{eqn_app_Ref}), we get
\begin{align}
\left|\partial_m f(\sqrt{E},m_c(E))\right| \ge \text{Re} \left[\partial_m f(\sqrt{E},m_c(E))\right] \ge c \label{lower_fE}
\end{align}
for some constant $c>0$.

Now for $w=E+i\eta$ with $\eta \le \epsilon$, we can expand $\partial_m f(\sqrt{w},m_c(w))$ around $\partial_m f(\sqrt{E},m_c(E))$,
$$ \partial_m f(\sqrt{w},m_c(w)) =  \partial_m f(E,m_c(E)) + O(\eta),$$
where we use (\ref{estimate2_bulk}). Combing with (\ref{lower_fE}), we see that $\left| \partial_m f(w,m_c(w)) \right| \sim 1$ for small enough $\epsilon$.
\end{proof}


{\noindent{\it{Case 2}}:}
We mimic the argument in the proof of Case 1. We see that it suffices to prove $|\alpha|+|\partial_u \alpha| \le C$ and $|\beta|\sim 1$ for $\alpha$, $\beta$ defined in (\ref{def_alpha}) and (\ref{def_beta}) and $|u-m_c|\le (\log N)^{-1/3}$. Using (\ref{estimate2_bulk}), it is not hard to prove that $|\alpha|+|\partial_u \alpha|+|\beta|\le C$. What remains is the proof of the lower bound $|\beta|\ge c$. For the case $\text{Im}\, w\sim 1$, it follows from Lemma \ref{lowerbound_beta1}. If $w\to 0$ in the case $|z|^2\ge 1+\tau$, then $m_c(w)=O(\sqrt{w})\to 0$ by (\ref{estimate1_out}).
Thus we can use (\ref{eqn_partialm}) to get directly that
$$\partial_m f(\sqrt{w},m_c)=1- |z|^{-2} + O(\sqrt{w}) \ge c.$$
Finally, we are left with the case $E=\text{Re}\, w\sim 1$ and $\eta=\text{Im}\, w \to 0$. Using (\ref{stj_rho1}), $m_c=\sqrt{w}(1+m_{1c})$, $|w| \sim 1$ and $\text{dist}(E,\text{supp}\, \rho_{1c})\ge \tau'$, we can get that
$$\left|\frac{\partial_{\sqrt{w}} f (\sqrt{w},m_c)}{\partial_m f (\sqrt{w},m_c)}\right|=\left|\frac{\partial m_c}{\partial \sqrt{w}}\right| \le C $$
for some constant $C>0$. Thus it suffices to prove that $\left|\partial_{\sqrt{w}} f(\sqrt{w},m_c) \right|$ has a lower bound. Using (\ref{eqn_partialw}) and noticing that $m_c(E)\in \mathbb R$, we get
$$\partial _{\sqrt{w}} f(\sqrt{E}, m_c(E)) = -1 -\frac{ m_c^2}{N}\sum\limits_{i = 1}^n l_i s _i \frac{\left( m_c^2  - |z|^2 \right)^2}{{\left[- (s_i + |z|^2) m_c^2+|z|^4+ \sqrt{E}\left( m_c^3  - |z|^2 m_c\right)\right]^2 }} \le -1.$$
Expanding $\partial _{\sqrt{w}} f(\sqrt{w}, m_c(w))$ around $\partial _{\sqrt{w}} f(\sqrt{E}, m_c(E))$, using (\ref{estimate2_bulk}) and $|m_{c}(E+i\eta)-m_{c}(E)|\sim \eta$, we get for $\eta$ small
$$\left|\partial _{\sqrt{w}} f(\sqrt{w}, m_c)\right| \ge 1+O(\eta)\ge c.$$

\vspace{5pt}

{\noindent{\it{Case 3}}:}
The case $\text{Im}\, w\ge \tau'$ can be proved with the same method as in the proof of case 1. Hence we only consider the case $|w-e_k|\le 2\tau'$ in the following. Note that $|w| \sim 1$ in this case. Suppose 
\begin{equation}
|w-e_k|\le 2\tau',\ \ |u-m_c|\le (\log N)^{-1/3}. \label{induction_ass1}
\end{equation}
Then we claim that
\begin{equation}
|\alpha| \sim 1, \ \ |\beta| \sim \sqrt{\kappa+\eta} \label{induction_ass2}
\end{equation}
for small enough $\tau'$. Using (\ref{induction_ass1}), (\ref{estimate2_bulk}), (\ref{regular2}) and Lemma \ref{lemm_edge_reg}, we can get that
\begin{align*}
\alpha = \frac{1}{2}\partial_m^2 f(\sqrt{e_k},m_c(e_k)) + O(|w-e_k|^{1/2}+ (\log N)^{-1/3}) \sim 1.
\end{align*}
To prove the estimate for $\beta$, we use (\ref{equationEm2}), (\ref{estimate2_bulk}) and Lemma \ref{lemm_edge_reg}, to get
\begin{align}
\beta & = \int_{e_k}^w \frac{d} {dw'} \partial_m f(\sqrt{w'},m_c(w')) dw'  = \int_{e_k}^w \frac{ \partial_{\sqrt{w'}} \partial_m f(\sqrt{w'},m_c(w'))}{2\sqrt{w'}} dw' + \int_{e_k}^w \partial^2_m f(\sqrt{w'},m_c(w')) \frac{dm_c(w')}{dw'}dw' \nonumber\\
& = \int_{e_k}^w \frac{ \partial_{\sqrt{w}} \partial_m f(\sqrt{e_k},m_c(e_k)) + O(|w-e_k|^{1/2}) }{2\sqrt{w'}}dw' + \int_{m_c(e_k)}^{m_c(w)}\left[ \partial^2_m f(\sqrt{e_k},m_c(e_k)) + O(|w-e_k|^{1/2}) \right]dm \nonumber\\
& = \partial^2_m f(\sqrt{e_k},m_k)(m_c(w)-m_c(e_k)) + O(|w-e_k|).
\end{align}
Thus we conclude for small enough $\tau'$ that
$$|\beta| \sim |w-e_k|^{1/2} \sim \sqrt{\kappa+\eta}.$$

With the estimate (\ref{induction_ass2}), we now proceed exactly as in the proof of \cite[Lemma 4.5]{isotropic}, by solving the quadratic equation (\ref{stab_bulk_fixed}) for $u-m_c$ explicitly. We select the correct solution by a continuity argument using that (\ref{Stability2}) holds by assumption at $z+iN^{-10}$. The second assumption of (\ref{induction_ass1}) is obtained by continuity from the estimate on $|u-m_c|$ at the neighboring point $z+iN^{-10}$. We refer to \cite[Lemma 4.5]{isotropic} for the full details. This concludes the proof.

\vspace{10pt}

{\noindent{\it{Case 4}}:}
The case when $\Im \, w\ge \tau'$ can be proved using the same method as in the proof of Case 1. Now we are left with the case $|w|\le 2\tau'$ for some sufficiently small $\tau'$. First we assume $|z|\ge c >0$ for some small $c>0$. Then mimicking the argument in the proof of Case 1, we see that it suffices to prove $|\alpha|+|\partial_u \alpha|\le C$ and $|\beta|\sim 1$ when $|u-m_c|\le (\log N)^{-1/3}$. Using (\ref{estimate2_bulk}), it is not hard to prove that $|\alpha|+|\partial_u \alpha|+|\beta|\le C$. The lower bound $|\beta|\ge c$ can be obtained easily from (\ref{partialm_t}).

Then suppose $|z|^2< c$, but $|w|^{1/2}+|z|^2 \ge \epsilon$. According to (\ref{partialm_smallmw}) and using that $\left| i|z|^2 + \sqrt{wt_0} \right| \sim |w|^{1/2}+|z|^2$, we can verify that
\begin{align*}
\beta = \partial_{m} f(\sqrt{w}, m_c(w)) \sim |w|^{1/2}+|z|^2 \sim 1.
\end{align*}
It is also easy to verify that
$$\partial_m^2 f(\sqrt{w},m_c(w))=O(|w|^{1/2}+|z|^2),\ \ \partial_m^3 f(\sqrt{w},m_c(w))=O(|w|^{1/2}+|z|^2).$$
Hence if $|u-m_c|\le (\log N)^{-1/3}$, we have
\begin{align*}
\alpha & =\frac{1}{2}\partial_m^2 f(\sqrt{w},m_c(w)) + O\left(\partial_m^3 f(\sqrt{w},m_c(w)) (\log N)^{-1/3} \right)= O(|w|^{1/2}+|z|^2).
\end{align*}
With a fixed point argument, we get (\ref{Stability2}).

\vspace{10pt}

{\noindent{\it{Case 5}}:}
Again we following the arguments in the proof of Case 1. However, instead of $f(\sqrt{w},m)$, we shall study $\Upsilon(w,m_1) $ in (\ref{def_stabD}) directly. 
We take over the notations in Definition \ref{def_stability} and abbreviate $R:=\Upsilon(w, u_1)$, so that $|R|\le \delta$. Then we write the equation $\Upsilon(w,u_1)-\Upsilon(w,m_{1c})=R$ as
\begin{equation}
\alpha(u_1)(u_1-m_{1c})^2+\beta(u_1-m_{1c})=R, 
\end{equation}
where we use the same symbol as in (\ref{stab_bulk_fixed}) for notational convenience. As in Case 1, we have $\beta = \partial_{m_1} \Upsilon(w,m_{1c}),$
and we can evaluate that $|\alpha| + |\partial_{u_1}\alpha|\le C$ for $w\in \mathbf D_L$ and $u_1$ satisfying $|u_1-m_{1c}|\ll |m_{1c}|.$ Now to conclude (\ref{StabilityL}), it suffices to prove $|\beta| \sim 1$ for $w\in \mathbf D_L$. In fact using (\ref{def_stabD}), we obtain that
$$\beta = 1 +O\left( \eta^{-1}\right) \sim 1,$$
for $\eta \ge \zeta^{-1}$. This concludes the proof.

\begin{proof}[Proof of Lemma \ref{lemm_density}]
The fact that $\rho_{1c}$ has compact support follows from Lemma \ref{lemm_rho}; $\rho_{1c}$ being integrable follows from Lemma \ref{lemm_edge0}. Note that in proving Lemmas \ref{lemm_rho} and \ref{lemm_edge0}, we do not make the regularity assumptions in Definition \ref{def_regular}. It remains to show that for fixed $w\in \mathbb C_+$ and $|z|\ne 1$, there exists a unique $m_{1c}(w) \in \mathbb C_+$ satisfying equation (\ref{eq_self3}). This follows from the $\eta\sim 1$ case in the proof of {\it Case 1} in this section.
\end{proof}

\noindent{\it{Remark:}} The estimate (\ref{estimate2_bulk}) has been used repeatedly during the proof of Lemma \ref{lemm_stability}.
Here we remark that it also gives the stability of the regularity conditions in Definition \ref{def_regular} under perturbations of $|z|$ and $\rho_{\Sigma}$. For example, we define the shifted empirical spectral density
\begin{equation}
\rho_{\Sigma, t} : = \frac{1}{N\wedge M}\sum_{i=1}^{N\wedge M} \delta_{\sigma_i + t},
\end{equation}
and the associated $m_c(w,t)$ and function $f(\sqrt{w},m,t)$. Given a regular edge $e_k$, it satisfies that
$$f(\sqrt{e_k},m_k,t=0)=0,\ \ \partial_m f(\sqrt{e_k},m_k,t=0)=0.$$
where we denote $m_k:=m_c(e_k)$. We have the Jacobian
$$ J:=\det\left(\begin{matrix} \partial_{\sqrt{w}} f  &  \partial_m f \\
   \partial_{\sqrt{w}} \partial_m f &  \partial_m^2 f \\
\end{matrix}\right)_{(\sqrt{w},m,t)=(\sqrt{e_k},m_k,0)} = \partial_{\sqrt{w}} f (\sqrt{e_k},m_k,0) \partial_m^2 f(\sqrt{e_k},m_k,0).$$
By (\ref{eqn_partialw}), we have $ \left|\partial_{\sqrt{w}} f (\sqrt{e_k},m_k,0)\right|\ge 1$. Combining with (\ref{regular2}), we get
$|J|\ge \epsilon$. Using (\ref{estimate2_bulk}), we can verify that $\partial_t f(\sqrt{e_k},m_k,0)=O(1)$ and $\partial_t \partial_m f(\sqrt{e_k},m_k,0)=O(1)$. Thus if we regard $e_k$ and $m_k$ as functions of $t$, then $\partial_t m_k(t=0)=O(1)$ and $\partial_t e_k(t=0) = O(1)$ by the implicit function theorem. Then it is easy to verify 
\begin{align*}
& \partial_m^2 f(\sqrt{e_k(t)},m_c(e_k,t)) = \partial_m^2 f(\sqrt{e_k},m_c(e_k))+O(t), \\ 
& \left|m_c(e_k,t)-a_i(e_k,t)\right|= \left|m_c(e_k)-a_i(e_k)\right|+O(t),
\end{align*}
and the similar estimates for $ \left|m_c-b_i \right|$ and $ \left|m_c+c_i \right|$. 
Thus if Definition \ref{def_regular} (i) holds for some $\rho_{\Sigma}$, then it holds for all $\rho_{\Sigma,t}$ provided that $t$ is small enough.

Now given a regular bulk component $[e_{2k}, e_{2k-1}]$ and $E\in [e_{2k}+\tau', e_{2k-1}-\tau']$. Differentiating the equation $f(\sqrt{E},m_c(E,t),t)=0$ in $t$ yields $$\partial_t m_c(E,t)= -\frac{\partial_t f(\sqrt{E},m_c(E,t),t)}{\partial_m f(\sqrt{E},m_c(E,t),t)}.$$
By (\ref{estimate2_bulk}), we find that $\partial_t f(\sqrt{E},m_c(E),0)=O(1)$, while by (\ref{estimate_alphabeta1}), $|\partial_m f(\sqrt{E},m_c(E),0)|=\beta \sim 1$. Thus $\partial_t m_c(E,0) = O(1)$. A simple extension of this argument shows that $m_c(E,t) = m_c(E)+O(t)$ and hence $\Im\, m_c(E,t)$ is bounded from below by some $c'=c'(\tau,\tau')$. Thus we conclude that if Definition \ref{def_regular} (ii) holds for some $\rho_{\Sigma}$, then it holds for all $\rho_{\Sigma,t}$ with $t$ in some fixed small interval around zero. Obviously, the above arguments also work for the perturbation of $|z|$.

\section{Proof of Lemma \ref{fluc_aver}}
\label{appendix3}

Our proof of (\ref{flucaver_ZZ}) is an extension of \cite[Lemma 4.9]{isotropic}, \cite[Lemma 7.3]{local_circular} and \cite[Theorem 4.7]{Semicircle}. Here we only prove the bound for $\|[Z]\|$. The proof for $\|\langle Z\rangle \|$ is exactly the same.
For $i\in \mathcal I_1$, we define $P_i=\mathbb E_{[i]}$ and $Q_i=1-P_i$. Recall that $Z_{[i]}  =  Q_i{G^{-1}_{[ii]} }$, we need to prove that
\begin{equation*}
[Z] =\frac{1}{N} \sum_{i=1}^N \pi_{[i]}\left(Q_i G_{[ii]}^{ - 1}\right)\pi_{[i]} \prec  \left|w\right|^{-1/2} {\Phi _o^2 } ,
\end{equation*}
for $w\in \mathbf D$. For $J\subset \mathcal I$, we define $\pi_{[i]}^{[J]}$ by replacing $m_{1,2}$ in (\ref{def_pi_i}) with $m_{1,2}^{[J]}$ defined in (\ref{def_M2}). As in (\ref{better_estimate_m12}), we can prove that $ | {m_{1,2}^{[i]}  - m_{1,2} } | \prec \left|w\right|^{-1/2} {\Phi _o^2}$, which further gives that
\begin{equation*}
[Z] =\frac{1}{N} \sum_{i=1}^N \pi_{[i]}^{[i]}\left(Q_i G_{[ii]}^{ - 1}\right)\pi_{[i]}^{[i]} + O_\prec\left(\left|w\right|^{-1/2} {\Phi _o^2 }\right)= \frac{1}{N} \sum_{i=1}^N Q_i \left(\pi_{[i]}^{[i]}G_{[ii]}^{ - 1}\pi_{[i]}^{[i]}\right) + O_\prec\left(\left|w\right|^{-1/2} {\Phi _o^2 }\right).
\end{equation*}
Thus if we abbreviate $B_i  := |w|^{1/2}Q_i \left(\pi_{[i]}^{[i]}G_{[ii]}^{ - 1}\pi_{[i]}^{[i]}\right)$, it suffices to prove that $B:=N^{-1}\sum_i B_i \prec \Phi_o^2$.
We estimate $B$ by bounding the $p$-th moment of its norm by $\Phi_o^{2p}$ for $p=2n$ with $n\in \mathbb N$, i.e. $\mathbb E\|B\|^{p} \prec \Phi_o^{2p}.$
The lemma then follows from the Chebyshev's inequality. Using $\|KK^\dag \|=\|K\|^2$ for any square matrix $K$, we get that for $p=2n$,
$$\text{Tr}(BB^\dag)^{n} \ge \left\|BB^\dag\right\|^{n} = \left\|B\right\|^{2n}. $$
Thus it suffices to prove that
\begin{equation}\label{flucaver_proof}
\mathbb E \text{Tr}(BB^\dag)^{p/2} \prec \Phi_o^{2p}, \ \text{ for }p=2n.
\end{equation}
This estimate can be proved with the same method in \cite[Appendix B]{Semicircle}, with the only complication being that $\pi_{[i]}$ is random and depends on $i$. In principle, this can be handle by using (\ref{eq_res3}) and (\ref{eq_res4}) to put any indices $j,k,...\in \mathcal I_1$ (that we wish to include) into the superscripts of $\pi_{[i]}$. This leads to a minor modification of the proof in \cite[Appendix B]{Semicircle}. Here we describe the basic ideas of the proof, without writing down all the details.

The proof is based on a decomposition of the space of random variables using $P_s$ and $Q_s$. It is evident that $P_s$ and $Q_s$ are projections, $P_s+Q_s=1$ and all of these projections commute with each other. For a set $J\subset \mathcal I$, we denote $P_J:=\prod_{s\in J} P_s$ and $Q_J:=\prod_{s\in J} Q_s$. Let $p=2n$ and introduce the shorthand notation $\tilde B_{k_s}:=B_{k_s}$ for $s\le p$ odd and $\tilde B_{k_s}:=B^\dag_{k_s}$ for $s \le p$ even. Then we get
\begin{equation}
\mathbb E\text{Tr}(B B^\dag )^{p/2} = \frac{1}{N^p}\sum\limits_{k_1 ,k_2 , \ldots ,k_p } \mathbb E \text{Tr}\prod\limits_{s = 1}^p {\tilde B_{k_s } } = \frac{1}{{N^p }}\sum\limits_{k_1 ,k_2 , \ldots ,k_p } {\mathbb E \text{Tr} \prod\limits_{s = 1}^p {\left( {\prod\limits_{r = 1}^p {\left( {P_{k_r }  + Q_{k_r } } \right)} \tilde B_{k_s } } \right)}  }.
\end{equation}
Introducing the notations $\mathbf k=(k_1,k_2,\ldots, k_p)$ and $\{\mathbf k\}=\{k_1,k_2,\ldots,k_p\}$, we can write
\begin{equation}
\mathbb E \text{Tr}(B B^\dag )^{p/2} = \frac{1}{{N^p }}\sum\limits_{\mathbf k} {\sum\limits_{I_1 , \ldots ,I_p  \subset \left\{ \mathbf k \right\}} {\mathbb E\text{Tr}\prod\limits_{s = 1}^p {\left( {P_{I_s^c } Q_{I_s} \tilde B_{k_s } } \right)} } }. \label{fluc_pf1}
\end{equation}
Following \cite[Appendix B]{Semicircle}, we claim that to conclude (\ref{flucaver_proof}) it suffices to prove that for $k\in I$
\begin{equation}
\left\|Q_I B_k\right\|\prec \Phi_o^{|I|}. \label{claim}
\end{equation}
As in \cite[Appendix B]{Semicircle}, it is not hard to prove for $k\in I$,
\begin{equation}
|w|^{-1/2}\left\|Q_I G_{[kk]}^{-1}\right\|\prec \Phi_o^{|I|}. \label{claim_assist}
\end{equation}
Now we extend the proof to obtain the estimate (\ref{claim}). For the case $|I|=1$ (i.e. $I=\{k\}$),
$$\| B_k \|= |w|^{1/2} \| \pi_{[i]}^{[i]}Z_{[k]} \pi_{[i]}^{[i]}\| \le |w|^{-1/2} \|Z_{[k]}\| \prec \Phi_o,$$
where we can prove $\|Z_{[k]}\| \prec |w|^{1/2}\Phi_o$ by modifying the proof in Lemma \ref{Z_lemma}.
For the case $|I|\ge 2$, WLOG, we may assume $k=1$ and $I=\{1,\ldots, t\}$ with $t\ge 2$. It is enough to prove that
\begin{equation}\label{fluc_aver_claim}
\left|w\right|^{1/2}\left\| Q_t  \ldots Q_2 \pi_{[1]}^{[1]}G_{[11]}^{ - 1}\pi_{[1]}^{[1]} \right\| \prec \Phi _o^t.
\end{equation}
We take the $t=3$ as an example to describe the ideas for the proof of (\ref{fluc_aver_claim}). Using (\ref{eq_res3}), we get
\begin{equation}\label{expand_pi112}
\pi_{[1]}^{[1]} = \pi_{[1]}^{[12]} + |w|^{1/2} \epsilon_{11}^{[1]} \pi_{[1]}^{[12]}A_1\pi_{[1]}^{[12]} + |w|^{1/2} \epsilon_{\bar 1 \bar 1}^{[1]} \pi_{[1]}^{[12]}A_2\pi_{[1]}^{[12]} + \text{error}_{1,2},
\end{equation}
where $\epsilon_{11}^{[1]}$ and $\epsilon_{\bar 1 \bar 1}^{[1]}$ are entries of 
$$\epsilon_{[1]}^{[1]}:=|w|^{1/2}\left(\frac{G_{[22]}^{[1]}}{N}+\frac{1}{N}\sum_{ k\notin \{1,2\}} {G_{[k2]}^{[1]} \left(G^{[1]}_{[22]}\right)^{-1}G_{[2k]}^{[1]} }\right)\prec \Phi _o^2,$$ $A_{1,2}$ are deterministic matrices with operator norm $O(1)$, and $\|\text{error}_{1,2}\| \prec |w|^{-1/2}\Phi _o^4$. Then we get
\begin{align}
& \pi_{[1]}^{[1]}G_{[11]}^{ - 1}\pi_{[1]}^{[1]} =  \pi_{[1]}^{[12]}G_{[11]}^{-1}\pi_{[1]}^{[12]} + |w|^{1/2} \epsilon_{11}^{[1]} \pi_{[1]}^{[12]}A_1\pi_{[1]}^{[12]} G_{[11]}^{ - 1}\pi_{[1]}^{[12]}  + |w|^{1/2}  \epsilon_{\bar 1 \bar 1}^{[1]} \pi_{[1]}^{[12]}A_2\pi_{[1]}^{[12]} G_{[11]}^{ - 1}\pi_{[1]}^{[12]}  \nonumber\\
&\quad \quad + |w|^{1/2}  \pi_{[1]}^{[12]} G_{[11]}^{ - 1}\epsilon_{11}^{[1]} \pi_{[1]}^{[12]}A_1\pi_{[1]}^{[12]} + |w|^{1/2}  \pi_{[1]}^{[12]} G_{[11]}^{ - 1}\epsilon_{\bar 1 \bar 1}^{[1]} \pi_{[1]}^{[12]}A_2\pi_{[1]}^{[12]} + O_\prec(|w|^{-1/2}\Phi _o^4). \label{expand_case3}
 \end{align}

We first handle the $\pi_{[1]}^{[12]}G_{[11]}^{-1}\pi_{[1]}^{[12]}$ term. By (\ref{claim_assist})
$$Q_2 \pi_{[1]}^{[12]}G_{[11]}^{-1}\pi_{[1]}^{[12]}=\pi_{[1]}^{[12]} \left(Q_2 G_{[11]}^{-1}\right)\pi_{[1]}^{[12]} \prec |w|^{-1/2}\Phi_o^2.$$ 
For the remaining term, we first expand
$\pi_{[1]}^{[12]}=\pi_{[1]}^{[123]}+ O_\prec(|w|^{-1/2}\Phi_o^2)$
and use (\ref{claim_assist}) to get
$$Q_3 Q_2\pi_{[1]}^{[12]}G_{[11]}^{-1}\pi_{[1]}^{[12]} = \pi_{[1]}^{[123]}\left(Q_3 Q_2G_{[11]}^{-1}\right)\pi_{[1]}^{[123]} +O_\prec \left(|w|^{-1/2}\Phi_o^4\right) \prec |w|^{-1/2}\Phi_o^3.$$
Then we deal with the second terms in (\ref{expand_case3}). We first expand $\epsilon_{[1]}^{[1]}= e_{[1]}^{[3]} +O_\prec(\Phi _o^3),$ where
$$e_{[1]}^{[3]}:=|w|^{1/2}\left(\frac{G_{[22]}^{[13]}}{N} +\frac{1}{N}\sum_{ k\notin \{1,2,3\}} {G_{[k2]}^{[13]} \left(G^{[13]}_{[22]}\right)^{-1}G_{[2k]}^{[13]} }\right).$$
Using the similar arguments as above, we have
\begin{align*}
Q_3 |w|^{1/2} e_{11}^{[3]} \pi_{[1]}^{[12]}A_1\pi_{[1]}^{[12]} G_{[11]}^{ - 1}\pi_{[1]}^{[12]} & = |w|^{1/2} e_{11}^{[3]} \pi_{[1]}^{[123]}A_1\pi_{[1]}^{[123]} \left(Q_3 G_{[11]}^{ - 1} \right)\pi_{[1]}^{[123]} +O_\prec (|w|^{-1/2}\Phi_o^4) \\
& \prec |w|^{-1/2}\Phi_o^4.
\end{align*}
Thus we have
$$Q_2Q_3 |w|^{1/2} \epsilon_{11}^{[1]} \pi_{[1]}^{[12]}A_1\pi_{[1]}^{[12]} G_{[11]}^{ - 1}\pi_{[1]}^{[12]} \prec  |w|^{-1/2}\Phi_o^3. $$
Obviously this estimate works for the rest of the terms in (\ref{expand_case3}). This proves (\ref{fluc_aver_claim}) when $t=3$.

We can continue in this manner for a general $t$. At the $l$-th step, we expand the leading order terms using (\ref{eq_res3}) and (\ref{eq_res4}), and after applying $Q_l \ldots Q_3 Q_2$ on them, the number of $\Phi_o$ factors increases by one at each step by (\ref{claim_assist}). Trough induction we can prove (\ref{fluc_aver_claim}). In fact the expansions can be performed in a systematic way using the method in \cite[Appendix B]{Semicircle}, and we leave the details to the reader. Also
we remark that similar techniques are used in the proof in Section \ref{section_isotropiclaw}, and we choose to present the details there (in fact the proof here is much easier than the one in Section \ref{section_isotropiclaw}).

\end{appendix}


\begin{thebibliography}{10}

\bibitem{Bai1997}
Z.~D. Bai.
\newblock Circular law.
\newblock {\em Ann. Probab.}, 25(1):494--529, 1997.

\bibitem{Bai2006}
Z.~D. Bai and J.~W. Silverstein.
\newblock {\em Spectral Analysis of Large Dimensional Random Matrices},
  volume~2 of {\em Mathematics Monograph Series}.
\newblock Science Press, Beijing, 2006.

\bibitem{Regularity2}
Z.~Bao, G.~Pan, and W.~Zhou.
\newblock Universality for the largest eigenvalue of sample covariance matrices
  with general population.
\newblock {\em Ann. Statist.}, 43(1):382--421, 2015.

\bibitem{isotropic}
A.~Bloemendal, L.~Erd{\H o}s, A.~Knowles, H.-T. Yau, and J.~Yin.
\newblock Isotropic local laws for sample covariance and generalized {W}igner
  matrices.
\newblock {\em Electron. J. Probab.}, 19(33):1--53, 2014.

\bibitem{principal}
A.~Bloemendal, A.~Knowles, H.-T. Yau, and J.~Yin.
\newblock On the principal components of sample covariance matrices.
\newblock {\em Prob. Theor. Rel. Fields}, 164(1):459--552, 2016.

\bibitem{Borodin2009}
A.~Borodin and C.~D. Sinclair.
\newblock The {G}inibre ensemble of real random matrices and its scaling
  limits.
\newblock {\em Commun. Math. Phys.}, 291(1):177--224, 2009.

\bibitem{local_circular}
P.~Bourgade, H.-T. Yau, and J.~Yin.
\newblock Local circular law for random matrices.
\newblock {\em Probab. Theory Relat. Fields}, 159:545--595, 2014.

\bibitem{local_circularII}
P.~Bourgade, H.-T. Yau, and J.~Yin.
\newblock The local circular law {II}: the edge case.
\newblock {\em Probab. Theory Relat. Fields}, 159(3):619--660, 2014.

\bibitem{Handbook_DS}
K.~R. Davidson and S.~J. Szarek.
\newblock Local operator theory, random matrices and banach spaces.
\newblock volume~1 of {\em Handbook of the Geometry of Banach Spaces}, pages
  317 -- 366. North-Holland, Amsterdam, 2001.

\bibitem{Edelman1997}
A.~Edelman.
\newblock The probability that a random real gaussian matrix has {$k$} real
  eigenvalues, related distributions, and the circular law.
\newblock {\em J. Multivar. Anal.}, 60(2):203 -- 232, 1997.

\bibitem{Regularity1}
N.~El~Karoui.
\newblock Tracy-widom limit for the largest eigenvalue of a large class of
  complex sample covariance matrices.
\newblock {\em Ann. Probab.}, 35(2):663--714, 2007.

\bibitem{Average_fluc}
L.~Erd{\H o}s, A.~Knowles, and H.-T. Yau.
\newblock Averaging fluctuations in resolvents of random band matrices.
\newblock {\em Ann. Henri Poincar\'e}, 14:1837--1926, 2013.

\bibitem{Delocal}
L.~Erd{\H o}s, A.~Knowles, H.-T. Yau, and J.~Yin.
\newblock Delocalization and diffusion profile for random band matrices.
\newblock {\em Commun. Math. Phys.}, 323:367--416, 2013.

\bibitem{Semicircle}
L.~Erd{\H o}s, A.~Knowles, H.-T. Yau, and J.~Yin.
\newblock The local semicircle law for a general class of random matrices.
\newblock {\em Electron. J. Probab.}, 18:1--58, 2013.

\bibitem{EdgraphI}
L.~Erd{\H o}s, A.~Knowles, H.-T. Yau, and J.~Yin.
\newblock Spectral statistics of {E}rd{\H o}s-{R}{\' e}nyi graphs {I}: Local
  semicircle law.
\newblock {\em Ann. Probab.}, 41(3B):2279--2375, 2013.

\bibitem{local_law}
L.~Erd{\H{o}}s, B.~Schlein, and H.-T. Yau.
\newblock Local semicircle law and complete delocalization for {W}igner random
  matrices.
\newblock {\em Commun. Math. Phys.}, 287(2):641--655, 2008.

\bibitem{Bulk_univ}
L.~Erd{\H o}s, H.-T. Yau, and J.~Yin.
\newblock Bulk universality for generalized {W}igner matrices.
\newblock {\em Probab. Theory Relat. Fields}, 154(1):341--407, 2012.

\bibitem{Forrester2007}
P.~J. Forrester and T.~Nagao.
\newblock Eigenvalue statistics of the real {G}inibre ensemble.
\newblock {\em Phys. Rev. Lett.}, 99:050603, 2007.

\bibitem{Ginibre}
J.~Ginibre.
\newblock Statistical ensembles of complex, quaternion, and real matrices.
\newblock {\em J. Math. Phys.}, 6(3):440--449, 1965.

\bibitem{Girko}
V.~Girko.
\newblock The circular law.
\newblock {\em Russ. Teor. Veroyatnost. i Primenen.}, 29(4):669--679, 1984.

\bibitem{gotze2010}
F.~G{\H o}tze and A.~Tikhomirov.
\newblock The circular law for random matrices.
\newblock {\em Ann. Probab.}, 38(4):1444--1491, 2010.

\bibitem{Single_ring}
A.~Guionnet, M.~Krishnapur, and O.~Zeitouni.
\newblock The single ring theorem.
\newblock {\em Ann. Math.}, 174(2):1189--1217, 2011.

\bibitem{Regularity5}
W.~Hachem, A.~Hardy, and J.~Najim.
\newblock Large complex correlated {W}ishart matrices: Fluctuations and
  asymptotic independence at the edges.
\newblock {\em arXiv:1409.7548}.

\bibitem{Anisotropic}
A.~Knowles and J.~Yin.
\newblock Anisotropic local laws for random matrices.
\newblock {\em arXiv:1410.3516}.

\bibitem{isotropic_deform}
A.~Knowles and J.~Yin.
\newblock The isotropic semicircle law and deformation of {W}igner matrices.
\newblock {\em Comm. Pure Appl. Math.}, 66(11):1663--1749, 2013.

\bibitem{Regularity3}
J.~O. Lee and K.~Schnelli.
\newblock Tracy-widom distribution for the largest eigenvalue of real sample
  covariance matrices with general population.
\newblock {\em arXiv:1409.4979}.

\bibitem{Random_polytopes}
A.~Litvak, A.~Pajor, M.~Rudelson, and N.~Tomczak-Jaegermann.
\newblock Smallest singular value of random matrices and geometry of random
  polytopes.
\newblock {\em Adv. Math.}, 195(2):491 -- 523, 2005.

\bibitem{Mehta}
M.~L. Mehta.
\newblock {\em Random matrices}, volume 142 of {\em Pure and Applied
  Mathematics}.
\newblock Elsevier, Amsterdam, 3 edition, 2004.

\bibitem{Regularity4}
A.~Onatski.
\newblock The tracy-widom limit for the largest eigenvalues of singular complex
  wishart matrices.
\newblock {\em Ann. Appl. Probab.}, 18(2):470--490, 2008.

\bibitem{PanZhou_circular}
G.~Pan and W.~Zhou.
\newblock Circular law, extreme singular values and potential theory.
\newblock {\em J. Multivar. Anal.}, 101(3):645--656, 2010.

\bibitem{Rud_Annal}
M.~Rudelson.
\newblock Invertibility of random matrices: norm of the inverse.
\newblock {\em Ann. Math.}, 168(2):575--600, 2008.

\bibitem{RudVersh_square}
M.~Rudelson and R.~Vershynin.
\newblock The littlewood-offord problem and invertibility of random matrices.
\newblock {\em Adv. Math.}, 218:600--633, 2008.

\bibitem{RudVersh_rect}
M.~Rudelson and R.~Vershynin.
\newblock The smallest singular value of a random rectangular matrix.
\newblock {\em Comm. Pure Appl. Math.}, 62:1707--1739, 2009.

\bibitem{RudVersh_smallball}
M.~Rudelson and R.~Vershynin.
\newblock Small ball probabilities for linear images of high-dimensional
  distributions.
\newblock {\em Int. Math. Res. Notices}, 2015(19):9594--9617, 2015.

\bibitem{Sinclair2007}
C.~D. Sinclair.
\newblock Averages over {G}inibre's ensemble of random real matrices.
\newblock {\em Int. Math. Res. Not.}, 2007:1--15.

\bibitem{TaoVu_circular}
T.~Tao and V.~Vu.
\newblock Random matrices: the circular law.
\newblock {\em Commun. Contemp. Math.}, 10(2):261--307, 2008.

\bibitem{TaoVu_local}
T.~Tao and V.~Vu.
\newblock Random matrices: {U}niversality of local spectral statistics of
  non-{H}ermitian matrices.
\newblock {\em Ann. Probab.}, 43(2):782--874, 2015.

\bibitem{tao2010}
T.~Tao, V.~Vu, and M.~Krishnapur.
\newblock Random matrices: Universality of {ESD}s and the circular law.
\newblock {\em Ann. Probab.}, 38(5):2023--2065, 2010.

\bibitem{local_circularIII}
J.~Yin.
\newblock The local circular law {III}: general case.
\newblock {\em Probab. Theory Relat. Fields}, 160(3):679--732, 2014.

\end{thebibliography}

\end{document}